\documentclass[11pt]{amsart}

\pdfoutput=1 % added for arXiv submission
\usepackage[utf8]{inputenc} % Required for inputting international characters
\usepackage[T1]{fontenc} % Output font encoding for international characters
\usepackage{tabto}
\usepackage{multicol}

\title[Odd annular Bar-Natan category and $\mathfrak{gl}(1|1)$]{Odd annular Bar-Natan category and $\boldsymbol{\mathfrak{gl}(1|1)}$} 

\author{Casey Necheles}
\address{Syracuse University; Department of Mathematics; 215 Carnegie; Syracuse, NY 13244}
\email{clnechel@syr.edu}

\author{Stephan Wehrli}
\thanks{SW was partially supported by a grant from the Simons Foundation (\#632059 Stephan Wehrli)}
\address{Syracuse University; Department of Mathematics; 215 Carnegie; Syracuse, NY 13244}
\email{smwehrli@syr.edu}

\usepackage{amsmath,amssymb,amsthm,amsfonts,amstext}

\usepackage{float}

\usepackage{mathtools}
\usepackage[percent]{overpic}
\usepackage{subfiles}
\usepackage{subcaption}

\usepackage[export]{adjustbox} % provides valign=c option in includegraphics and that makes vertical alignment easier.

\usepackage{array}
\usepackage{bbm} % provides command \mathbbm{1}
\usepackage{calc} % provides \widthof command

\usepackage{mleftright} % provides brackets for matrices
\usepackage{tikz}
\usepackage{tikz-cd}
\usetikzlibrary{braids} % to draw braids

\usepackage{bm} % bold math symbols
\usepackage{mathtools}
\usepackage{stmaryrd}
\usepackage{accents}
\usepackage{thm-restate}

\usepackage{graphicx}
\usepackage{pdfpages}

\theoremstyle{plain}
\newtheorem{theorem}{Theorem}
\newtheorem{lemma}[theorem]{Lemma}
\newtheorem{proposition}[theorem]{Proposition}
\newtheorem{corollary}[theorem]{Corollary}

\theoremstyle{definition}
\newtheorem{definition}[theorem]{Definition}
\newtheorem{remark}[theorem]{Remark}
\newtheorem{example}[theorem]{Example}

\theoremstyle{plain} % For theorems numbered by letters
\newtheorem{itheorem}{Theorem}

\newtheorem{icorollary}[itheorem]{Corollary}

\numberwithin{equation}{section} % Causes equations to be numbered by section

\DeclareMathOperator{\rlunion}{\raisebox{1pt}{\incg{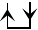}}}
\DeclareMathOperator{\lrunion}{\raisebox{1pt}{\incg{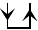}}}

\DeclareMathOperator{\Hom}{Hom}
\newcommand{\BNA}{\mathcal{BN}_{\!o}(\mathbb{A})}
\newcommand{\sBNA}{s\mathcal{BN}_{\!o}(\mathbb{A})}
\newcommand{\OBNA}{\mathcal{OBN}_{\!o}(\mathbb{A})}
\newcommand{\OBNR}{\mathcal{OBN}_{\!o}(\mathbb{R}^2)}
\newcommand{\UBNA}{\mathcal{BN}_{\!\pi}(\mathbb{A})}
\newcommand{\BBNA}{\mathcal{BBN}_{\!o}(\mathbb{A})}
\newcommand{\OBBNA}{\mathcal{OBBN}_{\!o}(\mathbb{A})}
\newcommand{\TL}{\mathcal{T\!L}_{o,\bullet}}
\newcommand{\TLor}{\mathcal{T\!L}_o^{or}}

\DeclareMathOperator{\M}{\mathcal{M}}
\DeclareMathOperator{\ann}{\mathbb{A}}
\DeclareMathOperator{\RG}{\mathcal{R}\!\mathit{g}_{\bullet}}
\DeclareMathOperator{\calO}{\mathcal{O}}
\DeclareMathOperator{\calP}{\mathcal{P}}
\DeclareMathOperator{\calF}{\mathcal{F}}
\DeclareMathOperator{\calG}{\mathcal{G}}

\newcommand{\incg}[1]{\includegraphics[valign=c]{Images/#1}}

\allowdisplaybreaks %Allows pagebreaks within equations

\begin{document}

\bibliographystyle{plain}

\begin{abstract}
We introduce two monoidal supercategories: the odd dotted Temperley-Lieb category $\TL(\delta)$, which is a generalization of the odd Temperley-Lieb category studied by Brundan and Ellis in \cite{BrundanEllis}, and the odd annular Bar-Natan category $\BNA$, which generalizes the odd Bar-Natan category studied by Putyra in \cite{PutyraChrono}.  We then show there is an equivalence of categories between them if $\delta=0$.  We use this equivalence to better understand the action of the Lie superalgebra $\mathfrak{gl}(1|1)$ on the odd Khovanov homology of a knot in a thickened annulus found by Grigsby and the second author in \cite{GL11Grigsby}. 
\end{abstract}

\maketitle

\tableofcontents

\section{Introduction}
    %\documentclass[../main.tex]{subfiles}
%\begin{document}

Around the year 2000, Khovanov introduced a new link invariant now known as \textit{Khovanov homology}. This invariant categorifies the Jones polynomial, in the sense that the Jones polynomial can be recovered as its graded Euler characteristic. However, Khovanov homology is strictly stronger than the Jones polynomial, and the homotopy type of Khovanov's complex is itself a link invariant. In addition, Khovanov homology extends to a functor on the category of links and smooth link cobordisms in $\mathbb{R}^4$.

After defining Khovnaov homology on knots with diagrams in the plane, it is natural to try to extend the homology theory to knots in general thickened surfaces. Such an extension was found by Asaeda, Przytycki, and Sikora in 2003 \cite{APS}. In the case where the surface in question is an annulus, the resulting homology theory is known as \textit{annular Khovanov homology}.

In \cite{SchurWeyl}, Grigsby, Licata, and Wehrli showed there exists a natural action of the Lie algebra $\mathfrak{sl}_2$ on the annular Khovanov homology of a knot $K$ embedded in the thickened annulus $\ann\times I$.  Separately, Bar-Natan \cite{BN} found an alternative way of calculating Khovanov homology, by first defining what is known as the Bar-Natan category $\mathcal{BN}(F)$. Objects of $\mathcal{BN}(F)$ are closed 1-manifolds embedded in a surface $F$, and morphsims are formal linear combinations of dotted cobordisms embedded in $F\times I$, modulo relations. Back in the annular setting, Russell \cite{Russell} showed implicitly that the additive closure of $\mathcal{BN}(\ann)$ is equivalent to the additive closure of a dotted version of the Temperley-Lieb category at $\delta=2$.

Khovanov's original theory is now known as {\it even Khovanov homology}.  A newer theory, known as {\it odd Khovanov homology}, was described by Ozsv\'{a}th, Rasmussen, and Szab\'{o} in \cite{ORS}.  Whereas even Khovanov homology implicitly uses a truncated symmetric algebra in its construction, this new odd Khovanov homology uses an exterior algebra.  The two theories are equivalent over $\mathbb{Z}_2$-coefficients, but differ in characteristic not equal to 2.  Analogous to the $\mathfrak{sl}_2$ action in the even setting, Grigsby and Wehrli found there is an action of the Lie superalgebra $\mathfrak{gl}(1|1)$ on the odd Khovanov homology of a knot in a thickened annulus \cite{GL11Grigsby}.  

\subsection{Background Motivation}
  In the even setting, we have a good understanding of \textit{why} there is an $\mathfrak{sl}_2$ action on a knot in the thickened annulus.  We can visualize this understanding in the following diagram which is explained below:

\[\begin{tikzcd}
{\mathcal{T\!L}_{e,\,\bullet}(2)^{\oplus}} \arrow[r, "\mathbbm{S}^1\times(-)"] \arrow[d] & \mathcal{BN}(\ann)^{\oplus} \arrow[d, "\mathcal{F}^{AKh}"]                           \\
\mathcal{R}ep(\mathfrak{sl}_2) \arrow[r] & \mathcal{M}od(\mathbbm{k})                                     
\end{tikzcd}
\]

The objects and functors in this diagram are attributed to several different people.  In the top right we have $\mathcal{BN}(\ann)^\oplus$, the additive closure of the  annular Bar-Natan category, in which objects are disjoint unions of circles and morphisms are formal linear combinations of cobordisms \cite[Subs. 11.6]{BN}.  The annular Khovanov functor $\mathcal{F}^{AKh}$, defined in~\cite[Subs. 4.2]{SchurWeyl}, sends objects and morphisms in $\mathcal{BN}(\ann)^\oplus$ to certain vector spaces and linear maps in $\mathcal{M}od(\mathbbm{k})$, and there is a known understanding of how $\mathcal{F}^{AKh}$ factors through $\mathcal{R}\mathit{ep}(\mathfrak{sl}_2)$.  On the top of the diagram we have the functor that sends the dotted Temperley-Lieb category, $\mathcal{T\!L}_{e,\,\bullet}(2)$, to the additive closure of $\mathcal{BN}(\ann)$. Using Russell's results \cite{Russell}, one can use this functor to show the two categories are equivalent.  The undotted Temperley-Lieb category, which is a quotient of the dotted version, is known to embed in $\mathcal{R}ep(\mathfrak{sl}_2$) \cite{Thys}, thereby giving us an intrinsic understanding of the action on the even Khovanov homology of a knot in a thickened annulus.

\subsection{Main Results}

One of the goals of this paper is to give a similar explanation in the odd case as to why a $\mathfrak{gl}(1|1)$ action exists on the odd Khovanov homology of a knot in a thickened annulus. More specifically, we set out to define categories and functors analogous to those in the diagram above that would explain the  $\mathfrak{gl}(1|1)$-action described in \cite{GL11Grigsby}.  We first defined two categories, an odd dotted Temperley-Lieb category $\TL(\delta)$ that generalizes the odd Temperley-Lieb supercategory from \cite{BrundanEllis}, and an odd dotted annular Bar-Natan category $\BNA$ that generalizes the odd dotted (non-annular) Bar-Natan category from \cite{PutyraChrono}, neither of which has, to our knowledge, appeared in the literature.  
 
Our main theorem is 
 \begin{itheorem}\label{thm:main}
There exists a superfunctor
\[
\mathcal{I}\colon\mathcal{T\!L}_{o,\bullet}(0)\longrightarrow\BNA
\]
which is a well-defined fully faithful embedding of monoidal supercategories, and which induces an equivalence between the additive closures of the supergraded extensions of the involved monoidal supercategories.
\end{itheorem}

The proofs that appear in this paper are not related to Russell's proofs of the analogous statement in the even setting \cite{Russell}.  Though it seemed the obvious path forward, Russell's proofs do not translate well to the odd setting.  On the other hand, the proofs in this paper do translate to the even setting and can be used to prove analogous statements for a unified theory that generalizes the even and odd theories.
 
The equivalence of $\TL(0)$ and $\BNA$ gives us the following corollary:
 
 \begin{icorollary} \label{cor:eqMonoidalCategories}
There is an equivalence of monoidal supercategories
\[
\overline{\mathcal{I}}\colon%
\mathcal{T\!L}_o(0)^{\oplus^s}\longrightarrow\mathcal{BBN}_{\!o}(\mathbb{A})^{\oplus^s}
\]
where $\mathcal{T\!L}_o(0)$ denotes the odd Temperley-Lieb supercategory at $\delta=0$, $\mathcal{BBN}_o(\mathbb{A})$ denotes the quotient of $\mathcal{BN}_o(\mathbb{A})$ by Boerner's (NDD) relation \cite{Boerner}, and $\mathcal{C}^{\oplus^s}$ denotes the additive closure of the supergraded extension of the monoidal supercategory $\mathcal{C}$.
\end{icorollary}

This gives us the following diagram which, after proving Theorem~\ref{thm:commute}, answers the motivating question of why there exists an action of the Lie superalgebra $\mathfrak{gl}(1|1)$ on the odd Khovanov homology of a knot in a thickened annulus.

\begin{equation}\label{eqn:commute}
\begin{tikzcd}[column sep=normal,row sep=large]
\mathcal{T\!L}_o(0)^{\oplus^s}\ar[r,"\overline{\mathcal{I}}"]\ar[d,"\mathcal{R}"']&
\BBNA^{\oplus^s}\ar[d,"\mathcal{F}_{\!o}^{AKh}"]\\
\mathcal{R}\!\mathit{ep}(\mathfrak{gl}(1|1))\ar[r]&\mathcal{SM}\mathit{od}(\Bbbk)
\end{tikzcd}
\end{equation}

\begin{restatable}{itheorem}{TheoremC}
\label{thm:commute}
Diagram~\ref{eqn:commute} commutes up to even supernatural isomorphism.
\end{restatable}

We end this introduction with the observation that the equivalence $\mathcal{I}$ from Theorem~\ref{thm:main} allows us to interpret the odd annular Khovanov bracket of an annular link diagram as a chain complex in the appropriate extension of $\mathcal{T\!L}_{o,\bullet}(0)$.

\subsection{Organization}

In Section~\ref{sec:Chapter2} we give background information about monoidal supercategories that will be used in the rest of the paper, including information about braidings, graded extensions, and filtrations.  We close out the section by defining the dotted odd Temperley-Lieb supercategory, $\TL(\delta)$.  In Section~\ref{sec:Chapter3} we define the second of our two main categories, $\BNA$, the odd annular Bar-Natan category.  These two categories can be seen along the top row of the diagram below.
\[
\begin{tikzcd}
    & \mathcal{R}\!\mathit{g} \arrow[ld, "\mathcal{G}"'] &&&\\
\mathcal{T\!L}_{o,\bullet}(0) \arrow[r, "\mathcal{I}"] \arrow[d] & \mathcal{BN}_{\!o}(\mathbb{A}) \arrow[r, "\simeq"] \arrow[d] & \arrow[lu, "\mathcal{F}"'] \mathcal{OBN}_{\!o}(\mathbb{A}) \arrow[r] \arrow[d] & \mathcal{OBN}_{\!o}(\mathbb{R}^2) \arrow[r, "\mathcal{F}^{Kh}_o"]& \mathcal{SM}\mathit{od}(\Bbbk)\\
    \mathcal{T\!L}_{o}(0) \arrow[r, "\overline{\mathcal{I}}"] \arrow[d,"\mathcal{R}"'] & \mathcal{BBN}_{\!o}(\mathbb{A}) \arrow[r, "\simeq"] &  \mathcal{OBBN}_{\!o}(\mathbb{A}) \arrow[d, "\mathcal{F}_o^{AKh}"]\\
\mathcal{R}ep\mleft(\mathfrak{gl}(1|1)\mright) \arrow[rr] && \mathcal{SM}\mathit{od}(\Bbbk)&&
\end{tikzcd}
\]

In the last two subsections of Section~\ref{sec:Chapter3}, %Subsections~\ref{sec:OrderedBNA} and~\ref{sec:AdmissFactorizations}, 
we define an ordering on cobordisms with sufficient flexibility for us to define an inverse of our main functor, $\mathcal{I}$.  In Section~\ref{sec:Chapter4} we define $\mathcal{I}:\TL(0)\rightarrow\BNA$ and prove it is full.

The next two sections are dedicated to showing that $\mathcal{I}$ is faithful.  In Section~\ref{sec:Chapter5} we do this by constructing an explicit left inverse, and in doing so we construct an intermediary category that we call a  \textbf{marked Reeb graph category}, $\RG$.  Showing $\calG\circ\calF$ is a left inverse will complete the proof that $\mathcal{I}$ is an equivalence of categories in a satisfyingly concrete manner.  The only downside is the more than three dozen relations to be checked.

In Section~\ref{sec:Chapter6} we prove that $\mathcal{I}$ is faithful in a more implicit manner, by showing the composition of functors along the top row of the diagram, which we call $\mathcal{J}:\TL(0)\rightarrow \mathcal{SM}\mathit{od}(\Bbbk)$, is faithful. In the final part of this section we define an annular version of the odd Khovanov functor from Putyra's paper \cite{PutyraChrono}, $\mathcal{F}^{AKh}_o:\BNA\rightarrow \mathcal{SM}od(\Bbbk)$, and prove that this functor descends to a functor $\mathcal{F}^{AKh}_o:\mathcal{OBBN}_{\!o}(\ann)\rightarrow\mathcal{SM}od(\Bbbk)$, where $\mathcal{OBBN}_{\!o}(\ann)$ is a quotient of $\mathcal{OBN}_{\!o}(\ann)$.  

Finally, in Sections\ref{sec:Chatper7} and \ref{sec:Chapter8} we discuss the connection to Grigsby and Wehrli's work \cite{GL11Grigsby} and fill in the left hand side of diagram.

\subsection{Acknowledgments}
The authors would like to thank Claudia Miller for many valuable suggestions. They would also like to thank Krzysztof Putyra for interesting discussions and for allowing them to use his macros for creating images. In particular, the definition of the Reeb graph category in Section~\ref{sec:Chapter5} was partly inspired by an idea that was communicated to the second author by Krzysztof Putyra.
%The second author was partially supported by a grant from the Simons Foundation (\#632059 Stephan Wehrli).

%\end{document}

\section{Categorical preliminaries}
    %\documentclass[../main.tex]{subfiles}
%
%\begin{document}
In this section we recall the notion of a monoidal supercategory as defined in~\cite{BrundanEllis}. We also define the supergraded extension and the additive closure of a monoidal supercategory, and we introduce the odd dotted Temperley--Lieb supercategory which generalizes the (undotted) odd Temperley--Lieb supercategory defined in~\cite{BrundanEllis}.

\subsection{Monoidal supercategories}

Throughout this section $\Bbbk$ will be a fixed commutative unital ring. A \textbf{supermodule} over $\Bbbk$ is a $\Bbbk$-module $V$ equipped with a $\mathbb{Z}_2$-grading $V=V_0\oplus V_1$. Given a homogeneous element $v\in V_i$ in a supermodule $V$, we will denote by $|v|:=i$ its $\mathbb{Z}_2$-degree and call it the \textbf{superdegree} of $v$.

We will say that a $\Bbbk$-linear map $f\colon V\rightarrow W$ between two supermodules is \textbf{even} if it preserves the supergrading and \textbf{odd} if it reverses the supergrading. A general $\Bbbk$-linear map $f\colon V\rightarrow W$ between two supermodules can be split uniquely
as a sum of an even part $f_0$ and an odd part $f_1$, and hence the space of $\operatorname{Hom}_\Bbbk(V,W)$, $\Bbbk$-linear maps from $V$ to $W$, is naturally a supermodule with superdegree $0$ elements given by even maps and superdegree $1$ elements given by odd maps. Likewise, the tensor product over $\Bbbk$ of two supermodules $V$ and $W$ is naturally a supermodule with homogeneous components $$(V\otimes W)_0:=(V_0\otimes W_0)\oplus(V_1\otimes W_1) \textnormal{ and } (V\otimes W)_1:=(V_0\otimes W_1)\oplus(V_1\otimes W_0).$$

\begin{definition}
A \textbf{supercategory} over $\Bbbk$ is a category $\mathcal{C}$ whose morphism sets are supermodules over $\Bbbk$ and the composition of morphisms is $\Bbbk$-bilinear. It is further required that $|f\circ g|=|f|+|g|$ if $f$ and $g$ are homogeneous. A \textbf{superfunctor} between two supercategories $\mathcal{C}$ and $\mathcal{D}$ is a functor $\calF\colon\mathcal{C}\rightarrow\mathcal{D}$ for which the assignment $f\mapsto \calF(f)$ restricts to an even $\Bbbk$-linear map on each morphism set.  A \textbf{supernatural transformation} between two superfunctors $\calF,\calG:\mathcal{C}\rightarrow \mathcal{D}$ is a collection of morphisms $\{\mathfrak{n}_x=\mathfrak{n}_{0,x}+\mathfrak{n}_{1,x}\colon \calF(x)\rightarrow \calG(x)\}_{x\in Obj(\mathcal{C})}$ such that for each homogeneous $h\in \Hom_{\mathcal{C}}(x,y)$ the following diagram commutes.
\[\begin{tikzcd}[column sep = huge]
\mathcal{F}(x) \arrow[r, "{\mathfrak{n}_{0,x}+\mathfrak{n}_{1,x}}"] \arrow[d, "\mathcal{F}(h)"'] & \mathcal{G}(x) \arrow[d, "\mathcal{G}(h)"] \\
\mathcal{F}(y) \arrow[r, "{\mathfrak{n}_{0,y}+(-1)^{|h|}\mathfrak{n}_{1,y}}"]                    & \mathcal{G}(y)                            
\end{tikzcd}\]

\end{definition}

Given a supercategory $\mathcal{C}$, one can define a new supercategory $\mathcal{C}\boxtimes\mathcal{C}$, whose objects are ordered pairs of objects in $\mathcal{C}$, and whose morphisms are given by \[\operatorname{Hom}_{\mathcal{C}\boxtimes\mathcal{C}}((x,y),(x',y')):=\operatorname{Hom}_\mathcal{C}(x,x')\otimes\operatorname{Hom}_\mathcal{C}(y,y').\] The composition in $\mathcal{C}\boxtimes\mathcal{C}$ is defined by
\[
(f'\boxtimes g')\circ (f\boxtimes g):=(-1)^{|g'||f|}(f'\circ f)\boxtimes(g'\circ g)  
\]
for homogeneous morphisms $f,g,f',g'$, where the notation $f\boxtimes g$ refers to the tensor product $f\otimes g$ viewed as
a morphism in $\mathcal{C}\boxtimes\mathcal{C}$.

\begin{definition}
A \textbf{monoidal supercategory} is a supercategory $\mathcal{C}$ equipped with a superfunctor
$-\otimes-\colon\mathcal{C}\boxtimes\mathcal{C}\longrightarrow\mathcal{C}$,
an object $e\in\mathcal{C}$, and even supernatural isomorphims
$\mathfrak{a}\colon(-\otimes-)\otimes-\Rightarrow-\otimes(-\otimes-)$, $\mathfrak{r}\colon-\otimes e\Rightarrow-$,
and $\mathfrak{l}\colon e\otimes-\Rightarrow-$, which are required to satisfy Mac Lane's pentagon axiom and the relation $(\mathbbm{1}_x\otimes\mathfrak{l}_y)^{-1}\circ(\mathfrak{r}_x\otimes\mathbbm{1}_y)=\mathfrak{a}_{(x,e,y)}$ for all objects $x,y\in\mathcal{C}$. Here we assume that the categories $(\mathcal{C}\boxtimes\mathcal{C})\boxtimes\mathcal{C}$ and $\mathcal{C}\boxtimes(\mathcal{C}\boxtimes\mathcal{C})$ are identified via the canonical isomorphism. The object $e$ is called the \textbf{supermonoidal unit object}, and the functor $-\otimes-$ is called the \textbf{supermonoidal product} or \textbf{tensor product}. A monoidal supercategory is called \textbf{strict} if $\mathfrak{a},\mathfrak{r},\mathfrak{l}$ are identity transformations.
\end{definition}

In this definition, the requirement that $-\otimes-$ is a functor on $\mathcal{C}\boxtimes\mathcal{C}$ means that $\mathbbm{1}_x\otimes\mathbbm{1}_y=\mathbbm{1}_{x\otimes y}$ for all objects $x,y\in\mathcal{C}$ and that the tensor product interacts with the composition of morphisms via the \textbf{super interchange law}
\begin{equation}\label{eqn:superinterchange}
(f'\otimes g')\circ (f\otimes g)=(-1)^{|g'||f|}(f'\circ f)\otimes(g'\circ g)
\end{equation}
whenever $g'$ and $f$ are homogeneous. Morally speaking, equation~\eqref{eqn:superinterchange} means that moving $g'$ past $f$ comes at the price of multiplying by a minus sign if $|g'|=|f|=1$. The equation above can also be understood graphically if one represents morphisms by string diagrams and $\circ$ and $\otimes$ by the following pictures:
\[
\begin{tikzpicture}
    \node (bottomf) at (0,-1.5) {};
    \node (topf) at (0,1.5) {};
    \node[circle,draw] (f) at (0,0) {${\scriptstyle f}$};
    \node (bottomg) at (1.6,-1.5) {};
    \node (topg) at (1.6,1.5) {};
    \node[circle,draw] (g) at (1.6,0) {${\scriptstyle g}$};
    \node (bottomfg) at (3.6,-1.5) {};
    \node (topfg) at (3.6,1.5) {};
    \node[circle,draw] (fr) at (3.6,0.55) {${\scriptstyle f}$};
    \node[circle,draw] (gr) at (3.6,-0.55) {${\scriptstyle g}$};
    \node () at (0.8,0) {$\circ$};
    \node () at (2.6,0) {$:=$};
    \draw (bottomf) -- (f.south);
    \draw (f.north) -- (topf);
    \draw (bottomg) -- (g.south);
    \draw (g.north) -- (topg);
    \draw (bottomfg) -- (gr.south);
    \draw (gr.north) -- (fr.south);
    \draw (fr.north) -- (topfg);
\end{tikzpicture}\qquad\quad
\begin{tikzpicture}
     \node () at (0,0) {and};
     \node () at (0,1.5) {};
     \node () at (0,-1.5) {};
\end{tikzpicture}\qquad\quad
\begin{tikzpicture}
    \node (bottomf) at (0,-1.5) {};
    \node (topf) at (0,1.5) {};
    \node[circle,draw] (f) at (0,0) {${\scriptstyle f}$};
    \node (bottomg) at (1.6,-1.5) {};
    \node (topg) at (1.6,1.5) {};
    \node[circle,draw] (g) at (1.6,0) {${\scriptstyle g}$};
    \node (bottomfr) at (3.6,-1.5) {};
    \node (topfr) at (3.6,1.5) {};
    \node (bottomgr) at (4.6,-1.5) {};
    \node (topgr) at (4.6,1.5) {};
    \node[circle,draw] (fr) at (3.6,0.55) {${\scriptstyle f}$};
    \node[circle,draw] (gr) at (4.6,-0.55) {${\scriptstyle g}$};
    \node () at (0.8,0) {$\otimes$};
    \node () at (2.6,0) {$:=$};
    \draw (bottomf) -- (f.south);
    \draw (f.north) -- (topf);
    \draw (bottomg) -- (g.south);
    \draw (g.north) -- (topg);
    \draw (bottomfr) -- (fr.south);
    \draw (fr.north) -- (topfr);
    \draw (bottomgr) -- (gr.south);
    \draw (gr.north) -- (topgr);
\end{tikzpicture}
\]
Equation~\eqref{eqn:superinterchange} now follows if one imposes the relation
\[
\begin{tikzpicture}
    \node (bottomfl) at (0,-1.5) {};
    \node (topfl) at (0,1.5) {};
    \node (bottomgl) at (1,-1.5) {};
    \node (topgl) at (1,1.5) {};
    \node (bottomfr) at (4.5,-1.5) {};
    \node (topfr) at (4.5,1.5) {};
    \node (bottomgr) at (5.5,-1.5) {};
    \node (topgr) at (5.5,1.5) {};
    \node () at (2.9,0) {$=\,\,\,\,(-1)^{|f||g|}$};
    \node[circle,draw] (fl) at (0,0.55) {${\scriptstyle f}$};
    \node[circle,draw] (gl) at (1,-0.55) {${\scriptstyle g}$};
    \node[circle,draw] (fr) at (4.5,-0.55) {${\scriptstyle f}$};
    \node[circle,draw] (gr) at (5.5,0.55) {${\scriptstyle g}$};
    \draw (bottomfl) -- (fl.south);
    \draw (fl.north) -- (topfl);
    \draw (bottomgl) -- (gl.south);
    \draw (gl.north) -- (topgl);
    \draw (bottomfr) -- (fr.south);
    \draw (fr.north) -- (topfr);
    \draw (bottomgr) -- (gr.south);
    \draw (gr.north) -- (topgr);
\end{tikzpicture}
\]
for homogeneous morphisms $f$ and $g$.
Back in the formal setting, the latter relation corresponds to the fact that the tensor product in a monoidal supercategory satisfies $f\otimes g=(f\otimes\mathbbm{1}_{y'})\circ(\mathbbm{1}_x\otimes g)
=(-1)^{|f||g|}(\mathbbm{1}_{x'}\otimes g)\circ(f\otimes\mathbbm{1}_y)$
where the first equation holds for arbitrary morphisms $f\colon x\rightarrow x'$ and $g\colon y\rightarrow y'$, and the second equation only holds if $f$ and $g$ are homogeneous. This also implies that the tensor product $f\otimes g$ of two isomorphisms $f$ and $g$ is again an isomorphism with inverse given by
\begin{equation}\label{eqn:inversetensor}
(f\otimes g)^{-1}=(\mathbbm{1}_x\otimes g^{-1})\circ(f^{-1}\otimes\mathbbm{1}_{y'})=(-1)^{|f||g|}(f^{-1}\otimes g^{-1}),
\end{equation}
where the second equation only holds if $f$ and $g$ are homogeneous.
\begin{example}\label{ex:SModk}
Let $\mathcal{SM}\mathit{od}(\Bbbk)$ denote the category whose objects are supermodules over $\Bbbk$ and whose morphisms are $\Bbbk$-linear maps. This category becomes a monoidal supercategory if one defines the tensor product of supermodules as described above, and the tensor product of $\Bbbk$-linear maps by
\[\label{eqn:tensorproductmap}
(f\otimes g)(v\otimes w):=(-1)^{|g||v|}f(v)\otimes g(w).
\]
Note that $(f\otimes g)(v\otimes w)$ can be different from $f(v)\otimes g(w)$ even if $f=\mathbbm{1}$. On the other hand, we always have $(f\otimes g)(v\otimes w)=f(v)\otimes g(w)$ if $g$ is even.
\end{example}

Given two monoidal supercategories $\mathcal{C}$ and $\mathcal{D}$, a \textbf{monoidal superfunctor} from $\mathcal{C}$ to $\mathcal{D}$ is a superfunctor $\calF\colon\mathcal{C}\rightarrow\mathcal{D}$ which intertwines the tensor products in $\mathcal{C}$ and $\mathcal{D}$, and which sends the unit object $e_\mathcal{C}$ to the unit object $e_\mathcal{D}$, up to coherent isomorphisms. More precisely, this means that $\calF$ comes along with an even isomorphism $i\colon e_\mathcal{D}\rightarrow F(e_\mathcal{C})$ and an even supernatural isomorphism $\mathfrak{m}\colon \calF(-)\otimes \calF(-)\rightarrow \calF(-\otimes-)$ such that
\begin{equation}\label{eqn:msfunit}
\mathfrak{m}_{x,e_\mathcal{C}}\circ(\mathbbm{1}_{\calF(x)}\otimes i)=\mathbbm{1}_{\calF(x)}=\mathfrak{m}_{e_\mathcal{C},x}\circ(i\otimes\mathbbm{1}_{\calF(x)})
\end{equation}
and
\begin{equation}\label{eqn:msfassociativity}
\mathfrak{m}_{x\otimes y,z}\circ(\mathfrak{m}_{x,y}\otimes\mathbbm{1}_{\calF(z)})=\mathfrak{m}_{x,y\otimes z}\circ(\mathbbm{1}_{\calF(x)}\circ\mathfrak{m}_{y,z}),
\end{equation}
where here we have assumed that $\mathcal{C}$ and $\mathcal{D}$ are strict.  If the coherence isomorphisms $i$ and $\mathfrak{m}$ are identities, $\calF$ is called a \textbf{strict monodial superfunctor}.

\begin{lemma}
The composition of monoidal superfunctors is a monoidal superfunctor.
\end{lemma}

\begin{proof}  Suppose $\mathcal{F}:\mathcal{A}\to\mathcal{B}$ and $\mathcal{G}:\mathcal{B}\to\mathcal{C}$ are monoidal superfunctors equipped with natural isomorphisms $\{\mathfrak{m}_{x,y}\}_{x,y\in Obj\mathcal{A}}$ and $\{\mathfrak{n}_{r,s}\}_{r,s\in Obj\mathcal{B}}$ that satisfy the definition above.  Suppose also that $k:x\to x'$ and $l:y\to y'$ are homogeneous morphisms in $\mathcal{A}$.  Then consider the following diagram.
\[\begin{tikzcd}[column sep=huge, row sep=large]
\mathcal{G}(\mathcal{F}(x))\otimes\mathcal{G}(\mathcal{F}(y)) \arrow[r, "{\mathfrak{n}_{\mathcal{F}(x),\mathcal{F}(y)}}"] \arrow[d, "\mathcal{G}(\mathcal{F}(k))\otimes\mathcal{G}(\mathcal{F}(l))"'] & \mathcal{G}(\mathcal{F}(x)\otimes\mathcal{F}(y)) \arrow[r, "{\mathcal{G}(\mathfrak{m}_{x,y})}"] \arrow[d, "\mathcal{G}(\mathcal{F}(k)\otimes\mathcal{F}(l))"'] & \mathcal{G}(\mathcal{F}(x\otimes y)) \arrow[d, "\mathcal{G}(\mathcal{F}(k\otimes l))"] \\
\mathcal{G}(\mathcal{F}(x'))\otimes\mathcal{G}(\mathcal{F}(y')) \arrow[r, "{\mathfrak{n}_{\mathcal{F}(x'),\mathcal{F}(y')}}"]                                                                         & \mathcal{G}(\mathcal{F}(x')\otimes\mathcal{F}(y')) \arrow[r, "{\mathcal{G}(\mathfrak{m}_{x',y'})}"]                                                            & \mathcal{G}(\mathcal{F}(x'\otimes y')) 
\end{tikzcd}\]
The left square commutes by definition, and the right square commutes because $\mathcal{G}$ is functorial and the underlying square commutes by definition.  Hence the entire diagram commutes and $\{\mathfrak{t}_{x,y}:=\mathcal{G}(\mathfrak{m}_{x,y})\circ\mathfrak{n}_{\mathcal{F}(x),\mathcal{F}(y)}\}_{x,y\in Obj\mathcal{A}}$ is a natural isomorphism allowing $\calG\circ\calF$ to intertwine the tensor products in $\mathcal{A}$ and $\mathcal{C}$.

Furthermore, there are even isomorphisms $i:e_{\mathcal{B}}\to\calF(e_{\mathcal{A}})$ and $j:e_{\mathcal{C}}\to\calG(e_{\mathcal{B}})$, so $h:=\mathcal{G}(i)\circ j:e_{\mathcal{C}}\to \calG(\calF(e_\mathcal{A}))$ is an isomorphism as well.  It is straight forward but tedious to check that $\mathfrak{t}$ and $h$ thus defined satisfy the two identities \eqref{eqn:msfunit} and \eqref{eqn:msfassociativity}.

\end{proof}

A \textbf{monoidal even supernatural transformation} between two strict monoidal superfunctors $\calF,\calG\colon\mathcal{C}\rightarrow\mathcal{D}$ is an even supernatural transformation $\mathfrak{n}\colon \calF\Rightarrow \calG$ satisfying $\mathfrak{n}_{e_\mathcal{C}}=\mathbbm{1}_{e_\mathcal{D}}$ and $\mathfrak{n}_{x\otimes y}=\mathfrak{n}_x\otimes\mathfrak{n}_y$ for all objects $x,y\in\mathcal{C}$.

By an \textbf{equivalence} between two monoidal supercategories $\mathcal{C}$ and $\mathcal{D}$, we shall mean a monoidal superfunctor $\calF\colon\mathcal{C}\rightarrow\mathcal{D}$ and a monoidal superfunctor $\calG\colon\mathcal{D}\rightarrow\mathcal{C}$ such that $\calG\circ \calF\cong\mathbbm{1}_\mathcal{C}$ and $\calF\circ \calG\cong\mathbbm{1}_\mathcal{D}$ via monoidal even supernatural isomorphisms. For later use, we prove:

\begin{lemma}\label{lem:automaticallymonoidal}
If $\mathcal{C}$ is a strict monoidal supercategory and $\calF\colon\mathcal{C}\rightarrow\mathcal{C}$ is a superfunctor such that $\mathbbm{1}_\mathcal{C}\cong \calF$ via an even supernatural isomorphism $\mathfrak{n}$ satisfying $\mathfrak{n}_{x\otimes y}=\mathfrak{n}_x\otimes\mathfrak{n}_y$ and $\mathfrak{n}_{e_\mathcal{C}}=\mathbbm{1}_{e_\mathcal{C}}$, then $\calF$ is automatically strict monoidal.
\end{lemma}

\begin{proof}
The equations $\mathfrak{n}_{x\otimes y}=\mathfrak{n}_x\otimes\mathfrak{n}_y$ and $\mathfrak{n}_{e_\mathcal{C}}=\mathbbm{1}_{e_\mathcal{C}}$ imply implicitly that $\calF(x\otimes y)=\calF(x)\otimes \calF(y)$ and $\calF(e_\mathcal{C})=e_\mathcal{C}$. Moreover, if $f\colon x\rightarrow x'$ and $g\colon y\rightarrow y'$ are morphisms in $\mathcal{C}$, then
\begin{align*}
\calF(f\otimes g)&\stackrel{(1)}=\mathfrak{n}_{x'\otimes y'}\circ(f\otimes g)\circ\mathfrak{n}_{x\otimes y}^{-1}\\
&\stackrel{(2)}=(\mathfrak{n}_{x'}\otimes\mathfrak{n}_{y'})\circ(f\otimes g)\circ(\mathfrak{n}_x\otimes\mathfrak{n}_y)^{-1}\\
&\stackrel{(3)}=(\mathfrak{n}_{x'}\otimes\mathfrak{n}_{y'})\circ(f\otimes g)\circ(\mathfrak{n}_x^{-1}\otimes\mathfrak{n}_y^{-1})\\
&\stackrel{(4)}=(\mathfrak{n}_{x'}\circ f\circ\mathfrak{n}_x^{-1})\otimes(\mathfrak{n}_{y'}\circ g\circ\mathfrak{n}_y^{-1})\\
&\stackrel{(5)}=F(f)\otimes F(g),
\end{align*}
where we have used that $\mathfrak{n}$ is natural and that $\mathfrak{n}_{x\otimes y}=\mathfrak{n}_x\otimes\mathfrak{n}_y$, and where the third and the fourth equations follow because $\mathfrak{n}$ is even. Thus, $\calF$ is monoidal with $i=\mathbbm{1}_{e_\mathcal{C}}$ and $\mathfrak{m}_{x,y}=\mathbbm{1}_{\calF(x)\otimes \calF(y)}$.
\end{proof}

In the remainder of this paper we will often treat non-strict monoidal supercategories as if they were strict. This is justified because the relevant non-strict monoidal supercategories can easily be replaced by explicit strict ones. More generally, there is a version of Mac Lane's coherence theorem~\cite{MacLane}, which says that every non-strict monoidal supercategory is equivalent to a strict one \cite{BrundanEllis}.

\begin{remark}
Working over coefficients in $\Bbbk[\pi]$, where $\pi$ denotes a formal variable with $\pi^2=1$, one can define a notion of a \textbf{$\pi$-monoidal supercategory} by replacing all factors of $-1$ in this subsection by factors of $\pi$.
\end{remark}

\subsection{Braidings in monoidal supercategories}\label{subs:braidings}

\begin{definition}
A \textbf{braiding} on a strict monoidal supercategory $\mathcal{C}$ is a collection of even isomorphisms $\{\tau_{x,y}:x\otimes y\to y\otimes x\}_{x,y\in Obj\mathcal{C}}$ such that, for any objects $x,y,z\in \mathcal{C}$, the following equations are true, 
\begin{equation}\label{eqn:braiding}
\tau_{x,y\otimes z}=(\mathbbm{1}_y\otimes\tau_{x,z})\circ(\tau_{x,y}\otimes\mathbbm{1}_z),\qquad\tau_{x\otimes y,z}=
(\tau_{x,z}\otimes\mathbbm{1}_y)\circ(\mathbbm{1}_x\otimes\tau_{y,z})
\end{equation}
and such that, for all homogeneous $f\colon x\rightarrow x'$ and $g\colon y\rightarrow y'$, we have
\begin{equation} \label{eqn:braidingmorphisms}
    (g\otimes f)\circ \tau_{x,y}=(-1)^{|f||g|}\tau_{x',y'}\circ (f\otimes g).
\end{equation}
\end{definition}
We can represent copies of $\tau$ by positive braid generators, in which case the equations in~\eqref{eqn:braiding} can be visualized as
\[\begin{tikzpicture}[scale=0.8]
\draw (-8,-3) .. controls (-8,-2) .. (-7,-1.5);% x bottom
\draw (-7,-1.5) .. controls (-6,-1) .. (-6,0); % x top
\draw (-6.5,-3) .. controls (-6.5,-2) .. (-7.1,-1.7); %y bottom
\draw (-7.4,-1.55) .. controls (-8,-1.25) .. (-8,0);% y top
\draw (-6,-3) .. controls (-6,-1.75) .. (-6.6, -1.45);% z bottom
\draw (-6.9,-1.3) .. controls (-7.5,-1) .. (-7.5,0);% z top
\node () at (-8,-3.5) {$x$};
\node () at (-6.25,-3.5) {$y\otimes z$};
\node () at (-5,-1.5) {$=$};
\pic at (-4,0) [braid/.cd,gap=0.1]{braid={s_2^{-1} s_1^{-1} }};
\node () at (-4,-3.5) {$x$};
\node () at (-2.75,-3.5) {$y$};
\node () at (-1.5,-3.5) {$z$};
\node () at (0,-1.5) {and};
\draw (2,-3) .. controls (2,-1.75) .. (2.75,-1.375);%x bottom
\draw (2.75,-1.375) .. controls (3.5,-1) .. (3.5,0);%x top
\draw (2.5,-3) .. controls (2.5, -2) .. (3.25,-1.625);% y bottom
\draw (3.25,-1.625) .. controls (4,-1.25) .. (4,0);%y top
\draw (4,-3) .. controls (4,-2) .. (3.4,-1.7);% z bottom
\draw (3.1,-1.55) -- (2.9,-1.45);% z middle
\draw (2.6,-1.3) .. controls (2,-1) .. (2,0); %z top
\node () at (2.25,-3.5) {$x\otimes y$};
\node () at (4,-3.5) {$z$};
\node () at (5,-1.5) {$=$};
\pic at (6,0) [braid/.cd,gap=0.1]{braid={s_1^{-1} s_2^{-1} }};
\node () at (6,-3.5) {$x$};
\node () at (7.25,-3.5) {$y$};
\node () at (8.5,-3.5) {$z$};
\end{tikzpicture}\]
and \eqref{eqn:braidingmorphisms} is visualized as
\[\begin{tikzpicture}
\pic at (-3,0) [braid/.cd,gap=0.1]{braid={s_1^{-1} }};
    \node[circle,draw] (gl) at (-3,1.5) {${\scriptstyle g}$};
    \node[circle,draw] (fl) at (-2,0.5) {${\scriptstyle f}$};
    \draw (-3,0) -- (gl.south);
    \draw (gl.north) -- (-3,2);
    \draw (-2,0) -- (fl.south);
    \draw (fl.north) -- (-2,2);
\node () at (0,0) {$=\,\,\,\,(-1)^{|f||g|}$};
\pic at (2,2) [braid/.cd,gap=0.1]{braid={s_1^{-1} }};
    \node[circle,draw] (fr) at (2,0) {${\scriptstyle f}$};
    \node[circle,draw] (gr) at (3,-1) {${\scriptstyle g}$};
    \draw (2,0.5) -- (fr.north);
    \draw (fr.south) -- (2,-1.5);
    \draw (3,0.5) -- (gr.north);
    \draw (gr.south) -- (3,-1.5);
\end{tikzpicture}\]
where the diagrams are to be read from bottom up. If we apply \eqref{eqn:braidingmorphisms} to $f=\mathbbm{1}_x$ and $g=\tau_{y,z}$, we get
$(\tau_{y,z}\otimes\mathbbm{1}_x)\circ\tau_{x,y\otimes z}=\tau_{x,z\otimes y}\circ(\mathbbm{1}_x\circ\tau_{y,z})$, which together with~\eqref{eqn:braiding} implies: %that $\tau$ satisfies the braid relation
\[
(\tau_{y,z}\otimes\mathbbm{1}_x)\circ(\mathbbm{1}_y\otimes\tau_{x,z})\circ(\tau_{x,y}\otimes\mathbbm{1}_z)=
(\mathbbm{1}_z\otimes\tau_{x,y})\circ(\tau_{x,z}\otimes\mathbbm{1}_y)\circ(\mathbbm{1}_x\otimes\tau_{y,z}).
\]
When visualized, this gives us the familiar braid relation:
\[
\begin{tikzpicture}[scale=0.8]
\pic at (0,0) [scale=0.8,braid/.cd,gap=0.1]{braid={s_1^{-1} s_2^{-1} s_1^{-1}}};
\node () at (3,-2) {$=$};
\pic at (4,0)[scale=0.8,braid/.cd,gap=0.1] {braid={s_2^{-1} s_1^{-1} s_2^{-1}}};
\node () at (0,-4) {$x$};
\node () at (1,-4) {$y$};
\node () at (2,-4) {$z$};
\node () at (4,-4) {$x$};
\node () at (5,-4) {$y$};
\node () at (6,-4) {$z$};
\end{tikzpicture}\vspace*{-0.13in}
\]

A braiding $\tau$ is called \textbf{symmetric} if it satisfies $\tau_{y,x}\circ\tau_{x,y}=\mathbbm{1}_{x\otimes y}$ for all $x,y\in\mathcal{C}$. In this case, one can assign to any collection of objects $x_1,\ldots x_n\in\mathcal{C}$ and any permutation $\sigma\in\mathfrak{S}_n$ a natural \textbf{permutation isomorphism}
\[
\textnormal{R}_{x_1,\ldots,x_n,\sigma}\colon x_1\otimes\ldots\otimes x_n\longrightarrow x_{\sigma^{-1}(1)}\otimes\ldots\otimes x_{\sigma^{-1}(n)}
\]
by writing $\sigma$ as a product of adjacent transpositions and replacing these transpositions by isomorphisms of the form $\mathbbm{1}_{x_{i_1}}\otimes\ldots\otimes\tau_{x_{i_j},x_{i_{j+1}}}\otimes\ldots\otimes\mathbbm{1}_{x_{i_n}}$.

\begin{example}\label{ex:SModkbraiding} One can check that the monoidal supercategory $\mathcal{SM}\mathit{od}(\Bbbk)$ has a symmetric braiding $\tau_{V,W}\colon V\otimes W\rightarrow W\otimes V$ given by $v\otimes w\mapsto(-1)^{|v||w|}w\otimes v$.
\end{example}

\subsection{Supergraded extension}\label{subs:supergradedext}

We can extend any monoidal supercategory $\mathcal{C}$ to a monoidal supercategory $\mathcal{C}^s$ in which objects come with formal grading shifts.  Objects of $\mathcal{C}^s$ are pairs $(x,d)$ where $x$ is an object of $\mathcal{C}$, and $d$ is an element of $\mathbb{Z}_2$, to be viewed as a formal grading shift.  We will use the notation $x_d:=(x,d)$ and we will identify $x$ with $x_0=(x,0)$. Morphisms in $\mathcal{C}^s$ are defined by \[\Hom_{\mathcal{C}^s}(x_d,y_{d'}):=\Hom_\mathcal{C}(x,y),\] and the superdegree, $|\cdot|^s$, of a homogeneous morphism $f\in \Hom_{\mathcal{C}^s}(x_d,y_{d'})$ is defined by $$|f|^s:=|f|+d+d'\in\mathbb{Z}_2$$ where on the right-hand side $|f|$ is the superdegree of $f$ viewed as a morphism in $\mathcal{C}$.  Given a morphism $f\in \Hom_\mathcal{C}(x,y)$, we will also write $f_d^{d'}$ for $f$ viewed as a morphism in $\Hom_{\mathcal{C}^s}(x_d,y_{d'})$. Composition of morphisms in $\mathcal{C}^s$ is defined via the composition in $\mathcal{C}$.

To define the tensor product $\otimes^s$ on objects of $\mathcal{C}^s$ we define $x_d\otimes^s y_{d'}:=(x\otimes y)_{d+d'}$.  On morphisms set
\begin{equation}\label{eqn:supergradedext}
f_a^b\otimes^s g_c^d:=(-1)^{ad+ac+a|g|+d|f|}(f\otimes g)_{a+c}^{b+d}
\end{equation}
where $f$ and $g$ are morphisms in $\mathcal{C}$, and $|f|$ and $|g|$ denote their superdegrees in $\mathcal{C}$.
It is tedious but straightforward to check that $\otimes^s$, defined as above, satisfies the super interchange law \eqref{eqn:superinterchange} and that the unit object $e=e_0$ of $\mathcal{C}$ is also a unit object with respect to $\otimes^s$
(compare with~\cite{BrundanEllis}, where supergraded extensions were introduced under the name \textbf{$\Pi$-envelope}).

\begin{remark}\label{rem:supergradedext}
Although the sign in \eqref{eqn:supergradedext} may seem somewhat random, it is actually completely determined if one requires that $(\mathbbm{1}_x)_0^b\otimes^s(\mathbbm{1}_y)_0^d=(\mathbbm{1}_{x\otimes y})_0^{b+d}$ for all objects $x,y$ in $\mathcal{C}$ and all $b,d\in\mathbb{Z}_2$. Indeed, since $(\mathbbm{1}_x)_a^0$, $(\mathbbm{1}_y)_c^0$, and $(\mathbbm{1}_{x\otimes y})_{a+c}^0$ are the inverses of $(\mathbbm{1}_x)_0^a$, $(\mathbbm{1}_y)_0^c$, and $(\mathbbm{1}_{x\otimes y})_0^{a+c}$, then we must have that
\begin{align*}
    (\mathbbm{1}_x)_1^0\otimes (\mathbbm{1}_y)_1^0 &=\mleft[(\mathbbm{1}_x)_0^1\mright]^{-1}\otimes \mleft[(\mathbbm{1}_y)_0^1\mright]^{-1}\\
    &=-1\mleft[(\mathbbm{1}_x)_0^1 \otimes (\mathbbm{1}_y)_0^1\mright]^{-1}\\
    &=-1\mleft[(\mathbbm{1}_{x\otimes y})_0^1\mright]^{-1}\\
    &=-1(\mathbbm{1}_{x\otimes y})_1^0
\end{align*}
which in turn implies that $(\mathbbm{1}_x)_a^0\otimes^s(\mathbbm{1}_y)_c^0=(-1)^{ac}(\mathbbm{1}_{x\otimes y})_{a+c}^0$. Moreover, we have
\begin{gather*}
f_a^b=(\mathbbm{1}_{x'})_0^b\circ f\circ(\mathbbm{1}_{x})_a^0 \qquad \qquad
g_c^d=(\mathbbm{1}_{y'})_0^d\circ g\circ(\mathbbm{1}_{y})_c^0\\
(f\otimes g)_{a+c}^{b+d}=(\mathbbm{1}_{x'\otimes y'})_0^{b+d}\circ (f\otimes g)\circ(\mathbbm{1}_{x\otimes y})_{a+c}^0
\end{gather*}
and equation \eqref{eqn:supergradedext} can now be deduced from the latter equations and from the super interchange law \eqref{eqn:superinterchange}.
\end{remark}

It is clear that any strict monoidal superfunctor $\mathcal{S}\colon\mathcal{C}\rightarrow\mathcal{D}$ between two strict monoid-al supercategories $\mathcal{C}$ and $\mathcal{D}$ extends to a strict monoidal superfunctor between $\mathcal{C}^s$ and $\mathcal{D}^s$. %For a strict monoidal supercategory $\mathcal{C}$, 
We further have:

\begin{lemma}\label{lem:extension}
Let $\mathcal{D}$ be the category of supermodules $\mathcal{SM}\mathit{od}(\Bbbk)$ or the representation category of a Lie superalgebra. If $\mathcal{C}$ is a strict monoidal supercategory, then any strict monoidal superfunctor $\mathcal{S}\colon\mathcal{C}\rightarrow\mathcal{D}$ extends to a (non-strict) monoidal superfunctor $\mathcal{S}^s\colon\mathcal{C}^s\rightarrow\mathcal{D}$.
\end{lemma}

We will prove this lemma in subsection~\ref{subs:gradingshifts}.

\subsection{$\mathbb{Z}$-gradings and filtrations}\label{subs:Zgradings}

To avoid confusion, we point out that a different (unrelated) notion of a filtered category than what is defined below sometimes appears in the literature (see e.g. \cite{Weibel}).

We will say that a supermodule $V=V_0\oplus V_1$ is \textbf{$\mathbb{Z}$-graded} if both of its homogeneous components are equipped with $\mathbb{Z}$-gradings
\[
V_0=\bigoplus_{j\in\mathbb{Z}}V_{0,j},\qquad V_1=\bigoplus_{j\in\mathbb{Z}}V_{1,j}.
\]
Equivalently, a $\mathbb{Z}$-graded supermodule can be viewed as a $\Bbbk$-module $V$ equipped with a $(\mathbb{Z}_2\oplus\mathbb{Z})$-grading, where the $\mathbb{Z}_2$-part in this grading corresponds to the supergrading. Given a homogeneous element $v\in V_{s,j}$ in a $\mathbb{Z}$-graded supermodule $V$, we will denote its $\mathbb{Z}$-degree by $\deg(v):=j\in\mathbb{Z}$.

Any tensor product of $\mathbb{Z}$-graded supermodules $V$ and $W$ is again $\mathbb{Z}$-graded via $\deg(v\otimes w)=\deg(v)+\deg(w)$ for any homogeneous elements $v,w$. Likewise, the set of linear maps $\operatorname{Hom}(V,W)$ can be equipped with a $\mathbb{Z}$-grading by declaring the $\mathbb{Z}$-degree of a linear map $f\in\operatorname{Hom}(V,W)$ to be $j$ if it satisfies $\deg(f(v))=\deg(v)+j$ for every homogeneous element $v\in V$.
 
By a \textbf{filtration} on a supermodule $V$, we shall mean a nested sequence of submodules
\[
V\supseteq
\ldots\supseteq F_kV\supseteq F_{k-1}V\supseteq\ldots
\]
such that $F_kV=(V_0\cap F_kV)\oplus (V_1\cap F_kV)$ for all $k\in\mathbb{Z}$. For a supermodule $V$ that is $\mathbb{Z}$-graded, we further require that
\[
F_kV=\bigoplus_{\substack{s\in\mathbb{Z}_2\\ j\in\mathbb{Z}}}(V_{s,j}\cap F_kV)
\]
for all $k\in\mathbb{Z}$. We will say that a filtration is \textbf{nonpositive} if $F_0V=V$ and \textbf{exhaustive} if
\[
\bigcup_kF_kV=V\qquad\mbox{and}\qquad \bigcap_kF_kV=\{0\}.
\]
Moreover, we will define the \textbf{filtered degree} of an element $v\in V$ as
\[
\operatorname{fdeg}(v):=\inf\{k\in\mathbb{Z}\,|\,v\in F_kV\}\in\mathbb{Z}\cup\{-\infty\},
\]
so that an element $v\in V$ has filtered degree at most $k$ if and only if it is in $F_kV$. Note that $\operatorname{fdeg}(v)$ is always an integer if the filtration is exhaustive and $v$ is nonzero, and always nonpositive if the filtration is nonpositive.

Given a linear map $f$ between two filtered supermodules $V$ and $W$, we say that $f$ is \textbf{filtered} of \textbf{filtered degree at most $k$} if $f(F_jV)\subseteq F_{j+k}W$ for all $j\in\mathbb{Z}$. In particular, this definition allows us to view $\operatorname{Hom}(V,W)$ as a filtered supermodule, whose submodule $F_k\operatorname{Hom}(V,W)$ is given by all linear maps that are filtered of filtered degree at most $k$. 

An equivalent definition of $F_k\Hom (V,W)$ is as follows: let $D$ denote the direct sum of all $\operatorname{Hom}(F_jV,W/F_{j+k}W)$, for all $j\in\mathbb{Z}$ and let $r:\Hom(V,W)\rightarrow D$ be the map induced by restriction ($r$).  Then
\[
F_k\operatorname{Hom}(V,W):=\ker\bigl(\operatorname{Hom}(V,W)\xrightarrow{r} D\bigl).
\] 

The tensor product $V\otimes W$ of two filtered supermodules is again a filtered supermodule, whose submodule $F_k(V\otimes W)\subseteq V\otimes W$ is defined as the span of all $(F_iV)\otimes (F_jW)$ such that $i+j=k$. The following definitions will be used in the remainder of this paper:

\begin{definition}
A \textbf{$\mathbb{Z}$-graded supercategory} is a supercategory whose morphism sets are equipped with $\mathbb{Z}$-gradings satisfying $\deg(f\circ g)=\deg(f)+\deg(g)$ whenever $f$ and $g$ are homogeneous morphisms. In the monoidal setting, we also require that $\deg(f\otimes g)=\deg(f)+\deg(g)$.
\end{definition}

\begin{definition}\label{def:filteredsupercategory}
A \textbf{filtered supercategory} is a supercategory whose morphism sets are equipped with nonpositive exhaustive filtrations satisfying $\operatorname{fdeg}(f\circ g)\leq\operatorname{fdeg}(f)+\operatorname{fdeg}(g)$ whenever $f$ and $g$ are nonzero morphisms. In the supermonoidal setting, we also require that $\operatorname{fdeg}(f\otimes g)\leq\operatorname{fdeg}(f)+\operatorname{fdeg}(g)$.
\end{definition}

Note that if $x$ is a nonzero object in a filtered supercategory, then its identity morphism, $\mathbbm{1}_x$, necessarily has filtered degree zero. Indeed, this follows because $\operatorname{fdeg}(\mathbbm{1}_x)$ is a nonpositive integer such that $\operatorname{fdeg}(\mathbbm{1}_x)=\operatorname{fdeg}(\mathbbm{1}_x\circ\mathbbm{1}_x)\leq 2\operatorname{fdeg}(\mathbbm{1}_x)$.

\begin{example}\label{ex:modf} Let $\mathcal{SM}\mathit{od}_f(\Bbbk)$ denote the category whose objects are $\mathbb{Z}$-graded supermodules over $\Bbbk$, and whose morphisms are given by linear maps that are non-increasing with respect to the $\mathbb{Z}$-grading. This category becomes a filtered monoidal supercategory if one defines the filtered degree of a morphism $f\colon V\rightarrow W$ to be at most $k\in\mathbb{Z}$ if $\deg(f(v))\leq\deg(v)+k$ for all homogeneous $v\in V\setminus\{0\}$.
\end{example}

If $\mathcal{C}$ is a filtered monoidal supercategory then the collection $F_k\mathcal{C}$ of all morphisms of filtered degree at most $k$ forms an ideal in $\mathcal{C}$.  In other words, $F_k\mathcal{C}$ is closed under composition and under tensor products with morphisms in $\mathcal{C}$. Hence the quotient $\mathcal{C}/F_k\mathcal{C}$ is itself a monoidal supercategory, and there is an obvious quotient functor $\mathcal{C}\rightarrow\mathcal{C}/F_k\mathcal{C}$. In the case where $k=-1$, this functor can be viewed as a projection onto the supercategory $\mathcal{C}/F_{-1}\mathcal{C}$ in which only morphisms of filtered degree zero survive.

\begin{example}\label{ex:modg} Let $\mathcal{SM}\mathit{od}_g(\Bbbk)$ denote the supercategory whose objects are $\mathbb{Z}$-graded supermodules over $\Bbbk$, and whose morphisms are grading-preserving linear maps. If $\mathcal{C}$ denotes the filtered supercategory $\mathcal{SM}\mathit{od}_f(\Bbbk)$ from Example~\ref{ex:modf}, then the projection $\mathcal{C}\rightarrow\mathcal{C}/F_{-1}\mathcal{C}$ can be identified with the obvious projection $\mathcal{SM}\mathit{od}_f(\Bbbk)\rightarrow\mathcal{SM}\mathit{od}_g(\Bbbk)$, which sends a linear map $f=f_0+f_{-1}+f_{-2}+\ldots$ with homogeneous components $f_k$ to its degree zero component $f_0$.
\end{example}

A superfunctor between two filtered supercategories $\mathcal{C}$ and $\mathcal{D}$ will be called \textbf{filtered} if the assignment $f\mapsto \calF(f)$ restricts to an even filtered linear map on each morphism set. Accordingly, an equivalence between two filtered supercategories $\mathcal{C}$ and $\mathcal{D}$ will be called filtered if it is given by filtered superfunctors $\mathcal{C}\rightarrow\mathcal{D}$ and $\mathcal{D}\rightarrow\mathcal{C}$. The following lemma follows immediately from the definitions:

\begin{lemma}
If $\mathcal{C}$ and $\mathcal{D}$ are two filtered supercategories, then every filtered equivalence $\mathcal{C}\rightarrow\mathcal{D}$ induces an equivalence $\mathcal{C}/F_{-1}\mathcal{C}\rightarrow\mathcal{D}/F_{-1}\mathcal{D}$.
\end{lemma}

Given a $\mathbb{Z}$-graded supercategory $\mathcal{C}$, we can define a $\mathbb{Z}$-graded extension $\mathcal{C}^g$ analogous to the supergraded extension $\mathcal{C}^s$ described in the previous subsection. Objects of $\mathcal{C}^g$ are pairs $(x,d)$, where $x$ is an object of $\mathcal{C}$ and $d$ is an integer, to be viewed as a formal grading shift. Morphisms in $\mathcal{C}^g$ are given by
\[
\operatorname{Hom}_{\mathcal{C}^g}((x,d),(y,d')):=\operatorname{Hom}_{\mathcal{C}}(x,y),
\]
and the composition in $\mathcal{C}^g$ is defined via the composition in $\mathcal{C}$. Moreover, the $\mathbb{Z}$-degree of a morphism $f\in\operatorname{Hom}_{\mathcal{C}^g}((x,d),(y,d'))$ is defined by
\[
\deg^g(f):=\deg(f)+d'-d\in\mathbb{Z},
\]
where $\deg(f)$ denotes the $\mathbb{Z}$-degree that $f$ has in $\mathcal{C}$. If $\mathcal{C}$ is a monoidal supercategory, then so is $\mathcal{C}^g$: on objects of $\mathcal{C}^g$, the tensor product is defined by $(x,d)\otimes^g(y,d'):=(x\otimes y,d+d')$, and on morphisms, it is defined to be the same as the tensor product in $\mathcal{C}$.

Note that the constructions of the supergraded extension and the $\mathbb{Z}$-graded extension `commute' with each other. In fact, if $\mathcal{C}$ is a $\mathbb{Z}$-graded supercategory then we can define a $(\mathbb{Z}_2\oplus\mathbb{Z})$-graded extension $\mathcal{C}^{sg}=(\mathcal{C}^s)^g=(\mathcal{C}^g)^s$, whose objects are of the form $(x,d,d')$ for an object $x\in\mathcal{C}$ and $(d,d')\in\mathbb{Z}_2\oplus\mathbb{Z}$. If $\mathcal{C}$ is also filtered then the morphism sets of any of the extensions above of $\mathcal{C}$ inherit filtrations because they can be identified with the morphism sets in $\mathcal{C}$. In particular, the formal shifts of the supergrading or the $\mathbb{Z}$-grading have no bearing on the filtrations on morphism sets or on the filtered degrees of morphisms.

\subsection{Additive closure}\label{subs:additiveclosure}
Recall that a category is called \textbf{preadditive} if its morphism sets are abelian groups, and \textbf{$\Bbbk$-linear} if its morphism sets are $\Bbbk$-modules. In both cases, one requires that the composition of morphisms is bilinear. Note that a preadditive category is the same thing as a $\mathbb{Z}$-linear category, and a $\Bbbk$-linear category becomes a preadditive category if one forgets the scalar multiplication on morphisms. A preadditive category is called \textbf{additive} if it contains a zero object and if it is closed under taking finite direct sums.

Given any preadditive category $\mathcal{C}$, we denote by $\mathcal{C}^\oplus$ its additive closure. Objects of $\mathcal{C}^\oplus$ are finite (possibly empty) sequences $(x_1,\ldots,x_n)$ of objects in $\mathcal{C}$, and a morphism $f\colon(x_1,\ldots,x_n)\rightarrow(y_1,\ldots,y_m)$ is given by an $m\times n$ matrix $[f_{ij}]$ of morphisms $f_{ij}\colon x_j\rightarrow y_i$ in $\mathcal{C}$. The composition of morphisms is modeled on matrix multiplication:
\[
[f_{ik}]\circ[g_{kj}]:=\left[\sum_k(f_{ik}\circ g_{kj})\right],
\]
where on the right-hand side, $f_{ik}\circ g_{kj}$ denotes the composition of $f_{ik}$ and $g_{kj}$ in $\mathcal{C}$. It is clear that $\mathcal{C}^\oplus$ is again preadditive because matrices can be added by adding their entries. In addition, $\mathcal{C}^\oplus$ has a zero object which is given by the empty sequence, and a direct sum operation given by concatenation of sequences.

We can embed $\mathcal{C}$ into $\mathcal{C}^\oplus$ by sending an object $x$ to the sequence $(x)$ and a morphism $f$ to the $1\times 1$ matrix $[f]$. Under this embedding, we can write any element of $\mathcal{C}^\oplus$ as a direct sum
\[
(x_1,\ldots,x_n)=x_1\oplus\ldots\oplus x_n
\]
of objects in $\mathcal{C}$. In particular, this implies that every additive functor from $\mathcal{C}$ to an additive category $\mathcal{A}$ extends to $\mathcal{C}^\oplus$ uniquely up to natural isomorphism. This also shows that if $\mathcal{C}$ is already additive, then $\mathcal{C}^\oplus$ is equivalent to $\mathcal{C}$ as an additive category. Moreover, if $\mathcal{S}\colon\mathcal{C}\rightarrow\mathcal{D}$ is an additive functor between two preadditive categories, then it extends to an additive functor $\mathcal{S}\colon\mathcal{C}^\oplus\rightarrow\mathcal{D}^\oplus$ given by
\begin{equation}\label{eqn:addclosureextension}
\mathcal{S}(x_1\oplus\ldots\oplus x_n):=\mathcal{S}(x_1)\oplus\ldots\oplus\mathcal{S}(x_n)\qquad\mbox{and}\qquad
\mathcal{S}([f_{ij}]):=[\mathcal{S}(f_{ij})].
\end{equation}

In the case where $\mathcal{C}$ is a monoidal supercategory, we can extend the supermonoidal structure to $\mathcal{C}^\oplus$ by defining the tensor products of objects by
\begin{align*}
(x_1,\ldots,x_n)\otimes(y_1,\ldots,y_m):=
(&x_1\otimes y_1,\ldots,x_1\otimes y_m,\\
&x_2\otimes y_1,\ldots,x_2\otimes y_m,\\
&\makebox[\widthof{$\,x_3\otimes y_1$}][c]{\vdots}%
\phantom{,\ldots,}%
\makebox[\widthof{$x_3\otimes y_m$}][c]{\vdots}\\
&x_n\otimes y_1,\ldots,x_n\otimes y_m).
\end{align*}
Given two morphisms $f$ and $g$ in $\mathcal{C}^\oplus$ with matrix entries $f_{ij}\colon x_j\rightarrow x_i'$ and $g_{kl}\colon y_l\rightarrow y_k'$, we define $f\otimes g$ as the morphism with matrix entries $f_{ij}\otimes g_{kl}\colon x_j\otimes y_l\rightarrow x_i'\otimes y_k'$. Finally, we declare a morphism $f$ in $\mathcal{C}^\oplus$ to be homogeneous of superdegree $a\in\mathbb{Z}_2$ if each of its entries $f_{ij}$ is homogeneous of that superdegree. It is easy to see that these definitions make $\mathcal{C}^\oplus$ into a monoidal supercategory over $\Bbbk$ whose unit object is given by the unit object $e=(e)$ of $\mathcal{C}$.

\subsection{Dotted odd Temperley-Lieb supercategory}\label{subs:TL}

To close out this section we will define a monoidal supercategory $\TL(\delta)$ which generalizes the odd Temperley-Lieb supercategory introduced in~\cite{BrundanEllis}, and which is formally generated by dotted flat tangles.

A \textbf{flat tangle} is a compact embedded $1$-manifold $T\subset\mathbb{R}\times I$ whose boundary consists of a (possibly empty) set of bottom endpoints in $\mathbb{R}\times\{0\}$ and a (possibly empty) set of top endpoints in $\mathbb{R}\times\{1\}$. A flat tangle will be called \textbf{dotted} if its interior, $T\setminus\partial T$, is decorated by at most finitely many distinct dots. Moreover, a dotted tangle will be called \textbf{chronological} if the height function $h\colon\mathbb{R}\times I\rightarrow I$ given by projection is a Morse function such that no two dots or critical points occur at the same height (and no dot occurs at the same height as a critical point).

We will identify two chronological flat tangles $T$ and $T'$ if they are related by \textbf{chronological isotopy}. Informally, this means they are isotopic through chronological flat tangles. More formally, it means there is a diffeomorphism $\varphi\colon(\mathbb{R}\times I)\times I\rightarrow\mathbb{R}\times I$ such that
\begin{itemize}
\item
for each $s\in I$, $\varphi_s:=\varphi(-,s)$ is a diffeomorphism of $\mathbb{R}\times I$,
\item
$\varphi_s$ takes each level set $\mathbb{R}\times\{t\}$ of the height function $h$ to a level set $\mathbb{R}\times\{t'\}$,
\item
$\varphi_0=\mathbbm{1}$ and $\varphi_1(T)=T'$.
\end{itemize}

If $T$ and $T'$ contain dots it is further required that $\varphi_1$ carries the dots on $T$ bijectively to the ones on $T'$.

To define the monoidal supercategory $\TL(\delta)$, we now fix a commutative unital ring~$\Bbbk$ and an element $\delta\in\Bbbk$. We further denote by $C(n,m)$ the free $\Bbbk$-module spanned by all chronological dotted flat tangles with $n$ bottom endpoints and $m$ top endpoints.
Objects of $\TL(\delta)$ are then given by nonnegative integers $n\geq 0$, and morphism sets are given by the quotients of the $C(n,m)$ by the following relations:

\begin{align}
&\includegraphics[valign=c]{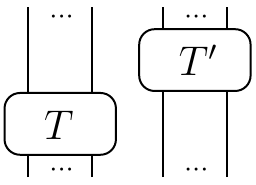}=(-1)^{|T||T'|}\,\,\includegraphics[valign=c]{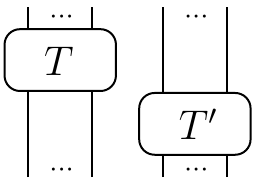}\label{eqn:TLinterchange}\\[0.3cm]
&\includegraphics[valign=c]{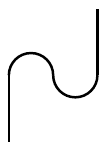} = \includegraphics[valign=c]{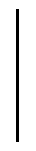} = - \reflectbox{\includegraphics[valign=c]{TLimages/TLisotopyLthenR.pdf}}\label{eqn:TLisotopy}\\[0.3cm]
&\includegraphics[valign=c]{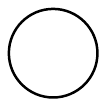}=\delta \,\,\,\,\,\,\,\,\includegraphics[valign=c]{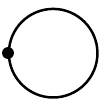}=0\label{eqn:TLcircles}\\[0.3cm]
&\reflectbox{\includegraphics[valign=c,angle=180,origin=c]{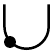}}=\includegraphics[valign=c,angle=180,origin=c]{TLimages/TLcupDotL.pdf}\,\,\,\,\,\,\,\,\includegraphics[valign=c]{TLimages/TLcupDotL.pdf}=\reflectbox{\includegraphics[valign=c]{TLimages/TLcupDotL.pdf}} \label{eqn:TLdotslide}\\[0.3cm]
&\includegraphics[valign=c]{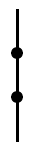}=0\label{eqn:TLtwodots}\\[0.3cm]
&\includegraphics[valign=c]{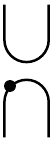}-\reflectbox{\includegraphics[valign=c,angle=180,origin=c]{TLimages/TLrelDotBottomL.pdf}}=\includegraphics[valign=c]{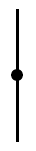}\includegraphics[valign=c]{TLimages/TLidentity.pdf}-\includegraphics[valign=c]{TLimages/TLidentity.pdf}\includegraphics[valign=c]{TLimages/TLidentityDot.pdf}\label{eqn:TLfourterm}
\end{align}
\begin{remark}
Each of the relations above is to be viewed as a local relation: the pictures in each relation only show a portion of the actual flat tangles, while the unshown portions are understood to be the same on both sides of the relation. It is further required that any dots or critical points that may occur in the unshown portions lie either above or below the shown portions. For example, the first relation in~\eqref{eqn:TLcircles} means that if $T$ is a (possibly empty) chronological dotted flat tangle and $\bigcirc$ is a disjoint undotted circle in $\mathbb{R}\times I$ which doesn't occupy the same height as any dot or critical point in $T$, then $T\sqcup\bigcirc=\delta T$.
\end{remark}
\begin{remark} \label{rmk:EvenAndOddTL}
Because the number of endpoints of a flat tangle is even, we must have either $\operatorname{Hom}(m,n)=0$ or  $m\equiv n \mod 2$ for $m,n\in \mathbbm{Z}$.  Thus, any of the tangle categories we discuss in this paper splits into two subcategories, with $\mathcal{T\!L}_\alpha (\delta)=\mathcal{T\!L}_\alpha^{even} (\delta)\oplus\mathcal{T\!L}_\alpha^{odd} (\delta)$, where the $even$ and $odd$ superscripts refer to the parity of the objects that generate each subcategory.  We note in the case of $\TL(\delta)$, $\mathcal{T\!L}_{\textnormal{o},\,\bullet}^{even}(\delta)$ carries a supermonoidal structure, but $\mathcal{T\!L}_{\textnormal{o},\,\bullet}^{odd}(\delta)$ is not closed under taking tensor products.
\end{remark}
The composition of morphisms in $\TL(\delta)$ is induced by vertical stacking of flat tangles, as shown below:
\[
T\circ T'\,:=\,\,\includegraphics[valign=c]{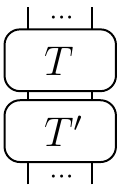}
\]
The supermonoidal product is induced by addition of integers and by horizontal stacking of tangles using the disjoint `right-then-left' union:
\[
T\otimes T'\,:=\,\,\includegraphics[valign=c]{TLimages/TLboxesRthenL.pdf}
\]
As a monoidal supercategory, $\TL(\delta)$ is generated by the following elementary tangles, which are all declared to have superdegree $1$:
\[
\includegraphics[valign=c,angle=180,origin=c]{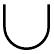}\qquad\quad\includegraphics[valign=c]{TLimages/TLcup.pdf}\qquad\quad\includegraphics[valign=c]{TLimages/TLidentityDot.pdf}
\]
Thus, the \textbf{superdegree} of a chronolocical dotted flat tangle $T$ is given by counting the total number of generating morphisms contained in $T$:
\[
|T|:=\#\{\textnormal{caps, cups, dots}\}\in \mathbb{Z}_2.
\]
It is easy to see that all of the relations above respect the superdegree, and hence the supergrading is well-defined.

In addition to the supergrading, we also define a $\mathbb{Z}$-grading, called the \textbf{quantum grading}, by
\[
q(T):=-2\#\{\textnormal{dots}\}\in \mathbb{Z}.
\]
This definition is also compatible with the relations above. The factor of $-2$ is inspired by the definition of the quantum grading in the odd Bar-Natan category, which we will discuss in subsection~\ref{sec:OddAnnularBN}. 

We also define an \textbf{annular grading} by the same formula:
\[
a(T):=-2\#\{\textnormal{dots}\}\in \mathbb{Z}.
\]
While it may seem redundant to have two identical gradings, the distinction between the quantum grading and the annular grading becomes relevant when we replace $\TL(\delta)$ by its $\mathbb{Z}$-graded extension with respect to the quantum grading: in the latter category, a dotted flat tangle representing a morphism between two objects with formal grading shifts $d,d'\in\mathbb{Z}$ has quantum degree $q(T)=-2\#\{\textnormal{dots}\}+d'-d$, while its annular degree is still given by $a(T)=-2\#\{\textnormal{dots}\}$.

In fact, we will think of the annular grading as corresponding to a filtration, which is given by declaring a morphism to have filtered degree at most $k$ if it can be expressed as a linear combination of dotted chronological flat tangles of annular degrees at most $k$. Taking the quotient of $\mathcal{C}=\mathcal{T\!L}_{o,\bullet}(\delta)$ by $F_{-1}\mathcal{C}$ corresponds to setting any tangle that contains a dot equal to zero, and hence
\[
\mathcal{C}/F_{-1}\mathcal{C}=\mathcal{T\!L}_o(\delta)
\]
is the (undotted) odd Temperley-Lieb supercategory introduced in~\cite{BrundanEllis}.

By replacing all minus signs in relations~\eqref{eqn:TLinterchange}-\eqref{eqn:TLfourterm} by plus signs, one obtains a monoidal category $\mathcal{T\!L}_{e,\bullet}(\delta)$, whose quotient by $F_{-1}\mathcal{C}$ is the usual (even) Temperley-Lieb category $\mathcal{T\!L}_e(\delta)$. One can also define a universal $\pi$-monoidal $\pi$-supercategory $\mathcal{T\!L}_{\pi,\bullet}(\delta)$ that generalizes the even and odd categories, and which is given by working over the ring $\Bbbk[\pi]$ for a formal variable $\pi$ with $\pi^2=1$, and replacing each minus sign in the relations above by a factor of $\pi$.  Note that the terms 'even' and 'odd' here have nothing to do with the parity of generating objects in Remark~\ref{rmk:EvenAndOddTL}.

In the remainder of this paper, we will be most interested in the odd dotted Temperley-Lieb supercategory for $\delta=0$, and in the even category for $\delta=2$. These two categories are generalized by the universal category $\mathcal{T\!L}_{\pi,\bullet}(\delta)$ for $\delta=1+\pi$.
We end this subsection with a lemma that will be used in Proposition~\ref{prop:Gwelldefined} and in Lemma~\ref{lm:generateGC2n}, the proof of which has the added benefit of demonstrating calculations in $\TL(0)$

\begin{lemma} \label{lm:TLidentities}
The following identities holds in $\TL(0)$.
\begin{align}
    &\reflectbox{\includegraphics[angle=180,origin=c,valign=b]{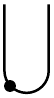}}\includegraphics[valign=b]{TLimages/TLidentity.pdf}\,\, + \,\,\, \includegraphics[angle=180,origin=c,valign=b]{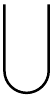}\includegraphics[valign=b]{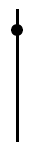}\,\, = \,\, \includegraphics[valign=b]{TLimages/TLidentity.pdf}\reflectbox{\includegraphics[angle=180,origin=c,valign=b]{TLimages/TLcupTallDotL.pdf}} \,\,\,+\,\, \includegraphics[valign=b]{TLimages/TLidentityDotHigh.pdf}\includegraphics[angle=180,origin=c,valign=b]{TLimages/TLcupTall.pdf}\label{eqn:TLdotcapidentity}\\[0.3cm]
    &\includegraphics[valign=t]{TLimages/TLcupTallDotL.pdf}\includegraphics[valign=t]{TLimages/TLidentity.pdf}\,\, + \,\,\, \includegraphics[valign=t]{TLimages/TLcupTall.pdf}\includegraphics[valign=t,angle=180,origin=c]{TLimages/TLidentityDotHigh.pdf} \,\,= \,\,\includegraphics[valign=t]{TLimages/TLidentity.pdf}\includegraphics[valign=t]{TLimages/TLcupTallDotL.pdf}\,\,\, +\,\, \includegraphics[valign=t,angle=180,origin=c]{TLimages/TLidentityDotHigh.pdf}\includegraphics[valign=t]{TLimages/TLcupTall.pdf} \label{eqn:TLdotcupidentity}\\[0.3cm]
    &\includegraphics[valign=t]{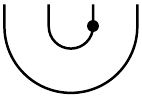}\,\,+\,\,\includegraphics[valign=t]{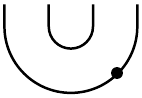}\,\,=\,\,\includegraphics[valign=t]{TLimages/TLcupDotL.pdf}\,\includegraphics[valign=t]{TLimages/TLcupTall}\,\,+\,\,\includegraphics[valign=t]{TLimages/TLcup}\,\includegraphics[valign=t]{TLimages/TLcupTallDotL.pdf}\label{eqn:TLtype1identity}\\[0.3cm]
    &\includegraphics[valign=t]{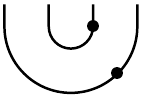}\,\,=\,\,\reflectbox{\includegraphics[valign=t]{TLimages/TLcupDotL.pdf}}\,\,+\,\,\reflectbox{\includegraphics[valign=t]{TLimages/TLcupTallDotL.pdf}}\label{eqn:TLtype2identity}
\end{align}
\end{lemma}
\begin{proof}
We will show the proof for \eqref{eqn:TLdotcapidentity} in detail and outline the proofs for \eqref{eqn:TLtype1identity} and \eqref{eqn:TLtype2identity}.  The proof for \eqref{eqn:TLdotcupidentity} is nearly identical to \eqref{eqn:TLdotcapidentity} and is omitted.
\begin{align*}
    \reflectbox{\includegraphics[angle=180,origin=c,valign=b]{TLimages/TLcupTallDotL.pdf}}\includegraphics[valign=b]{TLimages/TLidentity.pdf} + \includegraphics[angle=180,origin=c,valign=b]{TLimages/TLcupTall.pdf}\includegraphics[valign=b]{TLimages/TLidentityDotHigh.pdf} &\stackrel{(1)}{=}  %
    \includegraphics[angle=180,origin=c,valign=b]{TLimages/TLcupTallDotL.pdf}\includegraphics[valign=b]{TLimages/TLidentity.pdf} -  \includegraphics[angle=180,origin=c,valign=b]{TLimages/TLcupTall.pdf}\includegraphics[angle=180,origin=c,valign=b]{TLimages/TLidentityDotHigh.pdf}\\
    &\stackrel{(2)}{=}
    \includegraphics[angle=180,origin=c,valign=c]{TLimages/TLcup.pdf}\,\,\includegraphics[valign=b]{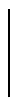}\circ \mleft(\includegraphics[valign=c]{TLimages/TLidentity.pdf}\otimes\mleft(\includegraphics[valign=c]{TLimages/TLidentityDot.pdf}\includegraphics[valign=c]{TLimages/TLidentity.pdf}-\includegraphics[valign=c]{TLimages/TLidentity.pdf}\includegraphics[valign=c]{TLimages/TLidentityDot.pdf}\mright)\mright)\\   
    &\stackrel{(3)}{=}
    \includegraphics[angle=180,origin=c,valign=c]{TLimages/TLcup.pdf}\,\,\includegraphics[valign=b]{TLimages/TLidentityShort.pdf}\,\circ \mleft(\includegraphics[valign=c]{TLimages/TLidentity.pdf}\otimes \mleft(\includegraphics[valign=c]{TLimages/TLrelDotBottomL.pdf}\,-\,\includegraphics[angle=180, origin=c,valign=c]{TLimages/TLrelDotBottomL.pdf}\mright)\mright)\\
    &\stackrel{(4)}{=}
    \includegraphics[angle=180,origin=c,valign=c]{TLimages/TLcup.pdf}\,\,\includegraphics[valign=b]{TLimages/TLidentityShort.pdf}\,\circ \mleft(\includegraphics[valign=c]{TLimages/TLidentity.pdf}\,\includegraphics[valign=c]{TLimages/TLrelDotBottomL.pdf}-\includegraphics[valign=c]{TLimages/TLidentity.pdf}\,\includegraphics[angle=180, origin=c, valign=c]{TLimages/TLrelDotBottomL.pdf}\mright)\\
    &\stackrel{(5)}{=}\includegraphics[valign=c]{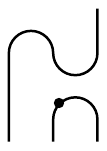} - \includegraphics[valign=c]{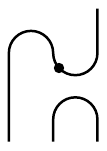}\\
    &\stackrel{(6)}{=}\includegraphics[valign=c]{TLimages/TLlemmaEQa.pdf}+\includegraphics[valign=c]{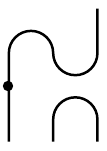}\\
    &\stackrel{(7)}{=}\includegraphics[valign=b]{TLimages/TLidentity.pdf}\reflectbox{\includegraphics[angle=180,origin=c,valign=b]{TLimages/TLcupTallDotL.pdf}} \,\,\,+\,\, \includegraphics[valign=b]{TLimages/TLidentityDotHigh.pdf}\includegraphics[angle=180,origin=c,valign=b]{TLimages/TLcupTall.pdf}
\end{align*}
In the first equation the minus sign comes from moving a dot past a cap.  We then decompose \includegraphics[angle=180,origin=c,scale=.5,valign=c]{TLimages/TLcup.pdf}\includegraphics[scale=0.5,valign=c,trim=0 0 0 3mm]{TLimages/TLidentityShort.pdf} from the top of both terms and factor out an undotted component from the left of each term to get equation 2. After using the relation \eqref{eqn:TLfourterm}, then tensoring and composing, we end up with equation 5.  A dotslide past the cap and then below the cup results in a sign change, and then relation \eqref{eqn:TLisotopy} gives the desired equality.

For \eqref{eqn:TLtype1identity} we can scale the diagram and slide dots to get the following:
\[ \includegraphics[valign=c]{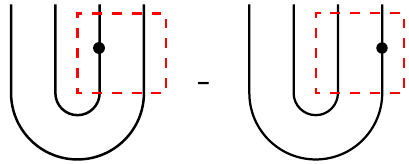} \]
Moving the dot on the outer cup past the critical point on the inner cup results in a sign change.  We then use relation \eqref{eqn:TLfourterm} inside of the red squares, followed by \eqref{eqn:TLisotopy} and \eqref{eqn:TLinterchange} to isotope to the right hand side of \eqref{eqn:TLtype1identity}.

For \eqref{eqn:TLtype2identity} we show the difference of the two sides is zero.
\[\includegraphics[valign=c]{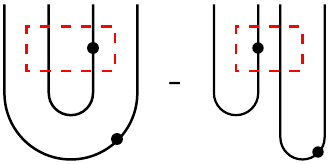}\]
Scaling and shifting the dot on the smaller cup does not change any signs.  We then use relation \eqref{eqn:TLfourterm} which results in one term containing a componant with two dots on it and the other term containing a circle.  Relations \eqref{eqn:TLtwodots} and \eqref{eqn:TLcircles} then give the desired result.
\end{proof}
%\end{document}
\label{sec:Chapter2}

\section{Odd annular Bar-Natan category}\label{sec:Chapter3}
    %\documentclass[../main.tex]{subfiles}
%\begin{document}

We define two versions of an odd annular Bar-Natan category; an unordered category denoted $\BNA$, and an ordered category denoted $\OBNA$.  We then show they are equivalent.
\subsection{Embedded chronological cobordisms}
Let $F$ be either the plane, $\mathbb{R}^2$, or the standard annulus
\[
\ann:=\{(x,y)\in\mathbb{R}^2\,|\,1\leq x^2+y^2\leq 4\}.
\]
By a \textbf{cobordism} in $F\times I$, we will mean an unoriented smooth compact embedded surface $S\subset (F\setminus\partial F)\times I$ such that $\partial S\subset F\times\partial I$. We will later regard such a cobordism as a morphism from its bottom boundary $S\cap(F\times\{0\})$ to its top boundary $S\cap(F\times\{1\})$. A cobordism will be called \textbf{dotted} if its interior, $S\setminus\partial S$, is decorated by at most finitely many distinct dots. Moreover, a dotted cobordism $S\subset F\times I$ will be called \textbf{chronological} if both of the following hold:
\begin{enumerate}
    \item[(1)] The height function $h\colon F\times I\rightarrow I$ restricts to a Morse function on $S$ such that no two dots or critical points occur at the same height (and no dot occurs at the same height as a critical point).
    \item[(2)] Each descending manifold of a critical point is equipped with an orientation, where the descending manifold is the set of all points on $S$ which flow into the critical point under the gradient flow.
\end{enumerate}

We will say that an orientation-preserving diffeomorphism $f\colon F\times I\rightarrow F\times I$ is \textbf{chron- ological} if it sends level sets of $h$ to level sets while preserving orientation. Moreover, we will say that two chronological dotted cobordisms $S$ and $S'$ are related by \textbf{chronological isotopy} if there is a smooth function $\varphi\colon (F\times I)\times I\rightarrow F\times I$ such that
\begin{itemize}
    \item for each $s\in I$, $\varphi_s:=\varphi(-,s)$ is a chronological diffeomorphism of $F\times I$,
    \item $\varphi_0=\mathbbm{1}$ and $\varphi_1(S)=S'$,
    \item $\varphi_1$ maps the dots on $S$ bijectively to the ones on $S'$.
\end{itemize}
Such a chronological isotopy will be called \textbf{relative to the boundary} if each $\varphi_s$ restricts to the identity on $F\times\partial I$. By abuse of notation, we will sometimes identify a chronological cobordism $S\subset F\times I$ with its rescaled version in $F\times[a,b]$ for $a<b$.

Given two chronological dotted cobordisms $S,S'\subset F\times I$ whose images under the projection $F\times I\rightarrow F$ are disjoint, we will call the union $S\cup S'$ a \textbf{disjoint `right-then-left' union} if every dot or critical point appearing in $S$ occurs at a greater height than every dot or critical point appearing in $S'$. In this case, we will use the notation $S\rlunion S'$ to denote $S\cup S'$. Schematically, $S\rlunion S'$ looks as in Figure~\ref{fig:BNrightthenleft}.
\begin{figure}[H]
\begin{center}
\includegraphics[valign=c]{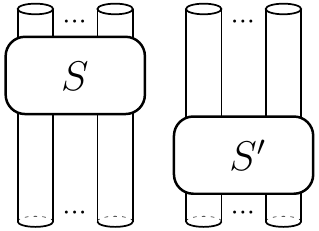}
\end{center}
\caption{Schematic depiction of the `right-then-left' union $S\rlunion S'$.}\label{fig:BNrightthenleft}
\end{figure}

Here and throughout this paper, we use the following definition:

\begin{definition}
A closed component $C\subset\ann$ is called \textbf{inessential} (or \textbf{trivial}) if it bounds a disk in $\ann$, and \textbf{essential} if it does not. Equivalently, the component $C$ is trivial if its homology class in $H_1(\ann;\mathbb{Z}_2)=\mathbb{Z}_2$ satisfies $[C]=0$, and essential if $[C]=1$.
\end{definition}

It is clear that if $C,C'\subset\ann$ are two closed embedded $1$-manifolds which can be connected by a cobordism $S\subset\ann\times I$, then $[C]=[C']$. We thus see that cobordisms in $\ann\times I$ must preserve the parity of the number of essential components. For example, if $S\subset\ann\times I$ is an elementary saddle cobordism, then either all three components of $\partial S$ are trivial, or exactly two of them are essential.

\subsection{Movies and surgery diagrams}

Given any chronological dotted cobordism $S\subset F\times I$, we can represent it by a \textbf{movie} of closed $1$-manifolds $C_0,C_1,\ldots,C_n\subset F$ obtained by intersecting $S$ with surfaces $F\times\{t_i\}$ for a generic partition $0=t_0<\ldots<t_n=1$ of $I$. In such a movie presentation, we assume that the $t_i$ are chosen so that each $S\cap (F\times t_i)$ contains no dots or critical points, $S\cap\mleft(F\times(t_{i-1},t_i)\mright)$ contains at most one dot or critical point, and so that consecutive $1$-manifolds in the movie are related by an isotopy, by a dotted identity cobordism, or by an elementary Morse modification (a merge, a split, or a birth or a death of a trivial circle). For example, the following movies represent respectively a split saddle and a death cobordism, where the arrows represent the orientations of the descending manifolds:
\[\includegraphics[valign=c]{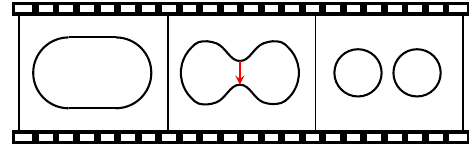} \hspace{2cm}   \includegraphics[valign=c]{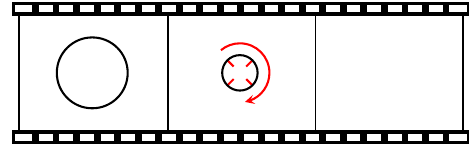} \]
We represent an identity cobordism decorated by a single dot by the following movie, where the dot in the first picture does not actually mean that there is a dot in that still of the movie, but rather that there is a dot occurring at some time after the first still but before the second:
\[
\includegraphics[valign=c]{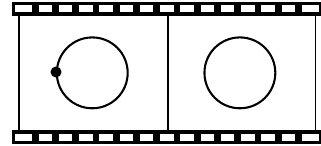}
\]
In movies representing chronological cobordisms in $\ann\times I$, we will sometimes use a star ($*$) to mark the location of the center of that annulus $\ann$ (i.e., the point $0\in\mathbb{R}^2$), as in the following example, where we also represent a birth of a trivial circle.

\[\includegraphics[valign=c]{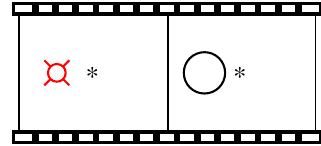}\]

If $S\subset F\times I$ is a chronological cobordism that contains no dots and only saddle critical points whose descending manifolds project to disjoint arcs $a_1,\ldots,a_n$ under the projection $F\times I\rightarrow F$, then we can represent $S$ by drawing a \textbf{surgery diagram} on $F$, consisting of the bottom boundary of $S$ along with the arcs $a_1,\ldots,a_n$. For instance, the following diagram represents a cobordism consisting of two merge saddles, where the arrows on the arcs indicate the orientations of the descending manifolds, and the numbering of the arcs indicate that the left saddle occurs before the right saddle in the movie of $S$:
\[\includegraphics[valign=c,scale=2.5]{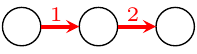}\]

\subsection{An odd annular Bar-Natan category} \label{sec:OddAnnularBN}

Fix a commutative unital ring $\Bbbk$. We define an odd annular Bar-Natan category, denoted $\BNA$, which has the structure of a monoidal supercategory. Objects in this category are closed unoriented 1-manifolds embedded in the annulus $\ann$. 

Morphisms in $\BNA$ are formal $\Bbbk-$linear combinations of dotted chronological cobordisms $S\subset 
\ann\times I$, modulo chronological isotopy relative boundary, and modulo the following relations.

\begin{enumerate}
\item[(1)] Disjoint union interchange:
   \begin{equation}\label{eqn:BNdisjoint}
    \includegraphics[valign=c]{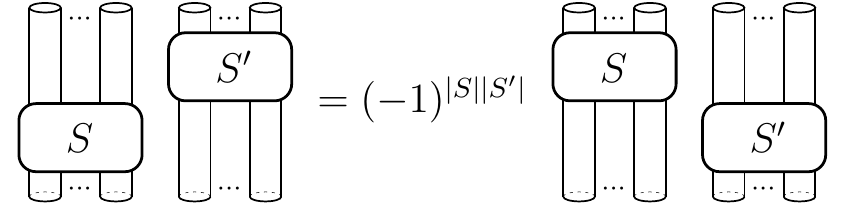}
    \end{equation}
\item[(2)] Connected sum interchange:
   \begin{equation}\label{eqn:BNconnected}
    \includegraphics[valign=c]{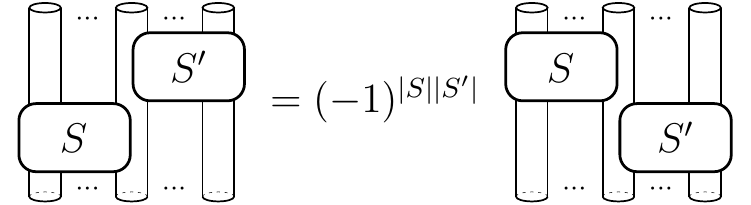}
    \end{equation}
 
\item[(3)] Exceptional $\Diamond$-interchange:
    \begin{equation}\label{eqn:BNdiamond}
    \includegraphics[valign=c]{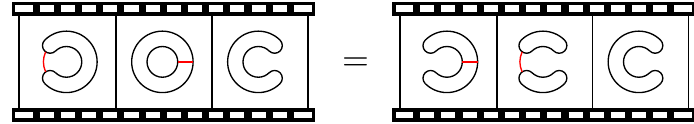}\hspace*{0.38cm}
    \end{equation}
\item[(4)] Exceptional $\times$-interchange:
   \begin{equation}\label{eqn:BNXchange}
    \includegraphics[valign=c]{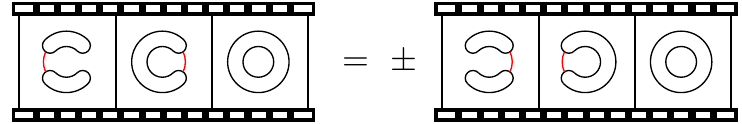}
    \end{equation}
\item[(5)] Creation/annihilation of critical points:
    \begin{equation}\label{eqn:BNcreation}
    \includegraphics[valign=c]{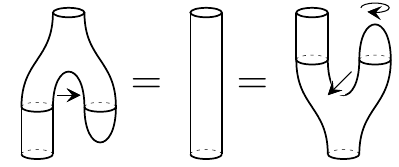}
    \end{equation}
\item[(6)] Orientation reversal:
    \begin{equation}\label{eqn:BNorientation}
    \includegraphics[valign=c]{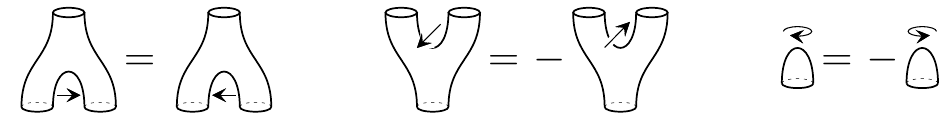}
    \end{equation}
\item[(7)] Bar-Natan relations:
    \begin{equation}\label{eqn:BNBN}
\includegraphics[valign=c]{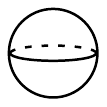}=0\hspace{1.5cm}\includegraphics[valign=c]{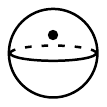}=1\hspace{1.5cm}\includegraphics[valign=c]{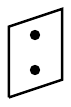}=0\hspace{1.5cm}\includegraphics[valign=c]{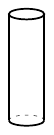}=\includegraphics[valign=c]{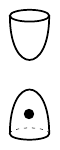}+\includegraphics[valign=c]{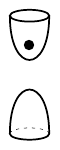}
    \end{equation}
\end{enumerate}

The cobordisms appearing in relations \eqref{eqn:BNdiamond} and \eqref{eqn:BNXchange} can also be represented by the surgery diagrams shown in Figure~\ref{fig:ExceptionalInterchanges}. We define the sign in~\eqref{eqn:BNXchange} to be a minus if the two arrows in the corresponding surgery diagram point to the same circle, and a plus otherwise.

\begin{figure}[H]
\begin{center}
\includegraphics[width=5cm]{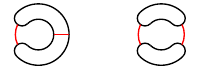}
\end{center}\vspace*{-0.1in}
\caption{Surgery diagrams corresponding to~\eqref{eqn:BNdiamond} and \eqref{eqn:BNXchange}.}\label{fig:ExceptionalInterchanges}
\end{figure}

In the Bar-Natan relations \eqref{eqn:BNBN}, we assume that all deaths are oriented clockwise. Moreover, we write $0$ or $1$ to indicate that the cobordisms shown on the left-hand side of each relation can be replaced by the scalars $0$ or $1$. For example, the second relation in~\eqref{eqn:BNBN} says that if a cobordism contains a component which looks like the dotted sphere shown in the relation, then this component can be dropped from the cobordism. Likewise, the third relation implies that a cobordism is zero if it contains a component which is decorated by two or more dots.

In all of the depicted relations, it is understood that the pictures only show portions of the actual cobordisms, while the portions that are not shown are assumed to be unchanged on both sides of each relation. In addition, we require that any dots or critical points that may appear in these unseen portions lie either above or below the shown portions.

Although the pictures in relations~\eqref{eqn:BNdisjoint} through \eqref{eqn:BNorientation} accurately depict the relative heights of the dots and the critical points, they do not reflect how the shown cobordisms are embedded in $\ann\times I$. For example, in relation~\eqref{eqn:BNdisjoint}, we only require that the cobordisms $S$ and $S'$ have disjoint projections to $\ann$, but we do not impose any further restrictions. In particular, this means that any of the boundary components of $S$ and $S'$ could be trivial or essential. In~\eqref{eqn:BNdiamond} and \eqref{eqn:BNXchange}, we assume that the pictures in the movies are drawn on the $2$-sphere $\mathbb{S}^2=\mathbb{R}^2\cup\{\infty\}$. We do not specify the locations of the points $0$ and $\infty$ in these pictures, but we require that these points lie at the same locations in all pictures belonging to the same relation. Note that the locations of the points $0$ and $\infty$ determine whether the components in the pictures are trivial or essential. Specifically, a component $C\subset\ann$ is trivial if the points $0$ and $\infty$ lie in the same connected component of $\mathbb{S}^2\setminus C$.

Unlike the other relations, the last relation in~\eqref{eqn:BNorientation} and the relations in~\eqref{eqn:BNBN} are completely local: they only involve small portions of a cobordism that are contained in a contractible subregion of $\ann\times I$. The last relation in~\eqref{eqn:BNBN} will be called the \textbf{vertical neck-cutting relation} because it can be used to remove vertical tubes from a cobordism. Note that this relation is only applicable if the belt circle of the vertical tube on the left-hand side of the relation bounds a horizontal disk $D\subset\ann\times I$ which is otherwise disjoint from the cobordism that contains the tube.

We define the \textbf{superdegree} of a dotted chronological cobordism $S$ by counting the total number of split saddles, deaths (i.e. local maxima), and dots that occur in $S$:
\[
|S|:=\#\{\textnormal{splits, deaths, dots}\}\in\mathbb{Z}_2.
\]
This definition induces a well-defined supergrading on $\BNA$ because all of the relations above are homogeneous with respect to the superdegree.

\begin{remark} If $2$ is invertible in $\Bbbk$, then the third relation in~\eqref{eqn:BNBN} is redundant; it can be deduced from relation~\eqref{eqn:BNconnected} and from the fact that a dot has superdegree $1$. However, this relation is no longer redundant in the even setting or in the universal category described below in Remark~\ref{rem:BNuniversal}. In fact in the even theory, imposing $\incg{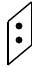}=0$ versus $\incg{BNsmalldoubledot}=1$ precisely corresponds to the distinction between Khovanov homology and Lee homology~\cite{Lee}.
\end{remark}

\begin{remark} \label{rmk:BNorientation}
Relations~\eqref{eqn:BNcreation} and \eqref{eqn:BNorientation} together imply that, for arbitrarily oriented critical points,
\[
\incg{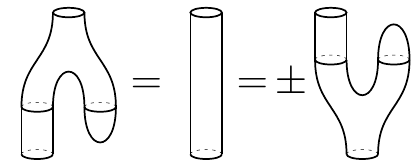}.
\]
The sign in the last term can be understood intrinsically as follows: rotate the arrow at the split saddle by $90^\circ$ in the direction specified by the orientation of the death. If the resulting rotated arrow points to the side that does not contain the death, then the sign is a plus, otherwise it is a minus.
\end{remark}

The \textbf{composition} $S\circ S'$ of two dotted chronological cobordisms is defined by stacking $S$ vertically on top of $S'$ and rescaling in the vertical direction. This composition operation extends bilinearly to arbitrary morphisms in $\BNA$.

To define the \textbf{supermonoidal product}, we divide the annulus $\ann$ into a union of two thinner annuli: an annulus $\ann_1$ consisting of all points in $\ann$ with $r\geq 3/2$, and an annulus $\ann_2$ consisting of all points in $\ann$ with $r\leq 3/2$, where here $r$ denotes the distance of a point from the origin of $\mathbb{R}^2\supset\ann$. On objects, we now define the tensor product by
\[
C\otimes C':=\varphi(C)\cup\psi(C'),
\]
where $\varphi\colon\ann\rightarrow\ann_1$ and $\psi\colon\ann\rightarrow\ann_2$ denote the obvious identifications given by rescaling the annulus $\ann$ in the radial direction (see Figure~\ref{fig:tensor}).
\begin{figure}[ht]
\begin{center}
\incg{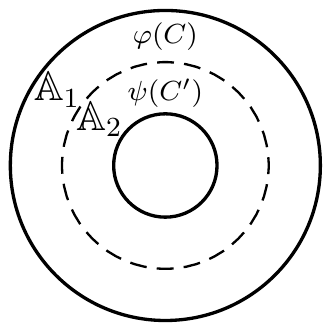}
\end{center}\vspace*{-0.1in}
\caption{Definition of $C\otimes C'$.}\label{fig:tensor}
\end{figure}

The tensor product of two dotted chronological cobordisms $S$ and $S'$ is defined by taking a disjoint `right-then-left' (or `inside-then-outside') union of a perturbed copy of $S$ in $\ann_1\times I$ and a perturbed copy of $S'$ in $\ann_2\times I$. More formally, choose an $\epsilon\in (0,1/2)$ such that all dots and critical points in $S$ occur at heights greater than $\epsilon$, and all dots and critical points in $S'$ occur at heights less than $1-\epsilon$. The tensor product of $S$ and $S'$ is then defined by
\[
S\otimes S':=\widetilde{\varphi}_\epsilon(S)\cup\widetilde{\psi}_{\epsilon}(S'),
\]
where $\widetilde{\varphi}_\epsilon:=\varphi\times f_\epsilon$ and $\widetilde{\psi}_\epsilon:=\psi\times f_{1-\epsilon}$ for $f_a$ a diffeomorphism of $I$ which sends the intervals $[0,a]$ and $[a,1]$ to $[0,1/2]$ and $[1/2,1]$, respectively. As with the composition, we extend the tensor product bilinearly to arbitrary morphisms.

\begin{remark}\label{rem:BNstricttensor}
The tensor product above is neither strictly associative nor strictly unital. However, one can easily replace $\BNA$ by an equivalent strict monoidal supercategory by considering objects up to isotopies of the annulus $\ann$ which can be written in polar coordinates $(r,\theta)$ as $\varphi_s(r,\theta)=(f_s(r),\theta)$ for an isotopy $f_s\colon [1,2]\rightarrow [1,2]$ and $s\in I$.
\end{remark}

\begin{remark}\label{rem:BNuniversal}
By removing all minus signs in the definition of $\BNA$ and forgetting the supergrading, we recover the even annular Bar-Natan category from~\cite[Subs.~11.6]{BN}. We can further define a universal category $\UBNA$, which generalizes the even and the odd category. To define this universal category, we extend coefficients to $\Bbbk[\pi]/(\pi^2-1)$ and replace all minus signs in the definition of $\BNA$ by factors of $\pi$.
\end{remark}

We now aim to derive a horizontal version of the neck-cutting relation. To obtain this horizontal version, we need the following lemma.

\begin{lemma}\label{lem:diamondzero}
The cobordisms that appear in relation~\eqref{eqn:BNdiamond} are equal to zero in $\BNA$.
\end{lemma}

\begin{proof} The surgery diagram that corresponds to the movies in~\eqref{eqn:BNdiamond} cuts $\mathbb{S}^2$ into four regions, which correspond bijectively to the four circles that appear in the center panels of the two movies. Indeed, each circle appearing in the two middle diagrams is a push-off of the boundary of exactly one of the four regions of the surgery diagram. Since the points $0,\infty\in\mathbb{S}^2$ can occupy at most two of the four regions, it follows that at least two of the circles are trivial. Moreover, since a trivial circle cannot contain an essential circle, at least one of the trivial circles, we will call it $C$, bounds a disk $D\subset\ann$ that does not contain the other circle in the same diagram.

Let $S$ denote the chronological cobordism represented by the movie that contains $C$. We can then view $C$ as a simple closed curve on $S$, and since $C=\partial D$, we can apply the vertical neck-cutting relation along this curve. Abstractly, this yields the sum
\[
\incg{BNDiamondZero.pdf}
\]
Using relation~\eqref{eqn:BNconnected} and recalling that dots, deaths, and splits have superdegree $1$, while births and merges have superdegree $0$, we can move the dot that appears in the first of these two terms to the location of the dot in the second term, at the cost of introducing an overall minus sign. Hence the two terms in the sum above cancel, which shows that $S$ is equal to zero. Relation~\eqref{eqn:BNdiamond} now implies that the chronological cobordism corresponding to the movie on its other side is also equal to zero.
\end{proof}

\begin{remark}
The lemma above does not hold in the even or in the universal setting. However, in the universal setting, the proof of the lemma shows that the cobordisms in~\eqref{eqn:BNdiamond} are equal to identity cobordisms decorated by a dot and multiplied by a factor of $1+\pi$. Since $\pi^2=1$, this implies that these cobordisms are annihilated by multiplication by $1-\pi$, or, equivalently, that multiplication by $\pi$ acts on them as the identity. We could therefore multiply either of side of relation~\eqref{eqn:BNdiamond} by a factor of $\pi$ without altering the significance of the relation.
\end{remark}

\begin{lemma}[Horizontal neck cutting]\label{lem:horizontalneck} In $\BNA$, we have
\begin{equation}\label{eqn:horizontalneck}
\incg{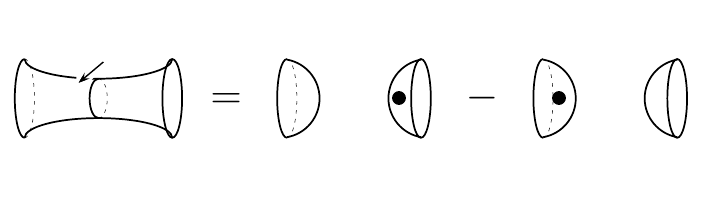}
\end{equation}
\end{lemma}
\begin{proof} By applying a chronological isotopy, we can move the two saddles on the left-hand side of~\eqref{eqn:horizontalneck} horizontally, until the horizontal tube in~\eqref{eqn:horizontalneck} looks as shown in Figure~\ref{fig:horizontalneck}.

\begin{figure}[H]
\begin{center}
\incg{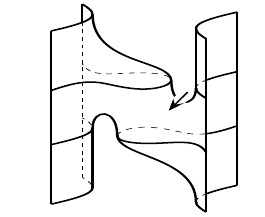}
\end{center}\vspace*{-0.1in}
\caption{Isotoped version of the horizontal tube in~\eqref{eqn:horizontalneck}.}\label{fig:horizontalneck}
\end{figure}

Suppose first that the tube in~\eqref{eqn:horizontalneck} connects two otherwise disconnected parts of a cobordism. Then the cobordism shown in Figure~\ref{fig:horizontalneck} can be viewed as a connected sum of a merge saddle and a split saddle. We can thus use relation~\eqref{eqn:BNconnected} to exchange the relative heights of these two saddles. This yields the cobordism
\begin{align*}
    \incg{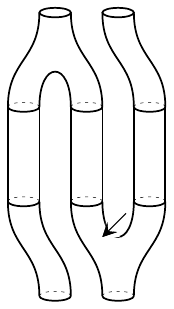}&\,\,=\,\phantom{-}\incg{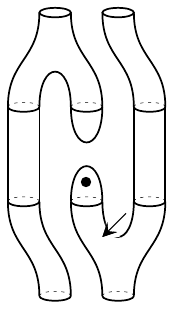}\,+\,\incg{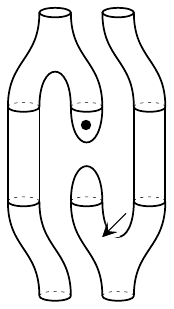}\\
            &\,=\,-\incg{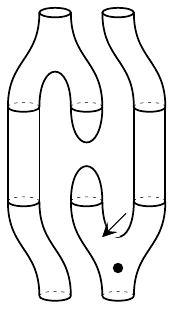}\,+\,\incg{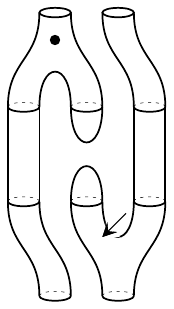}\,=\,\incg{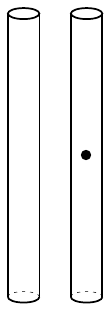}\,-\,\incg{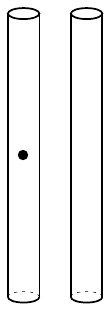}\,,
\end{align*}
where the sign change in the last equality follows from remark \ref{rmk:BNorientation}.  This proves the lemma for this case.

Now suppose that the horizontal tube in~\eqref{eqn:horizontalneck} connects a component of a cobordism to itself. Then the midsection of the cobordism shown in Figure~\ref{fig:horizontalneck} looks like one of the cobordisms appearing in relation~\eqref{eqn:BNdiamond}, and hence the left-hand side of~\eqref{eqn:horizontalneck} is zero, by Lemma~\ref{lem:diamondzero}. On the other hand, the two terms on the appearing on the right-hand side are equal in this case, and hence the right-hand side is zero as well.
\end{proof}

\begin{remark}
Lemma~\ref{lem:horizontalneck} also holds in the universal setting if one replaces the minus sign on the right-hand side by a factor of $\pi$.
\end{remark}

In addition to the supergrading, there is a $\mathbb{Z}$-grading on $\BNA$, called the \textbf{quantum grading}. On dotted chronological cobordisms, we define this grading by
\[
q(S):=\chi(S)-2\#\{\textnormal{dots}\}\in\mathbb{Z},
\]
where $\chi(S)$ denotes the Euler characteristic of $S$. Note that all relations in $\BNA$ are compatible with $q(S)$, and hence this grading is well-defined. In the following, we will denote by $\BNA^{sg}$ the $(\mathbb{Z}_2\oplus\mathbb{Z})$-graded extension of $\BNA^{sg}$. Objects in this category are of the form $(C,d,d')$, where $C$ is an object in $\BNA$, and where $d\in\mathbb{Z}_2$ and $d'\in\mathbb{Z}$ are formal shifts of the supergrading and of the quantum grading respectively. We will denote such objects by $C_{(d,d')}:=(C,d,d')$.

The following lemma holds in the additive closure of $\BNA^{sg}$, and is well-known in the even setting~\cite{BarNatanComputations}. In the (non-annular) odd setting, it was shown in \cite{PutyraChrono}.

\begin{lemma}[Delooping]\label{lem:delooping}
Let $C\subset\ann$ be a closed $1$-manifold and $\bigcirc\subset\ann$ be a trivial component which bounds a disjoint disk in $\ann$. There is a grading-preserving isomorphism
\[
C\sqcup\bigcirc\,\cong\, C_{(0,1)}\oplus C_{(1,-1)}.
\]
\end{lemma}

\begin{proof} The desired isomorphism is given by an identity cobordism on $C$ and by the following matrices on $\bigcirc$ where all deaths are assumed to be oriented clockwise:
\[
\begin{tikzcd}[column sep=large,ampersand replacement=\&]
\bigcirc\ar[r,"\begingroup
\renewcommand*{\arraystretch}{0.8}
\begin{bmatrix}
\includegraphics{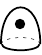}\\
\includegraphics{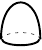}\end{bmatrix}
\endgroup"]
\&\left(\emptyset_{(0,+1)}\oplus\emptyset_{(1,-1)}\right)\ar[r,"{\begin{bmatrix}
\includegraphics{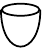}&\hspace*{-0.04in}
\includegraphics{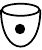}\vspace*{-0.05in}%
\end{bmatrix}}"]
\&\bigcirc
\end{tikzcd}
\]
A direct calculation using the Bar-Natan relations shows that the matrices above are indeed inverses of each other.
\end{proof}

As an immediate consequence of this lemma, we have the following:

\begin{corollary} \label{cor:deloopingequivalence}
Let $E_n$ denote a disjoint union of $n$ essential circles in $\ann$ and $\sBNA$ denote the full subcategory of $\BNA$ which contains all of the objects $E_n$, for all $n\geq 0$. Then the inclusion $\sBNA^{sg}\rightarrow\BNA^{sg}$ induces a graded monoidal equivalence between the additive closures of the involved monoidal supercategories.
\end{corollary}

\begin{proof}
First, note that every object $C\subset\ann$ without trivial components is isotopic in $\ann$ and hence isomorphic in $\BNA$ to one of the objects $E_n$. Using Lemma~\ref{lem:delooping} repeatedly, we thus see that every object $C\in\ann$ with $n$ essential components is isomorphic to a direct sum, henceforth called $D(C)$, of shifted copies of the object $E_n$. More specifically, we can define an isomorphism $\mathfrak{n}_C\colon C\rightarrow D(C)$ by applying the isomorphism from Lemma~\ref{lem:delooping} successively to the trivial components of $C$, starting with innermost trivial components. Here, we say that a trivial component of $C$ is innermost if it bounds a disk in $\ann$ that is disjoint from all other components of $C$. We can extend the assignment $C\mapsto D(C)$ to a functor 
\[
\mathcal{D}\colon(\BNA^{sg})^\oplus\longrightarrow(\sBNA^{sg})^\oplus
\]
by setting $\mathcal{D}(S):=\mathfrak{n}_{C'}\circ S\circ\mathfrak{n}_C^{-1}$ for every morphism $S\colon C\rightarrow C'$. By construction, this functor is a left-inverse for the embedding $\mathcal{E}\colon(\sBNA^{sg})^\oplus\rightarrow(\BNA^{sg})^\oplus$, and there is a graded even supernatural isomorphism $\mathbbm{1}\cong \mathcal{E}\circ\mathcal{D}$ given by the isomorphisms $\mathfrak{n}_C$.
This proves that $(\sBNA^{sg})^\oplus\simeq (\BNA^{sg})^\oplus$ as graded supercategories.

It remains to show that the equivalence above is monoidal. We first note that $\mathfrak{n}_C$ has the form of a column matrix whose entries are given by unions of dotted and undotted death cobordisms, along with cobordisms induced by isotopies. The signs of these entries are not unique, but rather may change if one changes the relative heights of undotted death cobordisms. In the remainder of this proof, we will assume for simplicity that objects and morphisms have been identified as described in Remark~\ref{rem:BNstricttensor}. In view of Lemma~\ref{lem:automaticallymonoidal}, it then suffices to show that the signs in the $\mathfrak{n}_C$ can be chosen so that $\mathfrak{n}_{C\otimes C'}=\mathfrak{n}_C\otimes^{sg}\mathfrak{n}_{C'}$, where $\otimes^{sg}$ denotes the tensor product in the supergraded extension.

To prove this, note that each nonempty object $C\subset\ann$ can be written uniquely as a tensor product
\begin{equation}
 C=C_1\otimes\ldots\otimes C_n \label{eqn:BNAtensor}   
\end{equation}
where the $C_i$ are objects that cannot be decomposed any further as tensor products of non-empty objects. We now start by choosing the sign of $\mathfrak{n}_{C'}$ arbitrarily for every nonempty object $C'$ that cannot be decomposed as a nontrivial tensor product. For a general object $C$ as above, we then define
\[
\mathfrak{n}_C:=\mathfrak{n}_{C_1}\otimes^{sg}\ldots\otimes^{sg}\mathfrak{n}_{C_n}
\]
which ensures that $\mathfrak{n}_C$ has the desired property.
\end{proof}

\begin{remark} \label{rmk:DisjUnionCalcs}
The objects $C_i$ in \eqref{eqn:BNAtensor} cannot be decomposed as tensor products of non-empty objects, but they may be written as the disjoint union of more than one object.  For example $C_i=\includegraphics[scale=0.5,valign=c]{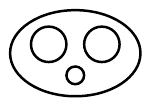}$ with all four circles trivial cannot be decomposed as a tensor product, but it can be written as $\includegraphics[scale=0.5, valign=c]{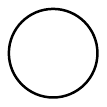}\sqcup\includegraphics[scale=0.5, valign=c]{BNimages/ObjCircle.pdf}\sqcup\includegraphics[scale=0.5, valign=c]{BNimages/ObjCircle.pdf}\sqcup\includegraphics[scale=0.5, valign=c]{BNimages/ObjCircle.pdf}$.  This is also true of cobordisms with trivial boundary componants.  In the next two sections we will define an ordering on components and cobordisms, replace the disjoint union $\sqcup$ with the disjoint right-then-left union $\rlunion$, and this decomposition will be well defined. Until then we simply remark that for the purpose of calculations, the disjoint right-then-left union of matrices in $(\sBNA^{sg})^\oplus$ is the usual tensor product of matrices.
\end{remark}

Given a connected component $S'$ of a cobordism $S'\subset\ann\times I$, we will say that $S'$ is \textbf{essential} if the inclusion-induced map $H_1(S';\mathbb{Z}_2)\rightarrow H_1(\ann\times I;\mathbb{Z}_2)$ is surjective. We further define
\[
a(S):=-2\#\{\textnormal{dots on essential components of $S$}\}\in\mathbb{Z}.
\]
Note that this does not define a grading on $\BNA$ because the vertical neck-cutting relation is in general not homogeneous with respect to $a(S)$. However, we can still define a filtration by declaring a morphism to have filtered degree at most $k$ if it can be written as a linear combination of dotted chronological cobordisms with $a(S)\leq k$. We will call this filtration the \textbf{annular filtration}.

\begin{lemma}\label{lem:annularfiltration} The annular filtration turns $\BNA$ into a filtered monoidal supercategory in the sense of Definition~\ref{def:filteredsupercategory}.
\end{lemma}

\begin{proof}
We must show that the filtered degree associated to the annular filtration satisfies $\operatorname{fdeg}(S\circ S')\leq\operatorname{fdeg}(S)+\operatorname{fdeg}(S')$ and $\operatorname{fdeg}(S\otimes S')\leq\operatorname{fdeg}(S)+\operatorname{fdeg}(S')$ for all nonzero morphisms $S,S'$. For this, it suffices to show that $a(S\circ S')\leq a(S)+a(S')$ and $a(S\otimes S')\leq a(S)+a(S')$ for any dotted chronological cobordisms $S,S'$. However, the latter is obvious because any dot that appears on an essential component of $S$ or $S'$ also appears on an essential component of $S\circ S'$ or $S\otimes S'$.
\end{proof}

\begin{example}\label{ex:annularsaddles}
Suppose $S_\mu$ and $S_\Delta$ are two saddle cobordisms such that $S_\mu$ merges two essential components into a trivial component $C$, and $S_\Delta$ splits $C$ into two essential components. Then $\operatorname{fdeg}(S_\mu)\leq 0$ and $\operatorname{fdeg}(S_\Delta)\leq 0$, while $\operatorname{fdeg}(S_\Delta\circ S_\mu)\leq-2$ because one can apply the vertical neck-cutting relation to a tubular neighborhood of the curve $C\subset S_\Delta\circ S_\mu$, resulting in two terms which each contain a dot on an essential component. We will see later in section~\ref{subs:annularTQFT} that the inequalities above for $\operatorname{fdeg}(S_\mu)$, $\operatorname{fdeg}(S_\Delta)$, and $\operatorname{fdeg}(S_\Delta\circ S_\mu)$ are actually equalities.
\end{example}

If $\mathcal{C}$ denotes the category $\BNA$ equipped with the annular filtration, then
\begin{equation}\label{eqn:defBBNA}
\mathcal{C}/F_{-1}\mathcal{C}=\mathcal{C}/F_{-2}\mathcal{C}=:\BBNA
\end{equation}
is precisely the quotient category obtained by setting dots on essential components equal to zero. In the even setting, this quotient category was studied in~\cite{Boerner}.

We end this section by noting that one can define the non-annular odd and universal categories $\mathcal{BN}_{\!o}(\mathbb{R}^2)$ and $\mathcal{BN}_{\!\pi}(\mathbb{R}^2)$ by considering objects and morphisms embedded in $\mathbb{R}^2$ and $\mathbb{R}^2\times I$ instead of $\ann$ and $\ann\times I$. Unlike the annular categories, the non-annular ones do not carry annular filtrations.

\subsection{An ordered annular Bar-Natan Category} \label{sec:OrderedBNA}

In this section we define an extension of the category $\BNA$ in which connected components of objects are ordered.

To start with, consider an object $C$ of $\BNA$, and let $E_1,\,\dots\,,E_n$ denote the essential components of $C$.  There is a natural total order on the $E_i$ given by setting $E_i\leq E_j$ whenever $U_i\subseteq U_j$, where $U_i$ denotes the connected component of $\ann\setminus E_i$ that contains the outer boundary of the annulus.

\begin{definition} \label{dfn:AdmissibleOrder}
We call a total order on the components of $C$ \textbf{admissible} if it restricts to the natural order on essential components.
\end{definition}

If $C$ is an object of $\BNA$, then one can define a distinguished admissible order on $\pi_o (C)$ as follows.  Let $\alpha:\mathbb{A}\rightarrow [1,2]\times [0,2\pi ]$ be the map which sends a point $P \in \ann$ to its polar coordinates, and give $[1,2]\times [0,2\pi]$ the lexicographic order where each factor is equipped with its usual total order coming from the total order on $\mathbb{R}$.  Since each connected component $C_i$ of $C$ is compact, it then follows that there is a unique point $P_i\in C_i$ for which $\alpha(P_i)\in \alpha(C_i)$ becomes minimal with respect to this order on $[1,2]\times [0,2\pi ]$.  For two components $C_i$ and $C_j$ of $C$, we can thus define $C_i \leq C_j$ whenever $\alpha(P_i)\geq \alpha(P_j)$ in the lexicographic order.  It is easy to see that the total order thus defined is indeed admissible.

\begin{remark} \label{rm:DistinguishedOrder}
The distinguished order just defined has the following property.  Suppose $C_i$ and $C_j$ are two components of $C$ such that $\alpha(C_i)\subseteq (r,2]\times [0,2\pi]$ and $\alpha(C_j)\subseteq [1,r)\times [0,2\pi]$ for some $r\in (1,2)$.  Then $C_i < C_j$ in the distinguished order.
\end{remark}

In what follows, we will often view a total order on components of $C$ as an order-preserving bijection on the fundamental group, $\calO :\pi_0(C)\rightarrow \{ 1,2,\dots , n_c\}$ where $n_c:=|\pi_0 (C)|$ and where $\{1,2,\dots n_c\}$ is equipped with its usual order.  Given any two total orders, $\calO$ and $\calO'$ on $\pi_0(C)$, we can then write $\calO'=\sigma\circ \calO$ where $\sigma$ denotes the permutation $\sigma:= \calO'\circ \calO^{-1} \in \mathfrak{S}_{n_c}$

We can now make the following definition.

\begin{definition} $\OBNA$ is the monoidal supercategory whose objects are pairs $(C,\calO)$ where $C$ is an object of $\BNA$ and $\calO$ is an admissible total order on $\pi_0(C)$, and whose morphisms are given by 

\begin{equation} \label{eq:OBN(A)Definition}
\Hom_{\OBNA}\left( (C,\calO),(C',\calO')\right):= \Hom_{\BNA}\left(C,C'\right).
\end{equation}
On objects of $\OBNA$, the supermonoidal product is defined by 
\begin{equation}
\left(C_1,\calO_1\right) \otimes \dots \otimes \left(C_n,\calO_n\right):= \left(C_1\otimes \dots \otimes C_n , \calO_1 \sqcup \ldots \sqcup \calO_n\right) 
\end{equation}
where $\calO_1 \sqcup \ldots \sqcup \calO_n$ denotes the total order on components of $C_1\otimes \dots \otimes C_n$ which restricts to $\calO_{i}$ on each $C_i$, and in which components of $C_i$ precede components of $C_j$ whenever $i< j$.  On morphisms, the supermonoidal product is induced by the one in $\BNA$.
\end{definition}

The definition above implies that if $C$ is an object of $\BNA$ and $\calO$ and $\calO'$ are two admissible orders of $\pi_0(C)$, then there is a canonical isomorphism \begin{equation*}
    \textnormal{R}_{C,\,\calO,\,\calO'}:(C,\,\calO)\to (C,\,\calO')
\end{equation*} 
in $\OBNA$ which corresponds to the identity morphism of $C$ under the identification \eqref{eq:OBN(A)Definition}.  This isomorphism satisfies 
\begin{equation} 
    \textnormal{R}_{C,\,\calO,\,\calO}=\mathbbm{1}_{(C,\,\calO)}\label{eqn:BNRid}
\end{equation}
    {\centering and}
\begin{equation}    
    \textnormal{R}_{C,\,\calO',\,\calO''}\circ \textnormal{R}_{C,\,\calO,\,\calO'}=\textnormal{R}_{C,\,\calO,\,\calO''}\label{eqn:BNRcomp}
\end{equation}
where the latter relation follows from the relation $\mathbbm{1}_{C} \circ \mathbbm{1}_{C} = \mathbbm{1}_{C}$ which holds in $\BNA$.

Now let $\calO_C$ denote the distinguished admissible order defined prior to Remark~\ref{rm:DistinguishedOrder}.  We then have:

\begin{lemma}\label{lem:BNAtoOBNA}
The functor $\BNA \rightarrow \OBNA$ which sends $C$ to $(C,\,\calO_{C})$ and which is given on morphisms by the identification \eqref{eq:OBN(A)Definition} is a fully faithful supermonoidal embedding.  Moreover, it is an equivalence with inverse equivalence given by the forgetful functor which sends $(C,\,\calO)$ to $C$.
\end{lemma}

\begin{proof}
It is clear that this functor is full and faithful and that the forgetful functor is a left-inverse.  The fact that the forgetful functor is, in fact, an inverse equivalence follows because the $\textnormal{R}_{C,\,\calO,\,\calO_{C}}$ provide a natural isomorphism between the identity functor of $\OBNA$ and the functor which sends $(C,\,\calO)$ to $(C,\,\calO_{C})$.  

Finally, the embedding above respects the supermonoidal product because if $C_1$ and $C_2$ are two objects of $\BNA$ and $C=C_1\otimes C_2$, then $\calO_{C}=\calO_{C_1}\sqcup \calO_{C_2}$ by Remark~\ref{rm:DistinguishedOrder}. 
\end{proof}

Next, suppose $C_1,\dots,C_n\subseteq\ann$ are mutually disjoint objects of $\BNA$ and $\calO_1,\dots,\calO_n$ are total orders on the components of the $C_i$.  We then define the disjoint union of the $(C_i,\calO_{i})$ by 
\begin{equation*}
(C_1,\calO_1)\sqcup\ldots\sqcup (C_n,\calO_n) := (C_1\sqcup\ldots\sqcup C_n,\calO_1\sqcup\ldots\sqcup \calO_n)  
\end{equation*}
 where $\calO_1\sqcup\ldots\sqcup \calO_n$ is defined as in the definition of the supermonoidal product in the category $\OBNA$.  Note that, unlike the supermonoidal product, the disjoint union is only defined if the $C_i$ are already disjoint.  In addition, the order $\calO_1\sqcup\ldots\sqcup \calO_n$ on the disjoint union may not be admissible even if each $\calO_{i}$ is individually admissible.  On the other hand, not every object of $\BNA$ or $\OBNA$ is ambient isotopic to a supermonoidal product of its components, but every object $(C,\calO)$ of $\OBNA$ can be written uniquely as 
\begin{equation}
    (C,\calO)=(C_1,\calO_1)\sqcup\dots\sqcup (C_n,\calO_n)
\end{equation}
where $C_1,\dots,C_n$ are the components of $C$, numbered so that $\calO(C_i)=i$, and $\calO_{i}$ denotes the unique total order on $\pi_0(C_i)=\{C_i\}$. 

Now consider chronological cobordisms $S_1,\ldots, S_n \subseteq \ann \times I$ whose projections to $\ann$ are mutually disjoint.  Suppose $S_i$ has lower boundary $C_i$ and upper boundary $C_i '$, and suppose $\pi_0(C_i)$ and $\pi_0(C_i')$ are equipped with total orders $\calO_{i}$ and $\calO_{i} '$, respectively.  We then denote by $S_1 \rlunion \dots \rlunion S_n$ the right-then-left union of the $S_i$, viewed as a cobordism 
\begin{equation} \label{dfn:Scobordism}
    S_i \rlunion \dots \rlunion S_n : (C_1,\calO_1)\sqcup\dots\sqcup(C_n,\calO_n)\longrightarrow(C_1',\calO_1')\sqcup\dots\sqcup(C_n',\calO_n').
\end{equation}

We now distinguish a specific case of $\textnormal{R}_{C,\calO,\calO'}$.  Let $C=C_1\sqcup C_2$ where $C_1$ and $C_2$ are two disjoint objects of $\BNA$ with total orders $\calO_1$ and $\calO_2$ on $\pi_0 (C_1)$ and $\pi_0(C_2)$.  If $\calO:=\calO_1\sqcup \calO_2$ and $\calO':=\calO_2\sqcup \calO_1$ are both admissible orders on $\pi_0(C)$, then we will use the notation $\textnormal{R}(C_1,C_2):=\textnormal{R}_{C,\calO,\calO'}$ for
\begin{equation}
    \textnormal{R}(C_1,C_2): (C_1\sqcup C_2,\calO_1\sqcup \calO_2)\rightarrow (C_1\sqcup C_2,\calO_2\sqcup \calO_1).
\end{equation}

With this understood, we have the following lemmas involving relations in $\OBNA$:

\begin{lemma} \label{lm:relationsA}
    Let $(C_1,\calO_1),\,(C_1',\calO'_1)\textnormal{ and }(C_2,\calO_2)$ denote disjoint objects of $\OBNA$, and let $S:C_1 \to C_1'$ denote a chronological cobordism in $\ann\times I$ such that $S$ is disjoint from $\mathbbm{1}_{C_2}:C_2 \to C_2$.  Then, assuming all involved orderings are admissible, the following two relations hold in $\OBNA$:
\begin{equation}
        \textnormal{R}(C_1',C_2)\circ(S\rlunion\mathbbm{1}_{C_2})=(\mathbbm{1}_{C_2} \rlunion S)\circ \textnormal{R}(C_1,C_2) \label{eq:Rel1}
\end{equation}        
\begin{equation}
        \textnormal{R}(C_2,C_1')\circ (\mathbbm{1}_{C_2}\rlunion S) = (S \rlunion \mathbbm{1}_{C_2})\circ \textnormal{R}(C_2,C_1) \label{eq:Rel2}
\end{equation} 
\end{lemma}

\begin{lemma} \label{lm:relationsB}
    Let $(C_1,\calO_1),\,(C_2,\calO_2),\,(C_3,\calO_3),\textnormal{ and } (C_4,\calO_4)$ denote disjoint objects of $\OBNA$.  Further, let $C=C_1\sqcup C_2 \sqcup C_3 \sqcup C_4$, $\calO=\calO_1\sqcup\calO_2\sqcup\calO_3\sqcup\calO_4$, and $\calO'=\calO_1\sqcup\calO_3\sqcup\calO_2\sqcup\calO_4$. Assuming $\calO$ and $\calO'$ are both admissible, then
    \begin{equation}        \label{eq:Rel3}
        \textnormal{R}_{C,\calO,\calO'}=\mathbbm{1}_{C_1} \rlunion \textnormal{R}(C_2,C_3)\rlunion \mathbbm{1}_{C_4}.
    \end{equation}  
\end{lemma}
\begin{lemma}\label{lm:relationsC}
    Let $S_m$ (resp. $S_{\Delta}$) denote the cobordism with a disjoint union of components in which one component contains a single merge (resp. split) saddle critical point and no other critical points, and all other components are identity cobordisms.  Let $C=C_i\sqcup C_j$ denote the bottom (resp. top) boundary of $S_m$ (resp. $S_\Delta$) and let $\calO=\calO_i\sqcup \calO_j$ be an admissible ordering on $\pi_0 (C)$.  Let $\calO'=\calO_j\sqcup\calO_i$ be the ordering obtained from $\calO$ by reversing the order of the two inputs (resp. outputs) of the component of $S_m$ (resp. $S_\Delta$) that contains the merge (resp. split), and let $S_m'$ (resp. $S_\Delta'$) be the same underlying cobordism as $S_m$ (resp. $S_\Delta$) but with this new ordering.  Then, assuming $\calO'$ is also admissible,
    \begin{equation}    \label{eq:Rel4}
        S_m\circ\textnormal{R}(C_j,C_i)=S_m'
    \end{equation}
    \begin{equation}    
        \textnormal{R}(C_i,C_j)\circ S_{\Delta} = S_{\Delta}' \label{eq:Rel5}
    \end{equation}
\end{lemma}

\begin{remark}
The assumption in Lemma~\ref{lm:relationsC} that both $\calO$ and $\calO'$ are admissible implies that at least one of the inputs (resp. outputs) on the component that contains $S_m$ (resp. $S_\Delta$) must be an inessential circle.
\end{remark}
\begin{proof}[Proof of Lemmas \ref{lm:relationsA}, \ref{lm:relationsB}, \ref{lm:relationsC}]
    The relations in each of these lemmas follows because the morphisms that appear on both sides of the relation have the same source object and the same target object in $\OBNA$, as well as the same underlying chronological cobordism.
\end{proof}

\begin{lemma} \label{lm:relationsD} 
    Let $S_1,\ldots,S_n$ denote chronological cobordisms representing morphisms $$S_i:(C_{i-1},\calO_{i-1})\to(C_{i},\calO_{i})$$ so that $S=S_n\circ\ldots\circ S_1:(C_0,\calO_0)\to (C_n,\calO_n)$ in $\OBNA$.  Suppose $S'=S_n'\circ\ldots\circ S_1'$ is similarly defined with $S_i':(C_{i-1}',\calO_{i-1}')\to(C_{i}',\calO_{i}')$.  If there are chronological isotopies $\phi_i : S_i\Rightarrow S_i'$ which are compatible with the given admissible orders, and if $\phi_1$ fixes $C_0$, $\phi_n$ fixes $C_n$, and $\phi_i,\,\phi_{i+1}$ agree on $C_i$, then we have:
    \begin{equation}
    S_n\circ\ldots\circ S_1 = S_n'\circ\ldots\circ S_1'.\label{eqn:BNRelisotopy}
\end{equation}
\end{lemma}
\begin{proof}
This relation holds in $\OBNA$ because in the definition of $\BNA$, cobordisms are considered up to chronological isotopy relative to the boundary $\partial$.
\end{proof}

\subsection{Admissible Factorizations}\label{sec:AdmissFactorizations}

In this subsection we will discuss factorizations of chronological cobordisms representing morphisms in $\OBNA$.

We first aim to extend the notion of an admissible order to cobordisms.  Let $S\subseteq \ann\times\textnormal{I}$ be a chronological cobordism, and let $S_1,\ldots,S_n$ denote the components of $S$ that are chronologically isotopic to identity cobordisms of essential circles, possibly decorated by dots.  We may refer to these as \textbf{essential vertical cylinders}.  Let $V_i$ denote the connected component of $(\ann\times \textnormal{I})\setminus S_i$ that contains the outer boundary of the annulus $\ann\times\{0\}$.  We define a natural total order on the $S_i$ by setting $S_i\leq S_j$ whenever $V_i \subseteq V_j$.  This total order on the $S_i$ extends to a natural partial order on all components of $S$ given by setting $S'<S_i$ whenever $S'$ is a component of $S$ with $S'\subseteq V_i$, and $S_i\leq S'$ whenever $S'$ is a component  with $S'\nsubseteq V_i$.

\begin{definition} \label{dfn:admissibleOrderS}
A total order on the components of $S$ is called \textbf{admissible} if it extends the partial order described above. Note that admissible orders always exist because every partial order can be extended to a total order.
\end{definition}

In the following lemma, we say a map $s$ between posets is called order-preserving if $C'\leq C''$ implies $s(C')\leq s(C'')$ for all $C',\,C''$ in the domain of $s$.

\begin{lemma} \label{lm:StotalOrder}
Suppose $S$ contains at most one critical point, and let $C$ and $C'$ denote the lower and the upper boundary of $S$.
\begin{enumerate}
    \item For any pair of total orders on $\pi_0(C)$ and $\pi_0(C')$, there is at most one compatible total order on $\pi_0 (S)$, where compatible means that the inclusion-induced maps $\pi_0(C)\rightarrow \pi_0(S)$ and $\pi_0(C')\rightarrow \pi_0(S)$ are order-preserving. \label{lm:TotalOrderSa}
    \item For any admissible order on $\pi_0(S)$, there exist compatible admissible orders on $\pi_0(C)$ and $\pi_0(C')$. \label{lm:TotalOrderSb}
\end{enumerate}
\end{lemma}

\begin{proof}

Under the assumption on $S$, at least one of the two maps $\pi_0 (C)\to \pi_0(S)$ is surjective.  Hence any pair of total orders on $\pi_0(C)$ and $\pi_0(C')$ uniquely determines the compatible total order on $\pi_0 (S)$, provided this compatible total order exists.  This proves the first statement.

To prove the second, let $E_1,\ldots , E_n$ denote the essential components of $C$, and let $s$ denote the map $\pi_0(C)\to\pi_0(S)$.  Further, suppose an admissible total order on $\pi_0(S)$ is given.  To prove that there is a compatible admissible order on $\pi_0(C)$, it suffices to show that under the map $s$, the given admissible order on $\pi_0(S)$ pulls back to a partial order on $\pi_0(C)$ that is consistent with the natural total order on the $E_i$ described prior to Definition~\eqref{dfn:AdmissibleOrder}.  More concretely, it suffices to show that if $s(E_i)>s(E_j)$ with respect to the given admissible order, then $E_i>E_j$ with respect to the natural order.

Thus, suppose $s(E_i)>s(E_j)$.  By the assumption on $S$, there can be at most one component of $S$ that has an essential boundary and that is not chronologically isotopic to an essential vertical cylinder.  Hence at least one of the components $s(E_i)$ and $s(E_j)$ is an essential vertical cylinder.  We will only consider the case where $s(E_i)$ is of this type, as the other case is similar.  Then the assumption $s(E_i)> s(E_j)$ implies $V_i\supseteq s(E_j)$, and this implies that $U_j \subsetneqq U_i$, where $U_i$ and $U_j$ are as in the discussion prior to Definition~\ref{dfn:AdmissibleOrder} and $V_i$ is as in the the discussion prior to Definition~\ref{dfn:admissibleOrderS}.  Hence $E_i>E_j$ with respect to the natural order, as desired.  This shows that there is a compatible admissible order on $\pi_0(C)$. The proof that there is a compatible admissible order on $\pi_0(C')$ is analogous.
\end{proof}

Now let $S\subseteq \ann\times\textnormal{I}$ be an arbitrary chronological cobordism (possibly decorated by dots) with lower and upper boundaries $C$ and $C'$.  Assume that admissible orders $\calO$ and $\calO'$ on $\pi_0(C)$ and on $\pi_0(C')$ are given, and let $S=S_n\circ\ldots\circ S_1$ be a decomposition of $S$ into chronological cobordisms $S_i:C_{i-1}\to C_i$ such that $C_0=C$ and $C_n=C'$.  Further, assume that the $\pi_0(C_i)$ are equipped with admissible orders $ \calO_{i}$ with $ \calO_0=\calO$ and $ \calO_n=\calO'$ so that the decomposition $S=S_n\circ\ldots\circ S_1$ can be viewed as a factorization of $S:(C,\calO)\to (C',\calO')$ in the category $\OBNA$.

\begin{definition} \label{def:cobtype}
The factorization $S=S_n\circ\ldots\circ S_1$, together with the admissible orders $ \calO_{i}$ is \textbf{admissible} if each $S_i$ is of one of the following two types:
    \begin{enumerate}
        \item[I:] $S_i$ contains a single critical point and no dots, or a single dot and no critical points.  Moreover, there is an admissible order on $\pi_0(S_i)$ which is compatible with $\calO_{i-1}$ and $ \calO_{i}$
        \item[II:] $S_i$ contains no dots and no critical points.
    \end{enumerate}
We will refer to these as \textbf{type I} and \textbf{type II} cobordisms.
\end{definition}

\begin{remark}\label{rem:typeIcobordism}
If $S_i$ is of type I, then Lemma~\ref{lm:StotalOrder} implies that the compatible admissible order on $\pi_0(S_i)$ is unique.  In particular, this means that $S_i$ can be written uniquely as 
\begin{equation*}
S_i=S_{i,1}\rlunion \dots \rlunion S_{i,n_i}    
\end{equation*} where the union is understood as in \eqref{dfn:Scobordism}, and where the $S_{i,j}$ are the connected components of $S_i$, numbered so that $S_{i,1}<\ldots<S_{i,n_i}$ with respect to the compatible admissible order on $\pi_0(S_i)$.  Moreover, since $S_i$ contains only one dot or critical point, all but one of the $S_{i,j}$ are chronologically isotopic to identity cobordisms.
\end{remark}

\begin{remark}\label{rem:typeIIcobordism}
If $S_i$ is of type II, then the $S_i$ - viewed as a morphism from $(C_{i-1},\calO_{i-1})$ to $(C_i, \calO_{i})$ - is chronologically isotopic, relative to $C_{i-1}$, to the isomorphism $\textnormal{R}_{C_{i-1},\calO_{i-1},\calO_{i-1}'}$ for a particular admissible order $\calO_{i-1}'$ on $\pi_0(C_{i-1})$.  To see this, first note that the inclusion-induced maps $s:\pi_0(C_{i-1})\to \pi_0(S_i)$ and $t:\pi_0(C_i)\to\pi_0(S_{i})$ are invertible in a type II cobordism because each component has exactly one lower and exactly one upper boundary component.   Thus, $\calO_{i-1}'=\sigma\circ \calO_{i-1}$, where $\sigma$ is the permutation given by
\[
\begin{tikzcd}
{\{1,\ldots,n_{C_{i-1}}\}} \arrow[r, "\mathcal{O}_{i-1}^{-1}"] & \pi_0(C_{i-1}) \arrow[r, "s"] & \pi_0(S_i) \arrow[r, "t^{-1}"] & \pi_0(C_{i}) \arrow[r, "\mathcal{O}_i"] & {\{1,\ldots,n_{C_{i-1}}\}}.
\end{tikzcd}
\]
\end{remark}

\begin{remark}
The permutation $\sigma$ from the previous remark has the property that it restricts to an order-preserving map on those $j$ for which $\calO_{i-1}^{-1}(j)$ is an essential component of $C_{i-1}$.  Indeed, this follows because both of the orders $\calO_{i-1}$ and $\calO_{i}$ are admissible, and because under the maps $s$ and $t$, the natural total orders on the essential components of $C_{i-1}$ and $C_i$ both correspond to the natural total order on the components of $S_i$ that are chronologically isotopic to essential vertical cylinders.
\end{remark}

The following lemma can be understood as saying that morphisms of types I and II generate the morphisms in $\OBNA$.

\begin{lemma}\label{lem:factorizationexistence}
Every chronological cobordism $S\colon(C,O)\rightarrow(C',O')$ in the thickened annulus admits an admissible factorization.
\end{lemma}

\begin{proof}
Given a chronological cobordism $S\subset\ann\times I$, we start by choosing a subdivision $0=t_0<\ldots<t_{2n+1}=1$ of $I$ such that $S_i:= S \cap (\ann\times[t_{i-1},t_i])$ contains 
\begin{itemize}
\item
no dots or critical points if $i$ is odd, and
\item
exactly one dot or critical point if $i$ is even.
\end{itemize}
By part~2 of Lemma~\ref{lm:StotalOrder}, we can then choose for each even $i$ an admissible order on $\pi_0(S_i)$, together with compatible admissible orders on $\pi_0(C_{i-1})$ and on $\pi_0(C_i)$. With respect to these orders, the cobordism $S_i$ is of type I if $i$ is even, and of type II if $i$ is odd.  Thus, the factorization of $S$ given by $S=S_{2n+1}\circ\ldots\circ S_1$ is admissible.
\end{proof}

While admissible factorizations are not unique, we have:

\begin{lemma}\label{lem:factorizationrelations}
Any two admissible factorizations of $S$ can be transformed into each other by repeatedly applying the relations~\eqref{eqn:BNRid}, \eqref{eqn:BNRcomp}, the relations~\eqref{eq:Rel1}, \eqref{eq:Rel2}, \eqref{eq:Rel3}, \eqref{eq:Rel4}, \eqref{eq:Rel5}, and \eqref{eqn:BNRelisotopy}, as well as the relation 
\begin{equation} \label{eq:SIdentity}
    S_i=\mathbbm{1}_{C_i}\circ S_i=S_i\circ \mathbbm{1}_{C_{i-1}}.
\end{equation}
\end{lemma}

\begin{proof}
We first consider the special case where $S$ is itself a cobordism of type I or II and $S_1$ and $S_2$ are the cobordisms $S_1:=S\cap(\ann\times [0,t])$ and $S_2:= S\cap (\ann\times [t,1])$ for a generic $t\in \textnormal{I}$.  In this case, it is clear that at most one of the two cobordisms $S_1$ and $S_2$ can contain a dot or a critical point.  It follows that there are isotopies $\phi_1$ and $\phi_2$ as in~\eqref{eqn:BNRelisotopy} which take one of the two cobordisms $S_1$ and $S_2$ to an identity cobordism, while taking the other one to a re-scaled copy of $S$ itself.  These isotopies take the components of $C'':=S\cap(\ann\times\{t\})$ bijectively to the components of the lower or the upper boundary of $S$, and thus there is a unique (and necessarily admissible) total order on $\pi_0(C'')$ which corresponds to the given total order on the relevant boundary of $S$ under this bijection.  If we equip $\pi_0(C'')$ with this total order, it then follows that the factorization $S=S_2\circ S_1$ is admissible, and related to one of the two factorizations $S=\mathbbm{1}\circ S$ and $S=S\circ\mathbbm{1}$ via~\eqref{eqn:BNRelisotopy}.  Thus, in this special case there is an admissible order on $\pi_0(C'')$ for which $S=S_2\circ S_1$ is an admissible factorization which can be transformed into the trivial factorization $S=S$ using relations ~\eqref{eqn:BNRelisotopy} and \eqref{eq:SIdentity}.

Now suppose $S\subseteq\ann\times\textnormal{I}$ is a general chronological cobordism, and $S=S_m\circ\ldots\circ S_1$, and $S=S_n'\circ\ldots\circ S_1'$ are two admissible factorizations of $S$.  Suppose these factorizations are obtained respectively by cutting $S$ along annuli $\ann\times\{t_i\}$ and $\ann\times\{t_j'\}$ for two partitions $0=t_0<\ldots <t_m=1$ and $0=t_0'<\ldots <t_n'=1$ of I.  Further, let 
$\{t_k''\} \supseteq \{t_i\}\cup\{t_j'\}$ be a common refinement of these two partitions, and let $S=S_l''\circ\ldots\circ S_1''$ be the decomposition obtained by cutting $S$ along the $\ann\times\{t_k''\}$.  Using the special case above and induction on $l-m$, it is then easy to see that there are admissible orders on the components of the $C_k'':=S\cap[\ann\times\{t_k''\}]$ for which $S=S_l''\circ\ldots\circ S_1''$ is an admissible factorization, and related to the factorization $S=S_m\circ\ldots\circ S_1$ via relations~\eqref{eqn:BNRelisotopy} and \eqref{eq:SIdentity}.  Likewise, there are (potentially different) admissible orders on the $\pi_0(C_k'')$ for which $S=S_l''\circ\ldots\circ S_1''$ is admissible, and related to the factorization $S=S_n'\circ\ldots\circ S_1'$ via relations~\eqref{eqn:BNRelisotopy} and \eqref{eq:SIdentity}.

The lemma now follows if we can prove that the relations mentioned in the lemma can be used to reorder components.  Thus, let $S\subseteq\ann\times\textnormal{I}$ be a general chronological cobordism, and let $S=S_n\circ\ldots\circ S_1$ and $S=S_n'\circ\ldots\circ S_1'$ be two admissible factorizations such that $S_i$ and $S_i'$ are identical cobordisms for all $i$, but with $S_i$ being viewed as a morphism from $(C_{i-1},\calO_{i-1})$ to $(C_i, \calO_{i})$, and $S_i'$ being viewed as a morphism from $(C_{i-1},\calO_{i-1}')$ to $(C_i, \calO_{i}')$ for closed 1-manifolds $C_i\subseteq\ann$ and admissible orders $ \calO_{i}$ and $ \calO_{i}'$.

Using~\eqref{eq:SIdentity}, we can insert between any two consecutive factors $S_{i+1}$ and $S_i$ of $S=S_n\circ\ldots\circ S_1$ a factor of $\mathbbm{1}_{C_i}$, and using~\eqref{eqn:BNRid} and~\eqref{eqn:BNRcomp}, we can replace $\mathbbm{1}_{C_i}$ by $\textnormal{R}_{C_i, \calO_{i}', \calO_{i}}\circ\textnormal{R}_{C_i, \calO_{i}, \calO_{i}'}$. The factorization $S=S_n\circ\ldots\circ S_1$ then becomes a product of the terms 
\[\textnormal{R}_{C_i, \calO_{i}, \calO_{i}'}\circ S_i\circ \textnormal{R}_{C_{i-1},\calO_{i-1}', \calO_{i-1}}\]
and we will now complete our proof by showing that each of these terms can be transformed into $S_i'$ by using the relations stated in the lemma.  We will distinguish between several cases: 

\textbf{Case 1:} Suppose $S_i$ is of type I, and suppose that it does not contain a saddle critical point.  By Remark~\ref{rem:typeIcobordism}, we can then write $S_i$ uniquely as a union $$S_i=S_{i,1}\rlunion\ldots\rlunion S_{i,n_i}$$ of its components . Because $S_i$ contains no saddles, each $S_{i,j}$ has at most one lower boundary component and at most one upper boundary component, thus there is no ambiguity in how these components can be ordered in $S_{i,j}$.  It follows that there is a permutation $\sigma\in \mathfrak{S}_{n_c}$ such that $$S_i'=S_{i,\sigma^{-1}(1)}\rlunion\ldots\rlunion S_{i,\sigma^{-1}(n_i)}.$$ Explicitly, $\sigma=\calO_{i-1}'\circ \calO_{i-1}^{-1}$ or $\sigma= \calO_{i}'\circ \calO_{i}^{-1}$, depending on whether $\pi_0(C_{i-1})\to\pi_0(S_i)$ or $\pi_0(C_i)\to\pi_0(S_i)$ is a bijection.  Since all but one of the $S_{i,j}$ can be thought of as being an identity cobordism, and since we can factor $\sigma$ as a product of transpositions, we can now use relations \eqref{eq:Rel1}, \eqref{eq:Rel2}, and \eqref{eq:Rel3} repeatedly to conclude that 
\begin{align*}
    \textnormal{R}_{C_i, \calO_{i}, \calO_{i}'}\circ S_i\circ\textnormal{R}_{C_i,\calO_{i-1}',\calO_{i-1}} & = \textnormal{R}_{C_i, \calO_{i}, \calO_{i}'}\circ\textnormal{R}_{C_i, \calO_{i}', \calO_{i}}\circ S_i'\\
    & =\mathbbm{1}_{C_i}\circ S_i'\\
    & =S_i'
\end{align*} as desired, where the last two equations follow from
relations~\eqref{eqn:BNRid} and~\eqref{eqn:BNRcomp} and from relation~\eqref{eq:SIdentity}.

\textbf{Case 2:} Suppose $S_i$ is a type I cobordism which contains a saddle critical point.  Then the component $S_{i,j}\subseteq S_i$ containing the critical point has either two lower boundary components or two upper boundary components.  We now claim that, after applying~\eqref{eqn:BNRcomp}, \eqref{eq:Rel4}, and~\eqref{eq:Rel5} if necessary, we can assume that the relevant primed and unprimed total orders in $$\textnormal{R}_{C_i, \calO_{i}, \calO_{i}'}\circ S_i\circ \textnormal{R}_{C_{i-1},\calO_{i-1}',\calO_{i-1}}$$ agree on the two lower or upper boundary components of $S_{i,j}$.  We can then argue as in Case 1 to relate the product above to $S_i'$.

To prove the claim, suppose for concreteness that $S_{i,j}$ has two upper boundary components $B_1$ and $B_2$.  Suppose further that $ \calO_{i}$ and $ \calO_{i}'$ disagree on $\{B_1,B_2\}$.  Since $ \calO_{i}$ and $ \calO_{i}'$ are both admissible, it follows that at most one of the components $B_1$ and $B_2$ is essential, and hence the total order $ \calO_{i}''$ obtained from $ \calO_{i}$ by reversing the order of $\{B_1,B_2\}$ is again admissible.  Likewise, if $S_i''$ denotes $S_i$ but viewed as a morphism from $(C_{i-1},\calO_{i-1})$ to $(C_i, \calO_{i}'')$, then $S_i''$ is again a type I cobordism.  Using~\eqref{eqn:BNRcomp} and \eqref{eq:Rel5}, we now see that \begin{align*}
    \textnormal{R}_{C_i, \calO_{i}, \calO_{i}'}\circ S_i\circ \textnormal{R}_{C_{i-1},\calO_{i-1}',\calO_{i-1}} &= \textnormal{R}_{C_i, \calO_{i}, O_1'} \circ \textnormal{R}_{C_i, \calO_{i}'', \calO_{i}} \circ S_i'' \circ \textnormal{R}_{C_{i-1},\calO_{i-1}',\calO_{i-1}}\\
    &=\textnormal{R}_{C_i, \calO_{i}'', \calO_{i}'}\circ S_i''\circ \textnormal{R}_{C_{i-1},\calO_{i-1}',\calO_{i-1}}
\end{align*}
and since the orders $ \calO_{i}$ and $ \calO_{i}'$ agree on $\{B_1,B_2\}$, this proves the claim.

\textbf{Case 3:}  Suppose $S_i$ is of type II. Remark~\ref{rem:typeIIcobordism} then tells us that $S_i$ is chronologically isotopic relative to $C_{i-1}$ to $\textnormal{R}_{C_{i-1},\calO_{i-1},\sigma\circ \calO_{i-1}}$ for some permutation $\sigma$. Suppose first that $C_i=C_{i-1}$, and that $S_i$ is not just isotopic, but equal to $\textnormal{R}_{C_{i-1},\calO_{i-1}, \calO_{i}}$.  In this case, $S_i'$ is equal to $\textnormal{R}_{C_{i-1},\calO_{i-1}', \calO_{i}'}$, and using~\eqref{eqn:BNRcomp} we get: 
\begin{align*}
    \textnormal{R}_{C_i, \calO_{i}, \calO_{i}'}\circ S_i\circ \textnormal{R}_{C_{i-1},\calO_{i-1}',\calO_{i-1}} &= \textnormal{R}_{C_i, \calO_{i}, \calO_i'} \circ \textnormal{R}_{C_i,\calO_{i-1}, \calO_{i}}\circ \textnormal{R}_{C_{i-1},\calO_{i-1}',\calO_{i-1}}\\
    &=\textnormal{R}_{C_i,\calO_{i-1}', \calO_{i}'} \\
    &=S_i'
\end{align*}

We can now generalize this argument as follows.  Let $S_i$ be a general type II cobordism and let $\M_{i-1}:=\sigma \circ \calO_{i-1}$ and $\M_{i-1}':=\sigma'\circ \calO_{i-1}'= \calO_{i}'\circ \calO_{i}^{-1}\circ \M_{i-1}$ where $\sigma:= \calO_{i}\circ t^{-1}\circ s \circ \calO_{i-1}^{-1}$ and $\sigma':= \calO_{i}'\circ t^{-1}\circ s\circ (\calO_{i-1}')^{-1}$, for $s$ and $t$ as in Remark~\ref{rem:typeIIcobordism}.  Further, let $S_i''$ denote the cobordism $S_i$, but viewed as a morphism from $(C_{i-1},\M_{i-1}')$ to $(C_i, \calO_{i}')$.  We then have \begin{align*}
    \textnormal{R}_{C_i, \calO_{i} \calO_{i}'}\circ S_i\circ \textnormal{R}_{C_{i-1},\calO_{i-1}',\calO_{i-1}} &= \mathbbm{1}_{(C_i, \calO_{i}')}\circ \textnormal{R}_{C_i, \calO_{i}, \calO_{i}'}\circ S_i\circ \textnormal{R}_{C_{i-1},\calO_{i-1}',\calO_{i-1}}\\
    & =S_i''\circ\textnormal{R}_{C_{i-1},\M_{i-1},\M_{i-1}'}\circ\textnormal{R}_{C_{i-1},\calO_{i-1},\M_{i-1}}\circ\textnormal{R}_{C_{i-1},\calO_{i-1}',\calO_{i-1}}\\
    & = S_i''\circ\textnormal{R}_{C_{i-1},\calO_{i-1}',\M_{i-1}'}\\
    & = \mathbbm{1}_{(C_i, \calO_{i}')}\circ S_i'\\
    &= S_i'  
\end{align*} where the first and the last equation follow from relation~\eqref{eq:SIdentity}, the third equation follows from relation~\eqref{eqn:BNRcomp}, and the second and the fourth equation follow because there are chronological isotopies as in~\eqref{eqn:BNRelisotopy} which take each factor on one side of each of these two equations to the corresponding factor on the other side.
\end{proof}
\begin{remark}
Relation~\eqref{eq:Rel2} is redundant, in the sense that it can be deduced from relation~\eqref{eq:Rel1} by composing from the left with $\textnormal{R}(C_2,C_1')$ and from the right with $\textnormal{R}(C_2,C_1)$, and using relations~\eqref{eqn:BNRid}, \eqref{eqn:BNRcomp}, and~\eqref{eq:SIdentity} (and relation~\eqref{eq:Rel3} for $C_1=\emptyset =C_4)$.
\end{remark}

\begin{remark}
In analogy with the category $\OBNA$, one can define an ordered version, $\OBNR$, of the non-annular odd Bar-Natan category $\mathcal{BN}_{\!o}(\mathbbm{R}^2)$.  The difference with the annular case is that, in the non-annular setting, every component of every object $C\subseteq\mathbbm{R}^2$ is trivial, and hence every total order on $\pi_0(C)$ is admissible.  The analogs of Lemmas~\ref{lem:factorizationexistence} and~\ref{lem:factorizationrelations} remain true in the non-annular setting.
\end{remark}
%\end{document}
    
\section{Embedding $\mathcal{I}$}\label{sec:Chapter4}
    %\documentclass[../main.tex]{subfiles}
%
%\begin{document}

In this section we introduce a monoidal superfunctor, $\mathcal{I}:\TL (0)\to\BNA$, and show that it is full.  In the next two sections we will give two proofs that $\mathcal{I}$ is faithful.  

\subsection{The main functor}
On objects, $\mathcal{I}$ is defined by
$$n\geq 0\,\,\longmapsto\,\mbox{$n$ essential circles}.$$ 

On generating morphisms we have

\begin{align*}
    \includegraphics[angle=180,origin=c,valign=c]{TLimages/TLcup.pdf} &\longmapsto \includegraphics[valign=c]{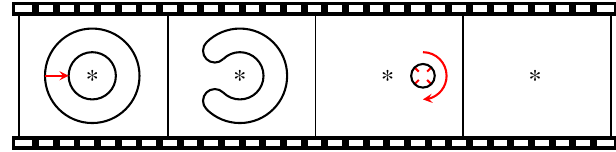}\\
    \includegraphics[valign=c]{TLimages/TLcup.pdf} &\longmapsto \includegraphics[valign=c]{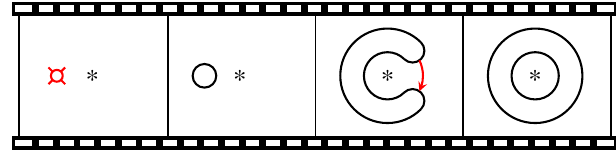}\\
    \includegraphics[valign=c]{TLimages/TLidentityDot.pdf} &\longmapsto \includegraphics[valign=c]{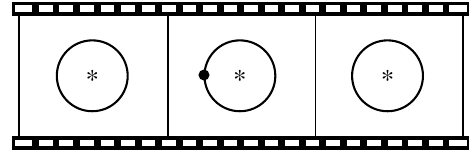}\\
\end{align*}

An intuitive way to visualize $\mathcal{I}$ is to rotate the objects in $\TL (0)$ around the annulus, but in such a way that they result in a chronological cobordism, meaning critical points are separated.

\begin{theorem}
$\mathcal{I}$ is well defined.
\end{theorem}

\begin{proof}

Because $\mathcal{I}$ is defined on generators, it suffices to check that the images of the relations on morphisms in $\TL (0)$ hold in $\BNA$.  There are 6 relations to verify.

For the Isotopy Relations, we have
    \begin{align}
        \includegraphics[valign=c]{TLimages/TLisotopyLthenR.pdf}&\longmapsto\includegraphics[valign=c]{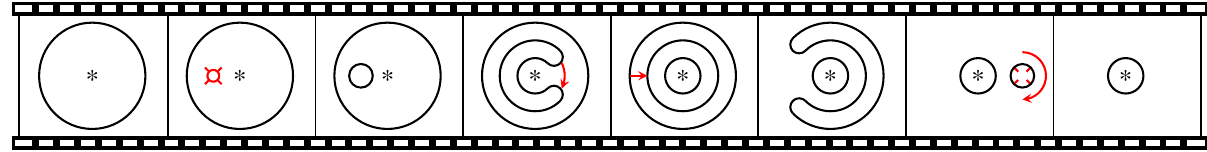}\label{im:LtoRSquiggle} \\
        \reflectbox{\includegraphics[valign=c]{TLimages/TLisotopyLthenR.pdf}}&\longmapsto \includegraphics[valign=c]{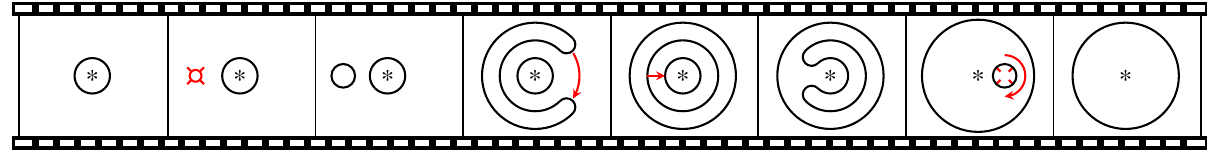} \label{im:RtoLSquiggle}.
    \end{align}
In \eqref{im:LtoRSquiggle}, it is clear that this cobordism can be isotoped to the identity cobordism.  Both \eqref{im:LtoRSquiggle} and the identity cobordism are degree $0$ (mod 2), so they are equivalent.  In \eqref{im:RtoLSquiggle}, if the arrow for the split in the fourth panel is rotated 90 degrees in the direction of the death, it will point towards the death.  Therefore, in order to use relation \eqref{eqn:BNcreation} and isotope this to the identity, we need to reverse the orientation of the split or the death.  This introduces a negative by relation \eqref{eqn:BNorientation}, so we have that \eqref{im:RtoLSquiggle} goes to minus the identity cobordism.

Moving on to the circle relation, \eqref{eqn:TLcircles}, we note that both the dotted and undotted circles in $\TL (0)$ go to a torus in $\BNA$.  The movie for the undotted torus is below.
\[  \includegraphics[valign=c]{TLimages/TLcircle.pdf}\longmapsto \includegraphics[valign=c]{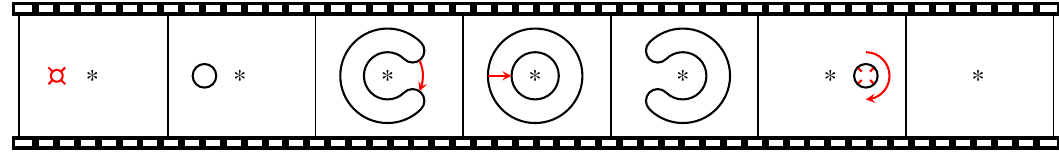}\]
Panels 3 to 5 of this movie are precisely the diamond relation from \eqref{eqn:BNdiamond}, which is shown to be zero in Lemma~\ref{lem:diamondzero}.  A dotted torus in $\BNA$ also contains the diamond relation, so the image of a dotted circle also evaluates to zero.

We have several dot relations to check.  First the dotted caps \eqref{eqn:TLdotslide}, which go to the following movies.
\begin{align*}
   \reflectbox{\includegraphics[angle=180,origin=c,valign=c]{TLimages/TLcupDotL.pdf}}&\longmapsto \includegraphics[valign=c]{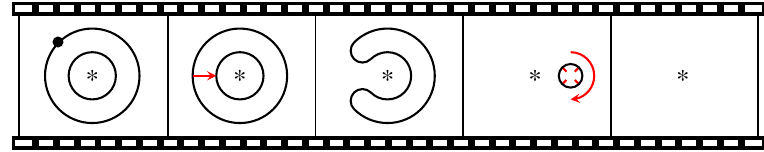} \\
   \includegraphics[angle=180,origin=c,valign=c]{TLimages/TLcupDotL.pdf}& \longmapsto \includegraphics[valign=c]{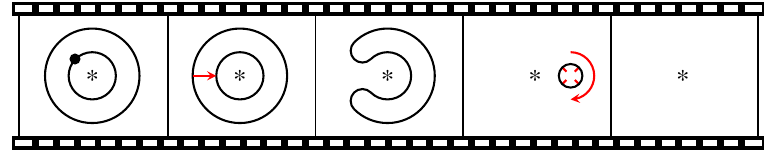} \\
\end{align*}
To move from the first movie to the second we slide the dot on the outer essential circle in the first panel upwards along the surface, past the merge in panel two, then along the level curve in panel three and down along the surface again to the inner essential circle in the first panel.  A merge has degree zero and a dot has degree 1, so no sign change results from these critical points passing each other.  The proof for the dotted cup relation, \eqref{eqn:TLdotslide}, is similar and thus omitted.

A surface with two dots in $\BNA$ evaluates to zero, satisfying the image of relation~\eqref{eqn:TLtwodots}.  

To show final dot relation, \eqref{eqn:TLfourterm}, we will show that both sides of the equation are equivalent to the following movie.
\begin{equation}
    \includegraphics[valign=c]{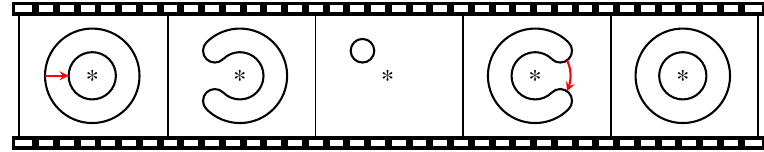}\label{eqn:MovieVneck}
\end{equation}
This movie can be visualized as rotating \includegraphics[valign=c,scale=0.5]{TLimages/TLcup.pdf} around the top of $\ann\times I$ and \includegraphics[angle=180,origin=c,scale=0.5]{TLimages/TLcup.pdf} around the bottom, then connecting the two components with a vertical neck.  The resulting cobordism is chronologically isotopic to the cobordism represented by the following movie in $\BNA$.
\begin{equation}
    \includegraphics[valign=c]{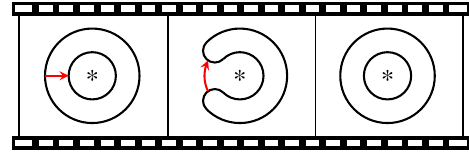} \label{eqn:MovieHneck}
\end{equation}

To show \eqref{eqn:TLfourterm} carries over, it will suffice to show that the first of these two movies is equivalent to the image of the left hand side of \eqref{eqn:TLfourterm} under $\mathcal{I}$, and the second movie is equivalent to the image of the right hand side.
\begin{align*}
    \eqref{eqn:MovieVneck}\,\,
    {=}&\includegraphics[valign=c,scale=0.8]{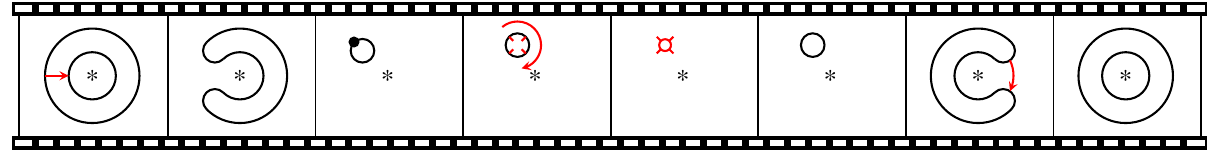}\\
    &\,+\includegraphics[valign=c,scale=0.8]{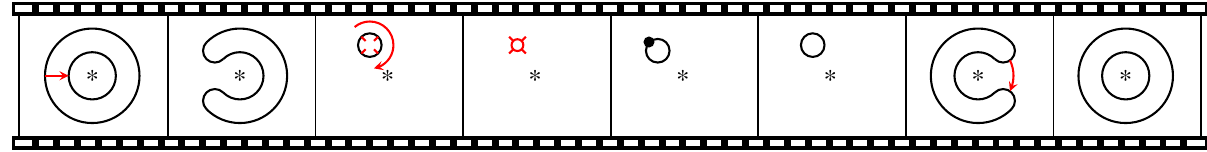}\\
    {=}&\includegraphics[valign=c,scale=0.8]{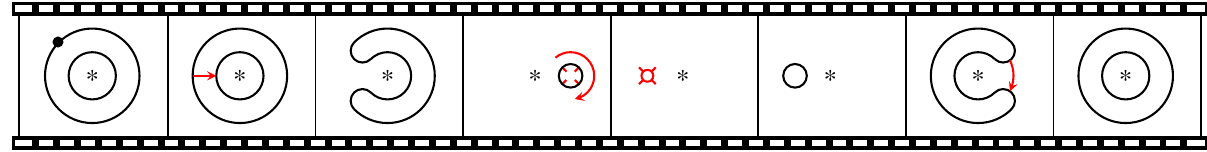}\\
    &\,-\includegraphics[valign=c,scale=0.8]{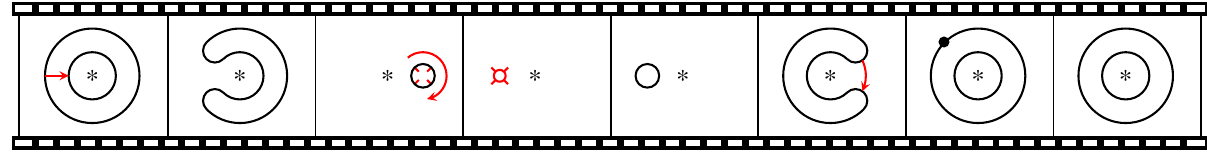}\\
    {=}&\,\,\,\,\,\,\mathcal{I}\mleft(\includegraphics[valign=c,scale=0.6]{TLimages/TLrelDotBottomL.pdf}\mright)-\mathcal{I}\mleft(\includegraphics[angle=180,origin=c,valign=c,scale=0.6]{TLimages/TLrelDotBottomL.pdf}\mright)
\end{align*}
We cut the vertical neck in \eqref{eqn:MovieVneck} to get the first equality, then slide the dots towards the outer panels for the second equality. When sliding the dot down to the first panel it passes
a merge, so no net change in sign.  When sliding it up in the second term, the dot passes
a split, which introduces a negative. 
The last equality follows from the definition of $\mathcal{I}$.
\begin{align*}
    \eqref{eqn:MovieHneck}\,\,
    {=}&\includegraphics[valign=c,scale=0.8]{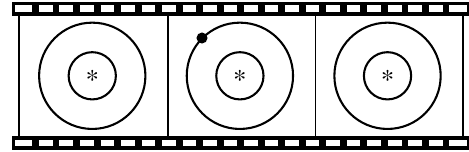}-\includegraphics[valign=c,scale=0.8]{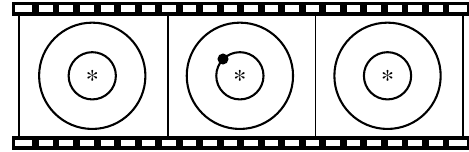}\\
    {=}&\,\,\,\,\,\,\mathcal{I}\mleft(\includegraphics[valign=c,scale=0.6]{TLimages/TLidentityDot.pdf}\includegraphics[valign=c,scale=0.6]{TLimages/TLidentity.pdf}\mright)-\mathcal{I}\mleft(\includegraphics[valign=c,scale=0.6]{TLimages/TLidentity.pdf}\includegraphics[valign=c,scale=0.6]{TLimages/TLidentityDot.pdf}\mright)
\end{align*}
We use that the two saddles in~\eqref{eqn:MovieHneck} form a horizontal tube as in Lemma~\ref{lem:horizontalneck}. Applying the horizontal neck cutting relation~\eqref{eqn:horizontalneck} to this tube yields the first equality, and the second equality follows again from the definition of $\mathcal{I}$.
\end{proof}
\subsection{Proof that $\mathcal{I}$ is full}
\begin{proposition}\label{prop:Ifull}
     $\mathcal{I}$ is full and its extension, $\TL(0)^\oplus\rightarrow\BNA^\oplus$, defined as in~\eqref{eqn:addclosureextension}, is essentially surjective on objects.
\end{proposition}
\begin{proof}
Recall Corollary~\ref{cor:deloopingequivalence} in which we defined a full subcategory of $\BNA$, called $\sBNA$, and showed that the additive closures of those two monoidal supercategories are equivalent. Recall also that $\sBNA$ contains the objects $E_n$, where $E_n$ denotes a collection of $n$ essential circles. The second statement in Proposition~\ref{prop:Ifull} now follows because these are precisely the objects that are in the image of $\mathcal{I}$, and because every object of $\BNA$ is isomorphic to an object of $\sBNA^\oplus$.

To prove that $\mathcal{I}$ is full, we give some details on the morphisms in $\sBNA$, and we prove that they are in the image of $\mathcal{I}$.

Let $S\subseteq \ann\times I$ be a morphism in $\sBNA$ from $E_n$ to $E_m$.  We can decompose $S$ so that $S=S_n\circ \cdots \circ S_1$ and each $S_i$ contains at most one critical point or dot.  Each $S_i$ can be further decomposed as the disjoint union of components $s_{i_j}$ $$S_i=s_{i_1}\rlunion\dots\rlunion s_{i_k} \rlunion\dots\rlunion s_{i_m}$$ where each $s_{i_j}$ for $j\neq k$ is an undotted identity cobordism, either inessential or essential, and $s_{i_k}$ contains exactly one critical point or one dot but not both.  Because $\mathcal{D}:\BNA^\oplus\to\sBNA^\oplus$ restricts to the identity on $\sBNA^\oplus$ and thus satisfies 
\[
S=\mathcal{D}(S)=\mathcal{D}(S_n)\circ\ldots\circ\mathcal{D}(S_1),
\]
it will suffice to show that $\mathcal{D}(S_i)$ is in the image of $\mathcal{I}\colon\TL(0)^\oplus\rightarrow\BNA^\oplus$. In fact, this will show that $S$ is in the image of the extended functor $\mathcal{I}\colon\TL(0)^\oplus\rightarrow\BNA^\oplus$, and since $S$ is in the original category $\BNA$, this will also show that $S$ is in the image of the original functor $\mathcal{I}\colon\TL(0)\rightarrow\BNA$.

We first consider the undotted inessential identity cobordisms in the disjoint union above. The delooping isomorphism defined in Lemma~\ref{lem:delooping} gives the following:

\begin{equation} \label{eqn:identitymatrix}
\mathcal{D}\mleft(\includegraphics[valign=c]{BNimages/CobIdentity.pdf}\mright)=
\begin{bmatrix}
\includegraphics[valign=c]{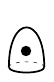}\\
\includegraphics[valign=c]{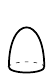}\end{bmatrix}
\circ
\includegraphics[valign=c]{BNimages/CobIdentity.pdf}
\circ
\begin{bmatrix}
\includegraphics[valign=c]{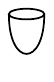}&\hspace*{-0.04in}
\includegraphics[valign=c]{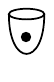}\vspace*{-0.05in}%
\end{bmatrix} 
=
\begin{bmatrix}
\includegraphics[valign=c,scale=0.75]{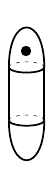}&\includegraphics[valign=c,scale=0.75]{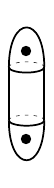}\\
\includegraphics[valign=c,scale=0.75]{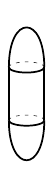}&\includegraphics[valign=c,scale=0.75]{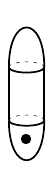}
\end{bmatrix}
=
\begin{bmatrix}
1&0\\0&1
\end{bmatrix}
\end{equation}

After considering each undotted inessential identity cobordism of $S_i$, we are left with a finite number of copies of essential identity cobordisms and $s_{i_k}$, where $s_{i_k}$ is one of the following:
\begin{enumerate}
    \item The components of $\partial s_{i_k}$ consist entirely of inessential circles, as in the following cobordisms:
    \begin{equation} \label{eqn:type1componant}
        \includegraphics[valign=c]{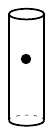} \hspace{1cm} \includegraphics[valign=c]{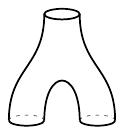} \hspace{1cm}             \includegraphics[valign=c]{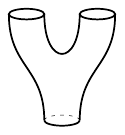} \hspace{1cm} \includegraphics[valign=c]{BNimages/CobBirth.pdf} \hspace{1cm}       \includegraphics[valign=c]{BNimages/CobDeath.pdf} 
    \end{equation}
    \item The critical point  is a dot, split, or merge, and $\partial s_{i_k}$ has exactly one essential circle on both its upper and lower boundaries. 
    \begin{equation} \label{eqn:type2componant}
        \includegraphics[valign=c]{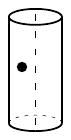}\hspace{1cm} \includegraphics[valign=c,scale=0.75]{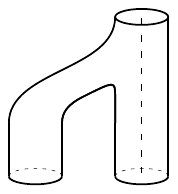}\hspace{1cm} \includegraphics[valign=c,scale=0.75]{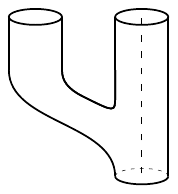} 
    \end{equation}
    \item The critical point of $s_{i_k}$ is a merge of two essential components that results in an inessential component, or it is a split of an inessential component into two essential components. As movies these are:
    \begin{equation} \label{eqn:type3componant}
        \includegraphics[valign=c]{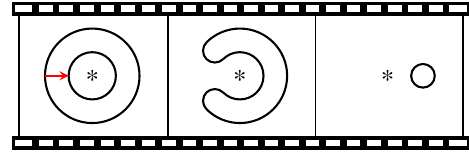}\hspace{.5cm}\textnormal{ and }\hspace{.5cm}\includegraphics[valign=c]{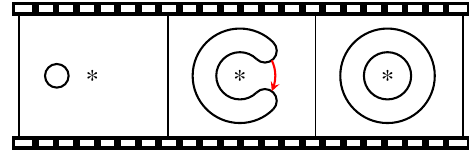}
    \end{equation}
\end{enumerate}

We next consider cobordisms with no essential boundary componants, as in~\eqref{eqn:type1componant}.  The functor $\mathcal{D}$ will send each of these cobordisms to a matrix of ones and zeros, as with \eqref{eqn:identitymatrix}.  The example we give here is of the merge cobordism, but the calculations for the other four of this type are similar. 
\begin{align*}
\mathcal{D}\mleft(\includegraphics[valign=c]{BNimages/CobMerge.pdf}\mright)
&=\begin{bmatrix}
\includegraphics[valign=c]{BNimages/CobDeathDot.pdf}\\
\includegraphics[valign=c]{BNimages/CobDeath.pdf}\end{bmatrix}
\circ
\includegraphics[valign=c]{BNimages/CobMerge.pdf}
\circ
\Bigl({\begin{bmatrix}
\includegraphics[valign=c]{BNimages/CobBirth.pdf}&\hspace*{-0.04in}
\includegraphics[valign=c]{BNimages/CobBirthDot.pdf}\vspace*{-0.05in}%
\end{bmatrix}} \rlunion {\begin{bmatrix}
\includegraphics[valign=c]{BNimages/CobBirth.pdf}&\hspace*{-0.04in}
\includegraphics[valign=c]{BNimages/CobBirthDot.pdf}\vspace*{-0.05in}%
\end{bmatrix}}\Bigr)\\ 
&= 
\begin{bmatrix}
\includegraphics[valign=c]{BNimages/CobDeathDot.pdf}\\
\includegraphics[valign=c]{BNimages/CobDeath.pdf}\end{bmatrix}
\circ
\includegraphics[valign=c]{BNimages/CobMerge.pdf}
\circ
{\begin{bmatrix}
\includegraphics[valign=t]{BNimages/CobBirth.pdf}\includegraphics[valign=t]{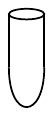}&\hspace*{-0.04in}
\includegraphics[valign=t]{BNimages/CobBirth}\includegraphics[valign=t]{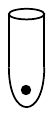}&\hspace*{-0.04in}
\includegraphics[valign=t]{BNimages/CobBirthDot.pdf}\includegraphics[valign=t]{BNimages/CobBirthLong.pdf}&\hspace*{-0.04in}
\includegraphics[valign=t]{BNimages/CobBirthDot.pdf}\includegraphics[valign=t]{BNimages/CobBirthLongDot.pdf}\vspace*{-0.05in}%
\end{bmatrix}}\\
 & =\begin{bmatrix}
1&0&0&0\\
0&1&1&0
\end{bmatrix} 
\end{align*} 
In the first equation, the disjoint right-then-left union of two $1\times 2$ matrices results in a $1\times 4$ matrix, by Remark~\ref{rmk:DisjUnionCalcs}.  We then compose the merge cobordism with the two matrices to get a $2\times 4$ matrix of cobordisms without boundary.  Using relations~\eqref{eqn:BNdisjoint}, \eqref{eqn:BNcreation} and~\eqref{eqn:BNBN}, we end with a matrix of ones and zeros.  

Thus, if $S_i=s_{i_1}\rlunion\dots\rlunion s_{i_k} \rlunion\dots\rlunion s_{i_m}$, where $s_{i_k}$ is one of the cobordisms in \eqref{eqn:type1componant} and all other $s_{i_j}$ are undotted trivial or essential identity cobordisms, then the components of $\mathcal{D}(S_i)$ consists solely of identity matrices \eqref{eqn:identitymatrix}, essential identity cobordisms, and matrices of ones and zeros.  By Remark~\ref{rmk:DisjUnionCalcs}, we tensor all matrices together and the result is a matrix of zeros and essential identity cobordisms.  Recalling that $\mathcal{I}\mleft(\includegraphics[scale=0.5,valign=c]{TLimages/TLidentityShort.pdf} \mright)=\includegraphics[scale=0.5,valign=c]{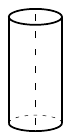}$, we conclude that $\mathcal{D}(S_i)$ is therefore in the image of $I^\oplus$.

In the second type of cobordism \eqref{eqn:type2componant} we use $\mathcal{D}$ to cap off the inessential boundary component and we are left with a matrix of zeros and dotted and undotted essential identity cobordisms. The example we give is of the split cobordism.
\begin{align*}
\mathcal{D}\mleft(\includegraphics[valign=c,scale=0.75]{BNimages/CobSplitEssential.pdf}\mright)
&=\mleft(
\begin{bmatrix}
\includegraphics[valign=c]{BNimages/CobDeathDot.pdf}\\
\includegraphics[valign=c]{BNimages/CobDeath.pdf}\end{bmatrix} \lrunion \includegraphics[valign=c]{BNimages/CobIdentityEssential.pdf} \mright)
\circ
\includegraphics[valign=c,scale=0.75]{BNimages/CobSplitEssential.pdf}\\
 &=
\begin{bmatrix}
\includegraphics[valign=b]{BNimages/CobDeathDot.pdf}\includegraphics[valign=b,scale=0.75]{BNimages/CobIdentityEssential.pdf}\\
\includegraphics[valign=b]{BNimages/CobDeath.pdf}\includegraphics[valign=b,scale=0.75]{BNimages/CobIdentityEssential.pdf}\end{bmatrix}
\circ
\includegraphics[valign=c,scale=0.75]{BNimages/CobSplitEssential.pdf}\\
&=
\begin{bmatrix}
\mp \includegraphics[valign=c,scale=0.75]{BNimages/CobIdentityEssentialDot.pdf} \\
\pm \includegraphics[valign=c,scale=0.75]{BNimages/CobIdentityEssential.pdf}
\end{bmatrix}
\end{align*}
This is the image of 
$\begin{bmatrix}
\mp\includegraphics[valign=c,scale=0.5]{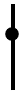}\\
\pm\includegraphics[valign=c,scale=0.5]{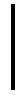}
\end{bmatrix}$ under $\mathcal{I}$.  The initial sign change on the dotted tube comes from moving the dot past the split cobordism, and then either both cobordisms will change sign when using Relation~\eqref{eqn:BNcreation}, or neither will, depending on the orientation of the split cobordism.  Recall that all deaths are assumed to be positive.

Finally, we have the two cobordisms with essential boundary components shown by the movies above \eqref{eqn:type3componant}.  Again, we may cap off the inessential boundary component with $\mathcal{D}$, and the result is cobordisms that are in the image of $\mathcal{I}$.
\[ \begin{bmatrix}
\includegraphics[valign=c]{BNimages/CobDeathDot.pdf}\\
\includegraphics[valign=c]{BNimages/CobDeath.pdf}\vspace*{0.05in}\end{bmatrix}
\circ \includegraphics[valign=c]{Movies/MergeAnn2Eto1I.pdf}=
\begin{bmatrix}
\mathcal{I}(\reflectbox{\includegraphics[angle=180,origin=c,valign=c,scale=0.75]{TLimages/TLcupDotL.pdf}})\vspace*{0.05in}\\ 
\mathcal{I}(\includegraphics[angle=180,origin=c,valign=c,scale=0.75]{TLimages/TLcup.pdf})
\end{bmatrix}\]
\[ \includegraphics[valign=c]{Movies/SplitAnn1Ito2E.pdf} \circ \begin{bmatrix}
\includegraphics[valign=c]{BNimages/CobBirth.pdf}&\hspace*{-0.04in}
\includegraphics[valign=c]{BNimages/CobBirthDot.pdf}\vspace*{-0.05in}%
\end{bmatrix} =
\begin{bmatrix}
\mathcal{I}(\includegraphics[valign=c,scale=0.75]{TLimages/TLcup.pdf}) &
\mathcal{I}(\includegraphics[valign=c,scale=0.75]{TLimages/TLcupDotL.pdf})
\end{bmatrix}
\]

\end{proof}
%\end{document}

\section{First proof that $\mathcal{I}$ is faithful}\label{sec:Chapter5}
    %\documentclass[../main.tex]{subfiles}
%
%\begin{document}
\subsection{Marked Reeb graphs}
The \textbf{Reeb graph} of a $2$-dimensional orientable cobordism $S$ equipped with a Morse function $f\colon S\rightarrow\mathbb{R}$ is the graph $\Gamma_S$ obtained by collapsing every connected component of every level set of $f$ to a point. It is usually assumed that $f(S)\subseteq I$ and $f(\partial S)\subseteq\partial I$, and that all critical values of $f$ lie in $I\setminus\partial I$. An example is shown in Figure~\ref{fig:ReebGraph}
\begin{figure}[H]
\begin{center}
    \begin{tikzpicture}
	\node  at (-0.24,0.95) {$f$};   
	\draw [->,>=stealth] (0,-1.9) -- (0,1);
    \end{tikzpicture}\,\,
    \raisebox{1.31cm}{
    \includegraphics[valign=c]{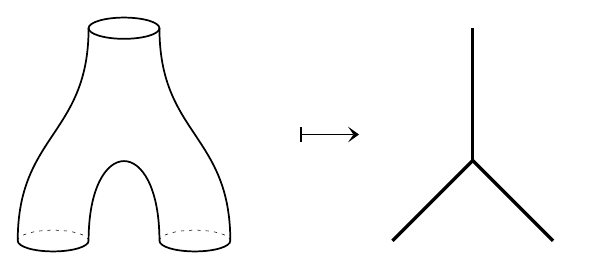}}
    \end{center}\vspace*{-0.4cm}
\caption{A cobordism and its Reeb graph.}\label{fig:ReebGraph}
\end{figure}
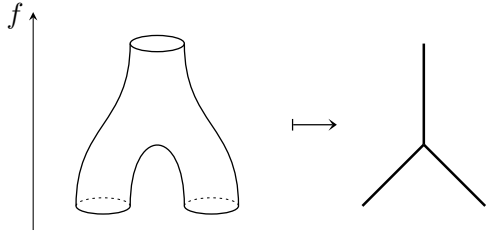

In general, the Reeb graph $\Gamma_S$ contains trivalent vertices with two inputs, corresponding to merge saddles, and trivalent vertices with two outputs, corresponding to split saddles. In addition, the Reeb graph contains univalent vertices corresponding to birth and death critical points, and univalent vertices corresponding to boundary components of $S$. We will mostly be interested in the case where $S$ is a chronological cobordism in $\ann\times I$, and where the Morse function, $f$, is given by the restriction of the height function $h\colon\ann\times I\rightarrow I$. In this case, each component of each level set is either trivial or essential, and we can keep track of this information by marking each edge of the Reeb graph $\Gamma_S$ as trivial or essential. If $S$ is decorated by dots, then we can further represent each dot on $S$ by a dot placed on the corresponding location on the Reeb graph $\Gamma_S$.

We also note that at each trivalent vertex of the Reeb graph, the number of essential edges has to be even. Indeed, this follows because if $S\colon C\rightarrow C'$ is a cobordism in $\ann\times I$, then its boundaries must satisfy $[C]=[C']\in H_1(\ann;\mathbb{Z}_2)$, and hence the parity of the number of essential components is the same in $C$ and $C'$. 

\begin{remark}\label{rem:essentialcycle} If $\Gamma=\Gamma_S$ is the Reeb graph of a chronological coboridsm $S\subset\ann\times I$, then we can define simplicial $1$-chain in $C^\Delta_1(\Gamma;\mathbb{Z}_2)$ by taking the \emph{sum of all essential edges} of $\Gamma$. This $1$-chain can be viewed equivalently as a relative $1$-chain $z\in C^\Delta_1(\Gamma,\partial\Gamma;\mathbb{Z}_2)=C^\Delta_1(\Gamma;\mathbb{Z}_2)/C_1^\Delta(\partial\Gamma;\mathbb{Z}_2)=C^\Delta_1(\Gamma;\mathbb{Z}_2)/0$, and as such it is a relative $1$-cycle because of the parity condition at the trivalent vertices. Since there are no simplicial $2$-chains and hence no simplicial $1$-boundaries in the graph $\Gamma$, we can further identify $z$ with its relative homology class $[z]\in H_1^{\Delta}(\Gamma,\partial\Gamma;\mathbb{Z}_2)$.
\end{remark}

\begin{remark} One can define the homology class $[z]$ from Remark~\ref{rem:essentialcycle} intrinsically, as follows. Consider the sequence of maps and identifications
\[
\mathbb{Z}_2=H^1(\ann\times I;\mathbb{Z}_2)\stackrel{i^*}{\longrightarrow}
H^1(S;\mathbb{Z}_2)\stackrel{PD}{=}H_1(S,\partial S;\mathbb{Z}_2)\stackrel{q_*}{\longrightarrow} H_1(\Gamma,\partial\Gamma;\mathbb{Z}_2),
\]
where $i\colon S\rightarrow\ann\times I$ denotes the inclusion, $q\colon S\rightarrow\Gamma$ denotes the obvious quotient map, and $PD$ denotes Poincar{\'e}-Lefschetz duality. By applying these maps and identifications to the generator $1\in\mathbb{Z}_2$, one obtains a homology class in $H_1(\Gamma,\partial\Gamma;\mathbb{Z}_2)$, and we claim that this homology class is precisely the homology class $[z]$ from Remark~\ref{rem:essentialcycle}.
\begin{figure}[H]
\begin{center}
\includegraphics[valign=c,scale=1.12]{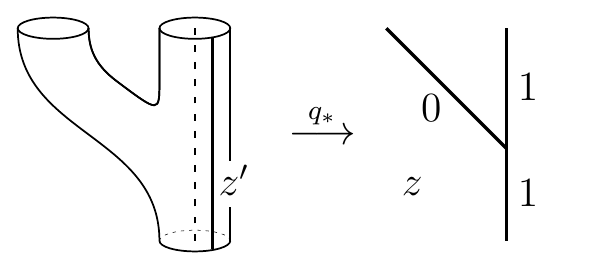}
\end{center}
\caption{The relative $1$-cycles $z$ and $z'$. The numbers $0$ and $1$ represent the coefficients of the $1$-simplices in $z$, and the bold vertical arc on the cobordism on the left represents the homology class $[z']$. The dashed vertical line marks the location of the center of the annulus.}\label{fig:ReebGraphCycle}
\end{figure}
To see this, note that there is a commutative diagram
\[
\begin{tikzcd}
H^1(S;\mathbb{Z}_2)\ar[r,"\cong"]&\operatorname{Hom}(H_1(S;\mathbb{Z}_2),\mathbb{Z}_2)\\
\mathclap{\mathbb{Z}_2=\hspace*{0.43in}}H^1(\ann\times I;\mathbb{Z}_2)\ar[u,"i^*"]\ar[r,"\cong"]&\operatorname{Hom}(H_1(\ann\times I;\mathbb{Z}_2),\mathbb{Z}_2)\ar[u,"(i_*)^*"']
\end{tikzcd}
\]
where the horizontal maps are induced by evaluating cohomology classes on homology classes. If $\alpha$ and $\beta$ are the generators of $H^1(\ann\times I;\mathbb{Z}_2)=\mathbb{Z}_2$ and $\operatorname{Hom}(H_1(\ann\times I;\mathbb{Z}_2)$, then the isomorphism at the bottom has to send $\alpha$ to $\beta$. The commutativity of the diagram thus implies that for any homology class $H\in H_1(S;\mathbb{Z}_2)$, we have
\[\langle i^*(\alpha),H\rangle= \bigl((i_*)^*(\beta)\bigr)(H)= \beta(i_*(H))=i_*(H),\] 
where $\langle-,-\rangle$ denotes the evaluation pairing, and where the last equality holds under the identification $H_1(\ann\times I;\mathbb{Z}_2)=\mathbb{Z}_2$. Hence the cohomology class $i^*(\alpha)$ evaluates to $0$ or $1$ on $H\in H_1(S,\mathbb{Z}_2)$ depending on whether $i_*(H)\in H_1(\ann\times I;\mathbb{Z}_2)$ is $0$ or $1$ under the identification with $\mathbb{Z}_2$. 
Consequently, the Poincar{\'e} dual of $i^*(\alpha)$ is represented by a relative $1$-cycle $z'\in C_1(S,\partial S)$ whose mod $2$ intersection number with a generic closed curve $C\subset S$ is $0$ or $1$, depending on whether $i_*([C])\in H_1(\ann\times I;\mathbb{Z}_2)$ is $0$ or $1$ (see Figure~\ref{fig:ReebGraphCycle}). In the case where $C$ is a component of a generic level set $S\cap(\ann\times\{t\})\subset S$, this means that the mod $2$ intersection number of $z'$ with $C$ is $0$ or $1$ depending on whether $C$ is trivial or essential. Comparing with the definition of $z$ in Remark~\ref{rem:essentialcycle}, it is now apparent that $q_*([z'])=[z]$ and hence $q_*(PD(i^*(\alpha)))=[z]$, which is what we wanted to prove. Note that, in general, the map $q_*$ is not injective, so the homology class $[z]$ does not determine $[z']$ (or the embedding of $S$ in $\ann\times I$).
\end{remark}

%Motivated by the discussion preceding the remark, 

We now define a class of decorated embedded graphs $\Gamma\subset\mathbb{R}\times I$, which we call marked Reeb graphs. These graphs will play a role in the next section.

\begin{definition}\label{def:dottedmarkedReeb} A \textbf{dotted marked Reeb graph} is an embedded uni-trivalent graph $\Gamma\subset\mathbb{R}\times I$, whose edges are marked as trivial or essential, and which is decorated by at most finitely many distinct dots. It is further required that no two dots or vertices occur at the same height, and no dot occurs at the same height as a vertex. In addition, we require that $\Gamma$ is locally modeled on one of the following pictures:
\[
\includegraphics[valign=c]{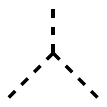}\,\quad\,\quad\,
\includegraphics[angle=180,origin=c,valign=c]{RGimages/RGmerge.pdf}\,\quad\,\quad\,
\includegraphics[valign=c]{RGimages/RGBirth.pdf}\,\quad\,\quad\,
\includegraphics[angle=180,origin=c,valign=c]{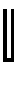}\,\quad\,\quad\,
\includegraphics[valign=c]{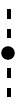}\,\quad\,\quad\,
\includegraphics[valign=c]{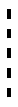}
\]
In these pictures, double lines represent trivial edges, and single lines represent essential edges. Moreover, dashed lines can be trivial or essential, with the provision that at each trivalent vertex, the number of essential edges has to be even. Note that edges are not allowed to have horizontal tangencies, and that trivalent vertices, as well as the univalent vertices shown in the third and the fourth picture (called births and deaths) have to lie in the interior of $\mathbb{R}\times I$. All other univalent vertices (called endpoints) have to lie on $\mathbb{R}\times\partial I$.
\end{definition}

We will regard a marked Reeb graph $\Gamma\subset\mathbb{R}\times I$ as a morphism from its bottom endpoints in $\mathbb{R}\times\{0\}$ to its top endpoints in $\mathbb{R}\times\{1\}$. The marking of the edges as trivial or essential induces a corresponding marking of the endpoints of $\Gamma$.  The symbols $\circ$ and $\bullet$ may be used to represent a trivial endpoint and an essential endpoint respectively, however in practice we will leave the endpoints unmarked because they can be easily deduced from the edges; see Remark~\ref{rem:RGembedding} below. We will identify two marked Reeb graphs if they are related by a chronological isotopy of $\mathbb{R}\times I$, defined as in section~\ref{subs:TL}.

\subsection{Reeb graph category $\RG$}
Let $\Bbbk$ be a commutative unital ring. We define the dotted marked Reeb graph category $\RG$ as the monoidal supercategory with the following objects and morphisms. Objects are finite (possibly empty) sequences in the symbols $\circ$ and $\bullet$, to be viewed as finite sequences of trivial and essential points placed on the real line. Morphisms are formal $\Bbbk$-linear combinations of marked Reeb graphs (in the sense of Definition~\ref{def:dottedmarkedReeb}), modulo the following relations:
\begin{align}
\vspace*{0.2cm}
&\includegraphics[valign=c]{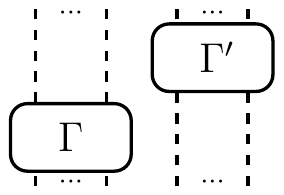}\,=\,(-1)^{|\Gamma||\Gamma'|}\,\includegraphics[valign=c]{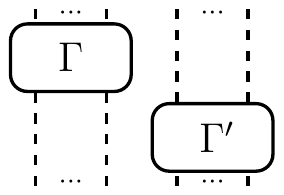}\label{eqn:RGinterchange}\\[0.3cm]
%&\incg{RGinterchange.pdf}
&\includegraphics[valign=c]{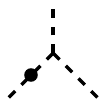}=\includegraphics[valign=c]{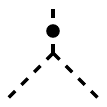}=\reflectbox{\includegraphics[valign=c]{RGimages/RGmergeDotL.pdf}}\,\quad\,\reflectbox{\includegraphics[angle=180,origin=c,valign=c]{RGimages/RGmergeDotL.pdf}}=-\includegraphics[angle=180,origin=c,valign=c]{RGimages/RGmergeDotC.pdf}=\includegraphics[angle=180,origin=c,valign=c]{RGimages/RGmergeDotL.pdf}\label{eqn:RGdotslide}\\[0.3cm]
%&\incg{RGdotslide.pdf}\\[0.3cm]
%&\incg{RGNX.pdf}\label{eqn:RGNX}\\[0.3cm]
&\includegraphics[valign=c]{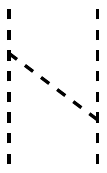}\,=\,\includegraphics[valign=c]{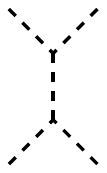}\,=\,\reflectbox{\includegraphics[valign=c]{RGimages/RGnx.pdf}}\label{eqn:RGNX}\\[0.3cm]
&\includegraphics[valign=c]{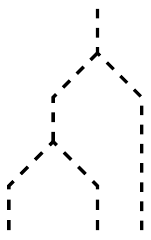}\,=\,\reflectbox{\includegraphics[valign=c]{RGimages/RGassocMerge.pdf}}\,\quad\quad\,\reflectbox{\includegraphics[angle=180,origin=c,valign=c]{RGimages/RGassocMerge.pdf}}\,=\,-\,\includegraphics[angle=180,origin=c,valign=c]{RGimages/RGassocMerge.pdf}\label{eqn:RGassociativity}\\[0.3cm]
%&\incg{RGassociativity.pdf}\label{eqn:RGassociativity}\\[0.3cm]
&\includegraphics[valign=c]{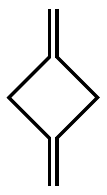}\,=\,\includegraphics[valign=c]{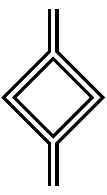}\label{eqn:RGdiamond}\\[0.3cm]
%&\incg{RGdiamond.pdf}\label{eqn:RGdiamond}\\[0.3cm]
&\includegraphics[valign=c]{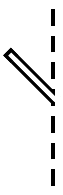}\,=\,\includegraphics[valign=c]{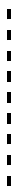}\,=\,\includegraphics[angle=180,origin=c,valign=c]{RGimages/RGwaveLup.pdf} \label{eqn:RGwave}\\[0.3cm]
&\includegraphics[valign=c]{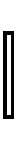}\,=\,0\,\quad\quad\,\includegraphics[valign=c]{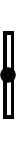}\,=\,1\,\quad\quad\,\includegraphics[valign=c]{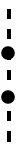}\,=\,0\,\quad\quad\,\includegraphics[valign=c]{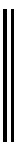}\,=\,\includegraphics[valign=c]{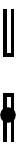}\,+\,\includegraphics[angle=180,origin=c,valign=c]{RGimages/RGneckBottomDot.pdf}\label{eqn:RGBN}
\end{align}
As in the definition of the odd dotted Temperley-Lieb supercategory in section~\ref{subs:TL}, all of the relations above are assumed to be local. Moreover, the number $1$ in relation~\eqref{eqn:RGBN} represents the scalar $1$ times the empty Reeb graph, and the two zeros indicate that the corresponding Reeb graphs are equal to zero in $\RG$. 
\begin{remark}\label{rem:RGembedding} Each dashed edge in the relations above is either adjacent to a boundary point of the shown portion of the graph, or to a trivalent vertex whose other legs are adjacent to such boundary points. Since the number of essential edges at each trivalent vertex has to be even, this implies that, in each relation, the marking of edges as trivial or essential is uniquely determined if one knows which of the endpoints of the shown portions are trivial or essential. There is also a more conceptual way of seeing this: in each relation involving dashed edges, every shown graph $\Gamma$ is a tree and thus has $H_1(\Gamma;\mathbb{Z}_2)=0$. Using the exact sequence for the pair $(\Gamma,\partial\Gamma)$, one therefore sees that the map $H_1(\Gamma,\partial\Gamma;\mathbb{Z}_2)\rightarrow H_0(\Gamma;\mathbb{Z}_2)$, $[z]\mapsto[\partial z]$, is injective, which proves that the relative homology class from Remark~\ref{rem:essentialcycle} is uniquely determined by its image in $H_0(\partial\Gamma;\mathbb{Z}_2)$.
\end{remark}
The \textbf{superdegree} that appears in relation~\eqref{eqn:RGinterchange} is defined by:
\[
|\Gamma|:=\#\left\{
\,\includegraphics[angle=180,origin=c,valign=c]{RGimages/RGmerge.pdf}\,\,,\,\,
\includegraphics[angle=180,origin=c,valign=c]{RGimages/RGbirth.pdf}\,\,,\,\,
\includegraphics[valign=c]{RGimages/RGidentityDot.pdf}\,
\right\}\in\mathbb{Z}_2.
\]

The \textbf{composition} is induced by vertical composition of Reeb graphs, and the \textbf{tensor product} is induced by concatenation of sequences on objects, and by the disjoint `right-then-left' union on morphisms:
\[
\Gamma\otimes\Gamma':=\Gamma\rlunion \Gamma'\,:=\,\,\includegraphics[valign=c]{RGimages/RGboxesRthenL.pdf}
\]
For later use, we also introduce the `left-then-right' union of Reeb graphs:
\[
\Gamma\lrunion \Gamma'\,\,:=\,\,\,\includegraphics[valign=c]{RGimages/RGboxesLthenR.pdf}
\]
Note that each relation involving dashed edges actually stands for multiple relations because there are different choices for the markings of the edges as trivial or essential. In each case, the choices made on the two sides of the relation are required to be consistent with each other, meaning that they have to agree near the endpoints of the shown portions of the involved Reeb graphs.

We also define a $\mathbb{Z}$-valued \textbf{quantum grading} on $\RG$ by
\[
q(\Gamma):=\#\{\textnormal{births, deaths}\}-\#\{\textnormal{trivalent vertices}\}-2\#\{\textnormal{dots}\}\in\mathbb{Z}
\]
for any marked Reeb graph $\Gamma$. Moreover, we define
\[
a(\Gamma):=-2\#\{\textnormal{dots on essential components of $\Gamma$}\}\in\mathbb{Z},
\]
where we say that a component of $\Gamma$ is essential if it contains an essential edge. Although $a(\Gamma)\in\mathbb{Z}$ is not preserved under the relations in $\RG$, we can define a filtration on $\RG$ by declaring a morphism to have filtered degree at most $k$ if it can be written as a linear combination of marked Reeb graphs with $a(\Gamma)\leq k$. We call this filtration the \textbf{annular filtration}. Arguing as in the proof of Lemma~\ref{lem:annularfiltration}, one can easily see that this filtration makes $\RG$ into a filtered monoidal supercategory, in the sense of Definition~\ref{def:filteredsupercategory}.

\begin{lemma}\label{lem:RGreversewave} In $\RG$, we have:
\[
\reflectbox{\includegraphics[angle=180,origin=c,valign=c]{RGimages/RGwaveLup.pdf}}\,=\,\includegraphics[valign=c]{RGimages/RGIdentityLong.pdf}\,=\,-\,\reflectbox{\includegraphics[valign=c]{RGimages/RGwaveLup.pdf}}
\]
\end{lemma}
\begin{proof} If the dashed line is trivial, then we have
\[
\includegraphics[valign=c]{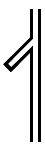}
\,\,=\,\,\includegraphics[valign=c]{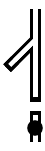}\,\,+\,\,\includegraphics[valign=c]{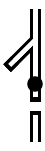}
\,\,=\,\,\includegraphics[valign=c]{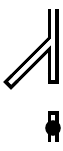}\,\,+\,\,\includegraphics[valign=c]{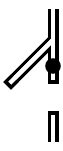}
\,\,=\,\, \includegraphics[valign=c]{RGimages/RGneckBottomDot.pdf} \,\,+\,\, \includegraphics[angle=180,origin=c,valign=c]{RGimages/RGneckBottomDot.pdf}
\,\,=\,\,\includegraphics[valign=c]{RGimages/RGidentityLongIness.pdf}
\]
where in the second-to-last equation we have used relations~\eqref{eqn:RGdotslide} and~\eqref{eqn:RGwave}. Similarly,
\[
\includegraphics[angle=180,origin=c,valign=c]{RGimages/RGWaveLemmaWhole.pdf}
\,\,=\,\,\includegraphics[angle=180,origin=c,valign=c]{RGimages/RGWaveLemmaShortTD.pdf}\,\,+\,\,\includegraphics[angle=180,origin=c,valign=c]{RGimages/RGWaveLemmaShortBD.pdf}
\,\,=\,\,\includegraphics[angle=180,origin=c,valign=c]{RGimages/RGWaveLemmaLongTD.pdf}\,\,-\,\,\includegraphics[angle=180,origin=c,valign=c]{RGimages/RGWaveLemmaLongBD.pdf}
\,\,=\,\, -\,\,\includegraphics[valign=c]{RGimages/RGneckBottomDot.pdf} \,\,-\,\, \includegraphics[angle=180,origin=c,valign=c]{RGimages/RGneckBottomDot.pdf}
\,\,=\,\,-\,\,\includegraphics[valign=c]{RGimages/RGidentityLongIness.pdf}
\]
Now suppose the dashed line is arbitrary. Then
\[
\includegraphics[valign=c]{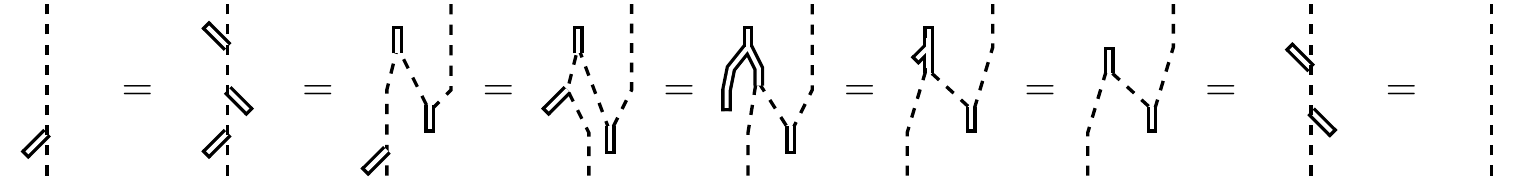}
\]
where the first and the last equation follow from~\eqref{eqn:RGwave}, the second and second-to-last equation follow from~\eqref{eqn:RGNX}, the fourth equation follows from~\eqref{eqn:RGassociativity}, and the sixth equation uses Lemma~\ref{lem:RGreversewave} in the case where the dashed line is trivial. The general proof of the second half of Lemma~\ref{lem:RGreversewave} is similar and is left to the reader.
\end{proof}

\begin{lemma}\label{lem:RGdiamondzero}
In $\RG$, the right-hand side (and hence the left-hand side) of~\eqref{eqn:RGdiamond} is equal to zero.
\end{lemma}
\begin{proof}
By applying the last relation of~\eqref{eqn:RGBN} to one of the edges in the diamond-shaped subgraph, one can expand the right-hand side of~\eqref{eqn:RGdiamond} as a sum of two terms. These two terms are negatives of each other because of relations~\eqref{eqn:RGinterchange} and~\eqref{eqn:RGdotslide}.
\end{proof}

\begin{remark} One can define a universal version of $\RG$ by replacing each minus sign in the definition of $\RG$ by a factor of $\pi$, for a formal variable $\pi$ with $\pi^2=1$.
\end{remark}

\subsection{Permutation isomorphisms in $\RG$}

We now come to the main part of this section, which is about defining symmetry isomorphisms similar to the $\textnormal{R}_{C,\calO,\calO'}$ from section \ref{sec:Chapter3}, but in the category $\RG$.

Let $\calP$ be an object of $\RG$ and write it as a tensor product $$\calP = p_1\otimes\ldots\otimes p_n$$ where $n\geq 0$ and where each $p_i$ is a trivial or an essential point.  Let $T_{\calP} :=\{i|p_i \textnormal{ is trivial}\}$.  Further, let $\calP '$ denote the object obtained from $\calP$ by removing all factors $p_i$ with $i\in T_{\calP}$ from the tensor product above.  For each subset $X\subseteq T_{\calP}$, we can then define a morphism $\varepsilon_{X}(\calP):\calP\to\calP'$ by 
$$\varepsilon_{X}(\calP):=\varepsilon_{X}(p_1)\lrunion\ldots\lrunion\varepsilon_{X}(p_n)$$ where 
\[ \varepsilon_{X}(p_i) := \begin{cases} 
          \includegraphics[valign=c]{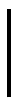} & \textnormal{if }i\notin T_{\calP}, \\
          \includegraphics[angle=180,origin=c,valign=c]{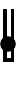} & \textnormal{if }i \in T_{\calP}\setminus X,\\
          \includegraphics[angle=180,origin=c,valign=c]{RGimages/RGbirth.pdf} & \textnormal{if }i \in X.
       \end{cases}
    \]
Dually, we can define a morphism $\varepsilon_{X}^*(\calP):\calP '\to\calP$ by 
\begin{align*}
    \varepsilon_{X}^*(\calP)&:=\varepsilon_{X}^*(p_1)\rlunion\ldots\rlunion \varepsilon_{X}^*(p_n)\\
    &\textnormal{ }=\varepsilon_{X}(p_1)\otimes\ldots\otimes \varepsilon_{X}(p_n)
\end{align*}
where
\[ \varepsilon_{X}^*(p_i) := \begin{cases}
       \includegraphics[valign=c]{RGimages/RGidentityEss.pdf} & \textnormal{if }i\notin T_{\calP}, \\
          \includegraphics[valign=c]{RGimages/RGBirth.pdf} & \textnormal{if }i \in T_{\calP}\setminus X,\\
          \includegraphics[valign=c]{RGimages/RGbirthDot.pdf} & \textnormal{if }i \in X.
\end{cases}
\]
\begin{example}
For $\calP=p_1\otimes p_2\otimes p_3$, $T_{\calP}=\{1,3\}$, and $X=\{1\}$, we have

\[\varepsilon_{X}(\calP)= \includegraphics[angle=180,origin=c,valign=b]{RGimages/RGbirth.pdf}\,\includegraphics[valign=b]{RGimages/RGidentityEss.pdf}\,\includegraphics[angle=180,origin=c,valign=b]{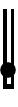}\hspace{0.5cm}\textnormal{ and }\hspace{0.5cm}\varepsilon_{X}^*(\calP)=\includegraphics[valign=t,trim=0 0 0 0.3cm]{RGimages/RGbirthDot.pdf}\,\includegraphics[valign=t,trim=0 0 0 0.3cm]{RGimages/RGidentityEss.pdf}\,\includegraphics[valign=t,trim=0 0 0 0.3cm]{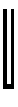}\]
\end{example}

In the following lemma, equation \eqref{eqn:RGEcomp} is a generalization of the first and second relations of \eqref{eqn:RGBN}, and equation \eqref{eqn:RGEid} generalizes the fourth relation of~\eqref{eqn:RGBN}.

\begin{lemma}
For any object $\calP\in\RG$ and $X,\textnormal{ }Y\subseteq T_{\calP}$, we have
\begin{align}
    \varepsilon_{X}(\calP)\circ\varepsilon_{Y}^*(\calP) &= \begin{cases}
    \mathbbm{1}_{\calP'} &\textnormal{if }X=Y,\\
    0 &\textnormal{if }X\neq Y,\\
    \end{cases}\label{eqn:RGEcomp}\\
    \mathbbm{1}_{\calP} &= \sum_{X\subseteq T_{\calP}} \varepsilon_{X}^*(\calP)\circ\varepsilon_{X}(\calP),\label{eqn:RGEid}
\end{align}
where $\mathbbm{1}_{\calP}$ and $\mathbbm{1}_{\calP'}$ denote the identity morphisms of $\calP$ and $\calP'$ (which are given by possibly empty collections of vertical lines).
\end{lemma}
\begin{proof}
If $n\notin T_{\calP}$ and if $\calP''$ denotes the object of $\RG$ obtained from $\calP$ by removing the last tensor factor, then $$\varepsilon_{X}(\calP)=\varepsilon_{X}(\calP'')\otimes\mathbbm{1}_{ p_n}\qquad \mbox{and}\qquad \varepsilon_{X}^*(\calP)=\varepsilon_{X}^*(\calP'')\otimes\mathbbm{1}_{ p_n} $$ and thus the lemma follows by induction on $n$.

On the other hand, if $n\in T_{\calP}$, then 
\begin{align*}
    \varepsilon_{X}(\calP)\circ\varepsilon_{Y}^*(\calP) & =\varepsilon_{X\setminus\{n\}}(\calP'')\circ(\mathbbm{1}_{\calP''}\otimes\varepsilon_{X}(p_n))\circ(\mathbbm{1}_{\calP''}\otimes\varepsilon_{Y}^*(p_n))\circ \varepsilon_{Y\setminus\{n\}}^*(\calP'')\\
    & =\varepsilon_{X\setminus\{n\}}(\calP'')\circ(\mathbbm{1}_{\calP''}\otimes(\varepsilon_{X}(p_n)\circ\varepsilon_{Y}^*(p_n)))\circ \varepsilon_{Y\setminus\{n\}}^*(\calP'')\\
    & = \begin{cases}
            \varepsilon_{X\setminus\{n\}}(\calP'')\circ \varepsilon^*_{Y\setminus\{n\}}(\calP'') & \textnormal{if } X\cap \{n\}=Y\cap\{n\},\\
            0 & \textnormal{if }X\cap\{n\}\neq Y\cap\{n\}
        \end{cases}\\
    & = \begin{cases}
            \mathbbm{1}_{\calP'} & \textnormal{if }X=Y,\\
            0 & \textnormal{if }X\neq Y,
        \end{cases}
\end{align*}
where the third equality follows from \eqref{eqn:RGBN} and the last equality from induction on $n$.  Likewise,
\begin{align*}
    \mathbbm{1}_{\calP} &=\mathbbm{1}_{\calP''}\otimes\mathbbm{1}_{ p_n}\\
    &=\sum_{X''\subseteq T_{\calP''}}\left(\varepsilon_{X}^*(\calP'')\circ\varepsilon_{X}(\calP'')\right)\otimes\mathbbm{1}_{ p_n}\\
    &=\sum_{X''\subseteq T_{\calP''}} (\varepsilon_{X''}^*(\calP'')\otimes\mathbbm{1}_{ p_n})\circ(\mathbbm{1}_{\calP'}\otimes\mathbbm{1}_{ p_n})\circ(\varepsilon_{X''}(\calP'')\otimes\mathbbm{1}_{ p_n})\\
    &=\sum_{\substack{X''\subseteq T_{\calP''} \\ X_n\subseteq T_{ p_n}}}(\varepsilon_{X''}^*(\calP'')\otimes\mathbbm{1}_{ p_n})\circ(\mathbbm{1}_{\calP'}\otimes(\varepsilon_{X}^*(p_n)\circ\varepsilon_{X}(p_n)))\circ(\varepsilon_{X''}(\calP'')\otimes\mathbbm{1}_{ p_n})\\
    &=\sum_{X\subseteq T_{\calP}}\varepsilon_{X}^*(\calP)\circ \varepsilon_{X}(\calP),
\end{align*}
where the second equality follows from induction on $n$, the fourth equality follows from \eqref{eqn:RGBN}, and the last equality follows from the definitions and because pairs of subsets of $T_{\calP''}$ and of $T_{ p_n}$ correspond to subsets of $T_{\calP}$.
\end{proof}

Now let $\sigma\in \mathfrak{S}_n$ be a permutation which restricts to an order-preserving map on the complement of $T_{\calP}$.  Let $\sigma(\calP)$ denote the object $$\sigma(\calP):= p_{\sigma^{-1}(1)}\otimes\dots\otimes p_{\sigma^{-1}(n)}.$$ 
For notational convenience we will let $\mathcal{Q}=\sigma(\calP)$, and $q_i=p_{\sigma^{-1}(i)}$ will denote the $i^{th}$ factor in $\mathcal{Q}$.  Then observe that $\sigma$ takes $T_{\calP}$ to $T_{\sigma(\calP)}=T_\mathcal{Q}$ because
\begin{align*}
    i\in T_{\calP} & \iff  p_i\textnormal{ is trivial}\\
    & \iff  p_{\sigma^{-1}(\sigma(i))}\textnormal{ is trivial}\\
    & \iff  q_{\sigma(i)} \textnormal{ is trivial} \\
    & \iff \sigma(i)\in T_{\mathcal{Q}}.
\end{align*}

To $\sigma$ and a subset $X\subseteq T_{\calP}$, we now assign a sign $$\textnormal{sgn}(\sigma|X):=(-1)^{|T(\sigma|X)|}$$ where $T(\sigma|X)$ denotes the set $$T(\sigma|X):=\{(i,j)\in X\times X |i<j\textnormal{ but }\sigma(i)>\sigma(j)\}.$$  Moreover we define a morphism $\textnormal{R}_{\calP,\sigma}:\calP\to\sigma(\calP)$ by
\[
\textnormal{R}_{\calP,\sigma}:=\sum_{X\subseteq T_{\calP}} \textnormal{sgn}(\sigma|X)\varepsilon_{\sigma(X)}^*(\sigma(\calP))\circ\varepsilon_{X}(\calP).
\]

\begin{example}
    For $\calP= p_1\otimes p_2\otimes p_3$, $T_{\calP}=\{1,3\}$, and $\sigma=(13)$, we have 
    
\[
    \textnormal{R}_{\calP,\sigma} =
    \includegraphics[angle=180,origin=c,valign=c]{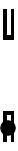}\,\includegraphics[valign=c]{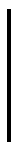}\,\includegraphics[valign=c]{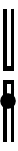}\,+\,
    \includegraphics[valign=c]{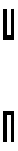}\,\includegraphics[valign=c]{RGimages/RGIdentityLongEss.pdf}\,\includegraphics[valign=c]{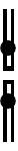}\,+\,
    \includegraphics[valign=c]{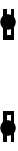}\,\includegraphics[valign=c]{RGimages/RGIdentityLongEss.pdf}\,\includegraphics[valign=c]{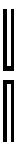}\,-\,
    \includegraphics[valign=c]{RGimages/RGneckShortBottomDot.pdf}\,\includegraphics[valign=c]{RGimages/RGIdentityLongEss.pdf}\,\includegraphics[angle=180,origin=c,valign=c]{RGimages/RGneckLongBottomDot.pdf}
\]   
    
\end{example}

The following lemma implies that $\textnormal{R}_{\calP,\sigma}$ is an isomorphism.

\begin{lemma}\label{lem:RGR}
We have
\begin{align}
    \textnormal{R}_{\calP,e}=\mathbbm{1}_{\calP}\label{eqn:RGRid}\\
    \textnormal{R}_{\sigma(\calP),\sigma'}\circ \textnormal{R}_{\calP,\sigma}=\textnormal{R}_{\calP,\sigma'\circ\sigma},\label{eqn:RGRcomp}
\end{align}
where $e\in\mathfrak{S}_n$ is the identity permutation and $\sigma$, $\sigma'$ are order-preserving on the complements of $T_{\calP}$ and $T_{\sigma(\calP)}$ respectively.
\end{lemma}
\begin{proof}
Relation~\eqref{eqn:RGRid} follows immediately from~\eqref{eqn:RGEid} and from the definition of $T_{\calP,\sigma}$.  As for~\eqref{eqn:RGRcomp}, we have 
\begin{align*}
    \hspace*{0.3in}&\hspace*{-0.3in}\textnormal{R}_{\sigma(\calP),\sigma'}\circ \textnormal{R}_{\calP,\sigma}\\&=\sum_{\substack{X\subseteq T_{\calP}\\Y\subseteq T_{\sigma(\calP)}}}
    \textnormal{sgn}(\sigma'|Y)\textnormal{sgn}(\sigma|X)\epsilon^*_{\sigma'(Y)}(\sigma'(\sigma(\calP)))\circ\epsilon_Y(\sigma(\calP))\circ\epsilon^*_{\sigma(X)}(\sigma(\calP))\circ\epsilon_X(\calP)\\
    & =\sum_{X\subseteq T_{\calP}}\textnormal{sgn}(\sigma'|\sigma(X))\textnormal{sgn}(\sigma|X)\varepsilon_{\sigma'(\sigma(X))}^*(\sigma'(\sigma(\calP)))\circ\varepsilon_{X}(\calP)\\
    &=\sum_{X\subseteq T_{\calP}} \textnormal{sgn}(\sigma'\circ\sigma|X)\varepsilon_{(\sigma'\circ\sigma)(X)}^*((\sigma'\circ\sigma)(\calP))\circ\varepsilon_{X}(\calP)\\
    &=\textnormal{R}_{\calP,\sigma'\circ\sigma}.
\end{align*}
Here the second equality follows because we know from~\eqref{eqn:RGEcomp} that $\varepsilon_{Y}(\sigma(\calP))\circ\varepsilon_{\sigma(X)}^*(\calP'')$ is zero unless $Y=\sigma(X)$, in which case it is an identity morphism.
To see the third equality, note first that $\sigma'(\sigma(X))=(\sigma'\circ\sigma)(X)$ and $\sigma'(\sigma(\calP))=(\sigma'\circ\sigma)(\calP)$. Indeed, the latter holds because if $ p_j':= p_{\sigma^{-1}(j)}$ denotes the $j^{\textnormal{th}}$ factor in $\sigma(\calP)$, then the $j^{\textnormal{th}}$ factor in  $\sigma'(\sigma(\calP))$ is given by
\[
p'_{(\sigma')^{-1}(j)}= p_{\sigma^{-1}((\sigma')^{-1}(j))}= p_{(\sigma'\circ\sigma)^{-1}(j)},
\]
which is precisely the $j^{\textnormal{th}}$ factor in $(\sigma'\circ\sigma)(\calP).$  We further have $$\textnormal{sgn}(\sigma'|\sigma(X))\textnormal{sgn}(\sigma|X)=\textnormal{sgn}(\sigma'\circ\sigma|X)$$ because a pair $(i,j)$ is in $T(\sigma'\circ\sigma|X)$ if either $(i,j)\in T(\sigma|X)$ or $(\sigma(i),\sigma(j))\in T(\sigma'|\sigma(X))$, but not both.
\end{proof}

Now, let $\calP$ and $\calP'$ be two objects in $\RG$ given by $\calP= p_1\otimes\dots\otimes p_m$ and $\calP'= p_{m+1}\otimes\dots\otimes p_{m+n}$, where $m,n\geq 0$ and where the $ p_i$ are trivial or essential points.  Suppose that either all factors in $\calP$ are trivial or all factors in $\calP'$ are trivial, or both.  We then define an isomorphism \[\textnormal{R}(\calP,\calP'):\calP\otimes\calP'\to\calP'\otimes\calP\] 
by $\textnormal{R}(\calP,\calP'):=\textnormal{R}_{\calP\otimes\calP',\sigma}$ where $\sigma\in\mathfrak{S}_{m+n}$ is the permutation 
\[
\sigma(i):=  \begin{cases}
              i+n &\textnormal{if }i\leq m,\\
              i-m &\textnormal{if }i>m.
            \end{cases}
\]
In the following lemma, $\calP_1, \calP_1'$, and $\calP_2$ denote objects in $\RG$, and it is assumed that either all factors in both $\calP_1$ and in $\calP_1'$ are trivial or that all factors in $\calP_2$ are trivial, or all factors are trivial.

\begin{lemma}
Let $\Gamma: \calP_1\to \calP_1'$ be a generating morphism of $\RG$.  Then 
\begin{align}
    \textnormal{R}( \calP_1', \calP_2)\circ(\Gamma\otimes\mathbbm{1}_{ \calP_2})=(\mathbbm{1}_{\calP_2}\otimes\Gamma)\circ \textnormal{R}(\calP_1,\calP_2),\label{eq:RGRel1}\\
    \textnormal{R}(\calP_2,\calP_1')\circ(\mathbbm{1}_{ \calP_2}\otimes\Gamma)=(\Gamma\otimes\mathbbm{1}_{\calP_2})\circ \textnormal{R}(\calP_2,\calP_1).\label{eq:RGRel2}
\end{align}
\end{lemma}
\begin{proof}
The second relation can be deduced from the first one by composing from the left with $\textnormal{R}(\calP_2,\calP_1')$ and from the right with $\textnormal{R}(\calP_2,\calP_1)$ and using Lemma~\ref{lem:RGR}.  We will therefore only prove the first relation.

\textbf{Case 1:} Suppose first that all factors in $\calP_1$ and in $\calP_1'$ are trivial, so that $T_{\calP_1}$ and $T_{\calP_1'}$ consist of all $i$ in the relevant range.  In this case, a straightforward case-by-case analysis shows that for all $X_1\subseteq T_{\calP_1}$ and $X_1'\subseteq T_{\calP_1'}$, we have
$$\varepsilon_{X'_1}(\calP_1')\circ\Gamma\circ\varepsilon_{X_1}(\calP_1)^*=a(\Gamma,X_1,X_1')\mathbbm{1}_{\emptyset}$$
for some scalar $a(\Gamma,X_1,X_1')\in\{0,\pm 1\}$, where $\emptyset$ denotes the empty tensor product.  If $\sigma$ and $\sigma'$ denote the permutations that appear in the definitions of $\textnormal{R}(\calP_1,\calP_2):=\textnormal{R}_{\calP_1\otimes \calP_2,\sigma}$ and $\textnormal{R}(\calP'_1,\calP_2):=\textnormal{R}_{\calP'_1\otimes \calP_2,\sigma'}$, we can thus rewrite the left-hand side of~\eqref{eq:RGRel1} as 
\begin{align*}
    \hspace*{0.6in}&\hspace*{-0.6in}\textnormal{R}(\calP_1',\calP_2)\circ(\Gamma\otimes\mathbbm{1}_{\calP_2})\\
    & \stackrel{(1)}{=}\sum_{X\subseteq T_{\calP_1'\otimes \calP_2}}\textnormal{sgn}(\sigma'|X)\varepsilon_{\sigma'(X)}(\calP_2\otimes\calP_1')\circ\varepsilon_{X}(\calP_1'\otimes \calP_2)\circ(\Gamma\otimes\mathbbm{1}_{\calP_2})\\
    &\stackrel{(2)}{=}\sum_{\substack{X_1'\subseteq T_{ \calP_1}'\\X_2\subseteq T_{ \calP_2}}}(-1)^{|X_1'||X_2|}(\varepsilon_{X_2}^*(\calP_2)\lrunion\varepsilon_{X_1}(\calP_1')^*)\circ(\varepsilon_{X_1}(\calP_1')\rlunion\varepsilon_{X_2}(\calP_2))\circ(\Gamma\otimes\mathbbm{1}_{\calP_2})\\
    &\stackrel{(3)}{=} \sum_{\substack{X_1'\subseteq T_{ \calP_1}'\\X_2\subseteq T_{ p_2}}}(\varepsilon_{X_2}^*(\calP_2)\rlunion\varepsilon_{X_1}^*(\calP_1'))\circ(\varepsilon_{X_1}(\calP_1')\rlunion\varepsilon_{X_2}(\calP_2))\circ(\Gamma\otimes\mathbbm{1}_{\calP_2})\\
    &\stackrel{(4)}{=}\sum_{X_1'\subseteq T_{ p_1}'}(\mathbbm{1}_{\calP_2}\otimes\varepsilon_{X_1}^*(\calP_1'))\circ
    \varepsilon_{X_2}^*(\calP_2)\circ\varepsilon_{X_2}(\calP_2)\circ((\varepsilon_{X_1}(\calP_1')\circ\Gamma)\otimes\mathbbm{1}_{\calP_2})\\
    &\stackrel{(5)}{=}\sum_{X_1'\subseteq T_{\calP_1}'}(\mathbbm{1}_{\calP_2}\otimes\varepsilon_{X_1}^*(\calP_1'))\circ ((\varepsilon_{X_1}(\calP_1')\circ\Gamma)\otimes\mathbbm{1}_{\calP_2})\\
    &\stackrel{(6)}{=}\sum_{\substack{X_1\subseteq T_{\calP_1}\\X_1'\subseteq T_{\calP_1'}}}(\mathbbm{1}_{\calP_2}\otimes\varepsilon_{X_1}^*(\calP_1'))\circ((\varepsilon_{X_1}(\calP_1')\circ\Gamma\circ\varepsilon_{X_1}^*(\calP_1)\circ\varepsilon_{X_1}(\calP_1))\otimes\mathbbm{1}_{\calP_2})\\
    &\stackrel{(7)}{=}\sum_{\substack{X_1\subseteq T_{\calP_1}\\X_1'\subseteq T_{\calP_1'}}}a(\Gamma,X_1,X_1')(\mathbbm{1}_{ \calP_2}\otimes\varepsilon_{X_1}^*(\calP_1'))\circ(\varepsilon_{X_1}(\calP_1)\otimes\mathbbm{1}_{\calP_2})\\
    &\stackrel{(8)}{=}\sum_{\substack{X_1\subseteq T_{ \calP_1}\\X_1'\subseteq T_{\calP_1'}}}(\mathbbm{1}_{ \calP_2}\otimes(\varepsilon_{X_1}^*(\calP_1')\varepsilon_{X_1}(\calP_1')\circ\Gamma\circ\varepsilon_{X_1}^*(\calP_1)))\circ(\varepsilon_{X_1}(\calP_1)\otimes\mathbbm{1}_{ \calP_2})\\
    &\stackrel{(9)}{=}\sum_{X_1\subseteq T_{ \calP_1}}(\mathbbm{1}_{ \calP_2}\otimes(\Gamma\circ\varepsilon_{X_1}^*(\calP_1)))\circ(\varepsilon_{X_1}(\calP_1)\otimes\mathbbm{1}_{\calP_2})\\
    &\stackrel{(10)}{=}\sum_{\substack{X_1\subseteq T_{ \calP_1}\\X_2\subseteq T_{\calP_2}}}(\mathbbm{1}_{ \calP_2}\otimes(\Gamma\circ\varepsilon_{X_1}^*(\calP_1)))\circ\varepsilon_{X_2}^*(\calP_2)\circ\varepsilon_{X_2}(\calP_2)\circ(\varepsilon_{X_1}(\calP_1)\otimes\mathbbm{1}_{ \calP_2})\\
    &\stackrel{(11)}{=}\sum_{\substack{X_1\subseteq T_{ \calP_1}\\X_2\subseteq T_{\calP_2}}}(\mathbbm{1}_{ \calP_2}\otimes\Gamma)\circ(\varepsilon_{X_2}^*(\calP_2)\lrunion\varepsilon_{X_1}^*(\calP_1)\circ(\varepsilon_{X_1}(\calP_1)\lrunion\varepsilon_{X_2}(\calP_2))\\
    &\stackrel{(12)}{=}\sum_{\substack{X_1\subseteq T_{ \calP_1}\\X_2\subseteq T_{ \calP_2}}}(-1)^{|X_1||X_2|}(\mathbbm{1}_{ \calP_2}\otimes\Gamma)\circ(\varepsilon_{X_2}^*(\calP_2)\rlunion\varepsilon_{X_1}^*(\calP_1))\circ(\varepsilon_{X_1}(\calP_1)\lrunion\varepsilon_{X_2}(\calP_2))\\
    &\stackrel{(13)}{=}\sum_{X\subseteq T_{\calP_1}\otimes T_{ \calP_2}}\textnormal{sgn}(\sigma|X)(\mathbbm{1}_{ \calP_2}\otimes\Gamma)\circ\varepsilon_{\sigma(X)}^*(\calP_2\otimes \calP_1)\varepsilon_{X}(\calP_1\otimes \calP_2)\\
    &\stackrel{(14)}{=}(\mathbbm{1}_{ \calP_2}\otimes\Gamma)\circ \textnormal{R}( \calP_1, \calP_2)
\end{align*}
where equations (1), (2), (13) and (14) follow from the definitions; equations (3) and (12) follow because $\varepsilon_{X_1}^*(\calP_1)$, $\varepsilon_{X_1}^*(\calP_1')$, and $\varepsilon_{X_2}^*(\calP_2)$ have superdegrees $|X_1|$, $|X_1'|$, and $|X_2|$; equations (4) and (10) follow because $\varepsilon_{X_2}^*(\calP_2)\otimes\mathbbm{1}_\emptyset=\varepsilon_{X_2}^*(\calP_2)=\mathbbm{1}_{\emptyset}\otimes\varepsilon_{X_2}^*(\calP_2)$ and $\mathbbm{1}_{\emptyset}\otimes\varepsilon_{X_2}(\calP_2)=\varepsilon_{X_2}(p_2)=\varepsilon_{X_2}(\calP_2)\otimes\mathbbm{1}_{\emptyset}$; and equations (5), (6), (8), (9) follow from~\eqref{eqn:RGEid}.

\textbf{Case 2:} Suppose now that all factors in $\calP_2$ are trivial.  Then
\begin{align*}
    \hspace*{0.6in}&\hspace*{-0.6in}\textnormal{R}(\calP_1', \calP_2)\circ(\Gamma\otimes\mathbbm{1}_{ \calP_2})\\
    &\stackrel{(1)}{=}\sum_{\substack{X_1'\subseteq  T_{\calP_1'}\\ X_2\subseteq T_{\calP_2}}}(\varepsilon_{X_2}^*(\calP_2)\lrunion\varepsilon_{X_1}^*(\calP_1'))\circ(\varepsilon_{X_1}(\calP_1')\lrunion\varepsilon_{X_2}(\calP_2))\circ(\Gamma\otimes\mathbbm{1}_{\calP_2})\\
    &\stackrel{(2)}{=}\sum_{\substack{X_1'\subseteq  T_{\calP_1'}\\ X_2\subseteq T_{ \calP_2}}}(\varepsilon_{X_2}^*(\calP_2)\otimes\mathbbm{1}_{\calP_1'})\circ\varepsilon_{X_1}^*(\calP_1')\circ\varepsilon_{X_1}(\calP_1')\circ(\mathbbm{1}_{ \calP_1'}\otimes\varepsilon_{X_2}(\calP_2))\circ(\Gamma\otimes\mathbbm{1}_{\calP_2})\\
    &\stackrel{(3)}{=}\sum_{X_2\subseteq T_{ \calP_2}}(\varepsilon_{X_2}^*(\calP_2)\otimes\mathbbm{ \calP_1'})\circ(\mathbbm{1}_{ \calP_1'}\otimes\varepsilon_{X_2}(\calP_2))\circ(\Gamma,\otimes\mathbbm{1}_{\calP_2})\\
    &\stackrel{(4)}{=}\sum_{X_2\subseteq T_{ \calP_2}}(-1)^{|X_2||\Gamma|}(\varepsilon_{X_2}^*(\calP_2)\otimes\mathbbm{1}_{ p_1'})\circ\Gamma\circ(\mathbbm{1}_{ \calP_1}\otimes\varepsilon_{X_2}(\calP_2))\\
    &\stackrel{(5)}{=}\sum_{X_2\subseteq T_{ \calP_2}}(\mathbbm{1}_{ \calP_2}\otimes\Gamma)\circ(\varepsilon_{X_2}^*(\calP_2)\otimes\mathbbm{1}_{\calP_1})\circ(\mathbbm{1}_{ \calP_1}\otimes\varepsilon_{X_2}(\calP_2))\\
    &\stackrel{(6)}{=}\sum_{\substack{X_1\subseteq T_{\calP_1}\\ X_2\subseteq T_{\calP_2}}}(\mathbbm{1}_{ \calP_2}\otimes\Gamma)\circ(\varepsilon_{X_2}^*(\calP_2)\otimes\mathbbm{1}_{ \calP_1})\circ\varepsilon_{X_1}^*(\calP_1)\circ\varepsilon_{X_1}(\calP_1)\circ(\mathbbm{1}_{ \calP_1}\otimes\varepsilon_{X_2}(\calP_2))\\
    &\stackrel{(7)}{=}\sum_{\substack{X_1\subseteq T_{\calP_1}\\ X_2\subseteq T_{\calP_2}}}(\mathbbm{1}_{ \calP_2}\otimes\Gamma)\circ(\varepsilon_{X_2}(\calP_2)^*\lrunion\varepsilon_{X_1}^*(\calP_1))\circ(\varepsilon_{X_1}(\calP_1)\lrunion\varepsilon_{X_2}(\calP_2))\\
    &\stackrel{(8)}{=}(\mathbbm{1}_{\calP_2}\otimes\Gamma)\circ \textnormal{R}(\calP_1,\calP_2),
\end{align*}
where equations (1)-(3) and (6)-(8) are deduced in the same way as equations (1)-(5) and (10)-(14) in the proof of Case 1, except that now one has to use that $\varepsilon_{X_1}(\calP_1)$, $\varepsilon_{X_1}(\calP_1')$, and $\varepsilon_{X_2}(\calP_2)$ have superdegrees $|X_1|$, $|X_1'|$, and $|X_2|$, respectively. 
Equations (4) and (5) follow because $\varepsilon_{X_2}(\calP_2)$ and $\varepsilon_{X_2}(\calP_2)^*$ both have superdegree $|X_2|$ and because $\Gamma\otimes\mathbbm{1}_{\emptyset}=\Gamma=\mathbbm{1}_{\emptyset}\otimes\Gamma$.
\end{proof}

We will also need the following lemma, in which $\calP_1$, $ \calP_2$, $\calP_3$, and $\calP_4$ denote objects of $\RG$ which are tensor products of $n_1$, $n_2$, $n_3$ and $n_4$ points, respectively, where each $n_i\geq 0$.  It is further assumed that all factors in $\calP_2$ or in $\calP_3$ are trivial.
%either all factors in $\calP_2$ or in $\calP_3$ are trivial (or both).

\begin{lemma}
If $\calP= \calP_1\otimes \calP_2\otimes \calP_3\otimes \calP_4$, then 
\begin{align}
    \textnormal{R}_{\calP,\sigma}=\mathbbm{1}_{\calP_1}\otimes \textnormal{R}( \calP_2, \calP_3)\otimes\mathbbm{1}_{ \calP_4},\label{eqn:RGR28}
\end{align}
where $\sigma\in\mathfrak{S}_{n_1+n_2+n_3+n_4}$ is the permutation that exchanges the factors of $\calP_2$ and $\calP_3$. Explicitly,
\begin{align*}
    \sigma(i)=\begin{cases}
            i &\textnormal{if }i\leq n_1 \textnormal{ or }i>n_1+n_2+n_3,\\
            i+n_3 &\textnormal{if }n_1<i\leq n_1+n_2,\\
            i-n_2 &\textnormal{if }n_1+n_2<i\leq n_1+n_2+n_3.
            \end{cases}
\end{align*}
\end{lemma}
\begin{proof}
By definition, the left-hand side of~\eqref{eqn:RGR28} is equal to the sum of all terms of the form 
\[
\begin{array}{c}
(-1)^{|X_2||X_3|}(\varepsilon_{X_1}^*(\calP_1)\rlunion\varepsilon_{X_3}^*(\calP_3)\rlunion\varepsilon_{X_2}^*(\calP_2)\rlunion\varepsilon_{X_4}^*(\calP_4))\circ\\ (\varepsilon_{X_1}(\calP_1)\lrunion\varepsilon_{X_2}(\calP_2)\lrunion\varepsilon_{X_3}(\calP_3)\lrunion\varepsilon_{X_4}(\calP_4))
\end{array}
\]
for $X_i\subseteq T_{\calP_i}$.  Using the definition of $\textnormal{R}(\calP_2,\calP_3)$ and the fact that $\varepsilon_{X_i}(\calP_i)$ and $\varepsilon_{X_i}^*(\calP_i)$ have the same superdegree, it is now easy to see that this sum is equal to the sum of all terms of the form $$(\mathbbm{1}_{\calP_1}\otimes \textnormal{R}(\calP_3, \calP_3)\otimes\mathbbm{1}_{ \calP_4})\circ((\varepsilon_{X_1}^*(\calP_1)\circ\varepsilon_{X_1}(\calP_1))\otimes\mathbbm{1}_{\calP_2\otimes \calP_3}\otimes(\varepsilon_{X_4}^*(\calP_4)\circ\varepsilon_{X_4}(\calP_4)))$$ for $X_1\subseteq T_{\calP_1}$ and $X_4\subseteq T_{\calP_4}$. The lemma thus follows from~\eqref{eqn:RGEid}.
\end{proof}

\begin{lemma}
Suppose $\Gamma$ is a merge (resp., split) morphism which has bottom (resp., top) endpoints $p_2$ and $p_1$ (resp., $p_1$ and $p_2$), at least one of which is trivial. Then composing $\Gamma$ with $\textnormal{R}(p_1,p_2)$ preserves (resp. changes) the sign. Graphically,
\begin{equation}
\includegraphics[valign=c,scale=0.8]{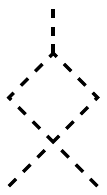}\,=\,\,\includegraphics[valign=c,scale=0.8]{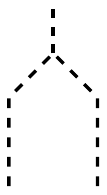}\,\quad\quad\,\mbox{ and }\,\quad\quad\,
\includegraphics[angle=180,origin=c,valign=c,scale=0.8]{RGimages/RGRthenMerge.pdf}\,=\,\,-\,\,\includegraphics[angle=180,origin=c,valign=c,scale=0.8]{RGimages/RGmergeTall.pdf} \label{eq:RGRelR}
\end{equation}
where the crossings represent copies of $\textnormal{R}(p_1,p_2)$.
\end{lemma}

\begin{proof}

We will only prove the lemma in the case where $\Gamma$ is a merge morphism, as the other case is similar. For concreteness, we will assume that $p_1$ is trivial, but the same proof (with reflected pictures) also works if $p_2$ is trivial.

To prove the desired result, we first use the definition of $\textnormal{R}(p_1,p_2)$ (and if necessary the last relation in~\eqref{eqn:RGBN}) followed by a dot slide:
\[
\includegraphics[valign=r,scale=.8]{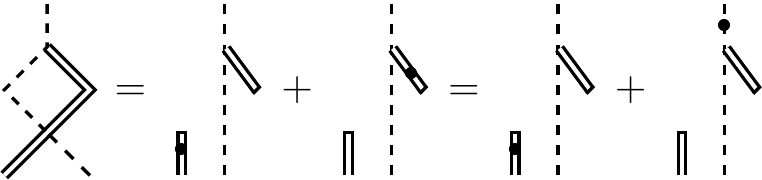}
\]
We then use~\eqref{eqn:RGwave} and Lemma~\ref{lem:RGreversewave}:
\[
\includegraphics[valign=r,scale=.8]{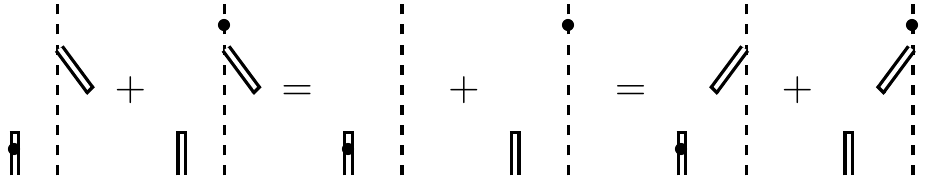}
\]
Finally, another dot slide and~\eqref{eqn:RGBN} will bring us to the conclusion:
\[\includegraphics[valign=r,scale=.8]{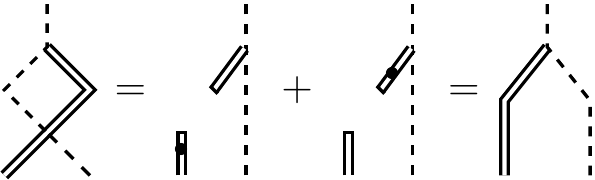}
\]
Note that in the case where $\Gamma$ is a split morphism, a sign change would occur when using Lemma~\ref{lem:RGreversewave} and with each dot slide.
\end{proof}

\subsection{Superfunctors $\mathcal{F}$ and $\mathcal{G}$}\label{subs:FandG}

To prove that $\mathcal{I}$ is faithful, we will define two monoidal superfunctors
\[
\OBNA\stackrel{\mathcal{F}}{\longrightarrow}\RG\stackrel{\mathcal{G}}{\longrightarrow}(\TL^s)^\oplus
\]
and we will show that their composition is a left inverse to $\mathcal{I}$.

On objects, we define
\[
\mathcal{F}(C,\mathcal{O}):=(f(C_1),\ldots,f(C_n)),
\]
where $C_1<\ldots<C_n$ are the components of $C$, numbered in increasing order with respect to $\mathcal{O}$, and
\[
f(C_i):=\begin{cases}
\circ&\mbox{if $C_i$ is trivial,}\\
\bullet&\mbox{if $C_i$ is essential.}
\end{cases}
\]

To define $\mathcal{F}$ on morphisms,
it suffices to define $\mathcal{F}$ on type~I and type~II cobordisms because
every chronological cobordism $S$ admits an admissible factorization bypassing Lemma~\ref{lem:factorizationexistence}. 

Suppose first that $S$ is a type~II cobordism. Then $S$ is chronologically isotopic to a permutation isomorphism $\textnormal{R}_{C,\mathcal{O},\sigma\circ\mathcal{O}}$, and we define
$\mathcal{F}(S):=\textnormal{R}_{\mathcal{F}(C,\mathcal{O}),\sigma}$,
where $\textnormal{R}_{\mathcal{F}(C,\mathcal{O}),\sigma}$ is defined as in the previous section.

Now suppose
that $S$ is a type~I cobordism. Then $S$ is a union of its components
\[
S=S_1\rlunion\ldots\rlunion S_n,
\]
where the order of the $S_i$ is determined by the given admissible orders on the boundary components on $S$ as in Remark~\ref{rem:typeIcobordism}, and where at most one $S_i$ contains a dot or a critical point. To define $\mathcal{F}(S)$, we replace each component $S_i$ by its dotted marked Reeb graph, up to a possible sign. More precisely, we set
\[
\mathcal{F}(S):=\mathcal{F}(S_1)\otimes\ldots\otimes\mathcal{F}(S_n),
\]
where the $\mathcal{F}(S_i)$ are defined by:
\[
\renewcommand*{\arraystretch}{3}
\begin{array}{cclccl}
\mathcal{F}\left(\includegraphics[valign=c]{BNimages/CobMerge.pdf}\right)&:=&\includegraphics[valign=c]{RGimages/RGmerge.pdf}\quad\qquad&\qquad\quad
\mathcal{F}\left(\includegraphics[valign=c]{BNimages/CobSplit.pdf}\right)&:=&\pm\includegraphics[angle=180,origin=c,valign=c]{RGimages/RGmerge.pdf}\\
\mathcal{F}\left(\includegraphics[valign=t]{BNimages/CobBirth.pdf}\right)&:=&\includegraphics[valign=c]{RGimages/RGbirth.pdf}\quad\qquad&\qquad\quad\mathcal{F}\left(\includegraphics[valign=b]{BNimages/CobDeath.pdf}\right)\,&:=&\,\pm\,\,\includegraphics[angle=180,origin=c,valign=c]{RGimages/RGbirth.pdf}\\
\mathcal{F}\left(\includegraphics[valign=c]{BNimages/CobIdentityDot.pdf}\right)&:=&\includegraphics[valign=c]{RGimages/RGidentityDot.pdf}\quad\qquad&\qquad\quad\mathcal{F}\left(\includegraphics[valign=c]{BNimages/CobIdentity.pdf}\right)&:=&\phantom{\pm}\includegraphics[valign=c]{RGimages/RGidentity.pdf}
\end{array}
\renewcommand*{\arraystretch}{1}
\]
Here, it is understood that dashed edges stand for trivial or essential edges, depending on whether their external endpoints correspond to trivial or essential boundary components of $S_i$.

In the formula for the merge cobordism, the orientation of the critical point is irrelevant and, thus, omitted. In the formula for the split cobordism, the sign is defined as follows: rotate the arrow at the critical point clockwise by $90^\circ$, and compare the resulting arrow to the given total order on the two upper boundary components of the split cobordism. If the rotated arrow points toward the higher-ordered component, then the sign is a plus, otherwise it is a minus. In the formula for the death cobordism, the sign is defined to be a plus if the critical point is oriented clockwise, and a minus otherwise.

\begin{proposition}\label{prop:Fwelldefined}
$\mathcal{F}$ is well-defined.
\end{proposition}
\begin{proof}
It suffices to show that $\mathcal{F}$ respects the defining relations of $\BNA$ and the relations from Lemma~\ref{lem:factorizationrelations}. For the relations from Lemma~\ref{lem:factorizationrelations}, this follows because $\mathcal{F}$ takes relations the~\eqref{eqn:BNRid}, \eqref{eqn:BNRcomp}, \eqref{eq:Rel1}, \eqref{eq:Rel2}, \eqref{eq:Rel3}, \eqref{eq:Rel4}, \eqref{eq:Rel5} to the relations~\eqref{eqn:RGRid}, \eqref{eqn:RGRcomp}, \eqref{eq:RGRel1}, \eqref{eq:RGRel2}, \eqref{eqn:RGR28}, and to the two relations in \eqref{eq:RGRelR}, respectively. Note that $\mathcal{F}$ trivially preserves relations~\eqref{eqn:BNRelisotopy} and~\eqref{eq:SIdentity} because $\mathcal{F}$ is a functor and invariant under chronological isotopy.

To prove that $\mathcal{F}$ is compatible with the defining relations of $\mathcal{BN}_{\!o}(\ann)$, suppose first that $S$ is an elementary split or death cobordism. Changing the orientation of the critical point in $S$ then reverses the sign of $\mathcal{F}(S)$, by our definition of the sign in $\mathcal{F}(S)$. On the other hand, if $S$ is an elementary merge cobordism, then $\mathcal{F}(S)$ remains unchanged under orientation changes, again by definition of $\mathcal{F}(S)$. It follows that $\mathcal{F}$ respects relation~\eqref{eqn:BNorientation}.

Next, note that changing the orientation of a critical point in one of the relations~\eqref{eqn:BNdisjoint}, \eqref{eqn:BNconnected}, \eqref{eqn:BNXchange} has the same effect on both sides of the relations. Since we have already shown that $\mathcal{F}$ is compatible with orientation changes, we can thus assume without loss of generality that the critical points in these relations are oriented in whichever way is the most suitable to us. Moreover, we can assume that the boundary components of the involved cobordisms are equipped with suitable admissible orderings. Indeed, any two admissible orderings are related by a permutation isomorphism, and composing the two sides of a relation with the same isomorphism does not affect the validity of the relation.

In relation~\eqref{eqn:BNconnected}, it further suffices to consider the case where $S$ and $S'$ are elementary cobordisms (caps, cups, dots, and saddles) because the general relation can be deduced from this special case. In view of Remark~\ref{rem:RGembedding}, we can then ignore the embedding of the corbordisms in~\eqref{eqn:BNconnected} and~\eqref{eqn:BNXchange} into $\ann\times I$.

With all of this in mind, it is now easy to see that $\mathcal{F}$ respects
relation~\eqref{eqn:BNdisjoint} because of relation~\eqref{eqn:RGinterchange};
relation~\eqref{eqn:BNconnected} because of relations~\eqref{eqn:RGdotslide}, \eqref{eqn:RGNX}, and \eqref{eqn:RGassociativity};
relation~\eqref{eqn:BNcreation} because of relation~\eqref{eqn:RGwave};
relation~\eqref{eqn:BNBN} because of relation~\eqref{eqn:RGBN}; and
relation~\eqref{eqn:BNdiamond} because it sends both sides of this relation to zero, by Lemma~\ref{lem:RGdiamondzero}.
\end{proof}

We next define $\mathcal{G}:\RG\to(\TL (0)^s)^\oplus$, which takes tensor products to 
and which takes values in the additive closure of the supergraded extension of $\TL$.

On objects, we define this functor by $$\mathcal{G}(p_1\otimes\ldots\otimes p_n):= g(p_1)\otimes\ldots\otimes(p_n),$$ 
\[
g(p_i):=\begin{cases}
0_0\oplus0_1 &\mbox{if $p_i$ is trivial,}\\
1&\mbox{if $p_i$ is essential}
\end{cases}
\]
where, in this definition, the normal size numbers $0$ and $1$ represent the objects $0,1\in\TL(0)$, and the subscripts $0$ and $1$ denote formal shifts of the supergrading.

\begin{remark}
To avoid confusion, we recall that the tensor product of two objects $k,l\in\TL(0)$ is defined as the sum $k\otimes l:=k+l$. This definition is extended to the additive closure of the supergraded extension of $\TL(0)$ via the constructions described in sections~\ref{subs:supergradedext} and~\ref{subs:additiveclosure}. Note that $0\in\TL(0)$ is the monoidal unit for both $\TL(0)$ and $(\TL(0)^s)^\oplus$. In particular, $0\in\TL(0)$ is \emph{not the zero object of} $(\TL(0)^s)^\oplus$, which is given by the empty sum of objects of $\TL(0)^s$.
\end{remark}

To define $\mathcal{G}$ on morphisms, it suffices to define it on the generating morphisms of $\RG$:
\[
\renewcommand*{\arraystretch}{3}
\begin{array}{cclccl}
\mathcal{G}\left(\includegraphics[valign=c]{RGimages/RGbirth.pdf} \right) &:=& %Birth
\renewcommand*{\arraystretch}{1.3}
    \mleft[\begin{array}{c}
         1\\ \hline
         0
    \end{array}\mright] &
\mathcal{G}\left(\includegraphics[angle=180,origin=c,valign=c]{RGimages/RGbirth.pdf} \right) &:=& %Death
\renewcommand*{\arraystretch}{1.3}
    \mleft[\begin{array}{c|c}
    0 & 1_1^0 
    \end{array} \mright]\\
\mathcal{G}\left(\includegraphics[valign=c]{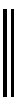}\right) &:=& %Inessential ID
\renewcommand*{\arraystretch}{1.3}
    \mleft[\begin{array}{c|c}
        1 & 0 \\ \hline
        0 & 1_1^1
    \end{array} \mright]&
\mathcal{G}\left(\includegraphics[valign=c]{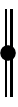}\right) &:=& % Inessential Dot
\renewcommand*{\arraystretch}{1.3}
    \mleft[\begin{array}{c|c}
        0 & 0 \\ \hline
        1_0^1 & 0
    \end{array} \mright]\\
\end{array}
\]
\[
\renewcommand*{\arraystretch}{3}
\begin{array}{cclccl}
\mathcal{G}\left(\includegraphics[valign=c]{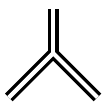}  \right) &:=& %Inessential Merge
\renewcommand*{\arraystretch}{1.3}
    \mleft[\begin{array}{c|c|c|c}
    1 & 0 & 0 & 0\\ \hline 
    0 & 1_1^1 & 1_1^1 & 0 \\
    \end{array}\mright] &
\mathcal{G}\left(\includegraphics[angle=180,origin=c,valign=c]{RGimages/RGmergeIness.pdf}\right) &:=& %Inessential Split
\renewcommand*{\arraystretch}{1.3}
    \mleft[\begin{array}{c|c}
    0 & 0 \\ \hline
    -1_0^1 & 0 \\ \hline
    1_0^1 & 0 \\ \hline
    0 & 1_1^0
    \end{array}\mright] \\
\mathcal{G}\left(\includegraphics[valign=c]{RGimages/RGidentityEss.pdf}\right) &:=& \includegraphics[valign=c]{TLimages/TLidentityShort.pdf}&% essential ID
\mathcal{G}\left(\includegraphics[valign=c]{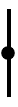}\right) &:=& \includegraphics[valign=c]{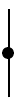}\\
% Essential Dot
\mathcal{G}\left(\includegraphics[valign=c]{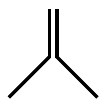}\right) 
&:=& % Essential merge
\renewcommand*{\arraystretch}{1.6}
    \mleft[\begin{array}{c}
    \reflectbox{\includegraphics[angle=180,origin=c,valign=c]{TLimages/TLcupDotL.pdf}} \\ \hline
    \includegraphics[angle=180,origin=c,valign=c]{TLimages/TLcup.pdf}\,_0^1
    \end{array}\mright] &
\mathcal{G}\left(\includegraphics[angle=180,origin=c,valign=c]{RGimages/RGmergeEss.pdf}\right) 
&:=& % Essential Split
\renewcommand*{\arraystretch}{1.3}
    \mleft[\begin{array}{c|c}
   \includegraphics[valign=c]{TLimages/TLcup.pdf}&-\includegraphics[valign=c]{TLimages/TLcupDotL.pdf}\,_1^0
    \end{array}\mright] \\
\mathcal{G}\left(\includegraphics[valign=c]{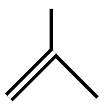}\right) &:=& %Mix Merge
\renewcommand*{\arraystretch}{1.3}
    \mleft[\begin{array}{c|c}
    \includegraphics[valign=c]{TLimages/TLidentityShort.pdf} & \includegraphics[valign=c]{TLimages/TLidentityShortDot.pdf}_1^0
    \end{array}\mright] &
\mathcal{G}\left(\reflectbox{\includegraphics[angle=180,origin=c,valign=c]{RGimages/RGmergeMixL.pdf}} \right) &:=& %Mix Split
\renewcommand*{\arraystretch}{1.3}
    \mleft[\begin{array}{c}
    -\includegraphics[valign=c]{TLimages/TLidentityShortDot.pdf}\\ \hline
    \includegraphics[valign=c]{TLimages/TLidentityShort.pdf}_0^1
    \end{array}
    \mright]\\
\mathcal{G}\left(\reflectbox{\includegraphics[valign=c]{RGimages/RGmergeMixL.pdf}}\right) &:=& %Mix Merge
\renewcommand*{\arraystretch}{1.3}
    \mleft[\begin{array}{c|c}
    \includegraphics[valign=c]{TLimages/TLidentityShort.pdf} & \includegraphics[valign=c]{TLimages/TLidentityShortDot.pdf}_1^0
    \end{array}\mright] &
\mathcal{G}\left(\includegraphics[angle=180,origin=c,valign=c]{RGimages/RGmergeMixL.pdf} \right) &:=& %Mix Split
\renewcommand*{\arraystretch}{1.3}
    \mleft[\begin{array}{c}
    \includegraphics[valign=c]{TLimages/TLidentityShortDot.pdf}\\ \hline
    -\includegraphics[valign=c]{TLimages/TLidentityShort.pdf}_0^1
    \end{array}
    \mright]\\
\end{array}
\]

\begin{proposition}\label{prop:Gwelldefined}
$\mathcal{G}$ is well-defined.
\end{proposition}

\begin{proof}
It suffices to show that $\mathcal{G}$ respects the defining relations of $\RG$. The proof involves verifying the images of the relations \eqref{eqn:RGinterchange} to \eqref{eqn:RGBN} directly via matrix composition and tensors of matrices.  The tensor product is the usual tensor product of matrices; however the tensor product of two morphisms within the matrix is found using the definition of the tensor product of morphisms in the supergraded extension, \eqref{eqn:supergradedext}.

The image of the first relation, \eqref{eqn:RGinterchange}, follows directly from how both $\RG$ and $\TL (0)$ are defined.  After that, there are 39 additional relations to check once we account for all admissible choices of essential and inessential components.  We will demonstrate the calculations for a select few representative ones.

There are 8 relations to check for \eqref{eqn:RGdotslide}, 4 for a dot slide past a merge and 4 for a dot slide past a split.  In the case of a merge where all components are inessential, we have
\begin{align*}
    \mathcal{G}\mleft(\includegraphics[valign=c]{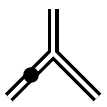}\mright) & = 
    \renewcommand{\arraystretch}{1.3}  % Inessential Merge
    \mleft[\begin{array}{cccc}
        1&0&0&0\\
        0&1_1^1&1_1^1&0
    \end{array}\mright]
    \circ
    \mleft(\renewcommand*{\arraystretch}{1.3} %Inessential Dot
    \mleft[\begin{array}{cc}
    0&0\\
    1_0^1&0
    \end{array}
    \mright]\otimes %tensor
    \mleft[\renewcommand*{\arraystretch}{1.3} %Inessential Identity
    \begin{array}{cc}
        1 & 0  \\
        0 & 1_1^1
    \end{array}
    \mright]
    \mright)\\
    &=
    \renewcommand{\arraystretch}{1.3}  % Inessential Merge
    \mleft[\begin{array}{cccc}
        1&0&0&0\\
        0&1_1^1&1_1^1&0
    \end{array}\mright] 
    \circ
    \renewcommand{\arraystretch}{1.3} % Tensor of dot and identity
    \mleft[\begin{array}{cccc}
        0&0&0&0\\
        0&0&0&0\\
        1_0^1&0&0&0\\
        0&1_1^0&0&0
    \end{array}\mright]\\
    &=
    \renewcommand{\arraystretch}{1.3} %Solution
    \mleft[\begin{array}{cccc}
        0&0&0&0\\
        1_0^1&0&0&0
    \end{array}\mright]
\end{align*}

Similarly
\begin{align*}
    \calG\mleft(\reflectbox{\includegraphics[valign=c]{RGimages/RGmergeInessDotL.pdf}}\mright) 
    &=&
    \renewcommand*{\arraystretch}{1.3} % inessential merge
    \mleft[\begin{array}{cccc}
    1 & 0 & 0 & 0\\  
    0 & 1_1^1 & 1_1^1 & 0 \\
    \end{array}\mright]
    \circ \mleft(
    \renewcommand*{\arraystretch}{1.3} % inessential id
    \mleft[\begin{array}{cc}
        1 & 0 \\
        0 & 1_1^1
    \end{array} \mright]  
    \otimes
    \renewcommand*{\arraystretch}{1.3} % inessential dot
    \mleft[\begin{array}{cc}
        0 & 0 \\
        1_0^1 & 0
    \end{array} \mright]    
   \mright) =
    \renewcommand{\arraystretch}{1.3} %Solution
    \mleft[\begin{array}{cccc}
        0&0&0&0\\
        1_0^1&0&0&0
    \end{array}\mright]\\
    \calG\mleft(\includegraphics[valign=c]{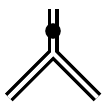}  \mright)
    &=&
    \renewcommand*{\arraystretch}{1.3}%Dotted ID
    \mleft[\begin{array}{cc}
        0 & 0 \\ 
        1_0^1 & 0
    \end{array} \mright]
    \circ
    \renewcommand*{\arraystretch}{1.3}%Merge
    \mleft[\begin{array}{cccc}
    1 & 0 & 0 & 0\\  
    0 & 1_1^1 & 1_1^1 & 0 \\
    \end{array}\mright]
    =
    \renewcommand{\arraystretch}{1.3} %Solution
    \mleft[\begin{array}{cccc}
        0&0&0&0\\
        1_0^1&0&0&0
    \end{array}\mright]\\
\end{align*}

There are 8 variations on relation \eqref{eqn:RGNX} to check.  To see that there are 8, recall that at each trivalent vertex the number of essential segments must be $0$ or $2$. One can check that once three of the four endpoints are chosen to be inessential or essential, the fourth one is forced.  The relation shown below is less immediately clear than the other seven.
\begin{align*}
    \mathcal{G}\mleft(\includegraphics[scale=0.8,valign=c]{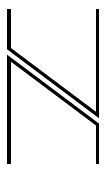}\mright) &= \mleft(
    \renewcommand*{\arraystretch}{1.3}
    \mleft[\begin{array}{c|c}
    \includegraphics[valign=c]{TLimages/TLidentityShort.pdf} & \includegraphics[valign=c]{TLimages/TLidentityShortDot.pdf}_1^0
    \end{array}\mright]
    \otimes
    \includegraphics[valign=c]{TLimages/TLidentityShort.pdf} \mright)
    \circ
    \mleft( \includegraphics[valign=c]{TLimages/TLidentityShort.pdf}
    \otimes
    \renewcommand*{\arraystretch}{1.3}
    \mleft[\begin{array}{c}
    -\includegraphics[valign=c]{TLimages/TLidentityShortDot.pdf}\\ \hline
    \includegraphics[valign=c]{TLimages/TLidentityShort.pdf}_0^1
    \end{array}
    \mright] \mright)\\
    &=
    \renewcommand*{\arraystretch}{1.3}
    \mleft[\begin{array}{c|c}
    \includegraphics[valign=c]{TLimages/TLidentityShort.pdf}\includegraphics[valign=c]{TLimages/TLidentityShort.pdf} & \includegraphics[valign=c]{TLimages/TLidentityShortDot.pdf}\includegraphics[valign=c]{TLimages/TLidentityShort.pdf}_1^0
    \end{array}\mright]
    \circ
    \renewcommand*{\arraystretch}{1.3}
    \mleft[\begin{array}{c}
    -\includegraphics[valign=c]{TLimages/TLidentityShort.pdf}\includegraphics[valign=c]{TLimages/TLidentityShortDot.pdf}\\ \hline
    \includegraphics[valign=c]{TLimages/TLidentityShort.pdf}\includegraphics[valign=c]{TLimages/TLidentityShort.pdf}_0^1
    \end{array}
    \mright]\\
    &= -\includegraphics[valign=c]{TLimages/TLidentityShort.pdf}\includegraphics[valign=c]{TLimages/TLidentityShortDot.pdf}+\includegraphics[valign=c]{TLimages/TLidentityShortDot.pdf}\includegraphics[valign=c]{TLimages/TLidentityShort.pdf}\\
    &= \includegraphics[valign=c]{TLimages/TLidentityShortDot.pdf}\includegraphics[valign=c]{TLimages/TLidentityShort.pdf}-\includegraphics[valign=c]{TLimages/TLidentityShort.pdf}\includegraphics[valign=c]{TLimages/TLidentityShortDot.pdf}
\end{align*}
and $\mathcal{G}\mleft(\reflectbox{\includegraphics[scale=0.5,valign=c]{RGimages/GdefinedNX.pdf}}\mright)$ gives the same result.  However, after calculating $\mathcal{G}\mleft(\includegraphics[scale=0.5,valign=c]{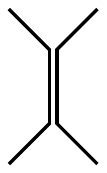}\mright)$ we will need to use relation \eqref{eqn:TLfourterm} to get the final equivalence.

\begin{align*}
    \mathcal{G}\mleft( \includegraphics[scale=0.8,valign=c]{RGimages/GdefinedNXc.pdf}\mright) &= 
    \renewcommand*{\arraystretch}{1.3}
    \mleft[\begin{array}{cc}
   \includegraphics[valign=c]{TLimages/TLcup.pdf}&-\,\includegraphics[valign=c]{TLimages/TLcupDotL.pdf}\,_1^0
    \end{array}\mright]
    \circ
    \renewcommand*{\arraystretch}{1.6}
    \mleft[\begin{array}{c}
    \reflectbox{\includegraphics[angle=180,origin=c,valign=c]{TLimages/TLcupDotL.pdf}} \\ 
    \includegraphics[angle=180,origin=c,valign=c]{TLimages/TLcup.pdf}\,_0^1
    \end{array}\mright]\\
    &= \includegraphics[valign=c]{TLimages/TLrelDotBottomL.pdf}\,-\,\reflectbox{\includegraphics[angle=180,origin=c,valign=c]{TLimages/TLrelDotBottomL.pdf}}\\
    &=\,\includegraphics[valign=c]{TLimages/TLidentityShortDot.pdf}\includegraphics[valign=c]{TLimages/TLidentityShort.pdf}-\includegraphics[valign=c]{TLimages/TLidentityShort.pdf}\includegraphics[valign=c]{TLimages/TLidentityShortDot.pdf}
\end{align*}

There are 16 total relations to check for associativity (\ref{eqn:RGassociativity}).  For the associativity of two merges, the three bottom endpoints may be freely chosen to be essential or inessential, but that will force the remainder of the graph, so we have 8 possible choices.  The same is true for the number of skew-associativity relations of the split maps.  Again, we demonstrate only one of the two relations that do not follow immediately from the calculations.  

\begin{align*}
  \calG\mleft(\,\includegraphics[scale=0.8,valign=c]{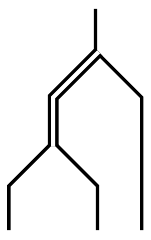}\,\mright)
  &=
  \renewcommand*{\arraystretch}{1.3}%Mix Merge
    \mleft[\begin{array}{c|c}
    \includegraphics[valign=c]{TLimages/TLidentityShort.pdf} & \includegraphics[valign=c]{TLimages/TLidentityShortDot.pdf}_1^0
    \end{array}\mright]
  \circ  \mleft(
  \renewcommand*{\arraystretch}{1.6} %Essential Merge
    \mleft[\begin{array}{c}
    \reflectbox{\includegraphics[angle=180,origin=c,valign=c]{TLimages/TLcupDotL.pdf}} \\ \hline
    \includegraphics[angle=180,origin=c,valign=c]{TLimages/TLcup.pdf}\,_0^1
    \end{array}\mright]
    \otimes 
    \includegraphics[valign=c]{TLimages/TLidentityShort.pdf}
    \mright)\\
    &=
    \reflectbox{\includegraphics[angle=180,origin=c,valign=b]{TLimages/TLcupDotL.pdf}}\includegraphics[valign=b]{TLimages/TLidentityShort.pdf}\,+\,\includegraphics[angle=180,origin=c,valign=b]{TLimages/TLcup.pdf} \includegraphics[valign=b]{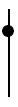}\\[0.3cm]
% %%%%%
    \calG\mleft(\,\reflectbox{\includegraphics[scale=0.8,valign=c]{RGimages/RGassocExample.pdf}}\,\mright)
     &=
    \renewcommand*{\arraystretch}{1.3}%Mix Merge
    \mleft[\begin{array}{c|c}
    \includegraphics[valign=c]{TLimages/TLidentityShort.pdf} & \includegraphics[valign=c]{TLimages/TLidentityShortDot.pdf}_1^0
    \end{array}\mright]
  \circ  \mleft(
    \includegraphics[valign=c]{TLimages/TLidentityShort.pdf}
    \otimes 
  \renewcommand*{\arraystretch}{1.6} %Essential Merge
    \mleft[\begin{array}{c}
    \reflectbox{\includegraphics[angle=180,origin=c,valign=c]{TLimages/TLcupDotL.pdf}} \\ \hline
    \includegraphics[angle=180,origin=c,valign=c]{TLimages/TLcup.pdf}\,_0^1
    \end{array}\mright]
    \mright)\\
    &=
    \includegraphics[valign=b]{TLimages/TLidentityShort.pdf}\,\reflectbox{\includegraphics[angle=180,origin=c,valign=b]{TLimages/TLcupDotL.pdf}}\,+\, \includegraphics[valign=b]{TLimages/TLidentityShortDotHigh.pdf}\,\includegraphics[angle=180,origin=c,valign=b]{TLimages/TLcup.pdf}
\end{align*}

These are equal by Lemma~\ref{lm:TLidentities}. The proofs of all other relations are straight forward calculations and are omitted.
\end{proof}

In the following proposition, we view $\mathcal{I}$ as a monoidal superfunctor on the additive closure of the supergraded extension of $\mathcal{T\!L}_{o,\bullet}(0)$.

\begin{proposition}\label{prop:leftinverse}
$\mathcal{G}\circ\mathcal{F}$ is a left inverse of $\mathcal{I}$.
\end{proposition}

\begin{proof}
We must show that the composition $\mathcal{G}\circ\mathcal{F}\circ\mathcal{I}$ is the identity functor. On objects, this composition acts as the identity because
\[
\mathcal{G}(\mathcal{F}(\mathcal{I}(n))) = \mathcal{G}(\mathcal{F}(n\textnormal{ essential circles}))=\mathcal{G}(n\textnormal{ essential points})=n.
%(\underbrace{\bullet\bullet\ldots\bullet}_{n\textnormal{ essential points}})=n,
\]
On generating morphisms, $\mathcal{G}\circ\mathcal{F}\circ\mathcal{I}$ is given by
\[
\begin{tikzcd}
\includegraphics[valign=c]{TLimages/TLcup.pdf} \arrow[r, "\mathcal{I}", maps to] & \includegraphics[valign=c]{Movies/CupAnn.pdf} \arrow[r, "\mathcal{F}", maps to] & \includegraphics[angle=180,origin=c,valign=c]{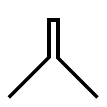} \arrow[r, "\mathcal{G}", maps to] & \includegraphics[valign=c]{TLimages/TLcup.pdf}\\
\includegraphics[angle=180,origin=c,valign=c]{TLimages/TLcup.pdf} \arrow[r, "\mathcal{I}", maps to] & \includegraphics[valign=c]{Movies/CapAnn.pdf} \arrow[r, "\mathcal{F}", maps to] & \includegraphics[valign=c]{RGimages/RGmergeEssCap.pdf} \arrow[r, "\mathcal{G}", maps to] & \includegraphics[angle=180,origin=c,valign=c]{TLimages/TLcup.pdf}\\
\includegraphics[valign=c]{TLimages/TLidentityShortDot.pdf} \arrow[r, "\mathcal{I}", maps to] & \includegraphics[valign=c]{Movies/IdentityDotAnn.pdf} \arrow[r, "\mathcal{F}", maps to] & \includegraphics[valign=c]{RGimages/RGidentityEssDot.pdf} \arrow[r, "\mathcal{G}", maps to] & \includegraphics[valign=c]{TLimages/TLidentityShortDot.pdf}\\
\end{tikzcd}
\]
and hence $\mathcal{G}\circ\mathcal{F}\circ\mathcal{I}$ is the identity on morphisms as well.
\end{proof}
Proposition~\ref{prop:leftinverse} shows that $\mathcal{I}$ is faithful, thus completing the proof of Theorem~\ref{thm:main}.

\begin{remark}
The formulas that define $\calG$ on generating morphisms can be obtained as follows: starting with a generating morphism $\Gamma$, apply the isomorphism $\circ\cong\emptyset_0\oplus\emptyset_1$ to all trivial endpoints of $\Gamma$. This results in a matrix $\mathcal{M}(\Gamma)$ whose entries are morphisms in the supergraded extension of $\RG$ without trivial endpoints. Now use the relations in $\RG$ to simplify these matrix entries as much as possible. Finally, use the following formulas to convert any remaining matrix entries in $\mathcal{M}(\Gamma)$ into morphisms in $\TL(0)^s$:
\begin{equation}\label{eqn:Gproperties}\calG\mleft(\includegraphics[angle=180,origin=c,valign=c]{RGimages/RGmergeEssCap.pdf}\mright):=\includegraphics[valign=c]{TLimages/TLcup.pdf},\qquad\quad\calG\mleft(\includegraphics[origin=c,valign=c]{RGimages/RGmergeEssCap.pdf}\mright):=\includegraphics[angle=180,valign=c]{TLimages/TLcup.pdf},\qquad\quad\calG\mleft(\includegraphics[valign=c]{RGimages/RGidentityEssDot.pdf}\mright):=\includegraphics[valign=c]{TLimages/TLidentityShortDot.pdf}
\end{equation}
In particular, $\mathcal{G}$ can be characterized as the unique monoidal superfunctor (up to natural isomorphism) from $\RG$ to the additive closure of $\TL(0)^s$ which satisfies $\calG(\bullet)=1$ and~\eqref{eqn:Gproperties}.
\end{remark}
%\end{document}

\section{Second proof that $\mathcal{I}$ is faithful}\label{sec:Chapter6}
    %\documentclass[../main.tex]{subfiles}

%\begin{document}
In this section we will give an alternative proof of the faithfulness of $\mathcal{I}$ by using the embedding $\ann\hookrightarrow\mathbb{R}^2$ together with the odd non-annular Khovanov TQFT functor. As a byproduct we will obtain explicit bases for the morphism sets in $\TL (0)$ (Corollary~\ref{cor:basisHomnm}).

\subsection{Odd non-annular Khovanov TQFT}\label{subs:nonannularTQFT}
Let
\[
V:=\Bbbk v_+\oplus\Bbbk v_-
\] be the free supermodule spanned by two homogeneous elements $v_+$ and $v_-$ with $|v_+|:=0$ and $|v_-|:=1$. In addition to the supergrading, we define a $\mathbb{Z}$-grading on $V$, called the quantum grading, by setting $q(v_\pm):=(-1)^{|v_\pm|}\in\mathbb{Z}$.

We further introduce four linear structure maps: a unit $\eta\colon\Bbbk\rightarrow V$, a counit $\epsilon\colon V\rightarrow\Bbbk$, a multiplication $\mu\colon V\otimes V\rightarrow V$, and a comultiplication $\Delta\colon V\rightarrow V\otimes V$. On generators, these maps are given by
\begin{equation}\label{eqn:TQFTeta}
        \eta=
        \left\{\begin{array}{l}
            1\,\,\mapsto\,\, v_+
		\end{array}\right.
\end{equation}
\begin{equation}\label{eqn:TQFT}
		\epsilon=
        \left\{\begin{array}{l}
            v_+\,\,\mapsto\,\, 0\\
            v_-\,\,\mapsto\,\, 1
		\end{array}\right.
\end{equation}
\begin{equation}\label{eqn:TQFTmu}
		\mu=\left\{\begin{array}{ll}
			v_+\otimes v_+ \,\,\mapsto\,\, v_+ \quad\ & v_-\otimes v_+ \,\,\mapsto\,\, v_-\\
			v_+\otimes v_- \,\,\mapsto\,\, v_-        & v_-\otimes v_- \,\,\mapsto\,\, 0
		\end{array}\right.
\end{equation}
\begin{equation}\label{eqn:TQFTdelta}
		\Delta=\left\{\begin{array}{l}
			v_+ \,\,\mapsto\,\, v_-\otimes v_+ -v_+\otimes v_- \\
			v_- \,\,\mapsto\,\, v_-\otimes v_-
		\end{array}\right.
\end{equation}

Following~\cite{PutyraChrono}, we now define a superfunctor
\[
\mathcal{F}_{\!o}^{Kh}\colon\OBNR\longrightarrow\mathcal{SM}\mathit{od}(\Bbbk)
\]
on the non-annular odd Bar-Natan category. On objects, this functor is given by
\[
\mathcal{F}_{\!o}^{Kh}(C,\mathcal{O}):=V^{\otimes n}
\]
where $C$ is a closed $1$-manifold in $\mathbb{R}^2$ with $n$ components and $\mathcal{O}$ is an arbitrary ordering of the components of $C$. Note that any ordering $\mathcal{O}$ is admissible because all components of $C$ are trivial in $\mathbb{R}^2$.

To define $\mathcal{F}_{\!o}^{Kh}$ on morphisms, it suffices to specify its values on (non-annular) type~I and type~II cobordisms (cf. Definition \ref{def:cobtype}). If $S=S_1\rlunion\ldots\rlunion S_n$ is a type~I cobordism with components $S_i$, then we set $\mathcal{F}^{Kh}_o(S):=\mathcal{F}^{Kh}_o(S_1)\otimes\ldots\otimes\mathcal{F}^{Kh}_o(S_n)$ and
\[
\renewcommand*{\arraystretch}{2}
\begin{array}{cclccl}
\mathcal{F}_{\!o}^{Kh}\left(\includegraphics[valign=c]{BNimages/CobMerge.pdf}\right)&:=&\mu\qquad&\qquad
\mathcal{F}_{\!o}^{Kh}\left(\includegraphics[valign=c]{BNimages/CobSplit.pdf}\right)&:=&\pm\Delta\\
\mathcal{F}_{\!o}^{Kh}\left(\includegraphics[valign=t]{BNimages/CobBirth.pdf}\right)&:=&\eta\qquad&\qquad\mathcal{F}_{\!o}^{Kh}\left(\includegraphics[valign=b]{BNimages/CobDeath.pdf}\right)&:=&\pm\epsilon
\\
\mathcal{F}_{\!o}^{Kh}\left(\includegraphics[valign=c]{BNimages/CobIdentityDot.pdf}\right)&:=&\mu_{v_-}\qquad&\qquad\mathcal{F}_{\!o}^{Kh}\left(\includegraphics[valign=c]{BNimages/CobIdentity.pdf}\right)&:=&
\mathbbm{1}_V
\end{array}
\renewcommand*{\arraystretch}{1}
\]
where the sign conventions are the same as in the definition of $\mathcal{F}$ (see section~\ref{subs:FandG}), and where $\mu_{v_-}\colon V\rightarrow V$ denotes the map $\mu_{v_-}:=\mu(v_-\otimes-)$.

If $S$ is a type~II cobordism, then $S$ is chronologically isotopic to a permutation cobordism $\textnormal{R}_{C,\calO,\sigma\circ\calO}$ for a permutation $\sigma\in\mathfrak{S}_n$. In this case, we set
$\mathcal{F}(S):=\textnormal{R}_{V^{\otimes n},\sigma}$
where $\textnormal{R}_{V^{\otimes n},\sigma}\colon V^{\otimes n}\rightarrow V^{\otimes n}$ is a permutation isomorphism (see section~\ref{subs:braidings}). Explicitly, $\textnormal{R}_{V^{\otimes n},\sigma}$ is obtained by writing $\sigma$ as a product of transpositions and the replacing each transposition by a map of the form $\mathbbm{1}_V\otimes\ldots\otimes\tau\otimes\ldots\otimes\mathbbm{1}_V$ where $\tau:=\tau_{V,V}\colon V\otimes V\rightarrow V\otimes V$ is as in Example~\ref{ex:SModkbraiding}. On generators:
\[  	\tau=\left\{\begin{array}{ll}
			v_+\otimes v_+ \,\,\mapsto\,\, v_+\otimes v_+ \quad\ & v_-\otimes v_+ \,\,\mapsto\,\, v_+\otimes v_-\\
			v_+\otimes v_- \,\,\mapsto\,\, v_-\otimes v_+        & v_-\otimes v_- \,\,\mapsto\,\, -v_-\otimes v_-
		\end{array}\right.
\]

\begin{proposition}\label{prop:nonannularTQFTwelldefined}
$\mathcal{F}_{\!o}^{Kh}$ is well-defined.
\end{proposition}

This is essentially a consequence of~\cite[Prop.~10.6]{PutyraChrono}. We sketch the proof by using an argument analogous to the one used in the proof of Proposition~\ref{prop:Fwelldefined}.

\begin{proof}
It suffices show that $\mathcal{F}_{\!o}^{Kh}$ respects
defining relations of the odd Bar-Natan category $\mathcal{BN}_{\!o}(\mathbb{R}^2)$ and the relations from (the non-annular version of) Lemma~\ref{lem:factorizationrelations}.

It is easy to see that $m\circ\tau=m$ and $\tau\circ\Delta=-\Delta$, which implies that $\mathcal{F}_{\!o}^{Kh}$ respects relations~\eqref{eq:Rel4} and~\eqref{eq:Rel5}. All other relations from Lemma~\ref{lem:factorizationrelations} are preserved under $\mathcal{F}_{\!o}^{Kh}$ because the maps $\tau_{U,W}\colon U\otimes W\rightarrow W\otimes U$ given by $u\otimes w\mapsto (-1)^{|u||w|}w\otimes u$ form a braiding on the monoidal supercategory $\mathcal{SM}\mathit{od}(\Bbbk)$.

Moreover, relations~\eqref{eqn:BNdisjoint} and~\eqref{eqn:BNorientation} are preserved under $\mathcal{F}_{\!o}^{Kh}$ because $\mathcal{SM}\mathit{od}(\Bbbk)$ is a monoidal supercategory and because of how we defined the sign of $\mathcal{F}_{\!o}^{Kh}(S)$ in the case where $S$ is an elementary cobordism. A direct calculation shows:
\begin{gather}
\mu\circ(\mu_{v_-}\otimes\mathbbm{1}_V)
            =\mu_{v_-}\circ\mu
		    =\mu\circ(\mathbbm{1}_V\otimes\mu_{v_-}),\label{eqn:TQFTdotslide}\\
(\mu\otimes\mathbbm{1}_V)\circ(\mathbbm{1}_V\otimes\Delta)
			= \Delta\circ\mu
			= (\mathbbm{1}_V\otimes\mu)\circ(\Delta\otimes\mathbbm{1}_V),\label{eqn:TQFTNX}\\
\mu\circ(\mu\otimes\mathbbm{1}_V) = \mu\circ(\mathbbm{1}_V\otimes\mu), \hskip 1cm
		(\Delta\otimes\mathbbm{1}_V)\circ\Delta = 
		 - (\mathbbm{1}_V\otimes\Delta)\circ\Delta,\label{eqn:TQFTassociativity}\\
		 \mu\circ\Delta=0\label{eqn:TQFTdiamond}\\
\mu\circ(\mathbbm{1}_V\otimes\eta) = \mathbbm{1}_V, \hskip 1cm
		(\epsilon\otimes\mathbbm{1}_V)\circ\Delta = \mathbbm{1}_V,\label{eqn:TQFTwave}\\
		\begin{split}
		    &\epsilon\circ\eta=0,\hskip 1cm \epsilon\circ\mu_{v_-}\circ\eta=\mathbbm{1}_{\Bbbk},\hskip 1cm\mu_{v_-}\circ\mu_{v-}=0,\\
		    &\hskip 2cm\mathbbm{1}_V=\eta\circ\epsilon\circ\mu_{v_-}+\mu_{v_-}\circ\eta\circ\epsilon.\label{eqn:TQFTBN}
		\end{split}
\end{gather}
Arguing as in the proof of Proposition~\ref{prop:Fwelldefined}, it is now easy to see that $\mathcal{F}_{\!o}^{Kh}$ respects
relation~\eqref{eqn:BNconnected} because of relations~\eqref{eqn:TQFTdotslide}, \eqref{eqn:TQFTNX}, and \eqref{eqn:TQFTassociativity};
relation~\eqref{eqn:BNcreation} because of relation~\eqref{eqn:TQFTwave};
relation~\eqref{eqn:BNBN} because of relation~\eqref{eqn:TQFTBN}; and
relation~\eqref{eqn:BNdiamond} because it sends both sides of this relation to zero, by relation~\eqref{eqn:TQFTdiamond}.
\end{proof}

\begin{remark}
The superfunctor $\mathcal{F}_{\!o}^{Kh}$ first appeared in Putyra's paper~\cite{PutyraChrono} and can be seen as an odd version of Khovanov's $(1+1)$-dimensional TQFT functor from~\cite{K}. The significance of $\mathcal{F}_{\!o}^{Kh}$ lies in the fact that it takes the odd Bar-Natan-Khovanov bracket~\cite{PutyraChrono} of a link diagram $L\subset\mathbb{R}^2$ to the odd Khovanov complex of $L$.
\end{remark}

\begin{remark}
Although our definition of $\mathcal{F}_{\!o}^{Kh}$ is based on~\cite{PutyraChrono}, the latter paper does not explicitly define an ordered version of the odd Bar-Natan category $\mathcal{BN}_{\!o}(\mathbb{R}^2)$. In fact, the issue of orderings (and the related Lemmas~\ref{lem:factorizationexistence} and~\ref{lem:factorizationrelations}) can be somewhat ignored in the non-annular setting because every ordering on the components of a closed $1$-manifold $C\subset\mathbb{R}^2$ is automatically admissible.
\end{remark}

\begin{remark}\label{rem:universalnonannularTQFT}
Working over $\Bbbk[\pi]$ where $\pi$ is a formal variable with $\pi^2=1$, one can define a superfunctor $\mathcal{F}^{Kh}_\pi$ on the category $\mathcal{OBN}_{\!\pi}(\mathbb{R}^2)$ (cf. Remark~\ref{rem:BNuniversal}). This functor is defined in the same way as $\mathcal{F}^{Kh}_o$, except that the factors of $-1$ that appear in the definitions of $\tau$ and $(f\otimes g)(v\otimes w)=(-1)^{|g||v|}f(v)\otimes g(w)$ are replaced by factors of $\pi$, and the term $v_-\otimes v_+-v_+\otimes v_-$ that appears in the definition of $\Delta$ is replaced by $v_-\otimes v_++\pi v_+\otimes v_-$.
\end{remark}

\subsection{Superfunctor $\mathcal{J}$}

Consider the composition of superfunctors
\[\TL (0) \stackrel{\mathcal{I}}{\longrightarrow}\BNA\longrightarrow \OBNA \longrightarrow \OBNR\stackrel{\mathcal{F}_{\!o}^{Kh}}{\longrightarrow}\mathcal{SM}\mathit{od}(\Bbbk)\]
where the second functor is the embedding from Lemma~\ref{lem:BNAtoOBNA} and the third functor is induced by the embedding $\ann\hookrightarrow\mathbb{R}^2$. We will call this composition $\mathcal{J}$.  To prove that $\mathcal{I}$ is faithful, we will show that $\mathcal{J}$ is faithful.

Note that $\mathcal{J}$ sends the object $n$ of $\TL (0)$ to the object $\mathcal{J}(n)=V^{\otimes n}$. On morphisms, $\mathcal{J}$ is therefore given by even linear maps
\[
\mathcal{J}\colon\operatorname{Hom}(n,m)\longrightarrow\operatorname{Hom}(V^{\otimes n},V^{\otimes m}),
\]
where $\operatorname{Hom}(n,m)$ denotes the set of all morphisms in $\TL (0)$ from $n$ to $m$. To prove that the maps above are injective and hence $\mathcal{J}$ is faithful, we will first consider the special case where $n=0$. In this case, we may further assume that $m$ is even, for otherwise $\Hom(n,m)=0$ by Remark~\ref{rmk:EvenAndOddTL}. Using this special case, we will then prove injectivity of $\mathcal{J}$ on $\Hom(n,m)$ for general $n,m$. 

\subsection{Injectivity of $\mathcal{J}$ on $\Hom(0,2n)$}

Let $[1,m]$ denote the set of all integers that lie non-strictly between 1 and $m$.  Following Khovanov~\cite{Kmatchings}, we will call a subset $I\subseteq [1,2n]$ \textbf{admissible} if $I\cap [1,m]$ has at most $\frac{m}{2}$ elements for each $m\in [1,2n]$.  Let $A_{2n}$ denote the set of all admissible subsets of $[1,2n]$.

By a \textbf{generalized cup diagram}, we shall mean a collection of disjoint cups and vertical rays which have a total of $2n$ upper endpoints lying on a horizontal line.  An example is shown in Figure~\ref{fig:generalizedcup}.
\begin{figure}[H]
\begin{center}
\incg{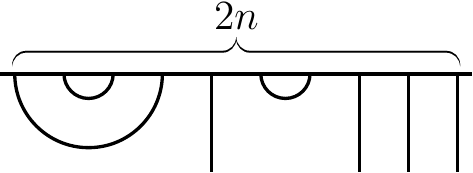}
\end{center}
\caption{A generalized cup diagram}\label{fig:generalizedcup}
\end{figure}
Here we assume that vertical rays remain disjoint from all cups when extended infinitely in the negative $y-$direction. 

Let $GC_{2n}$ denote the set of all generalized cup diagrams with $2n$ upper endpoints.

\begin{lemma}\label{lem:AGC} There is a bijection $A_{2n}\rightarrow GC_{2n}$
\end{lemma}

\begin{proof}
Let $I\in A_{2n}$ be an admissible subset of $[1,2n]$.  We will construct a generalized cup diagram as follows.  

Place $2n$ points on a horizontal line labeled $1\ldots 2n$ from left to right.   Starting from the left, for each $i\in I$ add a cup whose right endpoint is the $i$th endpoint and whose left endpoint is the nearest point to the left of $i$ that does not yet belong to a cup.  This is always possible because $I$ is admissible, so $I\cap [1,i]$ has at most $\frac{i}{2}$ elements. To complete the construction of $G(I)$, attach a vertical ray to each of the $2n$ points that does not belong to a cup.  An example for $n=4$ is shown below.
\[
I=\{3,4,8\}\quad\longmapsto\quad G(I)=\incg{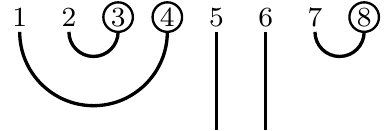}
\]

It is easy to see that $G(I)\in GC_{2n}$ and we thus obtain a map $A_{2n} \xrightarrow[]{\alpha} GC_{2n}$.  The inverse of this map, $\beta$, is given by sending $G\in GC_{2n}$ to the set of all $i\in[1,2n]$ for which the $i$th point is the right endpoint of a cup in $G$. 

It is clear that $\beta(\alpha(I))=I$ for any $I\in A_{2n}$, so $\beta$ is a left inverse of $\alpha$.  To see that $\beta$ is well defined, consider an arbitrary cup diagram $G\in GC_{2n}$.  Because $\beta$ takes right endpoints of cups to integers, this means that for each $i\in \beta (G)$ there is a $j_i\notin \beta (G)$ with $j_i<i$, and such that the $j_i$ corresponding to distinct $i$ are distinct.  Therefore, $\beta(G)$ satisfies the condition that $I\cap [1,m]$ has at most $\frac{m}{2}$ elements for each $m\in [1,2n]$. 

To show that $\alpha$ is a bijection, it now suffices to show that $A_{2n}$ and $GC_{2n}$ have the same cardinality.  It was mentioned in~\cite{Kmatchings} that $|A_{2n}|={\displaystyle \binom{2n}{n}}$, and we will show by induction that $|GC_{2n}|={\displaystyle\binom{2n}{n}}$ as well.

When $n=1$ there are only ${\displaystyle\binom{2}{1}=2}$ generalized cup diagrams possible, namely
\[
\incg{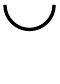}\qquad\mbox{and}\qquad\incg{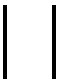}
\]
Assume there are ${\displaystyle\binom{2n}{n}}$ generalized cup diagrams on $2n$ points.  We consider two cases; those with all cups and no vertical rays, and those with at least two vertical rays.  It is known~\cite{JonesHecke} that the number of crossingless matchings on $2n$ points is given by the $n$th Catalan number, and hence the number of generalized cup diagrams with no vertical rays is given by ${\displaystyle\frac{1}{n+1}\binom{2n}{n}}$.  Each of these diagrams, $G$, can be extended to $2n+2$ points in exactly 2 ways:
\[
\incg{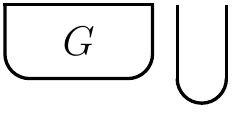}\qquad\qquad\qquad\incg{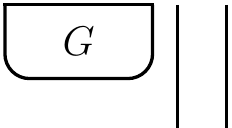}
\]
There are now ${\displaystyle\binom{2n}{n}-\frac{1}{n+1}\binom{2n}{n}=\frac{n}{n+1}\binom{2n}{n}}$ generalized cup diagrams remaining with at least two rays in their diagram.  Each of these can be extended in exactly 4 ways:  \begin{gather*}
\incg{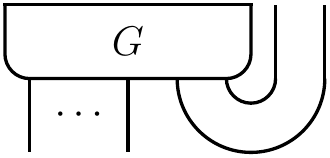}\qquad\qquad
\incg{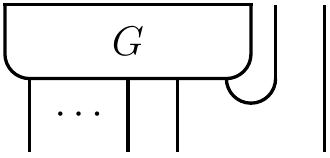}\\[0.2in]
\incg{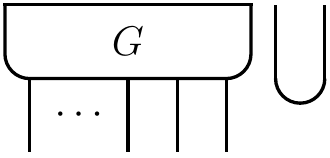}\qquad\qquad
\incg{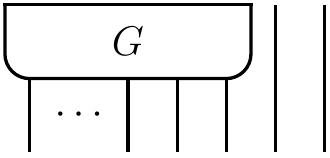}
\end{gather*}
Note that these are the only 4 ways to extend such diagrams because rays may not be inside of a cup.  Thus there are $$\frac{2}{n+1}\binom{2n}{n} + \frac{4n}{n+1}\binom{2n}{2} = \binom{2n+2}{n+1}$$ generalized cup diagrams on $2n+2$ points.
\end{proof}

Our next step is to define a map
\[
GC_{2n}\longrightarrow \Hom(0,2n).
\]
To define this map, let $G\in GC_{2n}$ and assume that the cups in $G$ have been positioned in such a way that cups whose right endpoints are further to the right occur at lower heights than cups whose right endpoints are further to the left.

Assuming that $G$ contains $2m$ rays, we can then regard $G$ as a well-defined morphism from $2m$ to $2n$ in the category $\TL (0)$.  Precomposing this morphism with 
\[
\phantom{\qquad\mbox{($m$ dotted cups)}}\incg{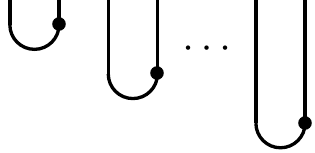}\qquad\mbox{($m$ dotted cups)}
\]
we obtain a morphism $D(G)\in \Hom(0,2n)$, and so the assignment $G\mapsto D(G)$ determines a map from $GC_{2n}$ to $\Hom (0,2n)$. Note that since a dotted cup has superdegree $0$, the relative heights of the dotted cups in the picture above could be changed without changing the underlying morphism.

\begin{lemma} \label{lm:generateGC2n}
The morphisms $D(G)$ for $G\in GC_{2n}$ generate the $\Bbbk$-module $\Hom(0,2n)$.
\end{lemma} 

\begin{proof}
Let $T$ be a chronological dotted flat $(2n,0)$-tangle representing a morphism in $\Hom(0,2n)$.  Because of the relations
\[
\incg{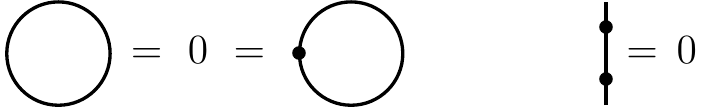}
\]
that hold in $\TL (0)$, we may assume without loss of generality that $T$ contains no closed components and at most one dot on each component.
The remaining relations that hold in $\TL (0)$ further allow us to apply arbitrary isotopies to $T$ at the possible cost of introducing a sign, and so we have $T=\pm D$ for a dotted cup diagram $D\in\Hom(0,2n)$.

Our aim is to show that $D$ is a linear combination of morphisms that are of the form $D(G)$ for $G\in GC_{2n}$.  There are two ways in which $D$ could fail to be of this form:
\begin{enumerate}
    \item A dotted arc $a$ in $D$ could be nested within one or more undotted arcs, in the sense that any path in $\mathbb{R}\times I$ that connects a point on $a$ to a point in $\mathbb{R}\times\{0\}$ intersects at least one undotted arc;
    \item The dotted arcs in $D$ may be nested within one or more dotted arcs.
\end{enumerate}
To deal with the first issue, we can use the relation
\[
\incg{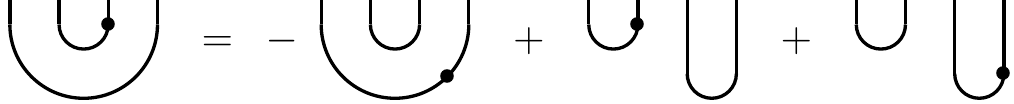}
\]
to reduce the nested-ness of dotted arcs, which follows from \eqref{eqn:TLtype1identity}.
Likewise, to deal with the second issue we use relation \eqref{eqn:TLtype2identity}.
\[
\incg{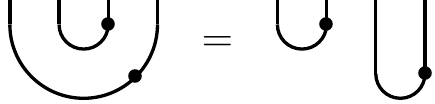}
\]
\end{proof}

We may think of a subset $I\subseteq [1,2n]$ as a sequence $s(I)=(\epsilon_1,\dots,\epsilon_{2n})\in \{\pm\}^{2n}$ where $\epsilon_i = +$ if $i\in I$ and $\epsilon_i = -$ if $i\notin I$. On such sequences, there is a natural partial order $\prec$ generated by the requirement that $(a,+,b,-,c)\prec (a,-,b,+,c)$ for any sequences $a, b, c$ in $\{\pm\}$ whose lengths add up to $2n-2$. This partial order can be extended to a total order $<$ by setting $+<-$ and by equipping sequences in $\{\pm\}$ with the induced lexicographic order.

\begin{remark}
The orders $\prec$ and $<$ induce corresponding orders on generalized cup diagrams via the bijection $A_{2n}\rightarrow GC_{2n}$. It is easy to see that the order induced by $\prec$ agrees with Khovanov's partial order from~\cite{Kmatchings} on cup diagrams with no vertical rays. For example, for $n=3$, the latter order is shown in Figure~\ref{fig:Khovanovorder}, where arrows point in direction of increasing order.
\begin{figure}[H]
\begin{center}
\begin{tikzcd}[row sep = tiny, column sep = small]
&\incg{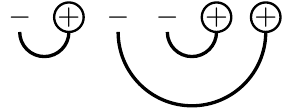}\ar[dr]&&\\
\incg{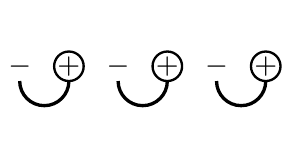}\ar[ur]\ar[dr]&&\incg{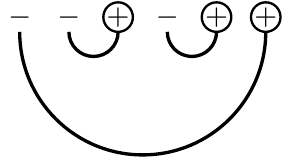}\ar[r]&\incg{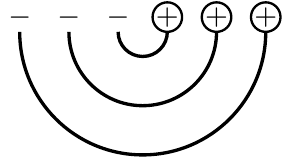}\\
&\incg{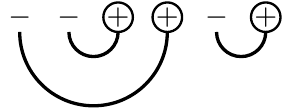}\ar[ur]&&
\end{tikzcd}
\end{center}
\caption{Khovanov's order on cup diagrams.}\label{fig:Khovanovorder}
\end{figure}
\end{remark}

Now let $I\subseteq [1,2n]$ and let $v_I\in V^{\otimes 2n}$ denote the standard basis vector
\[
v_I := v_{\epsilon_1}\otimes \dots \otimes v_{\epsilon_{2n}}
\]
where
\[   \epsilon_i := 
\begin{cases}
+ & \mbox{if $i \in I$,} \\
- & \mbox{if $i\notin I$,}
\end{cases}
\]
as before. For $I\subseteq [1,2n]$ admissible, let
\[
\tilde{v}_I := \mathcal{J}(D(G(I))) \in \Hom(\Bbbk, V^{\otimes 2n}) = V^{\otimes 2n}.
\]

\begin{lemma}
The $\tilde{v}_I$ for $I\in A_{2n}$ are linearly independent.  
\end{lemma}

\begin{proof}
Let $W_{2n}:=V^{\otimes 2n}/C_{2n}$ where $C_{2n}:=\textnormal{Span}\{v_I\,|\,I\subseteq[1,2n],I\notin A_{2n}\}$.  For $I\in A_{2n}$, let $\Bar{v}_I, \Bar{\tilde{v}}_I\in W_{2n}$ denote the images of $v_I$, $\tilde{v_I}\in V^{\otimes 2n}$ under the quotient map $V^{\otimes 2n}\rightarrow W_{2n}$. ~To prove the lemma, it will be sufficient to show that the $\{\Bar{\tilde{v}}_I | I\in A_{2n}\}$ form a basis for $W_{2n}$.

We first note that since the set of $v_J$ for all subsets $J\subseteq [1,2n]$ forms a basis of $V^{\otimes 2n}$, the set $\{\Bar{v}_J | J\in A_{2n}\}$ forms a basis for $W_{2n}$.  For $I\in A_{2n}$, we can thus write $\Bar{\tilde{v}}_I$ as
\[
\Bar{\tilde{v}}_I = \sum_{J\in A_{2n}}a_{IJ}\Bar{v}_J
\]
for suitable coefficients $a_{IJ}\in \Bbbk$. We will now show by induction that when $A_{2n}$ is given the total order $<$ described above, the matrix $[a_{IJ}]$ is upper triangular and has diagonal entries $\pm 1$. In particular, $[a_{IJ}]$ is invertible, and thus the $\Bar{\tilde{v}}_I$ for $I\in A_{2n}$ form a basis for $W_{2n}$. To prove that $[a_{IJ}]$ has the desired from, we will actually show
\begin{equation}\label{eqn:vtildev}
\tilde{v}_I=\pm v_I+\sum_{J\prec I}a_{IJ}v_J
\end{equation}
for suitable coefficients $a_{IJ}\in\Bbbk$, where the sum runs over all (possibly inadmissible) subsets $J\subseteq [1,2n]$ which are strictly less than $I$ in the partial order $\prec$.

\vspace*{0.1in}
\noindent
\textbf{Base Case:} $n=1$
\vspace*{0.1in}

\noindent
The admissible subsets of $[1,2]$ are $\{2\}$ and $\emptyset$, and the standard basis vectors associated with these sets are $v_{\{2\}}=v_{-}\otimes v_{+}$ and $v_\emptyset=v_{-}\otimes v_{-}$.
On the other hand, the associated morphisms $D(G(I))$ for these two sets are $D(\incg{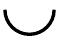})=\incg{GCmpsmall}$
and 
$D(\incg{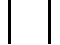})=\incg{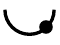}$
and their images under $\mathcal{J}$ are given by
\begin{align*}
\mathcal{J}\left(\incg{GCmpsmall}\right)&=\mathcal{F}_{\!o}^{Kh}\left(\incg{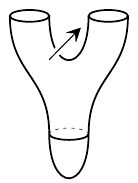}\right)=\Delta\circ\eta,\\
\mathcal{J}\left(\incg{GCmmsmalldot}\right)&=\mathcal{F}_{\!o}^{Kh}\left(\incg{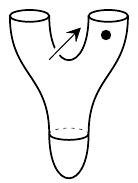}\right)=(\mathbbm{1}\otimes\mu_{v_-})\circ\Delta\circ\eta.
\end{align*}
Hence $\tilde{v}_{\{2\}}=(\Delta\circ\eta)(1)=v_-\otimes v_+-v_+\otimes v_-$ and $\tilde{v}_\emptyset=((\mathbbm{1}\otimes\mu_{v_-})\circ\Delta\circ\eta)(1)=v_-\otimes v_-$, which we can write as
\[
\tilde{v}_{-+}=v_{-+}-v_{+-}\qquad\mbox{and}\qquad\tilde{v}_{--}=v_{--},
\]
where here we have written subsets of $[1,2]$ as sequences in $\{\pm\}$.
Equation~\eqref{eqn:vtildev} thus follows because $(+,-)\prec (-,+)$.

\vspace*{0.1in}
\noindent
\textbf{Inductive Step:} $n>1$
\vspace*{0.1in}

\noindent
Suppose~\eqref{eqn:vtildev} holds for $n-1$ and let $I\in A_{2n}$. If $I=\emptyset$, then it is easy to see that $v_I=v_-\otimes\ldots\otimes v_-=\tilde{v}_I$, and hence this case is trivial. Now assume $I\neq\emptyset$. Then the generalized cup diagram $G(I)$ contains at least one cup. Let $I'\in A_{2n-2}$ be the admissible subset whose associated generalized cup diagram $G(I')$ is obtained from $G(I)$ by removing the cup whose right endpoint lies furthest to the left. Equivalently, this means that $s(I')$ is obtained from the sequence $s(I)$ by removing the $(-,+)$-consecutive subsequence that lies furthest to the left. By induction, we can now write $\tilde{v}_{I'}$ as
\[
\tilde{v}_{I'}=\pm v_{I'}+\sum_{J'\prec I'}a'_{I'J'}v_{J'}.
\]
Moreover, we can write each $v_{J'}$ in the formula above as $v_{J'}=v_{J'_1}\otimes v_{J'_2}$, where $J_1'$ and $J_2'$ correspond to the subsequences of $s(J)$ that lie to the left and to the right of the removed $(-,+)$-subsequence. The formula above then becomes
\[
\tilde{v}_{I'}=\pm (v_{I'_1}\otimes v_{I'_2})+\sum_{J'\prec I'}a'_{I'J'}(v_{J'_1}\otimes v_{J'_2}),
\]
and since $\mathcal{J}(\incg{GCmpsmall})=\Delta\circ\eta$, we obtain
\[
\tilde{v}_I=(\mathbbm{1}\otimes(\Delta\circ\eta)\otimes\mathbbm{1})\left(\pm (v_{I'_1}\otimes 1\otimes v_{I'_2})+\sum_{J'\prec I'}a'_{I'J'}(v_{J'_1}\otimes 1\otimes v_{J'_2})\right).
\]
Since $(\Delta\circ\eta)(1)=\tilde{v}_{-+}$, we can further write this as
\[
\tilde{v}_I=\pm(-1)^{|v_{I'_1}|}\bigl(v_{I'_1}\otimes \tilde{v}_{-+}\otimes v_{I'_2}\bigr)\,+\,\sum_{J'\prec I'}(-1)^{|v_{J'_1}|} a'_{I'J'}\bigl(v_{J'_1}\otimes \tilde{v}_{-+}\otimes v_{J'_2}\bigr),
\]
and using that $\tilde{v}_{-+}=v_{-+}-v{+-}$, it is now easy to see that $\tilde{v}_I$ is a sum of the term $\pm(-1)^{|v_{I'_1}|}v_I$, along with terms that are of the form $a_{IJ}v_J$ for $J\prec I$. Hence~\eqref{eqn:vtildev} follows.
\end{proof}

In view of the previous lemma and since $\mathcal{J}$ sends $D(G(I))$ to $\tilde{v_I}$, this implies:

\begin{corollary}\label{cor:basisHom}
    The $D(G)$ for $G\in GC_{2n}$ form a basis for $\Hom (0,2n)$.
\end{corollary}  

Since $\mathcal{J}$ sends the basis elements of $D(G)$ to the linearly independent vectors $\tilde{v_I}$, we thus obtain:

\begin{corollary}
    $\mathcal{J} : \Hom (0,2n)\rightarrow \Hom (\Bbbk, V^{\otimes 2n})$ is injective.
\end{corollary}

\subsection{Injectivity of $\mathcal{J}$ on $\Hom(n,m)$:}

Since $\Hom (n,m)=0$ if $n+m$ is odd, we will assume that $n+m$ is even.  In this case, we define a $\Bbbk-$linear map $\alpha:\Hom(n,m)\rightarrow\Hom(0,n+m)$ by
\[
\alpha \left(\incg{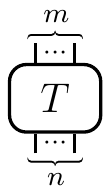}\right):=\incg{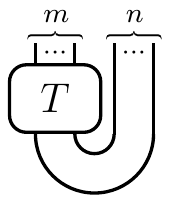}.
\]
Likewise, we define a $\Bbbk-$linear map $\beta : \Hom(0, n+m) \rightarrow \Hom(n,m)$ by
\[
\beta\left(\incg{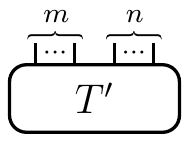}\right):=(-1)^{n|T'|+\frac{n(n-1)}{2}}\incg{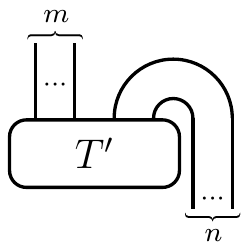}.
\]

\begin{lemma}\label{lem:abinverse}
    $\alpha$ and $\beta$ are mutually inverse isomorphisms.
\end{lemma} 

\begin{proof}
    We have
\begin{align*}    
    \beta\left(\alpha\left(\incg{GCpic1}\right)\right)&=\beta\left(\incg{GCpic2}\right)\\
    &=(-1)^{n|T|+n^2 + \frac{n(n-1)}{2}} \incg{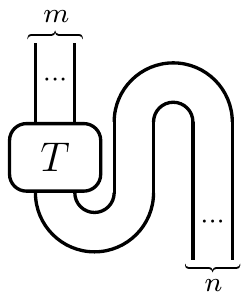}\\
&=(-1)^{n^2+\frac{n(n-1)}{2}}\incg{Gcpic6}=(-1)^{n^2}\incg{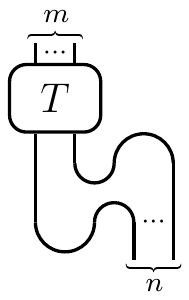}=\incg{GCpic1}
\end{align*}
where in the last equality we have used $n^2 \equiv n \textnormal{ mod }2$ and~\incg{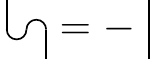}.  Likewise,
\begin{align*}
\alpha\left(\beta\left(\incg{GCpic3}\right)\right)&=(-1)^{n|T'|+\frac{n(n-1)}{2}}\alpha\left(\incg{GCpic4}\right)\\
&=(-1)^{n|T'|+\frac{n(n-1)}{2}}\incg{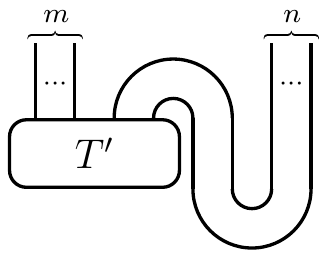}\\
&=(-1)^{\frac{n(n-1)}{2}}\incg{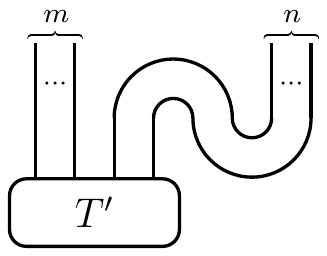}
=\incg{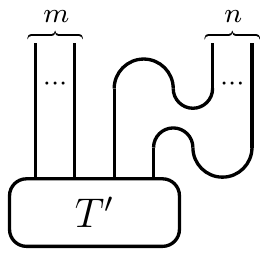}=\incg{GCpic3}
\end{align*}
where in the last equation we have used~\incg{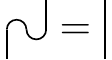}.
\end{proof}

\begin{lemma}
    The map $\mathcal{J}:\Hom(n,m)\rightarrow\Hom_{\Bbbk}(V^{\otimes n}, V^{\otimes m})$ is injective.
\end{lemma} 

\begin{proof}
    Since the isomorphism $\alpha$ can be defined entirely in terms of cup tangles and in terms of the super-monoidal structure of $\TL (0)$, there is a corresponding morphism, $\alpha '$, between morphism sets in $Mod(\Bbbk)$, which makes the following diagram commute:

\[
\begin{tikzcd}[row sep=large]
{\Hom(m,n)} \arrow[r, "\mathcal{J}"] \arrow[d, "\alpha"'] &  {\Hom(V^{\otimes n},V^{\otimes m})} \arrow[d, "\alpha'"]\\
{\Hom(0,m+n)} \arrow[r, "\mathcal{J}"']& {\Hom(\Bbbk,V^{\otimes (m+n)})}
\end{tikzcd}
\]

We have already seen that the horizontal arrow at the bottom is injective, and hence commutivity of the diagram implies that the horizontal arrow of the top is injective as well.
\end{proof} 

The lemma above completes the proof that $\mathcal{J}$ (and hence $\mathcal{I}$) is faithful. In view of Corollary~\ref{cor:basisHom}, we further obtain:
\begin{corollary}\label{cor:basisHomnm}
The $\beta(D(G))$ for $G\in GC_{n+m}$ form a basis for $\Hom(n,m)$.
\end{corollary}
\begin{corollary} The $\Bbbk$-module
$\Hom(n,m)$ is free of dimension
\[
\dim\Hom(n,m)=|GC_{n+m}|=\binom{n+m}{(n+m)/2}.
\]
\end{corollary}

\subsection{An alternative proof of Proposition~\ref{prop:Gwelldefined}}
We will now use the faithfulness of $\mathcal{J}$ to give an alternative proof of the fact that the functor $\mathcal{G}$ from section~\ref{subs:FandG} is well-defined. To this end, we denote by $\widetilde{\RG}$ the monoidal supercategory which is defined in the same way as $\RG$, but without imposing any of the relations, except for relation~\eqref{eqn:RGinterchange}. On $\widetilde{\RG}$, we can define a monoidal superfunctor
\[
\widetilde{\mathcal{G}}\colon\widetilde{\RG}\longrightarrow(\TL (0)^s)^\oplus
\]
by using the same formulas as in the definition of $\mathcal{G}$. Note that $\widetilde{\mathcal{G}}$ is well-defined because it sends relation~\eqref{eqn:RGinterchange} to relation~\eqref{eqn:TLinterchange}.

Next, consider the composition
\[
\mathcal{K}:=\mathcal{J}\circ\widetilde{\mathcal{G}}\colon\widetilde{\RG}\longrightarrow\mathcal{SM}\mathit{od}(\Bbbk),
\]
and notice that $\mathcal{K}$
satisfies $\mathcal{K}(\bullet)=\mathcal{J}(1)=V$
and $\mathcal{K}(\circ)=\mathcal{J}(0_0\oplus 0_1)=\Bbbk[0]\oplus\Bbbk[1]\cong V$, where the numbers in square brackets stand for formal shifts of the supergrading. A direct calculation further shows:
\[
\renewcommand*{\arraystretch}{2}
\begin{array}{cclccl}
\mathcal{K}\left(\incg{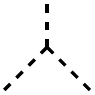}\right)&=&\mu,\qquad&\qquad
\mathcal{K}\left(\incg{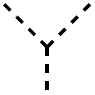}\right)&=&\Delta,\\
\mathcal{K}\left(\incg{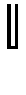}\right)&=&\eta,\qquad&\qquad\mathcal{K}\left(\incg{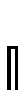}\right)&=&\epsilon,
\\
\mathcal{K}\left(\incg{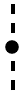}\right)&=&\mu_{v_-},\qquad&\qquad\mathcal{K}\left(\incg{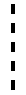}\right)&=&
\mathbbm{1}_V.
\end{array}
\renewcommand*{\arraystretch}{1}
\]
Now observe that $\mathcal{K}$ respects all of the relations in $\RG$ because it sends these relations to relations~\eqref{eqn:TQFTdotslide} through \eqref{eqn:TQFTBN} from the proof of Proposition~\ref{prop:nonannularTQFTwelldefined}. Hence, since $\mathcal{J}$ is faithful, $\widetilde{\mathcal{G}}$ must respect these relations as well. In conclusion, we have shown that $\widetilde{\mathcal{G}}$ descends to a well-defined functor on $\RG$, which re-proves Proposition~\ref{prop:Gwelldefined}.

\begin{remark}
Like the original proof of Propostion~\ref{prop:Gwelldefined}, the proof above requires a number of direct verifications. However, the proof above is less cumbersome than the original proof because, in verifying the relations from the proof of Proposition~\ref{prop:nonannularTQFTwelldefined}, one does not have to distinguish between trivial and essential edges.
\end{remark}

\subsection{Odd annular Khovanov TQFT}\label{subs:annularTQFT}
Since $\mathcal{J}$ is faithful and $\mathcal{I}$ is full and essentially surjective on objects, we obtain:

\begin{corollary}
The inclusion-induced functor $\OBNA\rightarrow\OBNR$ is faithful, but neither full nor essentially injective on objects.
\end{corollary}

In view of this corollary, we can regard $\OBNA$ as a (non-full) subcategory of $\OBNR$. Restricting the odd non-annular Khovanov TQFT to this subcategory, we thus obtain a monoidal superfunctor
\[
\mathcal{F}_{\!o}^{Kh}|\ann\colon\OBNA\longrightarrow\mathcal{SM}\mathit{od}(\Bbbk).
\]
Using the datum of the annulus, we will now promote this superfunctor to a superfunctor with values in the filtered supercategory $\mathcal{SM}\mathit{od}_f(\Bbbk)$ from Example~\ref{ex:modf}. To this end, we identify the supermodule $\mathcal{F}_{\!o}^{Kh}(C,\mathcal{O})=V^{\otimes n}$ assigned to an object $(C,\mathcal{O})$ of $\OBNA$ with the tensor product
\[
\mathcal{F}_{\!o}^{Kh}(C,\mathcal{O})=
\mathcal{F}_{\!o}^{Kh}(C_1)\otimes\ldots\otimes\mathcal{F}_{\!o}^{Kh}(C_n),
\]
where $C_1<\ldots<C_n$ are the components of $C$. We define the \textbf{annular grading} on this tensor product by $a(v_{\epsilon_{1}}\otimes\ldots\otimes v_{\epsilon_{n}}):=a(v_{\epsilon_{1}})+\ldots+a(v_{\epsilon_{n}})\in\mathbb{Z}$ where $\epsilon_i\in \{ \pm \}$, and where $a(v_{\epsilon_{i}})$ is given by:
\[
a(v_{\pm}):=\begin{cases}
0&\mbox{if $C_i$ is trivial},\\
\pm 1&\mbox{if $C_i$ is essential}.
\end{cases}
\]

The following is essentially Lemma~3 from~\cite{GL11Grigsby}:

\begin{lemma}
If $S$ is a morphism in $\OBNA$, then the linear map $\mathcal{F}_{\!o}^{Kh}(S)$ is non-increasing with respect to the annular grading.
\end{lemma}

\begin{proof}
It suffices to prove the lemma in the case where $S$ is an elementary cup, cap, or saddle cobordism in $\ann\times I$. If $S\subset\ann\times I$ is an elementary cup or a cap cobordism, then $S$ has a unique and necessarily trivial boundary component $C$. Hence the linear map $\mathcal{F}_{\!o}^{Kh}(S)$ (which is either given by $\eta$ or by $\epsilon$) is grading-preserving becasue $\mathcal{F}_{\!o}^{Kh}(C)$ is supported in annular degree zero. On the other hand, if $S$ is an elementary merge or split coobrdism, then the map $\mathcal{F}_{\!o}^{Kh}(S)$ is given by $\mu$ or $\Delta$. In this case, one has to take into account which of the boundary components of $S$ are trivial or essential. If all boundary components are trivial, then $\mathcal{F}_{\!o}^{Kh}(S)$ is again grading-preserving, whereas in the remaining cases, a direct inspection shows that $\mathcal{F}_{\!o}^{Kh}(S)$ splits into a component that preserves the annular grading and a component that lowers the annular grading by $2$. For a more explicit analysis of one of these cases, see Example~\ref{ex:SmuDelta} below.
\end{proof}

The lemma immediately implies that $\mathcal{F}_{\!o}^{Kh}|\ann$ can be viewed as a functor with values in the filtered category $\mathcal{SM}\mathit{od}_f(\Bbbk)$ from Example~\ref{ex:modf}, where the filtered structure comes from the annular grading. Composing with the quotient functor $\mathcal{SM}\mathit{od}_f(\Bbbk)\rightarrow\mathcal{SM}\mathit{od}_g(\Bbbk)$ which annihilates morphisms that strictly decrease the degree, one obtains a functor with values in the category $\mathcal{SM}\mathit{od}_g(\Bbbk)$ from Example~\ref{ex:modg}. This latter functor annihilates any cobordism $S\subset\ann\times I$ that contains a dot on an essential component, and thus descends to a functor on the ordered version of $\mathcal{BBN}_{\!o}(\ann)$. We will denote this induced functor by $\mathcal{F}_{\!o}^{AKh}$ and call it the \textbf{odd annular Khovanov TQFT}. Thus, the odd annular Khovanov TQFT is the unique functor $\mathcal{F}_{\!o}^{AKh}$ which makes the following diagram commute, where the horizontal arrows are quotient functors:
\[
\begin{tikzcd}
\OBNA\ar[r]\ar[d,"\mathcal{F}_{\!o}^{Kh}|\ann"']&\OBBNA\ar[d,"\mathcal{F}_{\!o}^{AKh}"]\\
\mathcal{SM}\mathit{od}_f(\Bbbk)\ar[r]&\mathcal{SM}\mathit{od}_g(\Bbbk)
\end{tikzcd}
\]

\begin{example}\label{ex:SmuDelta} As in Example~\ref{ex:annularsaddles}, let $S_\mu$ be a saddle cobordism which merges two essential components into a trivial component $C$, and $S_\Delta$ be a saddle cobordism which splits $C$ into two essential components. Then the maps that $\mathcal{F}_{\!o}^{Kh}|\ann$ associates to $S_\mu$ and $S_\Delta$ are given by the multiplication and the comultiplication
\begin{align*}
		\mu=&\left\{\begin{array}{ll}
			\fbox{$v_+\otimes v_+ \,\,\mapsto\,\, v_+,$} \quad\ & v_-\otimes v_+ \,\,\mapsto\,\, v_-,\\
			\hspace*{0.049in}v_+\otimes v_- \,\,\mapsto\,\, v_-,        & v_-\otimes v_- \,\,\mapsto\,\, 0,
		\end{array}\right. \\
		\Delta=&\left\{\begin{array}{l}
			\hspace*{0.049in}v_+ \,\,\mapsto\,\, v_-\otimes v_+ -v_+\otimes v_- ,\\
			\fbox{$v_- \,\,\mapsto\,\, v_-\otimes v_-,$}
		\end{array}\right.
\end{align*}
where the boxed components have degree $-2$ and all other components have degree $0$. Since $\mathcal{F}_{\!o}^{AKh}(S_\mu)$ and $\mathcal{F}_{\!o}^{AKh}(S_\Delta)$ are the degree $0$ parts of $(\mathcal{F}_{\!o}^{Kh}|\ann)(S_\mu)$ and $(\mathcal{F}_{\!o}^{Kh}|\ann)(S_\Delta)$, it follows that $\mathcal{F}_{\!o}^{AKh}(S_\mu)$ and $\mathcal{F}_{\!o}^{AKh}(S_\Delta)$ are given by the maps
\begin{align*}
		\mu':=&\left\{\begin{array}{ll}
			v_+\otimes v_+ \,\,\mapsto\,\, 0, \quad\ & v_-\otimes v_+ \,\,\mapsto\,\, v_-,\\
			v_+\otimes v_- \,\,\mapsto\,\, v_-,        & v_-\otimes v_- \,\,\mapsto\,\, 0,
		\end{array}\right. \\
		\Delta':=&\left\{\begin{array}{l}
			v_+ \,\,\mapsto\,\, v_-\otimes v_+ -v_+\otimes v_- ,\\
			v_- \,\,\mapsto\,\, 0.
		\end{array}\right.
\end{align*}
\end{example}

\begin{example}
The cobordisms $\mathcal{I}(\incg{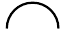})$ and $\mathcal{I}(\incg{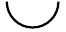})$ are obtained by capping off the trivial boundary components in cobordisms of the form $S_\mu$ and $S_\Delta$, where $S_\mu$ and $S_\Delta$ are as in the previous example. Hence $\mathcal{F}_o^{AKh}$ sends $\mathcal{I}(\incg{TLsmallcap})$ and $\mathcal{I}(\incg{TLsmallcup})$ to the maps
\begin{align}
		\mathcal{F}_o^{AKh}(\mathcal{I}(\incg{TLsmallcap}))=\epsilon\circ\mu'=&\left\{\begin{array}{ll}
			v_+\otimes v_+ \,\,\mapsto\,\, 0, \quad\ & v_-\otimes v_+ \,\,\mapsto\,\, 1,\\
			v_+\otimes v_- \,\,\mapsto\,\, 1,        & v_-\otimes v_- \,\,\mapsto\,\, 0,
		\end{array}\right.\label{eqn:epsilonmuprime} \\
		\mathcal{F}_o^{AKh}(\mathcal{I}(\incg{TLsmallcup}))=\Delta'\circ\eta=&\left\{\begin{array}{l}
			1 \,\,\mapsto\,\, v_-\otimes v_+ -v_+\otimes v_- .\label{eqn:Deltaprimeeta}
		\end{array}\right.
\end{align}
\end{example}

\begin{example} 
Suppose $S_\mu$ and $S_\Delta$ are as above. Then $\mathcal{F}_{\!o}^{Kh}|\ann$ sends the cobordism
$S_\Delta\circ S_\mu$ to the map
\begin{align*}
		\Delta\circ\mu=&\left\{\begin{array}{ll}
			\fbox{$v_+\otimes v_+ \,\,\mapsto\,\, v_-\otimes v_+ -v_+\otimes v_-,$} \quad & \fbox{$v_-\otimes v_+ \,\,\mapsto\,\, v_-\otimes v_-,$}\\[0.06in]
			\fbox{$v_+\otimes v_- \,\,\mapsto\,\, v_-\otimes v_-,$}        & \fbox{$v_-\otimes v_- \,\,\mapsto\,\, 0,$}
		\end{array}\right.
\end{align*}
where all components have degree $-2$. Thus, the map $(\mathcal{F}_{\!o}^{Kh}|\ann)(S_\Delta\circ S_\mu)$ has filtered degree $-2$, while the maps $(\mathcal{F}_{\!o}^{Kh}|\ann)(S_\mu)$ and
$(\mathcal{F}_{\!o}^{Kh}|\ann)(S_\Delta)$ have filtered degree zero because their degree zero components, which are given by $\mu'$ and $\Delta'$, are nonzero. Since the functor $\mathcal{F}_{\!o}^{Kh}|\ann$ is filtered, this implies that the inequalities from Example~\ref{ex:annularsaddles} are actually equalities.
\end{example}

\begin{remark}\label{rem:universalannularTQFT}
By using the functor $\mathcal{F}^{Kh}_\pi$ from Remark~\ref{rem:universalnonannularTQFT} instead of $\mathcal{F}^{Kh}_o$, one can define a functor $\mathcal{F}^{AKh}_\pi$ on the universal category $\mathcal{OBBN}_{\!\pi}(\ann)$ (cf. Remark~\ref{rem:BNuniversal}).
\end{remark}

%\end{document}

\section{Relationship with $\mathfrak{gl}(1|1)$}\label{sec:Chapter7}
    %\documentclass[../main.tex]{subfiles}
%
%\begin{document}
In this section, we will review basic facts about Lie superalgebras and then establish the commutativity of \eqref{eqn:commute} up to even supernatural isomorphism.
For more details on Lie superalgebras and on $\mathfrak{gl}(1|1)$, we refer the reader to~\cite{Kac,Sar}.

\subsection{Lie superalgebras}
In this section, all supermodules will be modules over a commutative unital ring $\Bbbk$, and the word linear will mean $\Bbbk$-linear unless otherwise stated.

\begin{definition}\label{def:Liesuperalgebra}
A \textbf{Lie superalgebra} is a supermodule $\mathfrak{g}$ together with a $\Bbbk$-bilinear map $[-,-]\colon \mathfrak{g}\times\mathfrak{g}\rightarrow\mathfrak{g}$ called the \textbf{Lie superbracket}, such that
\begin{enumerate}
    \item[(1)] $|[X,Y]|=|X|+|Y|$,
    \item[(2)] $[X,Y]+(-1)^{|X||Y|}[Y,X]=0$,
    \item[(3)] $[X,[Y,Z]]=[[X,Y],Z]+(-1)^{|X||Y|}[Y,[X,Z]]$
\end{enumerate}
for all homogeneous $X,Y,Z\in\mathfrak{g}$.
\end{definition}

\begin{example} If $V$ is a supermodule, then the general linear superalgebra $\mathfrak{gl}(V)$ is
the set of all linear endomorphisms $V\rightarrow V$, with Lie superbracket given by the supercommutator $[f,g]_s:=f\circ g-(-1)^{|f||g|}g\circ f$.
\end{example}

\begin{definition}\label{def:representation}
A \textbf{representation} of a Lie superalgebra $\mathfrak{g}$ is a supermodule $V$ together with an even linear map $\rho_V\colon\mathfrak{g}\rightarrow\mathfrak{gl}(V)$, called the \textbf{action} of $\mathfrak{g}$, such that
\[\rho_V([X,Y]_\mathfrak{g})=[\rho_V(X),\rho_V(Y)]_s\] for all $X,Y\in\mathfrak{g}$.
\end{definition}

\begin{example} If $V$ is a supermodule, then the identity map $\mathfrak{gl}(V)\rightarrow\mathfrak{gl}(V)$ is a representation of $\mathfrak{gl}(V)$, called the fundamental representation of $\mathfrak{gl}(V)$.
\end{example}

\begin{example}\label{ex:trivialrep}
If $\mathfrak{g}$ is a Lie superalgebra, then the zero map $\mathfrak{g}\rightarrow\mathfrak{gl}(\Bbbk)$ is a representation of $\mathfrak{g}$, called the trivial representation of $\mathfrak{g}$. Here, $\Bbbk$ is viewed as a supermodule over itself with homogeneous components $\Bbbk_0:=\Bbbk$ and $\Bbbk_1:=0$.
\end{example}

Given a representation $(V,\rho_V)$ of a Lie superalgebra, we will often write the action of $\mathfrak{g}$ on $V$ as $Xv:=\rho_V(X)(v)$ for $X\in\mathfrak{g}$ and $v\in V$. A homogeneous linear map $f\colon V\rightarrow W$ between two representations of $\mathfrak{g}$ is then called a \textbf{homomorphism} of representations if it satsifies
\begin{equation}\label{eqn:rephomomorphism}
Xf(v)=(-1)^{|f||X|}f(Xv)
\end{equation}
for all homogeneous $X\in\mathfrak{g}$ and all $v\in V$. More generally, we will say that a linear map $f\colon V\rightarrow W$ is a homomorphism of representations if its homogeneous components are homomorphisms of representations. It is easy to see that any composition of two homomorphisms of representations is again a homomorphism of representations.

For any two representations $V$ and $W$, the tensor product $V\otimes W$ is again a representation with action defined by
\[
X(v\otimes w):=(Xv)\otimes w+(-1)^{|X||v|}v\otimes(Xw).
\]
Note that this definition makes the tensor product of representations strictly unital and strictly associative in the sense that the canonical identifications of supermodules $\Bbbk\otimes V=V=V\otimes\Bbbk$ and $(U\otimes V)\otimes W=U\otimes(V\otimes W)$ intertwine the Lie superalgebra actions. The action on an $n$-fold tensor product is given explicitly as follows:
\[\label{eqn:nfoldtensor}
X(v_1\otimes\cdots\otimes v_n):=\sum_{i=1}^n(-1)^{|X|(|v_1|+\ldots+|v_{i-1}|)}v_1\otimes\cdots\otimes(Xv_i)\otimes\cdots\otimes v_n.
\]
One can further check that the twist map $\tau_{V,W}\colon V\otimes W\rightarrow W\otimes V$ given by 
\begin{equation}\label{eqn:twist}
\tau_{V,W}(v\otimes w):=(-1)^{|v||w|}w\otimes v
\end{equation}
provides an even isomorphism of representations $V\otimes W\cong W\otimes V$.

Like the tensor product of representations, the space of linear maps $\operatorname{Hom}(V,W)$ is again a representation with action given by
\[
(Xf)(v):=X(f(v))-(-1)^{|X||f|}f(Xv),
\]
where $f\in\operatorname{Hom}(V,W)$ and $v\in V$. In particular, it follows that a linear map $f\in\operatorname{Hom}(V,W)$ is a homomorphism of representations (in the sense described earlier) if and only if $Xf=0$ for all $X\in\mathfrak{g}$.

If $\Bbbk$ denotes the trivial representation, then it further follows that the canoncial identification $W=\operatorname{Hom}(\Bbbk,W)$ intertwines the Lie superalgebra actions. Moreover, the dual space $V^*:=\operatorname{Hom}(V,\Bbbk)$ of a representation is again a representation, with Lie superalgebra action given by
\[
(Xf)(v):=-(-1)^{|X||f|}f(Xv)
\]
for $f\in V^*$ and $v\in V$.

\begin{lemma}\label{lem:dualtensor}
The linear map $S\colon V^*\otimes W^*\rightarrow(V\otimes W)^*$ induced by sending a pair $f,g$ of linear functionals on $V,W$ to the tensor product $f\otimes g$ defined as in Example~\ref{ex:SModk} is an even homomorphism of representations.
\end{lemma}

\begin{proof} It is clear that $S$ is even. To see that $S$ intertwines the Lie superalgebra action, let $f\otimes g\in V^*\otimes W^*$ and $v\otimes w\in V\otimes W$. Then
\begin{align*}
\MoveEqLeft (XS(f\otimes g))\bigl(v\otimes w\bigr)=
-(-1)^{|X|(|f|+|g|)}S(f\otimes g)\bigl(X(v\otimes w)\bigr)\\
&=-(-1)^{|X|(|f|+|g|)}S(f\otimes g)\Bigl((Xv)\otimes w+(-1)^{|X||v|}v\otimes(Xw)\Bigr)\\
&=-(-1)^{|X|(|f|+|g|)}\Bigl((-1)^{|g|(|X|+|v|)}f(Xv)\otimes g(w)-(-1)^{|X||v|+|g||v|}f(v)\otimes g(Xw)\Bigr)\\
&=-(-1)^{|g||v|+|X||f|}f(Xv)\otimes g(w)-(-1)^{|X||f|+(|X|+|g|)|v|+|X||g|}f(v)\otimes g(Xw)\\
&=(-1)^{|g||v|}(Xf)(v)\otimes g(w)+(-1)^{|X||f|+(|X|+|g|)|v|}f(v)\otimes(Xg)(w)\\
&=S\Bigl((Xf)\otimes g)+(-1)^{|X||f|}f\otimes(Xg)\Bigr)\bigl(v\otimes w\bigr)\\
&=S(X(f\otimes g))\bigl(v\otimes w\bigr),
\end{align*}
and hence $XS(f\otimes g)=S(X(f\otimes g))$.
\end{proof}

A related result is the following:

\begin{lemma}
Suppose $f\colon V\rightarrow V'$ and $g\colon W\rightarrow W'$ are homomorphisms of representations, in the sense of~\eqref{eqn:rephomomorphism}. Then $f\otimes g$ is again a homomorphism of representations, where $f\otimes g$ denotes the linear map defined in Example~\ref{ex:SModk}.
\end{lemma}

\begin{proof}
Let $v\otimes w\in V\otimes W$. Then
\begin{align*}
\MoveEqLeft
X\bigl((f\otimes g)(v\otimes w)\bigr)=(-1)^{|g||v|}X\bigl(f(v)\otimes g(w)\bigr)\\
&=(-1)^{|g||v|}\Bigl((X(f(v)))\otimes g(w)+(-1)^{|X|(|f|+|v|)}f(v)\otimes(X(g(w)))\Bigr)\\
&=(-1)^{|g||v|+|X||f|}f(Xv)\otimes g(w)+(-1)^{|g||v|+|X|(|f|+|v|)+|X||g|}f(v)\otimes g(Xw)\\
&=(-1)^{|X||f|+|g||X|}(f\otimes g)((Xv)\otimes w)+(-1)^{|X|(|f|+|v|)+|X||g|}(f\otimes g)(v\otimes (Xw))\\
&=(-1)^{|X|(|f|+|g|)}(f\otimes g)\Bigl((Xv)\otimes w+(-1)^{|X||v|}v\otimes (Xw)\Bigr)\\
&=(-1)^{|X|(|f|+|g|)}(f\otimes g)\bigl(X(v\otimes w)\bigr),
\end{align*}
where the third equality follows from the assumption that $f$ and $g$ are homomorphisms of representations.
\end{proof}

Given a Lie superalgebra $\mathfrak{g}$, we thus see that the category $\mathcal{R}\!\mathit{ep}(\mathfrak{g})$ whose objects are representations of $\mathfrak{g}$ and whose morphisms are homomorphsims of representations is a monoidal supercategory. This category is symmetric with braiding given by the isomorphisms $\tau_{V,W}$.

\subsection{Evaluations and coevaluations}\label{subs:evcoev}

Let $V$ be a representation of a Lie superalgebra $\mathfrak{g}$ and consider the linear maps
\newlength{\evcoevlength}
\setlength{\evcoevlength}{\widthof{${\scriptstyle \tau_{V,V^*}}$}}
\[
\operatorname{End}(V)%
\xleftarrow{\makebox[\evcoevlength][c]{${\scriptstyle a}$}}%
V\otimes V^*%
\xrightarrow{\makebox[\evcoevlength][c]{${\scriptstyle \tau_{V,V^*}}$}}%
V^*\otimes V
\xrightarrow{\makebox[\evcoevlength][c]{${\scriptstyle ev}$}}%
\Bbbk
\]
where $\tau_{V,V^*}$ is the twist map defined in \eqref{eqn:twist} and $a$ and $ev$ are defined by
\[
a(v\otimes f)(w):=f(w)v,\qquad ev(f\otimes v):=f(v)
\]
for $v,w\in V$ and $f\in V^*$. We will call $ev$ the \textbf{evaluation map}.
Note that $\tau_{V,V^*}$ is an even isomorphism of representations. We further have:

\begin{lemma}
The maps $a$ and $ev$ are even homomorphisms of representations.
\end{lemma}
\begin{proof}
It is clear that $a$ and $ev$ are even. Moreover, 
\begin{align*}
(Xa(v\otimes f))(w)&=X(f(w)v)-(-1)^{|X|(|v|+|f|)}f(Xw)v\\
&=f(w)(Xv)+(-1)^{|X||v|}((Xf)(w))v\\
&=a\Bigl((Xv)\otimes f+(-1)^{|X||v|}v\otimes(Xf)\Bigr)(w)\\
&=a(X(v\otimes f))(w),
\end{align*}
where in the second equation we have used the definition of the dual representation and the fact that $f(w)$ is a scalar. Finally,
\begin{align*}
ev(X(f\otimes v))&=ev\Bigl((Xf)\otimes v+(-1)^{|X||f|}f\otimes(Xv)\Bigr)\\
&=(Xf)(v)+(-1)^{|X||f|}f(Xv)\\
&=-(-1)^{|X||f|}f(Xv)+(-1)^{|X||f|}f(Xv)\\
&=0\\
&=Xev(f\otimes v),
\end{align*}
where the last equation follows because the action on $\Bbbk$ is trivial.
\end{proof}

Now suppose $V$ is finitely generated and projective as a $\Bbbk$-module. Then the dual basis lemma implies the existence of degree $0$ elements and linear forms $v_1,\ldots,v_m\in V$ and $\lambda_1,\ldots,\lambda_m\in V^*$ and degree $1$ elements and linear forms $w_1,\ldots,w_n\in V$ and $\mu_1,\ldots,\mu_n\in V^*$ such that for every vector $v\in V$
\begin{equation}\label{eqn:dualbasis}
v=\sum_{i=1}^m\lambda_i(v)v_i+
\sum_{j=1}^n\mu_j(v)w_j
\end{equation}
or equivalently
\[
a\left(
\sum_{i=1}^mv_i\otimes\lambda_i+
\sum_{j=1}^nw_j\otimes\mu_j\right)=\mathbbm{1}_V.
\]
Let $b\colon\operatorname{End}(V)\rightarrow V\otimes V^*$ denote the linear map
\[
b(f):=\sum_{i=1}^mf(v_i)\otimes\lambda_i+
\sum_{j=1}^nf(w_j)\otimes\mu_j
\]
for $f\in\operatorname{End}(V)$.
\begin{lemma}
The map $b$ is the inverse of $a$. In particular, $a$ and $b$ are even isomorphisms of representations, and $b$ is independent of the choice of the $v_i,w_j,\lambda_i,\mu_j$.
\end{lemma}
\begin{proof}
Applying an $f\in\operatorname{End}(V)$ to both sides of~\eqref{eqn:dualbasis} yields $f(v)$ on the left-hand side and $a(b(f))(v)$ on the right-hand side. Thus, $a\circ b$ is the identity. Applying the same argument to $f\in V^*$ yields
\begin{equation}\label{eqn:flinear}
f(v)=\sum_{i=1}^m\lambda_i(v)f(v_i)+
\sum_{j=1}^n\mu_j(v)f(w_j)
\end{equation}
for any $v\in V$. Hence if $w\in V$ and $f\in V^*$, then
\begin{align*}
b(a(w\otimes f))&=\sum_{i=1}^mf(v_i)w\otimes\lambda_i+
\sum_{j=1}^nf(w_j)w\otimes\mu_j\\
&=w\otimes\left(\sum_{i=1}^mf(v_i)\lambda_i+
\sum_{j=1}^nf(w_j)\mu_j\right)=w\otimes f,
\end{align*}
where the first equality follows from the definitions, and the last equality follows from~\eqref{eqn:flinear}. Thus, $b\circ a$ is also the identity, which completes the proof.
\end{proof}

There is a canonical linear map $i\colon\Bbbk\rightarrow\operatorname{End}(V)$ given by $1\mapsto\mathbbm{1}_V$. This map is a map of representations because the action on $\mathbbm{1}_V$ is trivial since $\mathbbm{1}_V$ is a homomorphism of representations. We define the \textbf{coevaluation map} by
\[
coev:=b\circ i\colon\Bbbk\longrightarrow V\otimes V^*.
\]
Then $coev$ is an even homomorphism of representations and
\[
coev(1)=b(\mathbbm{1}_V)=\sum_{i=1}^mv_i\otimes\lambda_i+
\sum_{j=1}^nw_j\otimes\mu_j.
\]
In particular, $b$ can be written as
\[
b(f)=(f\otimes\mathbbm{1}_{V^*})(coev(1)),
\]
where we have used that $\mathbbm{1}_{V^*}$ has superdegree $0$. We now define the \textbf{supertrace} $\operatorname{str}(f)\in\Bbbk$ of a linear endomorphism $f\in\operatorname{End}(V)$ by
\[
\operatorname{str}(f):=(ev\circ\tau_{V,V^*}\circ b)(f)
=(ev\circ\tau_{V,V^*}\circ (f\otimes\mathbbm{1}_{V^*})\circ coev)(1)
\]
or explicitly
\[
\operatorname{str}(f)=\sum_{i=1}^m\lambda_i(f(v_i))-\sum_{j=1}^n\mu_j(f(w_j)).
\]
For grading reasons, we could replace the first occurrence of $f$ on the right-hand side by its component $f_{00}\colon V_0\rightarrow V_0$ and the second occurrence by its component $f_{11}\colon V_1\rightarrow V_1$. Hence we see that
\begin{equation}\label{eqn:strtr}
\operatorname{str}(f)=\operatorname{tr}(f_{00})-\operatorname{tr}(f_{11}),
\end{equation}
where the \textbf{trace} of a linear endomorphism $g\colon P\rightarrow P$ of an ordinary finitely generated projective $\Bbbk$-module $P$ is defined by
\[
\operatorname{tr}(g):=\sum_{k=1}^l\nu_k(g(u_k))
\]
for $u_1,\ldots,u_l\in P$ and $\nu_1,\ldots,\nu_k\in P^*=\operatorname{Hom}(P,\Bbbk)$ as in the dual basis lemma. Equation~\eqref{eqn:strtr} also shows that the supertrace of $f$ only depends on the even part of $f$, and hence is zero if $f$ is odd.

\begin{remark}
In the definitions above of the trace and the supertrace, we assumed that $\Bbbk$ is commutative. In the non-commutative case, one could still make sense of the definitions above, but one would have to assume that $V$ and $P$ are finitely generated projective right-modules, and one would have to interpret the trace and the supertrace as elements of the space of coinvariants of $\Bbbk$, viewed as a bimodule over itself.
\end{remark}

The supertrace satisfies the following key property
\begin{equation}\label{eqn:strreorder}
\operatorname{str}(f\circ g)=(-1)^{|f||g|}\operatorname{str}(g\circ f),
\end{equation}
where $f\colon V\rightarrow W$ and $g\colon W\rightarrow V$ are homogeneous linear maps between finitely generated projective supermodules $V$ and $W$. Assuming the corresponding property for the ordinary trace as given, one can deduce property~\eqref{eqn:strreorder} by noting that
\[
\operatorname{str}(f\circ g)=\operatorname{tr}((f\circ g)_{00})-\operatorname{tr}((f\circ g)_{11})=
\operatorname{tr}((g\circ f)_{00})-\operatorname{tr}((g\circ f)_{11})=\operatorname{str}(g\circ f)
\]
if $f$ and $g$ are both even, and
\[
\operatorname{str}(f\circ g)=\operatorname{tr}((f\circ g)_{00})-\operatorname{tr}((f\circ g)_{11})=
\operatorname{tr}((g\circ f)_{11})-\operatorname{tr}((g\circ f)_{00})=-\operatorname{str}(g\circ f)
\]
if $f$ and $g$ are both odd. Here we have used that $(f\circ g)_{aa}=f_{aa}\circ g_{aa}$ in the first case and $(f\circ g)_{aa}=f_{a,a+1}\circ g_{a+1,a}$ in the second. The remaining case where one of the maps $f$ and $g$ is even while the other one is odd is trivial because in this case both sides of~\eqref{eqn:strreorder} are zero.

Equation~\eqref{eqn:strreorder} implies that the supertrace of a supercommutator satisfies
\begin{equation}\label{eqn:strscomm}
\operatorname{str}([f,g]_s)=0
\end{equation}
for any $f,g\in\operatorname{End}(V)$.

\subsection{Supergrading shifts}\label{subs:gradingshifts}
Given a representation $V$ of a Lie superalgebra $\mathfrak{g}$, we denote by $V[1]$ the same representation but with reversed supergrading. In other words, $\rho_{V[1]}=\rho_V$, but an element $v\in V[1]$ has superdegree $a\in\mathbb{Z}_2$ in $V[1]$ if and only if it has superdegree $a+1\in\mathbb{Z}_2$ when viewed as an element of $V$. By the definitions of $V[1]$ and of the tensor product representation, we have
\[
V[1]=V\otimes(\Bbbk[1]),
\]
where $\Bbbk$ denotes the trivial representation, and where the equal sign means that the Lie superalgebra actions on the two sides of the equation agree with each other under the canonical identification of the underlying supermodules. Note that the oppositely ordered tensor product, $(\Bbbk[1])\otimes V$, is isomorphic but not equal to $V[1]$. In general, we have following lemma, in which we use the notation $V[0]:=V$:

\begin{lemma}\label{lem:tensorshift} 
Let $V_1,\ldots V_n$ be representations and $a_1,\ldots,a_n\in\mathbb{Z}_2$. Then there is an even isomorphism of representations
\[
\psi_{V_1,\ldots,V_n}^{a_1,\ldots,a_n}\colon (V_1[a_1])\otimes\cdots\otimes(V_n[a_n])\longrightarrow (V_1\otimes\cdots\otimes V_n)[a_1+\ldots+a_n]
\]
which is given by $v_1\otimes\cdots\otimes v_n\mapsto(-1)^m v_1\otimes\cdots\otimes v_n$ where
\[
m=\sum_{i=1}^{n-1}a_i(|v_{i+1}|+\ldots+|v_n|).
\]
In particular, we have
$(V[1])\otimes W\cong (V\otimes W)[1]\cong V\otimes(W[1])$ where the two isomorphisms are given respectively by the map
$v\otimes w\mapsto(-1)^{|w|}v\otimes w$ and by the identity map.
\end{lemma}

\begin{proof}
We first note that the representations $(V\otimes W)[1]$ and $V\otimes (W[1])$ are equal to each other because they are both equal to the representation $V\otimes W\otimes\Bbbk[1]$. Next, we observe that there is an even isomorphism of representations
\[
(V[1])\otimes W=V\otimes(\Bbbk[1])\otimes W%
\xrightarrow[\cong]{\mathbbm{1}_V\otimes\tau_{\Bbbk[1],W}}V\otimes W\otimes(\Bbbk[1])=(V\otimes W)[1].
\]
This isomorphism has the form stated in the lemma because $(\mathbbm{1}_V\otimes\tau_{\Bbbk[1],W})(v\otimes 1\otimes w)=v\otimes\tau_{\Bbbk[1],W}(1\otimes w)=(-1)^{|w|}v\otimes w\otimes 1$, where we have used that $\tau_{\Bbbk[1],W}$ has superdegree $0$, and $1\in\Bbbk[1]$ has superdegree $1$.

To prove the more general statement, we identify each factor $V_i[a_i]$ with $a_i=1$ with the representation $V_i[1]=V_i\otimes(\Bbbk[1])$. Using the twist maps $\tau_{\Bbbk[1],V_j}$, we then move each of the resulting copies of $\Bbbk[1]$ in the tensor product $V_1[a_1]\otimes\cdots\otimes V_n[a_n]$ to the right, in such a way that we never permute two copies of $\Bbbk[1]$ past each other. This results in an isomorphism of representations
\[
(V_1[a_1])\otimes\cdots\otimes (V_n[a_n])\longrightarrow (V_1\otimes\cdots\otimes V_n)\otimes (\Bbbk[1])^{\otimes(a_1+\ldots+a_n)},
\]
where in the exponent $a_1+\ldots+a_n$, we interpret each $a_i$ as the integer $0$ or $1$, rather than its reduction modulo $2$. Finally, we identify an even tensor power $(\Bbbk[1])^{\otimes 2k}$ with $\Bbbk^{\otimes 2k}$ using the identity map $\Bbbk^{\otimes 2k}\rightarrow\Bbbk^{\otimes 2k}$, so that in the end we are left with at most one copy of $\Bbbk[1]$. This yields the desired isomorphism $\psi_{V_1,\ldots,V_n}^{a_1,\ldots,a_n}$, and it is easy to see that it has the form stated in the lemma.
\end{proof}

The isomorphisms from Lemma~\ref{lem:tensorshift} satisfy the following identity:
\begin{lemma}\label{lem:tensorshiftassociativity}
If $U,V,W$ are representations, then
\[
\psi^{a+b,c}_{U\otimes V,W}\circ(\psi^{a,b}_{U,V}\otimes\mathbbm{1}_{W[c]})=\psi^{a,b,c}_{U,V,W}=\psi^{a,b+c}_{U,V\otimes W}\circ(\mathbbm{1}_{U[a]}\otimes\psi^{b,c}_{V,W})
\]
for all $a,b,c\in\mathbb{Z}_2$.
\end{lemma}
\begin{proof}
This can be seen by using the explicit formula for the isomorphisms from Lemma~\ref{lem:tensorshift}, or more conceptually, by recalling the construction of these isomorphisms in the proof of Lemma~\ref{lem:tensorshift}.
\end{proof}

Although we don't need it, we mention that there is also an even isomorphism of representations $\varphi\colon V^*[1]\rightarrow(V[1])^*$, which is given by the composition of the following identifications
\[
V^*[1]=V^*\otimes(\Bbbk[1])=V^*\otimes(\Bbbk^*[1])=V^*\otimes(\Bbbk[1])^*\stackrel{S}{\longrightarrow}(V\otimes\Bbbk[1])^*=(V[1])^*,
\]
where $S$ denotes the map from Lemma~\ref{lem:dualtensor}. Explicitly, $\varphi$ sends an element $f\in V^*[1]$ to the linear functional $\varphi(f)$ given by
\[
\varphi(f)(v)=S(f\otimes s^{-1})(v\otimes 1)=(-1)^{|v|}f(v),
\]
where $|v|$ denotes the superdegree of $v\in V[1]$ viewed as an element of $V$, and
\[
s\colon\Bbbk\longrightarrow\Bbbk[1]
\]
denotes the identity map of $\Bbbk$ viewed as a map from $\Bbbk$ to $\Bbbk[1]$.

Given any representation $V$, we can define an odd isomorphism of representations
\[
s_V\colon V\longrightarrow V[1]
\]
by using the identifications $V=V\otimes\Bbbk$ and $V[1]=V\otimes(\Bbbk[1])$ and setting $s_V:=\mathbbm{1}_V\otimes s$, for $s$ as above. Note that $s_V^{-1}=\mathbbm{1}_V\otimes (s^{-1})$ because $\mathbbm{1}_V$ has superdegree $0$. Explicitly, $s_V$ is given by
\[
s_V(v)=(\mathbbm{1}_V\otimes s)(v\otimes 1)=(-1)^{|v|}(v\otimes 1)=(-1)^{|v|}v,
\]
where $|v|$ denotes the superdegree of $v$ viewed as an element of $V$.
We now denote by $(\mathbbm{1}_V)_0^0\colon V[0]\rightarrow V[0]$ and $(\mathbbm{1}_V)_1^1\colon V[1]\rightarrow V[1]$ the maps given by the identity map of $V$, and by $(\mathbbm{1}_V)_0^1\colon V[0]\rightarrow V[1]$ the map $s_V$.

\begin{lemma}\label{lem:identityproducts} For $a,b,c,d\in\mathbb{Z}_2$, we have
\[
\psi_{V,W}^{b,d}\circ\left((\mathbbm{1}_V)_0^b\otimes(\mathbbm{1}_W)_0^d\right)=(\mathbbm{1}_{V\otimes W})_0^{b+d}
\]
and
\[
(\mathbbm{1}_V)_a^0\otimes(\mathbbm{1}_W)_c^0=(-1)^{ac}(\mathbbm{1}_{V\otimes W})_{a+c}^0\circ\psi_{V,W}^{a,c},
\]
where $\psi_{V,W}^{a,c}$ and $\psi_{V,W}^{b,d}$ denote the isomorphisms from Lemma~\ref{lem:tensorshift}.
\end{lemma}

\begin{proof} This is a straightforward calculation and thus left to the reader. Note that the two equations are equivalent to each other because $(\mathbbm{1}_U)_m^0$ and $(\mathbbm{1}_U)_0^m$ are inverses and have superdegree $m$.
\end{proof}

Given a linear map $f\colon V\rightarrow V'$ and $a,b\in\mathbb{Z}_2$, we now define $f_a^b\colon V[a]\rightarrow V'[b]$ to be the linear map
\begin{equation}\label{eqn:deffab}
f_a^b:=(\mathbbm{1}_{V'})_0^b\circ f\circ(\mathbbm{1}_V)_a^0,
\end{equation}
where $(\mathbbm{1}_V)_a^0$ denotes the inverse of $(\mathbbm{1}_V)_0^a$, so that in particular $(\mathbbm{1}_V)_1^0=s_V^{-1}$. Note that if $f$ is a homomorphism of representations, then so is $f_a^b$.

\begin{corollary}\label{cor:psiintertwine}
Let $f\colon V\rightarrow V'$ and $g\colon W\rightarrow W'$ be linear maps. Then
\[
\psi_{V',W'}^{b,d}\circ\left(f_a^b\otimes g_c^d\right)
=(-1)^{ad+ac+a|g|+d|f|}(f\otimes g)_{a+c}^{b+d}\circ\psi_{V,W}^{a,c}
\]
for all $a,b,c,d\in\mathbb{Z}_2$.
\end{corollary}

\begin{proof}
This is a straightforward calculation using equation~\eqref{eqn:deffab}, the super interchange law~\eqref{eqn:superinterchange}, and Lemma~\ref{lem:identityproducts}. Indeed, the idea is the same as the idea behind Remark~\ref{rem:supergradedext} from section~\ref{subs:supergradedext}.
\end{proof}

We are now ready to prove Lemma~\ref{lem:extension} from section~\ref{subs:supergradedext}.

\begin{proof}[Proof of Lemma~\ref{lem:extension}]
Let $\mathcal{S}\colon\mathcal{C}\rightarrow\mathcal{D}$ be a strict monoidal superfunctor, and assume first that $\mathcal{D}$ is the representation category $\mathcal{R}\!\mathit{ep}(\mathfrak{g})$ for a Lie superalgebra $\mathfrak{g}$.

As in section~\ref{subs:supergradedext}, we denote by $x_a$ the object $(x,a)$ in the supergraded extension of $\mathcal{C}$, and by $f^b_a$ the morphism $f\colon x\rightarrow y$, viewed as a morphism from $x_a$ to $x_b$. We then define $\mathcal{S}^s\colon\mathcal{C}^s\rightarrow\mathcal{D}$ by
$\mathcal{S}^s(x_a):=\mathcal{S}(x)[a]$ and \[\mathcal{S}^s(f^b_a):=\mathcal{S}(f)^b_a=\bigl(\mathbb{1}_{S(y)}\bigr)_0^b\circ\mathcal{S}(f)\circ\bigl(\mathbbm{1}_{S(x)}\bigr)_a^0.\]

To see that $\mathcal{S}^s$ is a monoidal superfunctor, note that the maps $\mathfrak{m}_{x_a,y_b}:=\psi_{\mathcal{S}(x),\mathcal{S}(y)}^{a,b}$ form an even supernatural isomorphism $\mathcal{S}^s(-)\otimes\mathcal{S}^s(-)\to\mathcal{S}^s(-\otimes^s-)$, where naturality follows from Corollary~\ref{cor:psiintertwine} and from the definition of $\otimes^s$ in section~\ref{subs:supergradedext}.
Lemma~\ref{lem:tensorshiftassociativity} implies that this isomorphism satisfies~\eqref{eqn:msfassociativity}, and~\eqref{eqn:msfunit} is trivially satisfied
for $i:=\mathbbm{1}_\Bbbk$ because $\psi_{\mathcal{S}(x),\Bbbk}^{a,0}=\mathbbm{1}_{\mathcal{S}(x)[a]}=\psi_{\Bbbk,\mathcal{S}(x)}^{0,a}$.

The case where $\mathcal{D}$ is the category $\mathcal{SM}\mathit{od}(\Bbbk)$ follows from the case where $\mathcal{D}$ is $\mathcal{R}\!\mathit{ep}(\mathfrak{g})$ because $\mathcal{SM}\mathit{od}(\Bbbk)$ embeds into $\mathcal{R}\!\mathit{ep}(\mathfrak{g})$ as the full subcategory given by all representations on which the $\mathfrak{g}$-action is trivial.
\end{proof}

\subsection{The Lie superalgebra $\mathfrak{gl}(1|1)$}

Recall~\cite{Kac} that the \textbf{general linear Lie superalgebra} $\mathfrak{gl}(1|1)$ is defined as the supermodule consisting of all $2\times 2$ matrices with entires in $\Bbbk$, where a nonzero $2\times 2$ matrix has superdegree $0$ if it is diagonal and superdegree $1$ if its diagonal entries vanish. The Lie superbracket is given by the supercommutator
\[
[A,B]_s:=AB-(-1)^{|A||B|}BA
\]
for homogeneous matrices $A$ and $B$. Clearly, $\mathfrak{gl}(1|1)$ is isomorphic to $\mathfrak{gl}(\Bbbk^{1|1})$, where $\Bbbk^{1|1}$ is the free supermodule spanned by two homogeneous elements $v_0$ and $v_1$ of superdegrees $|v_0|=0$ and $|v_1|=1$.

The \textbf{special linear superalgebra} $\mathfrak{sl}(1|1)$ is defined as the set of all $A\in\mathfrak{gl}(1|1)$ with $\operatorname{str}(A)=0$.
Note that this set forms an ideal in $\mathfrak{gl}(1|1)$ because of relation~\eqref{eqn:strscomm}, and hence $\mathfrak{sl}(1|1)$ is itself a Lie superalgebra.
In terms of generators, $\mathfrak{gl}(1|1)$ can be viewed as the span of the matrices
\[
H_1:=\begin{bmatrix}
1&0\\0&0
\end{bmatrix},\quad
H_2:=\begin{bmatrix}
0&0\\0&1
\end{bmatrix},\quad
E:=\begin{bmatrix}
0&1\\0&0
\end{bmatrix},\quad
F:=\begin{bmatrix}
0&0\\1&0
\end{bmatrix}.
\]
Likewise, $\mathfrak{sl}(1|1)$ is spanned by the identity matrix $H_+:=H_1+H_2$ and by the matrices $E$ and $F$. The superdegrees of these matrices are $|H_+|=|H_1|=|H_2|=0$ and $|E|=|F|=1$, and their supercommutators are given by
\begin{align*}
&[H_1,E]_s=E,\qquad[H_2,E]_s=-E,\\ &[H_1,F]_s=-F,\qquad[H_2,F]_s=F,\\
&[E,F]_s=H_+,\\
&[H_+,A]_s=[H_i,H_j]_s=[E,E]_s=[F,F]_s=0
\end{align*}
for $i,j\in\{1,2\}$ and $A\in\mathfrak{gl}(1|1)$.

\subsection{Representations of $\mathfrak{gl}(1|1)$}\label{subs:repgl11}

In this section, we will briefly review some aspects of the finite-dimensional representation theory of $\mathfrak{gl}(1|1)$. For the sake of simplicity, we will assume that $\Bbbk$ is an algebraically closed field of characteristic zero, although this assumption won't be used in the remainder of this thesis. We first observe:
\begin{lemma}\label{lem:highestweight}
Every finite-dimensional representation $V\neq 0$ of $\mathfrak{gl}(1|1)$ has a \textbf{highest weight vector}.
\end{lemma}
By a highest weight vector, we here mean a simultaneous eigenvector of $\rho(H_1)$ and $\rho(H_2)$ which is annihilated by $\rho(E)$, where $\rho$ denotes the given representation.

\begin{proof}[Proof of Lemma~\ref{lem:highestweight}]
Since $[H_1,H_2]_s=0$ and $|H_1|=|H_2|=0$, the operators $\rho(H_1)$ and $\rho(H_2)$ commute, and since $V$ is finite-dimensional and $\Bbbk$ is algebraically closed, it follows that these operators have a simultaneous eigenvector $v\in V$. Using the relations $[H_1,E]_s=E$ and $[H_2,E]_s=-E$, it is further easy to see that $\rho(E)(v)$ is again a simultaenous eigenvector, unless $\rho(E)(v)=0$. Since $\rho(E)^2=\rho([E,E]_s)/2=0$, we thus obtain that either $v$ or $\rho(E)(v)$ is a highest weight vector.
\end{proof}

\begin{lemma} Every finite-dimensional irreducible representation $V$ of $\mathfrak{gl}(1|1)$ is at most $2$-dimen-sional.
\end{lemma}

\begin{proof}
Suppose $V\neq 0$ is a finite-dimensional irreducible representation. By the previous lemma, $V$ contains a highest weight vector $v$, and using $\rho(F)^2=\rho([F,F]_s)/2=0$ and the remaining relations in $\mathfrak{gl}(1|1)$, it is easy to see that the vectors $v$ and $\rho(F)(v)$ span a subrepresentation $W\subseteq V$. Since $V$ is irreducible, we must have $W=V$, and hence $V$ is at most $2$-dimensional.
\end{proof}

If $V$ is a $1$-dimensional representation of $\mathfrak{gl}(1|1)$, then the operators $\rho(E),\rho(F)$, and $\rho(H_+)$ must act on $V$ by zero maps because these operators can be realized as commutators of operators, and every commutator of operators on a $1$-dimensional space is trivial. On the other hand, the operator $\rho(H_-)$ for $H_-:=H_1-H_2$ can act by multiplication by an arbitrary scalar.

To classify the $2$-dimensional irreducible representations of $\mathfrak{gl}(1|1)$, consider the set \[
\mathsf{P}:=\{(r,s)\in\Bbbk^2\,|\,r+s\neq 0\}.
\]
For each $\lambda=(r,s)\in\mathsf{P}$, we can define a $2$-dimensional representation $L(\lambda)$ by
\[
\renewcommand{\arraystretch}{1}
\begin{array}{rclrcl}
\rho(H_1)&:=&\begin{bmatrix}r&0\\0&r-1\end{bmatrix},\qquad\qquad& \rho(H_2)&:=&\begin{bmatrix}s&0\\0&s+1\end{bmatrix},\\[0.3in]
\rho(E)&:=&\begin{bmatrix}0&r+s\\0&0\end{bmatrix},\qquad\qquad& \rho(F)&:=&\begin{bmatrix}
0&0\\1&0\end{bmatrix},
\end{array}
\]
where these matrices act on $L(\lambda):=\Bbbk^{1|1}=\Bbbk v_0\oplus\Bbbk v_1$ by left-multiplication.

It is easy to see that $L(\lambda)$ is irreducible for $\lambda\in\mathsf{P}$, and that every $2$-dimensional irreducible representation of $\mathfrak{gl}(1|1)$ whose highest weight vector has superdegree $0$ is isomorphic to a unique $L(\lambda)$. In fact, suppose $V$ is such a representation. Then the corresponding $\lambda=(r,s)\in\mathsf{P}$ is the simultaneous eigenvalue with which the pair $(H_1,H_2)$ acts on a highest weight vector $v\in V$, and the desired isomorphism $V\cong L(\lambda)$ is given by sending the basis $\{v,Fv\}$ to the basis $\{v_0,v_1\}$.

Note that for $\lambda=\epsilon_1:=(1,0)$, the definition of $L(\lambda)$ implies $\rho(A)=A$ for all matrices $A\in\mathfrak{gl}(1|1)$. Hence $L:=L(\epsilon_1)$ is the fundamental representation. Similarly, one can see that the dual representation of $L$ is isomorphic to $L(-\epsilon_2)[1]$ for $\epsilon_2:=(0,1)$, where the isomorphism $L^*\cong L(-\epsilon_2)[1]$ is given by sending the dual basis $\{v_0^*,v_1^*\}$ to the basis $\{v_1,v_0\}$.

\begin{remark}
While finite-dimensional irreducible representations of $\mathfrak{gl}(1|1)$ are at most $2$-dimen-sional, there are higher-dimensional indecomposable representations, such as the representation $L\otimes L^*\cong L^*\otimes L$. This is in contrast to the representation theory of ordinary semisimple Lie algebras, where finite-dimensional indecomposable representations are automatically irreducible (still under the assumption that $\Bbbk$ is a field of characteristic zero).
\end{remark}

\subsection{Connection with $\mathcal{T\!L}_o(0)$ and proof of Theorem~\ref{thm:commute}}

To define the superfunctor $\mathcal{R}$ that appears in~\eqref{eqn:commute}, we will now analyze the evaluation and coevaluation maps
\[
ev\colon L^*\otimes L\longrightarrow\Bbbk,\qquad
coev\colon \Bbbk\longrightarrow L\otimes L^*
\]
associated with the fundamental representation $L$ of $\mathfrak{gl}(1|1)$. In view of the definition of the evaluation and the coevaluation in section~\ref{subs:evcoev}, it is clear that these maps are given by
\[
ev(v_i^*\otimes v_j)=\delta_{ij},\qquad coev(1)=v_0\otimes v_0^*+v_1\otimes v_1^*,
\]
where $\delta_{ij}:=1$ if $i=j$ and $\delta_{ij}:=0$ otherwise. We will also need the maps
\[
ev_\tau\colon L\otimes L^*\longrightarrow\Bbbk,\qquad coev_\tau\colon\Bbbk\longrightarrow L^*\otimes L
\]
defined by $ev_\tau:=ev\circ\tau_{L,L^*}$ and $coev_\tau:=\tau_{L,L^*}\circ coev$. These maps are given explicitly by
\[
ev_\tau(v_i\otimes v_j^*)=(-1)^{ij}\delta_{ij},\qquad coev_\tau(1)=v_0^*\otimes v_0-v_1^*\otimes v_1,
\]
where here we have used that the superdegrees of $v_i\in L$ and $v_j^*\in L^*$ are equal to $i$ and $j$ modulo 2, respectively.

We are actually most interested in the compositions
\[
\begin{tikzcd}[row sep=small,column sep=large,/tikz/column 1/.append style={anchor=base east}]
ev'=\Bigl\{&[-40pt]
(L^*[1])\otimes L\ar[r,"\psi^{1,0}_{L^*,L}"]& (L^*\otimes L)[1]\ar[r,"s^{-1}_{L^*\otimes L}"]&L^*\otimes L\ar[r,"ev"]&\Bbbk,\\
ev'_\tau=\Bigl\{&
L\otimes (L^*[1])\ar[r,"\psi^{0,1}_{L,L^*}"]& (L\otimes L^*)[1]\ar[r,"s^{-1}_{L\otimes L^*}"]&L\otimes L^*\ar[r,"ev_\tau"]&\Bbbk,
\end{tikzcd}
\]
and
\[\hspace*{-0.38cm}
\begin{tikzcd}[row sep=small,column sep=large,/tikz/column 1/.append style={anchor=base east}]
coev'=\Bigl\{&[-40pt]\Bbbk\ar[r,"coev"]&L\otimes L^*\ar[r,"s_{L\otimes L^*}"]&(L\otimes L^*)[1]\ar[r,"(\psi_{L,L^*}^{0,1})^{-1}"]&L\otimes(L^*[1]),\\
coev'_\tau=\Bigl\{&\Bbbk\ar[r,"coev_\tau"]&L^*\otimes L\ar[r,"s_{L^*\otimes L}"]&(L^*\otimes L)[1]\ar[r,"(\psi_{L^*,L}^{1,0})^{-1}"]&(L^*[1])\otimes L,
\end{tikzcd}
\]
where $s_V$ denotes the map defined prior to Lemma~\ref{lem:identityproducts} and $\psi^{a,b}_{V,W}$ denote the maps introduced in Lemma~\ref{lem:tensorshift} (in particular, $\psi^{0,1}_{L,L^*}=\mathbbm{1}$). Since the latter maps are odd and even homomorphisms of representations (in the sense of \eqref{eqn:rephomomorphism}), it is clear that the compositions above are themselves odd homomorphisms of representations.

A direct calculation shows:
\begin{equation}\label{eqn:evcoevprime}
\begin{split}
ev'(v_i^*\otimes v_j)=(-1)^{ij}\delta_{ij},\qquad coev'(1)=v_0\otimes v_0^*+v_1\otimes v_1^*,\\
ev'_\tau(v_i\otimes v_j^*)=(-1)^{ij}\delta_{ij},\qquad coev'_\tau(1)=v_0^*\otimes v_0+v_1^*\otimes v_1,
\end{split}
\end{equation}
where $v_i,v_j\in L$ and $v_i^*,v_j^*\in L^*$. Note that the term $v_1\otimes v_1^*$ in the formula for $coev'(1)$ comes with a plus sign because the map $(\psi_{L,L^*}^{0,1})^{-1}$ is given by an identity map, and the term $v_1^*\otimes v_1$ in the formula for $coev'_\tau(1)$ comes with a plus sign because the maps $\tau_{L,L^*}$ and $(\psi^{1,0}_{L^*,L})^{-1}$ both contribute a minus sign.

We will now use the maps above to define superfunctor from $\mathcal{T\!L}_o(0)$ to $\mathcal{R}\mathit{ep}(\mathfrak{gl}(1|1))$. To this end, we introduce a monoidal supercategory $\TLor (0)$, which can be viewed as an oriented version of $\mathcal{T\!L}_o(0)$. Objects in $\TLor (0)$ are finite (possibly empty sequences) in $\{+,-\}$, to be viewed as increasing sequences of distinct points on the real line, where each point is marked with a $+$ or a $-$. Morphisms in $\TLor (0)$ between a length $n$ sequence and a length $m$ sequences are defined in nearly the same way as morphisms in $\mathcal{T\!L}_o(0)$ between $n$ and $m$. The only difference is that now flat tangles are required to be oriented, and the orientation is required to point upward near endpoints that are marked by a $+$, and downward near endpoints that are marked by a $-$. The supermonoidal product in $\TLor (0)$ is defined by concatenation of sequences on objects, and by the right-then-left union, as in $\mathcal{T\!L}_o(0)$, on morphisms.

There is a superfunctor 
\[
\mathcal{R}_1\colon\mathcal{T\!L}_o(0)\longrightarrow\TLor (0)
\]
given by sending the object $n$ to an alternating length $n$ sequence $(+,-,+,-,\ldots)$ starting with a $+$. To define this functor on morphisms, let $T\subset\mathbb{R}\times I$ be an unoriented flat tangle representing a morphism in $\mathcal{T\!L}_o(0)$. Color the regions of $(\mathbb{R}\times I)\setminus T$ black and white in such a way that adjacent regions have different colors, and such that the left-most region is colored white. Now orient the strands of $T$ so that at each point of $T$, the region that lies to the left of $T$ when looking in the direction of the orientation is white; the resulting oriented flat tangle is $\mathcal{R}_1(T)$. It is clear that $\mathcal{R}_1$ is well-defined because, by definition, morphisms in $\TLor (0)$ satisfy the same relations as morphisms in $\mathcal{T\!L}_o(0)$. Note that $\mathcal{R}_1$ does not preserve the supermonoidal product, since concatenation of two alternating sequences starting with a $+$ is in general no longer alternating. On the other hand, $\mathcal{R}_1$ does restrict to a monoidal superfunctor on the full monoidal subcategory of $\mathcal{T\!L}_o(0)$ whose objects are given by the even integers $n\geq 0$, as in Remark~\ref{rmk:EvenAndOddTL}.

We can further define a monoidal superfunctor 
\[
\mathcal{R}_2\colon\TLor (0)\longrightarrow\mathcal{R}\mathit{ep}(\mathfrak{gl}(1|1))\]
by $\mathcal{R}_2(+):=L$ and $\mathcal{R}_2(-):=L^*[1]$, and by sending counterclockwise caps to $ev'$, clockwise caps to $ev'_\tau$, clockwise cups to $coev'$, and counterclockwise cups to $-coev'_\tau$ (note the minus sign!). Note that $\mathcal{R}_2$ sends vertical strands to the appropriate identity maps.

\begin{lemma}\label{lem:R2welldef}
$\mathcal{R}_2$ is well-defined.
\end{lemma}

The proof of Lemma~\ref{lem:R2welldef} will be given later.

We define
\[
\mathcal{R}\colon\mathcal{T\!L}_o(0)\longrightarrow\mathcal{R}\mathit{ep}(\mathfrak{gl}(1|1)).
\]
as the superfunctor $\mathcal{R}:=\mathcal{R}_2\circ\mathcal{R}_1$.

Note that $\mathcal{R}$ sends the object $m$ to the tensor product
\begin{equation}\label{eqn:Rm}
\mathcal{R}(m)=\mathcal{R}_2(+,-,+,-,\ldots)=L\otimes(L^*[1])\otimes L\otimes(L^*[1])\otimes\ldots,
\end{equation}
where there are $m$ tensor factors in total. We are now ready to prove Theorem~\ref{thm:commute}.

\TheoremC*

\begin{proof}
Recall that
\begin{equation}\label{eqn:FIm}
\mathcal{F}_o^{AKh}(\mathcal{I}(m))=V^{\otimes m},
\end{equation}
where $V=\Bbbk v_+\oplus\Bbbk v_-$ denotes the free supermodule spanned by two homogeneous elements of superdegrees $|v_+|=0$ and $|v_-|=1$. To relate~\eqref{eqn:FIm} to~\eqref{eqn:Rm}, we consider the isomorphisms of supermodules
\[
\mathfrak{n}_+\colon V\stackrel{\cong}{\longrightarrow} L\qquad\mbox{and}\qquad \mathfrak{n_-}\colon V\stackrel{\cong}{\longrightarrow} L^*[1]
\]
which are given by sending the basis $\{v_+,v_-\}$ to the bases $\{v_0,v_1\}$ and $\{-v_1^*,v_0^*\}$. Note that these isomorphisms both preserve the supergrading. Define
\[
\mathfrak{n}_m\colon\mathcal{F}_o^{AKh}(\mathcal{I}(m))\longrightarrow\mathcal{V}(\mathcal{R}(m))
\]
by $\mathfrak{n}_0:=\mathbbm{1}_\Bbbk$ and $\mathfrak{n}_m:=\mathfrak{n}_+\otimes\mathfrak{n}_-\otimes\mathfrak{n}_+\otimes\mathfrak{n}_-\otimes\ldots$ for $m>0$, where $\mathcal{V}\colon\mathcal{R}\mathit{ep}(\mathfrak{gl}(1|1))\rightarrow\mathcal{SM}\mathit{od}(\Bbbk)$ denotes the forgetful functor which forgets the $\mathfrak{gl}(1|1)$-action.

To complete the proof, we will now show that the maps $\mathfrak{n}_m$ define a supernatural isomorphism between $\mathcal{F}_o^{AKh}\circ\mathcal{I}$ and $\mathcal{V}\circ\mathcal{R}$. To this end, we introduce the odd linear maps $n\colon V\otimes V\rightarrow\Bbbk$ and $u\colon\Bbbk\rightarrow V\otimes V$ given by
\begin{equation}\label{eqn:nuformulas}
n(v_a\otimes v_b):=\delta_{ab},\qquad u(1):=v_-\otimes v_+-v_+\otimes v_-
\end{equation}
for $a,b\in\{+,-\}$. Comparing with~\eqref{eqn:epsilonmuprime} and \eqref{eqn:Deltaprimeeta}, we see that
\begin{equation}\label{eqn:nucapcup}
n=\mathcal{F}_{\!o}^{AKh}\!\left(\mathcal{I}(\incg{TLsmallcap})\right),\qquad
u=\mathcal{F}_{\!o}^{AKh}\!\left(\mathcal{I}(\incg{TLsmallcup})\right).
\end{equation}
On the other hand, a comparison with~\eqref{eqn:evcoevprime} shows that 
\[
n=ev'\circ (\mathfrak{n}_-\otimes \mathfrak{n}_+)=ev'_\tau\circ(\mathfrak{n}_+\otimes \mathfrak{n}_-)
\]
and
\[
u=(\mathfrak{n}_+^{-1}\otimes \mathfrak{n}_-^{-1})\circ coev'=-(\mathfrak{n}_-^{-1}\otimes \mathfrak{n}_+^{-1})\circ coev'_\tau.
\]
Because of~\eqref{eqn:nucapcup} and since $\mathcal{R}_2$ sends oriented caps and cups to one of the maps $ev'$, $ev'_\tau$, $coev'$, and $-coev'_\tau$, this shows that
\[
\mathfrak{n}_{m_2}\circ\mathcal{F}_o^{AKh}(\mathcal{I}(T))=\mathcal{V}(\mathcal{R}(T))\circ\mathfrak{n}_{m_1}
\]
whenever $T\colon m_1\rightarrow m_2$ is a flat tangle which contains at most one cap or cup. The theorem now follows because every flat tangle can be written as a composition of such tangles.
\end{proof}

%\todo[inline, color = red!40]{I added the following remark to clarify what the $\mathfrak{gl}(1|1)$-action on $\mathcal{F}^{AKh}_o(C,\mathcal{O})$ looks like. We can think about how to refer to this remark in other parts of the paper.}

\begin{remark}\label{rem:actiondescription}
Because $\mathcal{I}\colon\TL(0)^\oplus\rightarrow\BNA^\oplus$ is an equivalence of supercategories, Theorem~\ref{thm:commute} implies that $\mathcal{F}_o^{AKh}$ can be viewed as a functor with values in $\mathcal{R}\mathit{ep}(\mathfrak{gl}(1|1))$.
Hence each supermodule $\mathcal{F}_o^{AKh}(C,\mathcal{O})$ carries a natural $\mathfrak{gl}(1|1)$-action, which can be described as follows: if $C\subset\ann$ is a collection of $m$ essential circles, then $C$ is isotopic to $\mathcal{I}(m)$, and thus the supermodule
$\mathcal{F}_o^{AKh}(C,\mathcal{O})=\mathcal{F}_o^{AKh}(\mathcal{I}(m))=V^{\otimes m}$ is isomorphic to
\[
\mathcal{R}(m)=L\otimes (L^*[1])\otimes L\otimes (L^*[1])\otimes\ldots
\]
via the isomorphism $\mathfrak{n}_m$. Hence the $\mathfrak{gl}(1|1)$-action on $\mathcal{F}_o^{AKh}(C,\mathcal{O})$
is given by
$Xv:=\mathfrak{n}_m^{-1}(X\mathfrak{n}_m(v))$.
On the other hand, if $C\subset\ann$ is a trivial closed component, then $C$ is isomorphic to $\emptyset_0\oplus\emptyset_1=\mathcal{I}(0)_0\oplus\mathcal{I}(0)_1$, and thus $\mathcal{F}^{AKh}_o(C)=V$ is isomorphic to the representation
\[
\mathcal{R}(0)\oplus(\mathcal{R}(0)[1])=\Bbbk\oplus(\Bbbk[1]),
\]
on which the $\mathfrak{gl}(1|1)$-action is trivial. Generalizing these cases, one sees that if $(C,\mathcal{O})$ is any object of $\OBNA$ with components $C_1<\ldots<C_m$, then the action on the $i$th tensor factor of $\mathcal{F}^{AKh}_o(C,\mathcal{O})=V^{\otimes m}$ is
\begin{itemize}
\item
trivial if the component $C_i$ is trivial, 
\item
induced by the isomorphism $\mathfrak{n}_+\colon V\cong L$ if the component $C_i$ is essential and the number of essential components $C_j$ with $j<i$ is even,
\item
induced by the isomorphism $\mathfrak{n}_-\colon V\cong L^*[1]$ if the component $C_i$ is essential and the number of essential components $C_j$ with $j<i$ is odd.
\end{itemize}
\end{remark}

\begin{lemma}
The maps $n$ and $u$ from the proof of Theorem~\ref{thm:commute} satisfy
\begin{enumerate}
\item[\normalfont (a)] $(n\otimes\mathbbm{1}_V)\circ(\mathbbm{1}_V\otimes u)=\mathbbm{1}_V$.
\item[\normalfont (b)] $(\mathbbm{1}_V\otimes n)\circ(u\otimes\mathbbm{1}_V)=-\mathbbm{1}_V$.
\item[\normalfont (c)] $n\circ u=0$.
\end{enumerate}
\end{lemma}

\begin{proof}
This can be seen as a consequence of~\eqref{eqn:nucapcup} and of the fact that $\mathcal{I}$ is well-defined. Alternatively, (a) follows because
\begin{align*}
\bigl((n\otimes\mathbbm{1}_V)\circ(\mathbbm{1}_V\otimes u)\bigr)(v_\pm)&=\bigl((n\otimes\mathbbm{1}_V)\circ(\mathbbm{1}_V\otimes u)\bigr)(v_\pm\otimes 1)\\
&=\pm(n\otimes\mathbbm{1}_V)\bigl(v_\pm\otimes(v_-\otimes v_+-v_+\otimes v_-)\bigr)\\
&=\pm n(v_\pm\otimes v_-)v_+\mp n(v_\pm\otimes v_+)v_-\\
&=v_\pm,
\end{align*}
where we have used the explicit formulas~\eqref{eqn:nuformulas} for $u$ and $n$ and the fact that $u$ and $\mathbbm{1}_V$ have superdegree $1$ and $0$, respectively. Likewise, (b) follows because
\begin{align*}
\bigl((\mathbbm{1}_V\otimes n)\circ(u\otimes\mathbbm{1}_V)\bigr)(v_\pm)&=\bigl((\mathbbm{1}_V\otimes n)\circ(u\otimes\mathbbm{1}_V)\bigr)(1\otimes v_\pm)\\
&=(\mathbbm{1}_V\otimes n)\bigl((v_-\otimes v_+-v_+\otimes v_-)\otimes v_\pm\bigr)\\
&=-v_-n(v_+\otimes v_\pm)-v_+n(v_-\otimes v_\pm)\\
&=-v_\pm,
\end{align*}
where we have again used~\eqref{eqn:nuformulas} as well as the fact that $\mathbbm{1}_V$ and $n$ have superdegrees $0$ and $1$, respectively. Finally, (c) holds because
\[
n\circ u=ev'_\tau\circ coev'=ev_\tau\circ coev=ev\circ\tau_{L,L^*}\circ coev
\]
and
\[
\left(ev\circ\tau_{L,L^*}\circ coev\right)(1)=\operatorname{str}(\mathbbm{1}_L)=\operatorname{tr}(\mathbbm{1}_{\Bbbk})-\operatorname{tr}(\mathbbm{1}_{\Bbbk})=1-1=0,
\]
where we have used equation~\eqref{eqn:strtr}.
\end{proof}

\begin{proof}[Proof of Lemma~\ref{lem:R2welldef}]
The proof of Theorem~\ref{thm:commute} shows that the diagram
\[
\begin{tikzcd}[column sep=normal,row sep=large]
\TLor (0)\ar[r]\ar[d,"\mathcal{R}_2"']&\mathcal{T\!L}_o(0)\ar[r,"\mathcal{I}"]&
\mathcal{BBN}_o(\mathbb{A})\ar[d,"\mathcal{F}^{AKh}_o"]\\
\mathcal{R}\!\mathit{ep}(\mathfrak{gl}(1|1))\ar[rr]& &\mathcal{SM}\mathit{od}(\Bbbk)
\end{tikzcd}
\]
commutes up to even supernatural isomorphism, where the unlabeled arrows are the obvious forgetful functors. Because $\mathcal{I}$ and $\mathcal{F}^{AKh}_o$ are well-defined, and because the forgetful functor at the bottom is faithful, this shows that $\mathcal{R}_2$ is well-defined as well. Alternatively, $\mathcal{R}_2$ is well-defined because the previous lemma implies that it respects relations~\incg{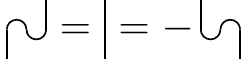} and $\bigcirc=0$.
\end{proof}

%\end{document}
    
\section{Odds and ends}\label{sec:Chapter8}
    %\documentclass[../main.tex]{subfiles}
%
%\begin{document}
\subsection{Connection with $\mathcal{T\!L}_e(0)$}
The usual even Temperley-Lieb category $\mathcal{T\!L}_e(0)$ at $\delta=0$ is defined in the same way as the odd Temperley-Lieb supercategory $\mathcal{T\!L}_o(0)$, except that all minus signs in the defining relations for $\mathcal{T\!L}_o(0)$ are replaced by plus signs, and the supergrading is discarded. Note that $\mathcal{T\!L}_e(0)$ is a monoidal category in the ordinary sense, rather than a monoidal supercategory.

\begin{proposition}\label{prop:eoTL}
There is an isomorphism of $\Bbbk$-linear categories $\mathcal{T\!L}_o(0)\cong\mathcal{T\!L}_e(0)$, which is given by the identity on objects.
\end{proposition}

\begin{proof}
Let $T\subset\mathbb{R}\times I$ be a flat chronological tangle representing a morphism in $\mathcal{T\!L}_o(0)$. Given a critical point $c\in T$ for the height function $T\rightarrow\mathbb{R}\times I\rightarrow I$, let $n(T,c)$ denote the number of points on $T$ that lie horizontally to the left $c$. Further, let
\[
a(T,c):= \begin{cases} 
          \biggl\lfloor\frac{n(T,\,c)+1}{2}\biggr\rfloor & \textnormal{if } c \textnormal{ is a local maximum}\\
          \biggl\lfloor\frac{n(T,\,c)}{2}\biggr\rfloor & \textnormal{if }c \textnormal{ is a local minimum}
       \end{cases}
\] 
Moreover, let
\[
a(T):=\sum_c a(T,c)
\]
where $c$ runs over all critical points on $T$. Now define $\mathcal{S}\colon\mathcal{T\!L}_o(0)\rightarrow\mathcal{T\!L}_e(0)$ by
\[
\mathcal{S}(T):=(-1)^{a(T)}T.
\]
We claim that this defines a well-defined functor.

First, note that $\mathcal{S}(T)$ is compatible with composition of tangles because
\[a(T\circ T')=a(T)+a(T'),\]
by definition of $a(T)$. Likewise, $a(T)=0$ if $T$ is an identity tangle. Next, suppose that each of the tangles $T_i$ in relation~\eqref{eqn:TLinterchange} contains a single critical point $c_i$. Let $T$ denote the tangle on the left-hand side of \eqref{eqn:TLinterchange} and $T'$ denote the tangle on the right-hand side. Then reversing the vertical order of $T_1$ and $T_2$ leaves $a(T,c_1)$ unchanged while changing $a(T,c_2)$ by $\pm 1$. Hence
\[
\mathcal{S}(T)=(-1)^{a(T)}T=(-1)^{a(T)}T'=-(-1)^{a(T')}T'=-\mathcal{S}(T'),
\]
where the second equation follows from the relation $T=T'$ which holds in $\mathcal{T\!L}_e(0)$, and the third relation follows because $a(T')=a(T)\pm 1$ by the preceding discussion. Hence $\mathcal{S}$ is compatible with the relation $T=-T'$ that holds in $\mathcal{T\!L}_o(0)$, at least in the case where $T_1$ and $T_2$ are as stated. The general case follows easily from this special case.

Finally, let $T$ be the first tangle that appears in relation~\eqref{eqn:TLisotopy} and $T'$ be the last tangle, and let $\mathbbm{1}$ denote the identity tangle that appears in the middle. Note that $T$ and $T'$ may contain vertical strands that are not shown in \eqref{eqn:TLisotopy}, and that lie to the left or to the right of the shown portions. In principle, there could also be nontrivial parts that lie above and below the shown portions, but we can ignore those parts since we have already shown that $\mathcal{S}$ is compatible with composition.

Now let $c_1$ and $c_2$ denote the local maximum and the local minimum in $T$, and $c_1'$ and $c_2'$ denote the local maximum and the local minimum in $T'$. It then follows from the definitions that $a(T,c_1)=a(T,c_2)$ but $a(T',c_1')=a(T',c_2')+1$. Hence $a(T)$ is even and $a(T')$ is odd, and therefore
\[
\mathcal{S}(T)=(-1)^{a(T)}T=T=\mathbbm{1}=T'=-(-1)^{a(T')}T'=-\mathcal{S}(T'),
\]
where the middle two equations hold because they hold in $\mathcal{T\!L}_e(0)$. Thus, the definition of $\mathcal{S}$ is compatible with \eqref{eqn:TLisotopy} as well.
\end{proof}

\begin{remark}
Note that the isomorphism from Proposition~\ref{prop:eoTL} does not preserve the tensor product of morphisms because in general the parity of $a(T\otimes T')$ differs from the parity of $a(T)+a(T')$.
\end{remark}

\begin{remark} If $2$ is invertible in $\Bbbk$, then the categories $\mathcal{T\!L}_e(0)$ and $\mathcal{T\!L}_e(2)$ are not equivalent as $\Bbbk$-linear categories. Indeed, the endomorphism ring of $n\in\mathcal{T\!L}_e(\delta)$ is the $n$th Temperley-Lieb algebra for $\bigcirc=\delta$, and the rank of this algebra (viewed as a $\Bbbk$-module) is independent of $\delta$ and increases with $n$. Therefore, any $\Bbbk$-linear equivalence between $\mathcal{T\!L}_e(0)$ and $\mathcal{T\!L}_e(2)$ would have to be the identity on objects. (For $n=0$ and $n=1$, the Temperley-Lieb algebras both have rank $1$, but this does not invalidate the argument because the even and odd integers $n\in\mathcal{T\!L}_e(\delta)$ generate two subcategories which have no nonzero morphisms between them.) It is now easy to see that the Temperley-Lieb algebras may depend nontrivially on the value of $\bigcirc=\delta$. For example, the Temperley-Lieb algebra for $n=2$ is isomorphic to $\Bbbk[z]/(z^2)$ if $\delta=0$, but isomorphic to $\Bbbk[z]/(z^2-2z)$ if $\delta=2$, where $z=\incg{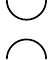}$.
\end{remark}

%\todo[inline, color = red!40]{Add comments about the universal theory.}

\subsection{Generalized Lie superalgebras}
One can define a notion of a \textbf{generalized Lie superalgebra} by replacing each factor of $-1$ that appears in Definition~\ref{def:Liesuperalgebra} by a formal variable $\pi$ with $\pi^2=1.$ For example, if $V$ is any supermodule over $\Bbbk[\pi]$, then $\mathfrak{gl}_\pi(V):=\operatorname{End}(V)$ is a generalized Lie superalgebra with generalized Lie superbracket
\[
[f,g]_\pi:=f\circ g-\pi^{|f||g|}g\circ f.
\]
Similarly, the set of all $2\times 2$ matrices with entries in $\Bbbk[\pi]$ constitutes a generalized Lie superalgebra $\mathfrak{gl}_\pi(1|1)$ with generalized Lie superbracket
\[
[A,B]_\pi:=AB-\pi^{|A||B|}BA,
\]
where the superdegrees of $A$ and $B$ are defined as in the definition of $\mathfrak{gl}(1|1)$.

A \textbf{representation} of a generalized Lie superalgebra $\mathfrak{g}$ is a supermodule $V$ together with an even linear map $\mathfrak{g}\rightarrow\mathfrak{gl}(V)$ which satisfies the condition in Definition~\ref{def:representation}, but for the generalized Lie superbrackets. For example, the identity map $\mathfrak{gl}(V)\rightarrow\mathfrak{gl}(V)$ is a representation of $\mathfrak{gl}(V)$, called the fundamental representation.

If $V$ and $W$ are two representations of the same generalized Lie superalgebra, then the set of $\operatorname{Hom}(V,W)$ of linear maps $f\colon V\rightarrow W$ is again representation with action given by
\[
(Xf)(v):=X(f(v))-\pi^{|X||f|}f(Xv).
\]
Likewise, the dual space $V^*:=\operatorname{Hom}(V,\Bbbk[\pi])$ is a representation with action given by
\[
(Xf)(v):=-\pi^{|X||f|}f(Xv)
\]
The representations of a generalized Lie superalgebra $\mathfrak{g}$ form a category $\mathcal{R}\mathit{ep}(\mathfrak{g})$, whose morphisms are given by linear maps $f\colon V\rightarrow W$ satisfying $Xf=0$ in $\operatorname{Hom}(V,W)$.
 
The results from section~\ref{sec:Chapter7} (with the exception of section~\ref{subs:repgl11}) now carry through in the generalized setting if one replaces all remaining factors of $-1$ by factors of $\pi$, and all categories and functors by their generalized versions. In particular, one can define a functor
\[
\mathcal{R}_\pi\colon\mathcal{T\!L}_{\pi}(1+\pi)\longrightarrow\mathcal{R}\mathit{ep}(\mathfrak{gl}_\pi(1|1))
\]
by sending caps and cups to maps of the form $ev'_\pi$, $ev'_{\tau,\pi}$, $coev'$, and $\pi coev'_{\tau,\pi}$, where now the definition of $\tau_{L,L^*}$ and $\tau_{L^*,L}$ involves a factor of $\pi$. After composing with a forgetful functor, the functor $\mathcal{R}_\pi$ becomes naturally isomorphic to the composition $\mathcal{F}^{AKh}_\pi\circ\mathcal{I}_\pi$, where $\mathcal{I}_\pi$ denotes the obvious generalization of $\mathcal{I}$, and denotes the functor $\mathcal{F}^{AKh}_\pi$ from Remark~\ref{rem:universalannularTQFT}. Note that the isomorphism $\mathfrak{n}_-\colon V\rightarrow L^*[1]$ from the proof of Theorem~\ref{thm:commute} has to be replaced by an iosmorphism which sends the basis $\{v_+,v_-\}$ to the basis $\{\pi v_1^*,v_0^*\}$. Moreover, the map $u$ from the proof of Theorem~\ref{thm:commute} has to be replaced by a map that sends $1$ to $v_-\otimes v_++\pi v_+\otimes v_-$.

\subsection{Action on odd annular Khovanov homology}

In~\cite{PutyraChrono}, Putyra defined an odd version of the formal Khovanov bracket~\cite{BN} for link diagrams in the plane $\mathbb{R}^2$. By repeating his construction in the annular setting, one can assign to each oriented annular link diagram $\mathcal{D}\subset\ann$ a formal chain complex $[\![\mathcal{D}]\!]$, which lives in the odd annular Bar-Natan category. Applying the odd annular Khovanov TQFT to this chain complex yields a chain complex
\[
\operatorname{ACKh}_1(\mathcal{D}):=\mathcal{F}_o^{AKh}\bigl([\![\mathcal{D}]\!]\bigr)
\]
in the category of bigraded supermodules, where the bigrading comes from the quantum grading and the annular grading. The complex $\operatorname{ACKh}_1(\mathcal{D})$ can be seen as an odd version of the annular Khovanov complex defined in~\cite{APS,LRoberts}, and in view of diagram~\eqref{eqn:commute}, it can be `lifted' to a chain complex in the representation category of $\mathfrak{gl}(1|1)$. Explicitly, this means that its chain groups carry natural $\mathfrak{gl}(1|1)$-actions, which supercommute with the differentials.

A slightly different $\mathfrak{gl}(1|1)$-action on a chain complex $\operatorname{ACKh}_2(\mathcal{D})$ was constructed in~\cite{GL11Grigsby}. Unlike the former action, this latter action properly commutes with the differentials. We now aim to prove:

\begin{theorem}\label{thm:actioncomparison}
Let $\mathcal{D}\subset\ann$ be an oriented link diagram. Then there is a chain isomorphism $\operatorname{ACKh}_1(\mathcal{D})\cong \operatorname{ACKh}_2(\mathcal{D})$ which supercommutes with the $\mathfrak{gl}(1|1)$-actions.
\end{theorem}

Before proving this result, we will briefly recall the construction of the involved chain complexes.

Let $\mathcal{D}\subset\ann$ be an oriented link diagram, and let $\chi$ denote the set of its crossings. Each crossing $x$ of $\mathcal{D}$ can be resolved in two possible ways, called its $0$-resolution and its $1$-resolution (see Figure~\ref{fig:01resolution}). To a function $I\colon\chi\rightarrow\{0,1\}$, we can thus assign an (unoriented) crossingless link diagram $\mathcal{D}(I)\subset\ann$ by replacing each crossing $x$ of $\mathcal{D}$ by its resolution of type $I(x)$.

\begin{figure}
    \centering
    \includegraphics[scale=1.4]{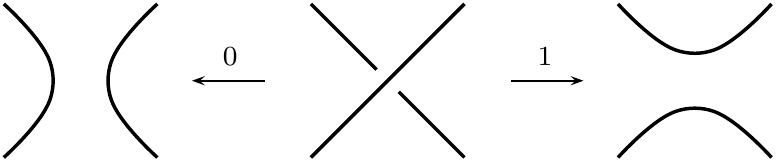}
    \caption{$0$-resolution and $1$-resolution of a crossing. The convention is such that, if one approaches the crossing along the overpass, then the \mbox{$0$-resolution} branches off to the left whereas the $1$-resolution branches off to the right.}
    \label{fig:01resolution}
\end{figure}

In what follows, it will be convenient to identify maps $I\colon\chi\rightarrow\{0,1\}$ with vertices of the hypercube $[0,1]^\chi\cong[0,1]^n$, where $n$ denotes the number of crossings of $\mathcal{D}$. We will further assume that the edges of the hypercube are oriented in direction of increasing $|I|$, where $|I|$ denotes the number of $x\in\chi$ with $I(x)=1$. If two vertices $I$ and $I'$ of the hypercube are connected by an oriented edge $I\rightarrow I'$, it then follows that the associated diagrams $\mathcal{D}(I)$ and $\mathcal{D}(I')$ differ at a single crossing $x$, where $\mathcal{D}(I)$ has a $0$-resolution and $\mathcal{D}(I')$ has a $1$-resolution. In particular, this means that there is an elementary merge or split cobordism $S_{I'I}\colon\mathcal{D}(I)\rightarrow\mathcal{D}(I')$, which has a saddle point near $x$, and which is given by an identity cobordism everywhere else.

We now aim to interpret the diagrams $\mathcal{D}(I)$ and the cobordisms  $S_{I'I}$ as objects and morphisms in the ordered odd annular Bar-Natan category. To this end, we equip each $\mathcal{D}(I)$ with an admissible ordering of its components, and we choose for each crossing of $\mathcal{D}$ an arrow connecting the two strands in its $0$-resolution. The chosen arrows induce orientations on the critical points in the cobordisms $S_{I'I}\subset\ann\times I$, which thus become chronological cobordisms.

Let $n_-$ denote the number of negative crossings in $\mathcal{D}$, and let $m(I)$ denote the number of components in $\mathcal{D}(I)$. To construct the chain complex $\operatorname{ACKh}_1(\mathcal{D})$, we replace each vertex $I$ of the hypercube $[0,1]^\chi$ by the supermodule
\[
F_1(I):=\mathcal{F}^{AKh}_o(\mathcal{D}(I))=V^{\otimes m(I)},
\]
where $V=\Bbbk v_+\oplus\Bbbk v_-$ denotes the supermodule from section~\ref{subs:nonannularTQFT}. We then define
\begin{equation}\label{eqn:ACKh1}
\operatorname{ACKh}_1(\mathcal{D})^i:=\bigoplus_{|I|=i+n_-}F_1(I),
\end{equation}
where the sum is over all vertices with $|I|=i+n_-$. The differential
\[\partial^i\colon\operatorname{ACKh}_1(\mathcal{D})^i\longrightarrow\operatorname{ACKh}_1(\mathcal{D})^{i+1}
\]
is given by a signed sum of maps $\mathcal{F}_o^{AKh}(S_{I'I})\colon F_1(I)\rightarrow F_1(I')$, where the sum ranges over all edges $I\rightarrow I'$ with $|I|=i+n_-$, and where the signs are chosen so that $\partial_i\circ\partial_{i+1}=0$ (see~\cite{ORS} for more details).

Note that since each $S_{I'I}$ is an elementary merge or split cobordism, the $\mathcal{F}^{AKh}_o(S_{I'I})$ are defined via the maps $\mu$ and $\Delta$ from section~\ref{subs:nonannularTQFT}. More precisely, each $\mathcal{F}^{AKh}_o(S_{I'I})$ can be written as a composition of maps of the form
$\mathbbm{1}_V\otimes\ldots\otimes \tau\otimes\ldots\otimes\mathbbm{1}_V$ with a map of the form $\mathbbm{1}_V\otimes\ldots\otimes\mu_0\otimes\ldots\otimes\mathbbm{1}_V$ or $\mathbbm{1}_V\otimes\ldots\otimes\Delta_0\otimes\ldots\otimes\mathbbm{1}_V$, where here $\mu_0$ and $\Delta_0$ denote the appropriate components of the maps $\mu$ and $\Delta$, respectively (cf.~section~\ref{subs:annularTQFT}). We emphasize that the differential in $\operatorname{ACKh}_1(\mathcal{D})$ does not preserve the supergrading because $\Delta$ has superdegree $1$.

The chain complex $\operatorname{ACKh}_2(\mathcal{D})$ from~\cite{GL11Grigsby} is defined in nearly same way as $\operatorname{ACKh}_1(\mathcal{D})$, except that in formula~\eqref{eqn:ACKh1}, the supermodule $F_1(I)$ is replaced by the shifted supermodule $F_2(I):=F_1(I)[s(I)]$ where
\[
s(I):=\frac{m(I)+|I|+n-3n_--|L|}{2}
\]
and where $|L|$ denotes the number of components of the annular link represented by $\mathcal{D}$, and $n$ denotes the number of crossings of $\mathcal{D}$. We leave it to the reader to verify that $s(I)$ is always an integer, and that it satisfies $s(I')=s(I)$ whenever $I$ and $I'$ are connected by an edge corresponding to a merge cobordism, and $s(I')=s(I)+1$ whenever $I$ and $I'$ are connected by an edge corresponding to a split cobordism. In particular, this implies that the differential in $\operatorname{ACKh}_2(\mathcal{D})$ (which is identical with the one in $\operatorname{ACKh}_1(\mathcal{D})$) preserves the supergrading in $\operatorname{ACKh}_2(\mathcal{D})$.

Because of diagram~\eqref{eqn:commute}, the supermodules $F_1(I)=\mathcal{F}_o^{AKh}(\mathcal{D}(I))$ can be interpreted as $\mathfrak{gl}(1|1)$-representations, and each map $\mathcal{F}_o^{AKh}(S_{I'I})$ can be viewed as a morphism in the representation category of $\mathfrak{gl}(1|1)$ (cf. Remark~\ref{rem:actiondescription}). Thus, the chain complex $\operatorname{ACKh}_1(\mathcal{D})$ carries a natural $\mathfrak{gl}(1|1)$-action, which supercommutes with the differentials. On $\operatorname{ACKh}_2(\mathcal{D})$, a similar action was defined in~\cite{GL11Grigsby}. The latter action is almost identical with the one on $\operatorname{ACKh}_1(\mathcal{D})$, but instead of formula~\eqref{eqn:nfoldtensor}, the following formula is used to define the action on the tensor product $F_2(I)=V^{\otimes m(I)}[s(I)]$:
\[\label{eqn:nfoldtensormodified}
X(v_1\otimes\cdots\otimes v_{m(I)}):=\sum_{i=1}^{m(I)}(-1)^{|X|(|v_{i+1}|+\ldots+|v_{m(I)}|)}v_1\otimes\cdots\otimes(Xv_i)\otimes\cdots\otimes v_{m(I)}
\]
(note that the grading-shift in $V^{\otimes m(I)}[s(I)]$ has no bearing on the action). It was shown in~\cite{GL11Grigsby} that the resulting action on $\operatorname{ACKh}_2(\mathcal{D})$ properly commutes with the differentials. Since, in $\operatorname{ACKh}_2(\mathcal{D})$, the differentials have superdegree $0$, this means that this action also supercommutes with the differentials.

We are now ready to prove Theorem~\ref{thm:actioncomparison}.

\begin{proof}[Proof of Theorem~\ref{thm:actioncomparison}]
For a vertex $I$ of the hypercube $[0,1]^\chi$, consider the representation
\[
F_3(I):=F_1(I)[s(I)],
\]
whose $\mathfrak{gl}(1|1)$-action agrees with the one on $F_1(I)$, but whose supergrading agrees with the one $F_2(I)$. Define a map of representations $f_I\colon F_1(I)\rightarrow F_3(I)$ by
\[
f_I:=\begin{cases}
\mathbbm{1}_{F_1(I)}&\mbox{if $s(I)$ is even,}\\
s_{F_1(I)}&\mbox{if $s(I)$ is odd},
\end{cases}
\]
where $s_{F_1(I)}$ is defined as in section~\ref{subs:gradingshifts}. Moreover, define $g_I\colon F_3(I)\rightarrow F_2(I)$ by
\[
g_I(v_1\otimes\cdots\otimes v_{m(I)}):=(-1)^{\sum_{j<k}|v_j||v_k|}v_1\otimes\cdots\otimes v_{m(I)},
\]
where $|v_j|$ denotes the superdegree of $v_j\in V$. It is easy to see that $g_I$ intertwines the $\mathfrak{gl}(1|1)$-actions on $F_3(I)$ and on $F_2(I)$, which are given respectively by~\eqref{eqn:nfoldtensor} and~\eqref{eqn:nfoldtensormodified}. Finally, let $h_I\colon F_1(I)\rightarrow F_2(I)$ be the map
\[
h_I:=(-1)^{\bigl\lfloor\frac{s(I)+1}{2}\bigr\rfloor}g_I\circ f_I.
\]
We leave it to the reader to verify that $h_I$ satisfies
\[
h_{I'}\circ\mathcal{F}^{AKh}_o(S_{I'I})=\mathcal{F}^{AKh}_o(S_{I'I})\circ h_I
\]
for every edge $I\rightarrow I'$ in the hypercube $[0,1]^\chi$ (the proof uses the explicit definition of the edge maps $\mathcal{F}^{AKh}_o(S_{I'I})$). In conclusion, we see that the maps $h_I$ define a chain isomorphism $\operatorname{ACKh}_1(\mathcal{D})\cong\operatorname{ACKh}_2(\mathcal{D})$ which supercommutes with the $\mathfrak{gl}(1|1)$-actions.
\end{proof}
%\end{document}

\bibliography{bibliography.bib}

\end{document}